\documentclass{memo-l}
\pdfoutput=1 
\usepackage[T1]{fontenc}
\usepackage{amsmath,amssymb,amsfonts}
\usepackage{hyperref}
\usepackage{xcolor}
\usepackage{mathtools}
\usepackage{esint}
\usepackage{bbm}

\newtheorem*{theorem*}{Theorem}
\newtheorem*{lemma*}{Lemma}
\newtheorem*{corollary*}{Corollary}

\newtheorem{theorem}{Theorem}[chapter]
\newtheorem{lemma}[theorem]{Lemma}
\newtheorem{corollary}[theorem]{Corollary}

\theoremstyle{definition}
\newtheorem{definition}[theorem]{Definition}

\theoremstyle{remark}
\newtheorem{remark}[theorem]{Remark}

\numberwithin{section}{chapter}
\numberwithin{equation}{chapter}

\makeindex

\begin{document} 

\frontmatter

\title[Hodge Decomposition and Potentials in Variable Exponent Spaces]{Hodge Decomposition and Potentials in Variable Exponent Lebesgue and Sobolev Spaces}

\author{Anna~Balci}
\address[Anna~Balci]{Charles University  in Prague, Department of Mathematical Analysis,
	Sokolovsk\'a 49/83, 186 75 Praha 8, Czech Republic and University of Bielefeld, Department of Mathematics,  Postfach 10 01 31, 33501 Bielefeld, Germany.}
\email{akhripun@math.uni-bielefeld.de}
\thanks{Anna Kh. Balci's research is funded  by the Deutsche Forschungsgemeinschaft (DFG, German Research Foundation) - SFB 1283/2 2021 - 317210226 and   by Charles University  PRIMUS/24/SCI/020 and Research Centre program No. UNCE/24/SCI/005.}

\author{Swarnendu Sil}
\address[Swarnendu Sil]{Department of Mathematics, Indian Institute of Science, Bengaluru 560012, Karnataka, India}
\email{swarnendusil@iisc.ac.in }
\thanks{Swarnendu Sil's research is supported by SERB MATRICS MTR/2023/000885.}

\author{Mikhail Surnachev}
\address[Mikhail Surnachev]{Keldysh Institute of Applied Mathematics, Miusskaya sq., 4 Moscow, 125047, Russia}
\email{peitsche@yandex.ru}

\thanks{Mikhail Surnachev's research is supported by  Moscow Center of Fundamental and Applied Mathematics under Agreement with the Ministry of Science and Higher Education of the Russian Federation.\iffalse, No. 075-15-2022-283\fi}

\date{}




\begin{abstract}
		The objective of this work is to establish a systematic study of boundary value problems within the framework of differential forms and variable exponent spaces. Specifically, we investigate the Hodge Laplacian and related first order systems like the div-curl systems, Hodge-Dirac systems, and Bogovskii-type problems in the context of variable exponent spaces. Our approach yields both existence theorems and elliptic estimates. These estimates provide key results such as the Hodge decomposition theorem, Gaffney inequality, and gauge fixing. These findings are crucial for advancing the nonlinear theory related to problems involving differential forms.  \end{abstract}
\maketitle
\tableofcontents

\mainmatter

\chapter{Introduction}
\section{Boundary value problems for differential forms}	In this paper we present a set of results on regularity theory in Lebesgue and Sobolev spaces with variable exponents for classical linear problems with differential forms on a compact smooth Riemannian manifold $M$. We treat boundary value problems for 
	\begin{enumerate}
		\item the Hodge Laplacian;
		\item first order ``div-curl'' systems;
		\item Hodge-Dirac systems;
		\item ``Bogovskii'' type problems. 
	\end{enumerate}
	As a corollary, we obtain results on cohomology resolution, the Hodge decomposition, Gaffney's inequality and gauge fixing in variable exponent spaces.

	The classical results on differential forms are collected, for example, in the books by de Rham~\cite{deRham}, Warner~\cite{Warner},   H.~Cartan~\cite{Car70}, Federer~\cite{Fed69}, M.~Spivak~\cite{Spivak, Spi_I,Spi_IV}, Bott and Tu~\cite{BotTu82}, V.I.~Arnold~\cite{Arn89}, H.~Flanders~\cite{Fl89},~R. Abraham, J. E.  Marsden  and T. Ratiu  in~\cite{AbrMarRat}; D.~Lovelock and H.~Rund\cite{LovRun75}. 
    
    More recent results on Sobolev spaces of differential forms, Gaffney inequalities, and problems of Hodge theory include works by Iwaniec and Lutoborski~\cite{IwaLut93}, Iwaniec and Martin~\cite{IwaMar93}, Scott~\cite{Sco95}, Iwaniec, Scott, and Stroffolini~\cite{IwaScoStr99}, Schwarz~\cite{Schwarz}, Mitrea, Mitrea and Shaw~\cite{MMS08}. Recent monographs summarizing these developments include the books by D.~Mitrea, I.~Mitrea, M. Mitrea, and M. Taylor~\cite{MitMitMitTay16},  by  Csato, Dacorogna, and Kneuss~\cite{CsaGyuDar12} and by Agarwal,  Ding, and Nolder~\cite{AgaDinNol09}. 
    
    Modern considerations in the framework of Calculus of Variations for differential forms are studied in Bandopadhyay, Dacorgna, and Sil~\cite{BandDacSil}, Bandopadhyay and Sil~\cite{BandSil_extint} and ~\cite{BandSil_classical}, Sil~\cite{Sil19}, and Sil's PhD thesis~\cite{SilT16}. 
	
	The history of the question goes back to \cite{Hodge34} where W.V.D. Hodge studied a variational method for certain boundary value problems for forms defined on domains in Euclidean space. Weyl in~\cite{Wey40} obtained the famous decomposition results. An important result known as the Gaffney inequality was presented in \cite{Gaf51}, see also \cite{Gaf54} and \cite{Gaf55} and Friedrichs~\cite{Fri55}. The variational method was applied to general compact Riemannian manifolds without the boundary in~\cite{MorEel55} and to such manifolds with the boundary in~\cite{Mor56}. The results obtained by the variational approach and the corresponding regularity theory were summarized in \cite[Chapter 7]{Morrey1966}, and we mostly follow this classical source. 
    
    A different version of the variational approach, based on the generalized Lax-Milgram theorem in Banach spaces \cite{KozYan13} was used to obtain the {$L^r$}-{H}elmholtz-{W}eyl decomposition for three dimensional exterior domains in \cite{HKSSY}.
    
    A parallel approach was presented by the integral equation method, initially developed by Kodaira~\cite{Kod49} (on analytic closed complex manifolds), Bidal, and de Rham~\cite{BiddeRha46}. For manifolds with boundaries, this method was adapted by Duff and Spencer~\cite{DufSpe51,DufSpe56}, by reducing the problem to boundary integral equations. This approach was utilized by Kress in \cite{Kre70, Kress} for problems in domains in Euclidean space, and later by Bolik in \cite{bolikphd,Bolik97,Bolik01, Bolik04, Bolik07}, first for domains in Euclidean space and later for compact Riemannian manifolds. More recent additions by this method are covered in the monographs by Mitrea, Mitrea, Taylor~\cite{MitMitTay01} and Mitrea, Mitrea, Mitrea, Taylor~\cite{MitMitMitTay16}. Estimates of the gradient of vector fields through div and curl in variable exponent {Sobolev} spaces in 3 dimensions were studied  in \cite{SinRi2020}. 

    Yet another famous technique was presented by Milgram and Rosenbloom~\cite{MilRos51}, who used an analogue of the heat equation on manifolds for differential forms; a similar result is also contained in~\cite{Gaf54_1}. 
    
    Schwarz in \cite{Schwarz} studied the Hodge decomposition of differential forms in manifolds with boundary using the theory of elliptic operators in vector bundles. More or less, in the classical spaces the regularity properties for the classical boundary value problems with differential forms may be treated as a particular case of the general elliptic theory of Agmon, Douglis and Nirenberg, see~\cite{AgmDouNir59} and Solonnikov~\cite{SolI, SolII}, more precisely, the theory of elliptic operators on vector bundles (see for instance \cite{Hor85II, Hor85III}) and the theory of elliptic complexes, see ~\cite{RemSch82}. However, these classical methods need the manifold to be at least $C^{2,1}$, and quite often (for instance, in \cite{Schwarz}) only smooth manifolds are considered. 
    
    If methods of calculus of pseudodifferential operators are used the usual assumption is also smoothness of the domain.  A variant using calculus of pseudodifferential operators was implemented in \cite{FKP98}, and even for the Euclidian case and Hilbert space setting of $W^{1,2}$ the calculations are heavy (which is expected from the PDO calculus in application to boundary value problems) and smooth boundary is required.

    Another approach to regularity results for more general Hodge-type systems, based on Campanato-type methods, which does not depend on representation formulas and potential theory,  was initiated in Sil~\cite{Sil17} and was further pursued in recent works Sengupta and Sil~\cite{SenguptaSil_MorreyLorentz_Hodge} and Kumar and Sil~\cite{KumarSil24}.
        
   The main aim of our present work is twofold. The first is to extend these classical results to the setting of variable exponent spaces and the second is to cover the case of $C^{1,1}$ manifolds. For non-standard spaces one has to develop this theory from the very foundations, though naturally it follows via the well-established route. We follow mainly the variational approach of \cite{Morrey1966}. We remark, though, that generally the methods of potential theory (see \cite{MitMitTay01}, \cite{MitMitMitTay16}) allow for less regular  boundaries, at least in the Euclidian case. However, in the variable exponent setting even the structure of the trace spaces is still not completely understood, so we rely on the combination of the variational approach and a priori estimates.

\section{Variable exponent spaces and the log-H\"older condition} 
In this paper $M$ is a finite dimensional compact Riemannian manifold with boundary, which is not necessarily orientable, with appropriate smoothness. We work with variable exponent Lebesgue spaces $L^{p(\cdot)}(M)$ and the Sobolev spaces $W^{k,p(\cdot)}(M)$, $k\in \mathbb{N}$ on $M$ and the corresponding spaces for  differential forms $L^{p(\cdot)}(M,\Lambda)$, $W^{k,p(\cdot)}(M,\Lambda)$, (see Chapter~\ref{sec:LebSob} for the notations) with the variable exponent bounded away from $1$ and $\infty$ and satisfying the log-H\"older condition:
\begin{align}\label{eq:p1}
	1<p_{-} \leq p(x)\leq p_{+} <\infty, \quad x\in M,\\
	|p(x)-p(y)| \leq \frac{L}{ \log (e + \operatorname{dist}(x,y)^{-1})}, \quad x,y\in M, \label{eq:p2}
\end{align}
where $\operatorname{dist}(x,y)$ denotes the Riemannian distance between $x,y\in M$. 
This condition was introduced by Zhikov in \cite{Zhi95} as a sufficient condition for density of smooth functions in variable exponent Sobolev spaces. As is well-known, in the absence of this condition, smooth functions may be not dense in the variable exponent Sobolev space \cite{Zhi86,Zhi95}. Another notable feature which may occur for variable exponents with lower regularity is the Lavrentiev phenomenon, which is that the infimum of a variational problem over the natural domain of the integral functional can be strictly less than then infimum over smooth functions. Corresponding examples and discussions may be found in \cite{Zhi86,Zhi95},\cite{FonMalMin04}, \cite{BalDieSur20}, \cite{BalSur21}, \cite{BalDieSur23_Arx} (for scalar problems) and in \cite{BalSur24} for spaces of differential forms. Due to the log-H\"older-condition, in the setting of this paper these effects do not arise.

The log-H\"older condition also guarantees that many properties can transferred from constant exponent Lebesgue spaces to variable exponent case. Majority of the theory developed so far is under this condition, see ~\cite{Zhi11,DieHHR11,CruFio13,KokMesRafSam16}. In particular, the key ingredient which allows us to extend the  classical regularity results of~\cite{Mor56} to variable exponent spaces are boundedness estimates for classical integral operators in the whole space and the half-space obtained by Diening and Růžička in~\cite{Dieruz03_1}, \cite{DieRuz03}, \cite{DieRuz04}. These works established estimates in variable exponent spaces for the Dirichlet and Neumann problem for the Poisson problem and also the Stokes system. However, the case of differential forms (which are not $0$-forms, i.e. functions) and the special nature of the boundary conditions in our case necessitates further modifications and adaptations.  

\section{Results}	
	Let $M$ be a $n$-dimensional  compact Riemannian $C^{1,1}$ manifold with boundary $bM$, not necessarily orientable. The notation used in this Section is more or less classical and self-evident. However, it will be detailed in the next section. In particular, $\mathcal{H}_T(M)$ stands for the space of harmonic fields on $M$ with vanishing tangential component, and in the following  we denote
\begin{align*}
	(f,g) &= \int\limits_M f\wedge \ast g = \int\limits_M \langle f,g\rangle \, dV,\\
	[f,g]&:= \int \limits_{b M} \langle\nu \wedge f, g \rangle\, d\sigma = \int\limits_{bM} tf\wedge \ast ng = \int\limits_{bM} \langle f, \nu \lrcorner g \rangle\, d\sigma.
\end{align*}

Now we describe a sample of our main results. For convenience, the first three results are stated for tangential boundary condition and in terms of existence of \emph{a solution satisfying the given estimate}. However, for all the three results, \emph{any} weak solution of the system can be obtained from the solution given by the results by adding a tangential harmonic fields. Since tangential harmonic fields are as smooth as the regularity of the manifold $M$ allows, the same regularity conclusion holds for any weak solution with an estimate which has an additional term $\left\lVert \omega \right\rVert_{L^{1}\left( \Omega, \Lambda  \right) }$ on the right hand side.  Also, by Hodge duality, all three results have their analogous versions for normal boundary conditions as well and similar remarks are valid for \emph{any} weak solutions for those systems as well, due to the regularity of normal harmonic fields. 

Our first result extends the standard elliptic theory for the Hodge Laplacian to variable exponent setting. By $c_{\mathrm{log}}(p)$ we denote the minimal constant $L$ such that \eqref{eq:p2} holds:
$$
c_{\mathrm{log}}(p) = \max_{x,y\in M} |p(x)-p(y)|\cdot\log (e + \mathrm{dist}(x,y)^{-1}).
$$
Let $s\in \{0\}\cup \mathbb{N}$  and we denote $\texttt{data} :=( p_{-}, p_{+}, c_{\mathrm{log}} (p), M )$. 

\begin{theorem*}[Main estimate for the Hodge Laplacian] 
		Consider the boundary value problem 
        	\begin{align}\label{eq:hti}
		\triangle \omega = \eta  \quad \text{on}\ M, \quad t\omega = t\varphi, \quad t\delta \omega = t\psi \quad\text{on}\ bM.   
	\end{align}
		\begin{enumerate}
			\item Let $M$ be $C^{1,1}$ and let  $\eta \in L^{p(\cdot)}(M,\Lambda)$, and $\varphi,\psi \in W^{1,p(\cdot)}(M,\Lambda)$ satisfy $(\eta,h_T) = [\psi,h_T]$ for all $h_T \in \mathcal{H}_T(M).$ Then there exists a weak solution $\omega \in \varphi+W^{1,p(\cdot)}_T(M,\Lambda)$ of \eqref{eq:hti}, such that for some constant $C = C(\mathrm{data})>0,$ we have 
			$$
			\|\omega\|_{1,p(\cdot),M} \leq C ( \|\eta-d\psi\|_{p(\cdot),M} + \|\varphi\|_{1,p(\cdot),M} + \|\psi\|_{p(\cdot),M}). 
			$$
             Moreover, if $d\varphi \in W^{1,p(\cdot)}(M,\Lambda)$ then $\alpha = \delta \omega$ and $\beta = d \omega$ belong to $W^{1,p(\cdot)}(M,\Lambda)$, satisfy 
            \begin{equation}\label{eq:relAB_intro}
            \begin{gathered}
            d\alpha + \delta \beta = \eta, \quad t\alpha= t\psi, \quad t\beta = td\varphi,\\
            (d\alpha,d\zeta) = (\eta,d\zeta), \quad (\delta\beta,\delta\zeta) = (\eta-d\psi,\delta \zeta)
            \end{gathered}
            \end{equation}
            for all $\zeta \in \mathrm{Lip}_T(M,\Lambda)$ and there holds
            \begin{equation}\label{eq:estAB_intro}
            \begin{gathered}
            \|\alpha\|_{1,p(\cdot),M} \leq  C(\mathrm{data}) (\|\eta-d\psi\|_{p(\cdot),M}+ \|\psi\|_{1,p(\cdot),M)}),\\  
            \|\beta\|_{1,p(\cdot),M} \leq  C(\mathrm{data})  (\|\eta - d \psi\|_{p(\cdot),M} + \|d\varphi\|_{1,p(\cdot),M}) .
            \end{gathered}
            \end{equation}
                 
			\item Let $M$ be of the class $C^{s+2,1}$. Let $\eta\in W^{s,p(\cdot)}(M,\Lambda)$, $\varphi\in W^{s+2,p(\cdot)}(M,\Lambda)$, and $\psi \in W^{s+1,p(\cdot)}(M,\Lambda)$ be such that $(\eta,h_T) = [\psi,h_T]$ for all $h_T \in \mathcal{H}_T(M)$. Then there exists a solution $\omega\in W^{s+2,p(\cdot)}(M,\Lambda)$ of the boundary value problem \eqref{eq:hti}
			such that for some constant $C = C(\mathrm{data},s)>0,$ we have 
			$$
			\|\omega\|_{s+2,p(\cdot),M} \leq C ( \|\eta\|_{s,p(\cdot),M} + \|\varphi\|_{s+2,p(\cdot),M} + \|\psi\|_{s+1,p(\cdot),M}).
			$$
		  Moreover, the potentials $\alpha = \delta\omega$, $\beta = d \omega$ satisfy \eqref{eq:relAB_intro} and there holds 
                      \begin{equation*}
            \begin{gathered}
            \|\alpha\|_{s+1,p(\cdot),M} \leq  C(\mathrm{data},s) (\|\eta-d\psi\|_{s,p(\cdot),M} +\|\psi\|_{s+1,p(\cdot),M}),\\
            \|\beta\|_{s+1,p(\cdot),M} \leq C(\mathrm{data},s) ( \|\eta - d \psi\|_{s,p(\cdot),M}+\|d\varphi\|_{s+1,p(\cdot),M}).
            \end{gathered}
            \end{equation*}
				\end{enumerate}
	\end{theorem*}

Clearly, for $\omega$ to be $W^{1, p(\cdot)}$, the requirement that $\eta \in L^{p(\cdot)}$ is suboptimal and excessive, but for applications to div-curl type problems we are more interested in the behavior of $d\omega$ and $\delta \omega$.  See Theorem~\ref{T:D3} for a proof. Compare Theorem 16 in \cite{Sil17} for the constant exponent case. This result and its counterpart for the Neumann boundary conditions imply the following Hodge decomposition theorem in variable exponent Lebesgue spaces. 
\begin{theorem*}[Hodge decomposition]
		Let $\omega\in W^{s,p(\cdot)}(M,\Lambda)$. Then there exist $\alpha,\beta \in W^{s+1,p(\cdot)}(M,\Lambda)$ and $h\in W^{s,p(\cdot)}(M,\Lambda) \cap \mathcal{H}(M)$ such that 
		\begin{gather*}
			\omega = h + d\alpha + \delta \beta,\\
			t\alpha =0, \quad \delta \alpha =0,\quad n\beta =0, \quad d \beta =0,\\
			\|h\|_{s,p(\cdot),M},  \|\alpha\|_{s+1,p(\cdot),M}, \|\beta\|_{s+1,p(\cdot),M} \leq C(\mathrm{data},s) \|\omega\|_{s,p(\cdot),M}.
		\end{gather*}
	\end{theorem*}
Here $\mathcal{H}(M)$ is the set of harmonic fields on $M$ (without any boundary conditions). See Section~\ref{sec:Hodge} for the corresponding results.
    
Another important corollary from the main theorem is the following result for the so-called `div-curl' or Cauchy-Riemann systems. 
\begin{theorem*}[Solvability of first order systems]
	Let $M$ be of class $C^{s+1,1}$ and let $f\in W^{s,p(\cdot)}(M,\Lambda)$, $df=0$;  $v\in W^{s,p(\cdot)}(M,\Lambda)$, $\delta v=0$, $\varphi \in W^{s+1,p(\cdot)}(M,\Lambda)$ satisfy $(v,h_T)=0$ and $(f,h_T) =[\varphi,h_T]$ for all $h_T \in \mathcal{H}_T(M)$; and $t(f-d\varphi)=0$. Then there exists a  solution $\omega \in W^{s+1,p(\cdot)}(M,\Lambda)$ of the boundary value problem
	\begin{align}\label{divcurl tangential}
		d\omega = f , \quad \delta \omega =v \quad\text{on}\ M, \quad t\omega = t\varphi \quad\text{on}\ bM,
	\end{align}
	such that that for some $C = C(\mathrm{data},s)>0$, we have
	$$\|\omega\|_{W^{s+1,p(\cdot)}(M,\Lambda)} \leq C ( \|f\|_{W^{s,p(\cdot)}(M,\Lambda)} + \|v\|_{W^{s,p(\cdot)}(M,\Lambda)} + \|\varphi\|_{W^{s+1,p(\cdot)}(M,\Lambda)}).$$ 
\end{theorem*}
See Theorem \ref{L:sysD} for a proof. Compare Theorem 14 in \cite{Sil17} for the constant exponent case, where the method needs $M$ to be $C^{s+2,1}$. This also yields estimates for the so-called Hodge-Dirac operator $D = d+\alpha \delta$, $\alpha \in \mathbb{R}\setminus \{0\}$. 
\begin{theorem*}[Hodge-Dirac system]
	Let $M$ be $C^{s+1,1}$  and let $f\in W^{s,p(\cdot)}(M,\Lambda)$ be such that $(f,h_T) =0$ for all $h_T \in \mathcal{H}_T(M)$ and let $\alpha \neq 0$ be a real number. Then there exists $\omega \in  W^{s+1,p(\cdot)}(M,\Lambda)\cap W^{1,p(\cdot)}_T(M,\Lambda)$, satisfying 
	\begin{align}\label{hodge direc tangential}
		D\omega:= d\omega + \alpha\delta \omega =f \text{ in } M,\qquad t\omega =0 \text{ on } bM,   
	\end{align} such that  $(\omega, \mathcal{H}_T) =0$ and for some $C=C(\mathrm{data}, \alpha,s)>0,$ we have the estimate 
	\begin{align*}
		\|\omega\|_{W^{s+1,p(\cdot)}(M,\Lambda)} \leq C \|f\|_{W^{s,p(\cdot)}(M,\Lambda)}.
	\end{align*}
\end{theorem*}
See Theorem \ref{Theorem on Hodge-Dirac tangential} for a proof.
The last result concerns non-elliptic first-order systems which correspond to a Poincar\'{e} type lemma with Dirichlet data and are often called Bogovskii type problems (see ~\cite{Bogovski1}, 
~\cite{Bogovski2}).
\begin{theorem*}[Poincar\'{e} lemma with full Dirichlet data]
	Let $M$ be $C^{s+1,1}.$ Given $f\in W^{s,p(\cdot)}(M,\Lambda)$  and $\varphi \in W^{s+1,p(\cdot)}(M,\Lambda)$ satisfying the conditions 
	\begin{align*}
		df =0, \quad t(f-d\varphi)=0,\quad  (f,h_T) =[\varphi,h_T] \qquad \text{ for all } h_T \in \mathcal{H}_T(M).   
	\end{align*}
	Then there exists $\omega \in W^{s+1,p(\cdot)}(M,\Lambda)$ satisfying $d\omega = f$ in $M$, $\omega=\varphi $ on $b M$ and for some $C = C(\mathrm{data},s)>0$,  we have the estimate 
	\begin{align*}
		\|\omega\|_{W^{s+1,p(\cdot)}(M,\Lambda)} \leq C ( \|f\|_{W^{s,p(\cdot)}(M,\Lambda)} + \|\varphi\|_{W^{s+1,p(\cdot)}(M,\Lambda)} ).
	\end{align*}
\end{theorem*}
See Theorem \ref{Poincare lemma with dirichlet data tangential} for a proof. Compare Theorem 8.2 in \cite{CsaGyuDar12} for $p \geq 2$ and Theorem 2.47 in \cite{SilT16} for $1< p < \infty$ for constant exponent case. Once again, the arguments there need $M$ to be $C^{s+2,1}$. Unlike the last three results, this system is genuinely non-elliptic with an infinite dimensional kernel (any closed form which vanishes on the boundary is in the kernel) and thus one can not expect regularity results to hold for \emph{any} weak solution. Once again, by Hodge duality, an analogous version holds  for the problem $\delta\omega = f$ in $M$, $\omega=\varphi $ on $b M$ as well. 

To the best of our knowledge, our results for the variable exponent cases are all new, except for the case of $0$-forms, where these reduce to the cases treated by Dienning and  Růžička.  Some of our conclusions for $C^{1,1}$ manifolds are also new even for the constant exponent case. 

\section{Gauge fixing and nonlinear problems}	
	Study of non-linear problems with differential forms started from the seminal paper by K.~Uhlenbeck \cite{Uhl77}, who obtained classical results on the H\"older continuity (for the scalar equation this reads as $C^{\alpha}$ property of the gradient) in the general framework of elliptic complexes. These results were extended by Hamburger in~\cite{Ham92}. Both of these results concern \emph{homogeneous} quasilinear equations and systems of $p$-Laplace type. Beck and Stroffolini \cite{BecStr13} considered partial regularity results for general quasilinear systems for differential forms. The study of everywhere regularity results for inhomogeneous $p$-Laplace type systems was initiated in Sil~\cite{Sil119}, where one of the authors of the present contribution introduced a \emph{gauge fixing procedure} to transfer the regularity from the right hand side to the exterior derivative of a solution. This resulted in a Nonlinear Stein theorem for differential forms and a variety of estimates in $\mathrm{BMO}, \mathrm{VMO}$ and H\"{o}lder spaces in~\cite{Sil119}. Recently, these ideas were successfully adapted in by Lee, Ok and Pyo in ~\cite{LeeOkPyo24} to  obtain Calder\'on-{Z}ygmund estimates for nonlinear equations of differential forms with suitably `small' {BMO} coefficients. This gauge fixing procedure, which is crucial to handle inhomogeneous systems, however, in turn relies on the well-developed linear theory. It is exactly this need that initiated writing this manuscript. We hope that the style of writing adopted here, with minimalistic use of geometric terminology will provide an easily accessible reference source for analysts working on non-linear problems with differential forms in variable exponent spaces.


	Our analysis of the linear problem leads to the important corollaries for nonlinear problems with differential forms, for results in the scalar case( 0-forms) see for example \cite{AceMin01,AceMin02, AceMin05}. The typical result of this form is the following:
	\begin{corollary*}
		Let $\Omega \subset \mathbb{R}^{n}$ be a $C^{1,1}$ bounded, contractible domain. Let $a:\Omega \rightarrow [\gamma, L]$ be measurable, where $ 0 < \gamma < L < \infty.$ Let $p(\cdot)$ satisfy \eqref{eq:p1}, \eqref{eq:p2};  $F\in L^{p'(\cdot)}\left( \Omega, \Lambda\right)$, $p'(x)=\frac{p(x)}{p(x)-1}$, and $u_{0} \in W^{1,p(\cdot)}\left( \Omega, \Lambda\right).$ Then the following minimization problem 
		\begin{align*}
			\inf \left\lbrace \int\limits_{\Omega} \left[ {\frac{a\left( x\right)}{p\left( x\right)}}\left\lvert du \right\rvert^{p(x)} - \left\langle F, du \right\rangle\right]\, dx\,:\, u \in u_{0} + W^{1,p(\cdot)}_{\delta, T}\left( \Omega, \Lambda\right),\ {\int\limits_{\Omega}u = \int\limits_\Omega u_0}\right\rbrace 
		\end{align*}
		admits a unique minimizer. Moreover, the minimizer $\bar{u} \in W^{1,p(\cdot)}\left( \Omega,\Lambda\right)$ is the unique weak solution to the system 
		\begin{align*}
			\left\lbrace \begin{aligned}
				\delta \left( a\left(x\right)\left\lvert d\bar{u}\right\rvert^{p(x)-2}d\bar{u}\right) &= \delta F &&\text{ in } \Omega, \\
				\delta\bar{u} &= \delta u_{0} &&\text{ in } \Omega, \\
				t\bar{u} &=t u_{0} &&\text{ on } \partial\Omega. 
			\end{aligned}\right. 
		\end{align*} 
	\end{corollary*}
    See Corollary \ref{nonlinear minimization} for a proof.  This type of result is also instrumental to derive comparison estimates for using perturbation techniques in the context of nonlinear problems.

\section{Organization}	
	
	The rest of the article is organized as follows. Lebesgue and Sobolev spaces of differential forms on manifolds are introduced in Chapter~\ref{sec:LebSob}. This Chapter contains approximation results (density of smooth forms in variable exponent Sobolev spaces).  Auxilliary results on potentials and interpolation in variable exponent spaces are contained in Chapter~\ref{sec:auxiliary}. In Chapter \ref{sec:boundary} we discuss classical results for the Hodge Laplacian and div-curl systems. In Chapter~\ref{sec:Dirichlet} we state the results on solvability of Dirichlet and Neumann boundary value problems for the Hodge Laplacian in variable exponent spaces. The main ingredient of the proof --- a priori estimates --- is postponed until the last section. Based on these results, we study theory of first-order systems in variable exponent spaces in Chapter~\ref{sec:theory}. This includes Hodge decomposition theorems, cohomology resolution, Gauge fixing, solvability of `div-curl' type first order systems, Hodge-Dirac systems, Gaffney’s inequality and certain non-elliptic first-order boundary value problems, namely Poincar\'{e} lemma with full Dirichlet data, in variable exponent spaces.  In Chapter~\ref{Sec:parametrix} we prove a priori estimates in variable exponent spaces, which provide a basis for the whole investigation.

	\chapter{Lebesgue and Sobolev Spaces of Differential Forms on Manifolds}\label{sec:LebSob}
	
	In this Chapter we introduce the notation, spaces of differential forms, and basic approximation results.

	\section{Bundle of differential forms}

	Let $M$ be a $C^{s,1}$, $s\in \mathbb{N}$, compact Riemannian manifold with boundary $bM$, which is not necessarily orientable  and $g_{ij}$ be its metric tensor.  The dimension of $M$ will be further denoted by $n$. This means the transition functions are $C^{s,1}$ and the metric tensor is $C^{s-1,1}.$
	
	We shall follow \cite{deRham} and work with the totality $\Omega M$ of differential forms of all degrees $0,\ldots,\mathrm{dim}\, M$ and parities (odd/even). This is done in order to cater for the non-orientable manifolds. If $M$ is assumed to orientable, then we can simply work with the usual differential forms, i.e. even forms.  On an $n$-dimensional manifold each element of $\Omega M$ can be decomposed into $2(n+1)$ homogeneous forms:
	\begin{equation}\label{eq:forms}
		f = \sum_{r=0}^n f_e^r + \sum_{r=0}^n f_o^r, \quad \mathrm{deg}\, f_e^k, f_o^k = k,\quad k=0,\ldots,n
	\end{equation}
	the forms $f_e^k$ even and $f_o^k$ are odd. The coordinates of the form $f$ is the set of coordinates of all its homogeneous components. Under the coordinate transformation $x=x(y)$ the degree of the form is preserved, the coefficients of even forms (``usual''  differential forms) of degree $r$ are transformed as 
	$$
	\omega_I'(y) = \sum_{J} \omega_J(x(y)) \frac{\partial x^J}{\partial y^I}
	$$
	and the coefficients of odd forms (``pseudoforms'' or densities) are transformed as 
	$$
	\omega_I'(y) = \biggl(\mathrm{sign} \frac{\partial x}{\partial y} \biggr) \sum_{J} \omega_J(x(y)) \frac{\partial x^J}{\partial y^I}.
	$$ 
	Here the summation is over all ordered $r$-tuples $J=\{1\leq j_1<j_2<\ldots j_r\leq n\}$. Further we denote the set of ordered $r$-tuples by $\mathcal{I}(r)$. For two $r$-tuples $J$ and $I$ the expression 
	$$
	\frac{\partial x^J}{\partial y^I} = \frac{\partial (x^{j_1},\ldots, x^{j_r})}{ \partial (y^{i_1},\ldots, y^{i_r})},
	$$
	and $\partial x/\partial y$ is the Jacobian of the transform. 
	
	On $r$-forms of the same parity (both even or both odd) the scalar product $\langle \cdot,\cdot \rangle$ is defined in the standard way:
	$$
	\langle\omega,\eta\rangle = \sum_{I} \omega_I \eta^I = \sum_{IK} G^{IK} \omega_I \eta_K,
	$$
	where the summation is over ordered sets $I,K \in \mathcal{I}(r)$, $\eta^I = g^{i_1j_1}\ldots g^{i_r j_r} \eta_{j_1\ldots j_r}$ with the summation over all $r$-tuples $(j_1,\ldots,j_r)$, and $G^{IK}$ is the determinant of the matrix at the intersection of rows $I$ and columns $K$ of the matrix $\{g^{ij}\}$.
	
	The notion of scalar product is extended to $\Omega M$ by linearity and assumption that the scalar product of forms of different degrees or parities is zero: for $f$ given by \eqref{eq:forms} and $v$ given by the similar expression
	\begin{equation}\label{eq:formsg}
		v = \sum_{r=0}^n v_e^r + \sum_{r=0}^n v_o^r, \quad \mathrm{deg}\, v_e^k, v_o^k = k,\quad k=0,\ldots,n
	\end{equation}
	the forms $v_e^k$ even and $v_o^k$ odd, we set
	$$
	\langle f,v\rangle =  \sum_{r=0}^ n\langle f_e^r,v_e^r\rangle + \sum_{r=0}^ n \langle f_o^r, v_o^r\rangle.
	$$ 
	We denote $|\omega| = \sqrt{\langle \omega,\omega \rangle}$.
	
	
	The volume form $dV$ with coordinates $(dV)_{i_1\ldots i_n} = \sqrt{g} \cdot \mathrm{sign}\, (i_1\ldots i_n)$ is an odd form. Here and below $\sqrt{g}$ is the square root of the determinant of the matrix $\{g_{ij}\}$. The notions of the exterior and interior product, the Hodge star operator, and the differential and codifferential operators $d$ and $\delta$ are also extended to $\Omega M$ by linearity. In local coordinates for $k$-form $\omega$ and $l$-form $\eta$, the wedge product is given by
	$$
	(\omega\wedge \eta)_{i_1\ldots i_{k+l}} = \sum_{\sigma\in Sh(k,l)} \omega_{i_{\sigma(1)}\ldots i_{\sigma(k)}} \eta_{i_{\sigma(k+1)}\ldots i_{\sigma(k+l)}} \mathrm{sign}\, \sigma,
	$$
	the interior product $\omega \lrcorner \eta$ for $k<l$ has coordinates 
	$$
	(\omega \lrcorner \eta)_{i_1\ldots i_{l-k}} = \sum_{J\in \mathcal{I}(r)} \omega^J \eta_{J i_1\ldots i_{l-k}},
	$$
	while the Hodge star operator $* \omega = \omega \lrcorner dV$ is given in coordinates by
	$$
	(*\omega)_{i_1 \ldots i_{n-k}}  = \sum_{J\in \mathcal{I}(r)} \omega^J  \sqrt{g}\cdot\mathrm{sign}\,(J i_1\ldots i_{n-k})
	$$
	The parity of the wedge and interior products is the product of the parities of its factors, the operators $d$ and $\delta$ preserve form's parity, and so the Hodge Laplacian $\triangle = d\delta + \delta d$, while the Hodge operator $*$ changes the form's parity. If $\epsilon$ is a continuous odd form of degree zero satisfying $\epsilon^2=1$ (such a form always exists locally and if a \emph{continuous} global choice of $\epsilon$ exists for some atlas in $M$, then $M$ is orientable and $\epsilon$ is an ``orientation'' for $M$ ),  then for an odd/even form $\omega$ the form $\epsilon \omega$ is even/odd and $\epsilon^2 \omega = \omega$. Multiplication by $\epsilon$ commutes with $d$, $\delta$, $\triangle$, $\nabla$. 
	
	The relation $\alpha \wedge * \beta = \langle \alpha,\beta\rangle dV$ holds modulo those terms which are not odd of degree $n$ (i.e. in this relation only the odd components of degree $n$ are equal). 
	
	By definition, the integral of a form $\omega$ on $M$ can be not zero only if $\omega$ is an odd form of degree $\mathrm{dim}\, M$, and if $\omega$ is an odd form of degree $\mathrm{dim}\, M$ its integral is defined in the standard way using a partition of unity and local coordinates. Thus the integrals of the form $\int\limits_M f dV$, with $f$ a scalar function, are well-defined. To define this integral for functions which are more general than continuous we can either stick to the local formulation via partition of unity, or introduce the (unique) Lebesgue measure and therefore integration on $M$ starting from well-defined integrals of continuous functions, which obviously gives the same result, or any other procedure for Lebesgue measure construction.
	
	On $r$-forms in local coordinates the operators $d$ and $\delta$ have the form 
	\begin{align}\label{eq:d}
		\begin{aligned}
			(d\omega)&_{i_1\ldots i_{r+1}} \\&= \sum_{j=1}^{r+1} (-1)^{j-1} \frac{\partial \omega_{i_1\ldots i_{j-1}i_{j+1}\ldots i_{r+1}}}{\partial x^{i_j}} = \sum_{j=1}^{r+1} (-1)^{j-1} \nabla_{i_j} \omega_{i_1\ldots i_{j-1}i_{j+1}\ldots i_{r+1}},
		\end{aligned}
	\end{align}
	\begin{align}
		(\delta \omega)_{i_1\ldots i_{r-1}} &= - \nabla^k \omega_{k i_1\ldots i_{r-1}} \nonumber\\&=- g^{kl} \nabla_l \omega_{ki_1 \ldots i_{r-1}} \nonumber \\&=- g^{kl}\left( \begin{aligned} \frac{\partial \omega_{ki_1\ldots i_{r-1}}}{\partial x^l} &- \Gamma^{s}_{kl}\omega_{s i_1\ldots i_{r-1}} \\ &- \Gamma^s_{i_1 l} \omega_{ksi_2\ldots i_r} - \ldots - \Gamma^{s}_{i_{r-1}l} \omega_{ki_1\ldots i_{r-2}s} \end{aligned}  \right),  \label{eq:delta}
	\end{align}
	where we employed the Einstein summation convention, $\nabla_{l}$ denotes the covariant derivative in the $l$-th coordinate direction and 
	\begin{align*}
		\Gamma^k_{lm} &= \frac{1}{2} g^{ls} \left(\frac{\partial g_{sm}}{\partial x^l}+ \frac{\partial g_{ls}}{\partial x^m} -\frac{\partial g_{lm}}{\partial x^s} \right) \qquad \text{ are the Christoffel symbols}.     
	\end{align*}
	 Note that on a $C^{1,1}$ manifold, where the components of the metric tensor are only Lipchitz, the derivatives above are understood in the sense of weak derivatives and the Christoffel symbols are bounded functions. On $C^{2,1}$ manifolds,  by the Bochner-Weizenb\"ock formula for homogeneous forms, the difference of the Hodge Laplacian and the operator $\nabla^* \nabla $ is a bundle endomorphism determined by the Riemann curvature tensor, $d\delta + \delta d = \nabla^* \nabla +  R^W\omega$: for an $r$-form $\omega$,
	$$
	(\triangle \omega)_{i_1\ldots i_r} = - \nabla^j \nabla_j \omega_{i_1\ldots i_r} +\sum_{k=1}^r (-1)^k (\nabla_{i_k} \nabla^l - \nabla^l \nabla_{i_k}) \omega_{l i_1 \ldots \hat{\imath}_k \ldots i_r}.
	$$
	For the invariant definition of $R^W$ see \cite[Section 1.2]{Schwarz}. We shall not use this formula explicitly. The principal symbol of $\triangle$ is $-|\xi|^2_g$.
	
	The notions of tangential and normal parts of a form on the boundary are extended to $\Omega M$  also by linearity (see Section~\ref{ssec:adm} below).
	

	For two forms $f$ and $v$ we denote 
	\begin{align}\label{eq:scp}
		(f,v) = \int\limits_{M} \langle f,v\rangle\, dV = \int\limits_{M} f\wedge \ast v.
	\end{align}
	For $f$ given by \eqref{eq:forms} and $v$ given by \eqref{eq:formsg} we have
	$$
	(f,v) = \sum_{r=0}^ n(f_e^r,v_e^r) + \sum_{r=0}^ n(f_o^r,v_o^r).
	$$
	If $f$, $v$ have different parity or different degree then clearly $(f,v)=0$.

	Let $k\in \{0\}\cup\mathbb{N}$, $\alpha \in [0,1]$, On $C^{k+1,\alpha}$-manifold $M$  we use the notation $C^{k,\alpha}(M,\Lambda)$ for forms that have $C^{k,\alpha}$ coefficients in any local coordinate system. By $C^{k,\alpha}_0(M,\Lambda)$ we denote the set of $C^{k,\alpha}(M,\Lambda)$ forms vanishing in a neighbourhood of $bM$. We denote $\mathrm{Lip}(M,\Lambda) = C^{0,1}(M,\Lambda)$, $\mathrm{Lip}_0(M,\Lambda) = C^{0,1}_0(M,\Lambda)$.

	\section{Variable exponent Lebesgue and Sobolev spaces in \texorpdfstring{$\mathbb{R}^n$}{Rn}}
	Let $\Omega$ be a bounded domain in $\mathbb{R}^n$ and assume that $p(\cdot)$ satisfies \eqref{eq:p1} with $M=\Omega$, and for now $p_{-}=1$ is allowed.
	
	Let $L^0(\Omega)$ denote the set of measurable function on~$\Omega$
	and $L^1_{\mathrm{loc}}(\Omega)$ denote the space of locally integrable 
	functions. We define the generalized Orlicz norm by
	\begin{align*}
		\|f\|_{p(\cdot),\Omega} &:= \inf \left\{\gamma > 0\,:\, \int_\Omega
			\big(|f(x)/\gamma|\big)^{p(x)}\,dx \leq 1\right\}.
	\end{align*}
The generalized Orlicz space~$L^{p(\cdot)}(\Omega)$ is defined as the
	set of all measurable functions with finite generalized Orlicz norm
	\begin{align*}
		L^{p(\cdot)}(\Omega) &:= \left\{ f \in L^0(\Omega)\,:\, \|f\|_{p(\cdot),\Omega}  \right\}<\infty .
	\end{align*}
    For example the generalized Orlicz function $p=\mathrm{const}$ generates the usual Lebesgue space~$L^p(\Omega)$. 	

	For $k\in \mathbb{N}$ we introduce the spaces 
	$$
	W^{k,p(\cdot)}(\Omega) =\{u\in W^{k,1}(\Omega)\,:\, \nabla^k \omega \in L^{p(\cdot)}(\Omega)\}.
	$$
	The norm in this space can be defined as 
	$$
	\|u\|^{(1)}_{k,p(\cdot),\Omega}=\sum_{l=0}^{k-1} \|\nabla^l u\|_{1,\Omega} + \|\nabla^k u\|_{p(\cdot),\Omega}.
	$$
	If the domain $\Omega$ is Lipschitz, and for us this is always the case, using the Sobolev embedding we can replace the norms in $L^1(\Omega)$ in this expression by the norms in $L^{p_{-}}(\Omega)$. In general smooth functions are not dense in~$W^{k,p(x)}(\Omega)$. Therefore, we define~$H^{k,p(\cdot)}(\Omega)$
	as
	\begin{align*}
		H^{k,p(\cdot)}(\Omega) &:= \big\{\text{closure of}\ C^\infty(\Omega) \cap W^{k,p(\cdot)}(\Omega)\ \text{in}\ W^{k,p(\cdot)}(\Omega)\big\}.
	\end{align*}
	However, if the exponent $p(\cdot)$ satisfies the log-H\"older condition \eqref{eq:p2} (with $M=\Omega$) then smooth functions are dense in $W^{k,p(\cdot)}(\Omega)$, that is (the result of Meyers and Serrin extended to the variable exponent case)
	$$
	H^{k,p(\cdot)}(\Omega) = W^{k,p(\cdot)}(\Omega).
	$$
	Assuming that the exponent is log-H\"older we finally define the norm in $W^{k,p(\cdot)}(\Omega)$ as 
	$$
	\|u\|_{k,p(\cdot),\Omega} = \sum_{l=0}^k \|\nabla^k u\|_{p(\cdot),\Omega},
	$$
	which gives an equivalent norm to the one defined by $\|\cdot\|^{(1)}_{k,p(\cdot),\Omega}$. 
	
	We shall work with finite sets of functions $f=\{f_I, I\in \mathcal{I}\}$. In this case we slightly abuse the notation and use the same notation  $L^{p(\cdot)}(\Omega)$, $W^{k,p(\cdot)}(\Omega)$. The (vector)-function $f\in W^{k,p(\cdot)}(\Omega)$ iff $f_I \in W^{k,p(\cdot)}(\Omega)$ for all $I\in \mathcal{I}$. The norm is defined as $\|f\|_{W^{k,p(\cdot)}(\Omega)} = \sum_{I\in \mathcal{I}}\|f_I\|_{W^{k,p(\cdot)}(\Omega)}$.
	
	Let $W_c^{1,p(\cdot)}(\Omega)$ be the set of $W^{1,p(\cdot)}(\Omega)$ functions with compact support in $\Omega$. For a function $w$ defined on $\Omega$ let $E_0 w$ be the extension of $w$ by zero outside $\Omega$. We introduce in $W^{1,p(\cdot)}(\Omega)$ the following subspaces with zero boundary values:
	\begin{align*}
		H_0^{1,p(\cdot)}(\Omega) &:= \big\{\text{closure of}\ C^\infty_0(\Omega)\ \text{in}\  W^{1,p(\cdot)}(\Omega) \big\},\\
		\widetilde W_0^{1,p(\cdot)}(\Omega) &: = \big\{\text{closure of}\ W_c^{1,p(\cdot)}(\Omega)\ \text{in}\  W^{1,p(\cdot)}(\Omega)  \big\},\\
		W^{1,p(\cdot)}_0(\Omega) &:= W^{1,1}_0(\Omega)\cap W^{1,p(\cdot)}(\Omega),\\  
		\widehat W_0^{1,p(\cdot)}(\Omega) &:= \big\{w \in W^{1,p(\cdot)}(\Omega)\,:\ E_0\omega \in W^{1,1}(\mathbb{R}^n) \big\}
	\end{align*}
	with same norm as in~$W^{1,p(\cdot)}(\Omega)$. In general situation, we have the obvious relation 
	$$
	H_0^{1,p(\cdot)}(\Omega) \subset\widetilde W_0^{1,p(\cdot)}(\Omega) \subset W_0^{1,p(\cdot)}(\Omega) \subset \widehat W_0^{1,p(\cdot)}(\Omega) .
	$$
	However, under the log-H\"older condition \eqref{eq:p2} and if $\Omega$ is a Lipschitz domain all these four spaces coincide:
	$$
	H_0^{1,p(\cdot)}(\Omega) =\widetilde W_0^{1,p(\cdot)}(\Omega) = W_0^{1,p(\cdot)}(\Omega) = \widehat W_0^{1,p(\cdot)}(\Omega).
	$$
	That is, a $W^{1,p(\cdot)}(\Omega)$ function can be approximated by $C_0^\infty(\Omega)$ functions if and only if its extension by zero to $\mathbb{R}^n\setminus \Omega$ belongs to $W^{1,1}(\mathbb{R}^n)$.
	
	We say that a function $u$ from $W^{1,p(\cdot)}(\Omega)$ has zero trace on a part $\Gamma$ of $\partial \Omega$ if it can be approximated in this space by functions vanishing near $\Gamma$. See \cite{DieHHR11} for further properties of these spaces.

\section{Variable exponent Lebesgue spaces on manifolds}
Let $M$ be a compact $n$-dimensional Riemannian manifold with boundary, which is at least $C^{1,1}$ and $(U_\alpha,\varphi_\alpha)$ is a finite atlas on $M$. We assume without loss that $\varphi_\alpha(U_\alpha)$ are bounded Lipschitz domains in $\mathbb{R}^n$ and by refining the atlas we can assume that $\varphi_\alpha(U_\alpha)$ is either a ball $B_R$ (for an interior patch) or half ball $\overline{B}_R^+ = B_R\cap \{x^n\geq 0\}$ for boundary patches. Here and below $B_R$ is the open ball of radius $R$ centered at the origin.
	
	Let $p(\cdot)$ be a Borel function on $M$ satisfying \eqref{eq:p1} (here $p_{-}=1$ is allowed). The Lebesgue space $L^{p(\cdot)}(M,\Lambda)$ is the set of all forms (elements of $\Omega M$) such that in each chart $(U_\alpha,\varphi_\alpha)$ the form has Lebesgue integrable coordinates on $\varphi_\alpha(U_\alpha)$ and the norm in $L^{p(\cdot)}(M,\Lambda)$ is defined as 
	$$
	\|\omega\|_{L^{p(\cdot)}(M,\Lambda)} = \inf \left\{\lambda>0 \, :\,  \int\limits_M (|\omega|\lambda^{-1})^{p(x)}\, dV<\infty\right\}.
	$$
	We shall also use the notation
	$$
	\|\omega\|_{p(\cdot),M}=\|\omega\|_{L^{p(\cdot)}(M,\Lambda)}.
	$$
	The space $L^\infty(M,\Lambda)$ is the set of all forms $\omega$ with measurable components with the finite quantity $\|\omega\|_{L^\infty(M,\Lambda)}:= \mathop{\mathrm{ess\,sup}}_{M} |\omega|$.
	
	The space $L^{p(\cdot)}_{\mathrm{loc}}(M,\Lambda)$ is the set of all forms such that in each chart the form has measurable components, and for each compact $K\subset M$ with positive distance from $bM$ there holds $\omega\chi_K \in L^{p(\cdot)}(M,\Lambda)$, where $\chi_K$ denotes the indicator function of $K$.

	\section{Admissible coordinate systems}\label{ssec:adm}
In the following we mostly use the notation from \cite[Section 7.5]{Morrey1966} for Riemainnian manifolds. Let~$M$ be a  Riemannian manifold of class~$C^{s,\mu}$ (with boundary $bM$).   Any two admissible coordinate systems are related by a transformation of class~$C^{s,\mu}$, where~$0\le \mu\le 1$.  We assume that the manifold $M$ is compact and of class at least~$C^{1,1}$. Boundary values and the normal and tangent components are defined via admissible boundary coordinate system.
	
	\begin{definition}
		An admissible boundary coordinate system on a manifold~$M$  with boundary~$bM$ of class~$C^{s,\mu}$ is a coordinate system $\varphi$ of class $C^{s,\mu}$ which maps some boundary neighbourhood $U$ onto a (Lipschitz)  domain~$G\cup \sigma$, $G\subset \mathbb{R}^n_+$, $\sigma \subset \partial G$, in such a way, that~$\sigma\subset \{x^n=0\}$ is not empty, open, and ~$\varphi(U \cap bM) = \sigma$, and the metric is of the form
		\begin{align*}
			ds^2=\sum_{i,j=1}^{n-1} g_{ij} (x'_n,0) dx^i dx^{j} +(dx^n)^2\quad\text{ on } \sigma.
		\end{align*}
	\end{definition}
We state the corresponding result from Morrey.
\begin{lemma}[\cite{Morrey1966}, Lemma 7.5.1]
		If~$M(\cup bM)$ is of class~$C^{s,\mu}$, then each point~$P_0$ of $bM$ is
		in the range of an admissible boundary coordinate system. Moreover if~$s\ge 2$ each point $P_0$ of $bM$ is in the image of a coordinate
		system of class $C^{s-1,\mu}$   such is  that the metric takes the form $ds^2$  with $(x'_n,0)$ replaced by any $x\in G$. If $M$ is of class $C^\infty$ this coordinate system may be
		taken to be of class $C^\infty$ and hence admissible. If~$x$ and~$y$ are overlapping admissible boundary coordinate sytems, then
		\begin{align*}
			y^\alpha_{,\nu}(x'_{\nu},0)= y^\nu_{,\alpha}(x'_{\nu},0)=0,\quad \alpha<\nu,\quad y^\nu_{,\nu}(x'_\nu,0)=1.
		\end{align*}
\end{lemma} 
Using an admissible coordinate system we can define on $bM$ (in coordinates on $\sigma$) the ``exterior'' unit normal $\nu = -dx^n$ and the normal and tangential parts of a form. Namely, in an admissible coordinate system we split a form $\omega$ on the boundary $bM$ as
	$$
	\omega =t\omega+ n\omega,\quad t\omega = dx^n \lrcorner (dx^n \wedge \omega) + dx^n \wedge (dx^n\lrcorner \omega).
	$$
	For an $r$-form $\omega = \sum_{I\in \mathcal{I}(r)} \omega_I dx^I$ there holds
	$$ 
	t\omega = \sum_{I\in \mathcal{I}(r)\,:\, i_r<n} \omega_I dx^I, \quad n\omega = \sum_{I\in \mathcal{I}(r)\,:\, i_r=n} \omega_I dx^I,
	$$
	and  this extends by linearity to $\Omega M$. Setting $t\omega$ is equivalent to setting $\nu \wedge \omega$, and setting $n\omega$ is equivalent to setting $\nu \lrcorner \omega$. In particular, $t\omega = t\varphi$ is equivalent to $\nu\wedge (\omega - \varphi)=0$ on $bM$ and $n\omega = n\psi$ is equivalent to $\nu \lrcorner (\omega - \psi) =0$ on $bM$. Using a partition of unity subordinate to the (admissible) coordinate cover $\{U_\alpha\}$ (see, for instance, \cite[Chapter 1, Theorem 1]{deRham} or \cite[Volume 1, Chapter 2, Theorem 15]{Spivak}, or \cite[Theorem 1.11]{Warner}) the form $\nu$ can be extended to a $C^{s-1,1}(M,\Lambda)$ form. The extension procedure clearly is not unique, but this does not affect further definitions. Thus below $\nu$ is an even $1$-form which belongs to $C^{s-1,1}(M,\Lambda)$ and satisfies $\nu=-dx^n$ on $bM$ in any admissible coordinate system. We define the subspaces of $C^{k,\alpha}(M,\Lambda)$ of forms with vanishing tangential/normal part:
	\begin{align*}
		C^{k,\alpha}_T(M,\Lambda) &= \{\omega \in C^{k,\alpha}(M,\Lambda)\, :\, \nu \wedge \omega =0 \quad \text{on}\quad bM\},\\
		C^{k,\alpha}_N(M,\Lambda) &= \{\omega \in C^{k,\alpha}(M,\Lambda)\, :\, \nu \lrcorner \omega =0 \quad \text{on}\quad bM\}.
	\end{align*}
We denote $\mathrm{Lip}_T(M,\Lambda) = C^{0,1}_T(M,\Lambda)$, $\mathrm{Lip}_N(M,\Lambda) = C^{0,1}_N(M,\Lambda)$.
	
	\section{Variable exponent Sobolev spaces on manifolds}

	Let $M$ be at least $C^{s,1}$, $s\in \mathbb{N}$ and $(U_\alpha, \varphi_\alpha)$ be a finite atlas of $M$. We assume each coordinate domain $\varphi_\alpha (U_\alpha)$ to be at least Lipschitz.

	On $\varphi_\alpha (U_\alpha) \subset \mathbb{R}^n$ we denote $p_\alpha(\cdot) = p(\varphi_\alpha^{-1}(\cdot))$. Further we assume that $p$ satisfies \eqref{eq:p1} (where $p_{-}=1$ is allowed) and has the logarithmic modulus of continuity \eqref{eq:p2}. The variable exponent $p_\alpha$ satisfies 
	$$
	|p_\alpha (x)-p_{\alpha}(y)| \leq  \frac{L_1}{\log \bigl(e+(|x-y|)^{-1}\bigr)}, \quad x,y\in \varphi_\alpha(U_\alpha),
	$$
	with a constant $L_1= L_1(M,c_{\mathrm{log}}(p))$ independent of $\alpha$.

	The Sobolev space $W^{s,p(\cdot)}(M,\Lambda)$ is the set of all forms such that in each  coordinate system the form has coordinates from $W^{s,p_\alpha(\cdot)}(\varphi_\alpha (U_\alpha))$, and for $f$ defined by \eqref{eq:forms} its norm in this space is 
	\begin{align*}
		\|f\|_{s,p(\cdot),M} =\sum_{r=0}^{\mathrm{dim}\, M} & \sum_{\alpha} \sum_{I \in \mathcal{I}(r)} \sum_{l=0}^s\sum_{|\beta|=l} \| |D^\beta (f^{r}_e)^{(\alpha)}_I| \|_{L^{p_\alpha(\cdot)} (\varphi_\alpha (U_\alpha))} \\
		&+ \sum_{r=0}^{\mathrm{dim}\, M}  \sum_{\alpha} \sum_{I \in \mathcal{I}(r)} \sum_{l=0}^s\sum_{|\beta|=l} \| |D^\beta ( f_o^r)^{(\alpha)}_I| \|_{L^{p_\alpha(\cdot)} (\varphi_\alpha (U_\alpha))}.
	\end{align*}
	Here $(f)^{(\alpha)}_I$ denotes the corresponding component of $f$ in the coordinate system $\varphi_\alpha$, and $D^\beta$ is the partial derivative 
	$$
	D^\beta =\frac{\partial^{|\beta|}}{\partial (x^1)^{\beta_1} \ldots \partial (x^n)^{\beta_n}}. 
	$$
	Such norms clearly depend on the choice of the atlas but are equivalent (by the standard change of variables formula in the Lebesgue integral).
	
	For $s=0$ the space $W^{s,p(\cdot)}(M,\Lambda)$ is the space $L^{p(\cdot)}(M,\Lambda)$, that is we set $W^{0,p(\cdot)}(M,\Lambda) = L^{p(\cdot)}(M,\Lambda)$.
	
	For $p(\cdot)\equiv \infty$ we use the same definition for $W^{s,\infty}(M,\Lambda)$. For $p(\cdot)\equiv \infty$, $W^{1,\infty}(M,\Lambda) = \mathrm{Lip}\, (M,\Lambda)$.

	For an an (even or odd) $r$-form $f$ denote  
	$$
	|\nabla^l f|^2 = \frac{1}{r!}\sum_{IJ} g^{i_1 j_1} \ldots g^{i_{l+r} j_{l+r}}(\nabla_{i_1\ldots i_l} f_{i_{l+1}\ldots i_{l+r}})(\nabla_{j_1\ldots j_l} f_{j_{l+1}\ldots j_{l+r}}) 
	$$
	where the summation is over $(l+r)$-tuples $I$, $J$. Then we can define a norm on $W^{s,p(\cdot)}(M,\Lambda)$ as
	$$
	\|f\|^{(inv)}_{s,p(\cdot),M} = \sum_{r=0}^{\mathrm{dim}\, M} \sum_{l=0}^s \| |\nabla^l f^{r}_e| \|_{p(\cdot),M} + \sum_{r=0}^{\mathrm{dim}\, M} \sum_{l=0}^s \| |\nabla^l f^{r}_o| \|_{p(\cdot),M}.
	$$
	This expression does not depend on the choice of the atlas and is equivalent to the norms introduced above.
	

	By $W^{s,p(\cdot)}_c(M,\Lambda)$ we denote the set of $\omega \in W^{s,p(\cdot)}(M,\Lambda)$ such that  $\omega$ vanishes in a neighbourhood of $bM$. 
    
	By $W_0^{1,p(\cdot)}(M,\Lambda)$ we denote the closure of $\mathrm{Lip}_0(M,\Lambda)$ in $W^{1,p(\cdot)}(M,\Lambda)$. If $\phi \in \mathrm{Lip}_0(M,\Lambda)$ then for any $\omega \in W^{1,p(\cdot)}(M,\Lambda)$ we have $\phi \wedge \omega, \phi \lrcorner \omega \in W_0^{1,p(\cdot)}(M,\Lambda)$.  See Lemma~\ref{L:approx1} below.

	By $W^{1,p(\cdot)}_T(M,\Lambda)$ we denote the subspace of forms from $W^{1,p(\cdot)}(M,\Lambda)$ such that $\nu \wedge \omega \in W_0^{1,p(\cdot)}(M,\Lambda)$. In an admissible boundary coordinate system this is equivalent to $\omega_I =0$ on $\{x_n=0\}$ (in the sense of trace) if $n\notin I$.

	By $W^{1,p(\cdot)}_N(M,\Lambda)$ we denote the subspace of forms from $W^{1,p(\cdot)}(M,\Lambda)$ such that $\nu \lrcorner \omega \in W_0^{1,p(\cdot)}(M,\Lambda)$. In an admissible boundary coordinate system this is equivalent to $\omega_I =0$ on $\{x_n=0\}$ (in the sense of trace) if $n\in I$.
	
	The differential and codifferential operators $d$ and $\delta$ are readily extended to (variable exponent) Sobolev spaces in the usual way by formulas \eqref{eq:d}, \eqref{eq:delta} (with the partial derivatives understood in the weak sense) and are bounded linear mappings from $W^{s,p(\cdot)}(M,\Lambda)$ to $W^{s-1,p(\cdot)}(M,\Lambda)$. In particular, they map $\mathrm{Lip}(M,\Lambda)=W^{1,\infty}(M,\Lambda)$ to $L^\infty(M,\Lambda)$. The Hodge-Laplace operator $d\delta+\delta d$ is a bounded linear mapping from $W^{s,p(\cdot)}(M,\Lambda)$, $s\geq 2$, to $W^{s-2,p(\cdot)}(M,\Lambda)$. 
	
	On $C^{s,1}$ manifold for $k\leq s$ the Hodge star operator is a bounded invertible linear mapping on $W^{k,p(\cdot)}(M,\Lambda)$, from $W_T^{k,p(\cdot)}(M,\Lambda)$ to $W_N^{k,p(\cdot)}(M,\Lambda)$.

	\section{Approximation results for variable exponent Sobolev spaces}\label{ssec:appox1}
	Let $M$ be of the class $C^{s,1}$, $s\in \mathbb{N}$. 
	
	We shall use a particular case of Zhikov's lemma, see \cite{Zhi95}. Let $D$ be a bounded domain in $\mathbb{R}^n$ and $p: D \to [p_{-},p_{+}]$, $p_{-}\geq 1$, be a log-H\"older exponent: 
	\begin{equation}\label{eq:LH1}
		|p(x)-p(y)| \leq L (e+\log |x-y|^{-1})^{-1}, \quad x,y \in D.
	\end{equation}
	Using the McShane extension lemma (see \cite{Mcshane_extension}) we can assume that $p$ is extended to $\mathbb{R}^n$ satisfying the same conditions.
	
	Let $\gamma_\varepsilon$ be a Friedrichs mollifier, that is $\gamma_\varepsilon(x)=\varepsilon^{-n} \gamma(\varepsilon^{-1}x)$ where $\gamma$ is a nonnegative smooth function with support in the ball $B_1$ and unit integral over $\mathbb{R}^n$. Let $T_\varepsilon: \mathbb{R}^n\to \mathbb{R}^n$ be linear transformations converging to the identity as $\varepsilon \to 0$ such that $\|T_\varepsilon x - x\|\leq k \varepsilon \|x\|$. We assume that $1/2\leq \mathrm{det}\, T_\varepsilon \leq 2$ and $\|x\|/2\leq \|T_{\varepsilon}x\|\leq 2\|x\|$. Denote 
	$$
	f_\varepsilon(x) = \int \limits_{\mathbb{R}^n} f(x_0+T_\varepsilon (y-x_0)) \gamma_\varepsilon (x-y)\, dy.
	$$ 
	\begin{lemma}\label{L:Zhikov0}
		Let $f\in L^{p(\cdot)}(D)$. Extend $f$ by zero to $\mathbb{R}^n \setminus D$. Then 
		$$
		\int\limits_{\mathbb{R}^n} |f_\varepsilon(x)-f(x)|^{p(x)}\, dx \to 0 \qquad \text{ as } \varepsilon \to 0.
		$$
	\end{lemma}
	\begin{proof}
		By the H\"older inequality,
		\begin{align*}
			|f_\varepsilon (x)| &\leq \biggl(\int\limits_{B_{\varepsilon}(x)} |f(x_0+T_\varepsilon (y-x_0))|^{p_{-}}\, dy \biggr)^{1/p_{-}} \biggl(\int \limits_{B_\varepsilon} |\gamma_\varepsilon|^{p_{-} } \, dx\biggr)^{1/p_{-}'} \\
			&\leq 2\biggl(\int\limits_D |f|^{p_{-}}\, dz \biggr)^{1/p_{-}} \biggl(\int \limits_{B_\varepsilon} |\gamma_\varepsilon|^{p_{-} } \, dx\biggr)^{1/p_{-}'} \leq c \varepsilon^{-n/p_{-}}.
		\end{align*}
		Denote 
		$$
		\widetilde B_\varepsilon(x) = \{y \in \mathbb{R}^n\, :\,  x-x_0 -T_\varepsilon^{-1} (y-x_0) \in \mathrm{supp}\, \gamma_\varepsilon \},\quad   p_\varepsilon(x) = \min\limits_{\widetilde B_\varepsilon(x)} p.
		$$
		It is clear that $\widetilde B_\varepsilon(x) \subset B_{K\varepsilon}(x)$ with $K$ is a constant depending on $k$ and $D$, $x_0$. By the log-condition we have
		$$
		|f_\varepsilon (x)|^{p(x)} \leq C(n,p_{-},L,k,D,x_0) |f_\varepsilon(x)|^{p_\varepsilon(x)}.
		$$
		By the Jensen inequality and the definition of $p_\varepsilon$ we get 
		\begin{align*}
			|f_\varepsilon (x)|^{p_{\varepsilon}(x)} &\leq \int\limits_{\mathbb{R}^n} |f(x_0+T_\varepsilon (y-x_0))|^{p_{\varepsilon}(x)} \gamma_\varepsilon(x-y)\, dy\\
			&\leq 1 + \int\limits_{\mathbb{R}^n} |f|^{p(y)} \gamma_\varepsilon(x-y)\, dy.
		\end{align*}
		Therefore,
		$$
		|f_\varepsilon (x)|^{p(x)} \leq C \biggl(1 +  \int\limits_{\mathbb{R}^n} |f|^{p(y)} \gamma_\varepsilon(x-y)\, dy \biggr).
		$$
		Since $|f|^{p(\cdot)} \in L^{1}(\mathbb{R}^n)$, the last integral converges in $L^1(\mathbb{R}^n)$, and thus the sequence $|f_\varepsilon (x)|^{p(x)}$ is equiintegrable. It remains to use the Lebesgue theorem. 
	\end{proof}
	
	\begin{remark}
		For $p_{-}>1$ one can use the boundedness of the maximal function in $L^{p(\cdot)}(\mathbb{R}^n)$ \cite{Die04}, the estimate $|f_\varepsilon(x)| \leq C M[f] (x)$, and the Lebesgue dominated convergence theorem. 
	\end{remark}
	For $\varepsilon>0$ denote
	$$
	D_\varepsilon = \{x\in D \,:\, x_0+T_\varepsilon(x - y-x_0)\in D, \, \text{ for all } y\in \mathrm{supp}\, \gamma_\varepsilon\} .
	$$ 
	The set $D_\varepsilon$ is open, since $\operatorname{supp}\, \gamma_\varepsilon$ is closed.
	\begin{lemma}\label{L:Zhikov}
		Let $u\in W^{s,p(\cdot)}(D)$. Then $u_\varepsilon \in C^\infty(D_\varepsilon)$ and for any domain $\widetilde D \subset \bigcap\limits_{0<\varepsilon <a} D_\varepsilon$,  there holds
		$$
		\int\limits_{\widetilde D} |\nabla^k u_\varepsilon -\nabla^k u|^{p(x)}\, dx \to 0, \quad k=0,\ldots, s, \qquad \text{ as } \varepsilon \to 0.
		$$
	\end{lemma}
	\begin{proof} 
		Follows from Lemma~\ref{L:Zhikov0} and $\nabla^k u_\varepsilon = (\nabla^k u)_{\varepsilon}$.
	\end{proof}
	
	In the following lemmas we denote 
	$$
	\widetilde B^{-}_{\varepsilon} = B_\varepsilon \cap \{x_n<-\varepsilon/2\}, \quad \widetilde B^{+}_{\varepsilon} = B_\varepsilon \cap \{x_n>\varepsilon/2\}.
	$$
	Let the exponent $p:M\to [p_{-},p_{+}]$, $p_{-}\geq 1$, be a log-H\"older exponent, i.e. assume \eqref{eq:p2} holds. 
	Then in each chart the exponent $p(\cdot)$ satisfies \eqref{eq:LH1}.
	\begin{lemma}\label{L:approx1}
		\hspace{2em}
		\begin{enumerate}
			\item Let $\omega \in W^{l,p(\cdot)}(M,\Lambda)$ ($W^{l,p(\cdot)}_c(M,\Lambda)$), $0\leq l \leq s$. Then there exists a sequence $\omega_\varepsilon \in C^{s-1,1}(M,\Lambda)$ (resp. $\omega_\varepsilon \in C_0^{s-1,1}(M,\Lambda)$) converging to $\omega $ in $W^{l,p(\cdot)}(M,\Lambda)$.
			
			\item Moreover, if $\omega \in W^{1,p(\cdot)} (M,\Lambda)\cap W_T^{1,1}(M,\Lambda)$ (resp. $\omega \in W^{1,p(\cdot)} (M,\Lambda)\cap W_N^{1,1}(M,\Lambda)$), then one can choose $\omega_\varepsilon$ with $t\omega_\varepsilon =0$ (resp. $n\omega_\varepsilon=0$). 
			
			\item Let $\omega\in {\mathrm{Lip}}(M,\Lambda)$.  Then there exists a sequence $\omega_\varepsilon \in C^{s-1,1}(M,\Lambda)$  converging to $\omega $ in $W^{1,q}(M,\Lambda)$, $1\leq q <\infty$, and such that ~$\omega_\epsilon$ is uniformly bounded in~$W^{1,\infty}(M, \Lambda)$.  If additionally $\omega=0$ on $bM$ then one can additionally choose $\omega_\varepsilon\in C^{s-1,1}_0(M,\Lambda)$.
			
			\item If $\phi \in {\mathrm{Lip}}(M,\Lambda)$ and $\phi =0$ on $bM$ then for any $\omega \in W^{1,p(\cdot)}(M,\Lambda)$ we have $\phi \wedge \omega, \phi \lrcorner \omega \in W^{1,p(\cdot)}_0(M,\Lambda)$.
			
			\item 	There holds
			\begin{gather*}
				W_T^{1,p(\cdot)}(M,\Lambda) = W_T^{1,1}(M,\Lambda) \cap W^{1,p(\cdot)}(M,\Lambda),\\
				W_N^{1,p(\cdot)}(M,\Lambda) = W_N^{1,1}(M,\Lambda) \cap W^{1,p(\cdot)}(M,\Lambda).
			\end{gather*}

			\item If $\omega_\varepsilon \in \mathrm{Lip}_T(M,\Lambda)$ (resp. $\mathrm{Lip}_N(M,\Lambda)$) then in (c) one can additionally choose $\omega_\varepsilon$ so that $\nu\wedge \omega_\varepsilon$ (resp. $\nu \lrcorner \omega_\varepsilon)$ belongs to  $C_0^{s-1,1}(M,\Lambda)$.
			
			\item If $\omega_\varepsilon \in W^{1,p(\cdot)}_T(M,\Lambda)$ ($W^{1,p(\cdot)}_N(M,\Lambda)$) then in (a) one can take $\omega_\varepsilon$ such that $\nu \wedge\omega_\varepsilon$ (resp. $\nu \lrcorner \omega_\varepsilon$) belongs to $C_0^{s-1,1}(M,\Lambda)$.

		\end{enumerate}
	\end{lemma}
	\begin{proof}
		It is sufficient to prove this for a homogeneous form $\omega$. Using a partition of unity subordinate to the coordinate atlas $\{(U_j,\varphi_j)\}$, we reduce these statements to one of the two cases: $\omega$ is compactly supported in one of interior patches or $\omega$ is compactly supported in one of boundary patches, where we can assume the coordinate system to be admissible. 
		
		For the first case, we mollify the coordinates of $\omega$ in the local coordinate system, and thus get 
		$$
		\omega_{\varepsilon,I}(x) = \int \omega_I (y) \gamma_\varepsilon(x-y)\, dy,
		$$
		where $\gamma_\varepsilon$ if a Friedrichs mollifier with $\varepsilon$ so small that the support of $\omega_{\varepsilon}$ is still in the same coordinate neighbourhood. This gives a required approximation for (a)-(g). The rest of the proof concerns the second case, when $\omega$ is compactly supported in an admissible boundary coordinate neighbourhood $U$. By refining the coordinate cover we can assume without loss that $|\nu|>1/2$ in $U$.
		
		(a) In this case we take the Friedrichs mollifier $\gamma_\varepsilon$ with $\operatorname{supp}\, \gamma_\varepsilon \in \widetilde B_\varepsilon^{-}$. Then for sufficiently small $\varepsilon$ the form $\omega_\varepsilon$ is supported in the same coordinate neighbourhood and
		$$
		\|\omega_\varepsilon -\omega\|_{l,p_j(\cdot), \varphi_j (U)} \to 0
		$$
		as $\varepsilon \to 0$. Clearly, this holds in any other coordinate chart.
		If the support of $\omega$ is separated from $bM$, then so is the support of $\omega_\varepsilon$.

		(b). Now, let $t\omega =0$. Then in interior patches and in boundary patches for normal components we follow the same procedure, while the tangential components are extended by zero across $x^n=0$, and the Friedrichs mollifier for them has support in $\widetilde B^{+}_\varepsilon$. Then all the tangential components of $\omega_{\varepsilon}$ vanish on $\{x^n<\varepsilon/2\}$. For the $n\omega=0$ condition the proof is similar with the roles of tangential and normal components reversed.

		(c). If $\omega \in {\mathrm{Lip}}(M,\Lambda)$ the approximation is the same as in the first case. If  additionally $\omega=0$ on $bM$, in boundary patches we extend the components of $\omega$ by zero to $\{x^n<0\}$ and then use the Friedrichs mollifier with support in $\widetilde B^{+}_\varepsilon$. 
		
		(d). First, approximate $\varphi$ by $\varphi_\varepsilon\in C_0^{s-1,1}(M,\Lambda)$ using (c). Then approximate $\varphi_\varepsilon \wedge \omega, \varphi_\varepsilon\lrcorner \omega \in W^{1,p(\cdot)}_c(M,\Lambda)$ using (a).
		
		(e)  Let $\omega \in W^{1,p(\cdot)}(M,\Lambda)\cap W_T^{1,1}(M,\Lambda)$. Arguing as in (b) we approximate $\omega$ by $C^{s-1,1}(M,\Lambda)$ forms $\omega_\varepsilon$  with support in the same coordinate chart such that $dx^n \wedge \omega_\varepsilon$ vanishes for $x^n<\varepsilon/2$. Now, $\nu \wedge \omega_\varepsilon = dx^n \wedge \omega_\varepsilon + (\nu - dx^n)\wedge \omega_\varepsilon$. Since $\nu - dx^n$ vanishes on $bM$, we approximate it as in (c) by $C^{s-1,1}_0(M,\Lambda)$ forms on the support of $\eta_j$.  Thus $\nu \wedge \omega_\varepsilon$ can be approximated in $W^{1,p(\cdot)}(M,\Lambda)$ by $C_0^{s-1,1}(M,\Lambda)$ forms, and so $\nu \wedge \omega$. By definition, this means $\omega \in W^{1,p(\cdot)}(M,\Lambda)$.
		
		(f). Now let $t\omega =0$. It remains to split
		$$
		\omega = \frac{1}{|\nu|^2} \left(\nu \wedge (\nu \lrcorner \omega) + \nu \lrcorner (\nu \wedge \omega) \right),
		$$
		approximate $(\nu \lrcorner \omega)$ by $C^{s-1,1}(M,\Lambda)$ forms $(\nu \lrcorner \omega)_\varepsilon$ and $(\nu \wedge \omega)$ by $C_0^{s-1,1}(M,\Lambda)$ forms $(\nu \wedge \omega)_\varepsilon$ as above, and write
		$$
		\omega_\varepsilon = \frac{1}{|\nu|^2} \left( \nu \wedge (\nu \lrcorner \omega)_\varepsilon + \nu \lrcorner (\nu \wedge \omega)_\varepsilon \right).
		$$
		The argument for $n\omega=0$ is similar, or via the Hodge dual.
		
		(g). Follows from (b) and (e).
		
	\end{proof}
	
	\section{Integration-by-parts formula}
	We denote 
	\begin{equation}\label{eq:bound}
		[f,v]:= \int \limits_{b M} \langle\nu \wedge f, v \rangle\, d\sigma = \int\limits_{bM} tf\wedge \ast nv = \int\limits_{bM} \langle f, \nu \lrcorner v \rangle\, d\sigma,
	\end{equation}
	where $d\sigma$ is the $(n-1)$-volume on $b M$, in the local admissible coordinate system $d\sigma = \sqrt{g} dx^1\ldots dx^{n-1}$ ($d\sigma = \sqrt{g}\sqrt{g^{nn}} dx^1\ldots dx^n$ if we drop ``admissible''), and $\nu$ is the ``exterior'' conormal, in local admissible coordinates just $\nu = -dx^n$ (resp. $\nu = -(g^{nn})^{-1/2} dx^n$). If $f$ and $v$ are homogeneous forms with different parities or $\mathrm{deg}\, f +1\neq \mathrm{deg}\, v$ then $[u,v]=0$.  
	
	For  $f$ given by \eqref{eq:forms} and $v$ given by \eqref{eq:formsg} we have  
	$$
	[f,v] = \sum_{r=0}^{n-1} [f^r_e,v_e^{r+1}]+ \sum_{r=0}^{n-1} [f^r_o,v_o^{r+1}].
	$$

	Then (see \cite{Morrey1966}, \cite{Schwarz}, \cite{MitMitMitTay16}) for any $f\in W^{1,q}(M,\Lambda)$, $v\in W^{1,q'}(M,\Lambda)$, $q=\mathrm{const}$, $1\leq q \leq \infty$, there holds 
	\begin{equation}\label{eq:by_parts}
		(df,v)=(f,\delta v) +[f,v].
	\end{equation}
	In particular, the integration-by-parts formula \eqref{eq:by_parts} holds if $f\in W^{1,1}(M,\Lambda)$ and $v\in \mathrm{Lip}(M,\Lambda)$, or if $v\in W^{1,1}(M,\Lambda)$ and $f\in \mathrm{Lip}(M,\Lambda)$.
	
	The integration-by-parts formula \eqref{eq:by_parts} for an $r$-form $f$ and an $(r+1)$-form $v$ is nothing else as the Gauss formula
	$$
	\int\limits_M (\mathrm{div}\, X)\, dV = \int\limits_{bM} \nu(X)\, d\sigma, \quad \mathrm{div}\, X = \nabla_j X^j = \frac{1}{\sqrt{g}} \frac{\partial (\sqrt{g} X_j)}{\partial x^j}, \quad \nu(X)=\nu_jX^j,
	$$
	for a $W^{1,1}(M,\Lambda)$ vector field $X^j = \sum_{I \in \mathcal{I}(r)}f_{I} v^{jI}$, together with the pointwise relation 
	\begin{align*}
		\mathrm{div}\, X &=\frac{1}{r!}\nabla_j \bigl(f_{I} v^{jI}\bigr) \\ &= \frac{1}{r!}(\nabla_j f_{I}) v^{jI} + \frac{1}{r!}f_{I} \nabla_j v^{jI} \\
		&=  \frac{1}{(r+1)!} (df)_{jI} v^{jI} + \frac{1}{(r+1)!}((r+1)f_{I;j} - (df)_{jI} )v^{jI}- \langle f,\delta v\rangle \\
		&= \langle df,v\rangle - \langle f,\delta v\rangle.
	\end{align*}
	Here one uses that in
	$$
	(r+1)f_{I;j} - (df)_{jI} = \sum_{k=1}^r ((-1)^{k+1}  f_{j I'_k;i_k} + f_{I;j} ), \quad I'_k = (i_1\ldots i_{k-1}i_{k+1}\ldots i_r)
	$$
	the expression in parentheses is the same if we transpose $j$ and $i_k$ in $jI$, and so $(f_{j I'_k;i_k} + f_{I;j} ) v^{jI}=0$ by the skew-symmetry of $v$ for all $k$. In fact, the proof just says that $df$ is the alternation of the tensor $(r+1)\nabla f$, the difference of $(r+1)\nabla f-df$ is a symmetric tensor and its contraction with skew-symmetric $v$ is zero.

	By using a partition of unity subordinate to an (admissible) coordinate cover the formula is reduced to the case when the support of $X$ is in the range of an admissible boundary coordinate system, and in this case the required formula reduces to 
	$$
	\int\limits_{\mathbb{R}^n_+} (\sqrt{g} X_n)_{,n} \, dx = - \int_{\mathbb{R}^{n-1}} X_n(x',0)\, \sqrt{g}dx',
	$$
	where $\sqrt{g}$ is at least Lipschitz and $X_n \in W^{1,1}(\mathbb{R}^n_+)$.

For the Euclidian case, when $M$ is a domain $D$ in $\mathbb{R}^n$, $(\delta v)_I = -\partial_{x^j} v_{jI}$ and the integration by parts formula is
\begin{gather*}
\int\limits_D \bigl(\sum_{I\in \mathcal{I}(r+1)} (df)_I v_I - \sum_{I \in \mathcal{I}(r)} f_I (\delta v)_I \bigr)\, dx \\
= \int\limits_{\partial D} \sum_{I\in\mathcal{I}(r)}\sum_{j=1}^n\nu_j f_I v_{jI}\, d\sigma = \int\limits_{\partial D} \sum_{I\in\mathcal{I}(r+1)} (\nu \wedge f)_I v_I\ d\sigma. 
\end{gather*}

	Various generalizations to ``bad'' domains are known, see \cite{MitMitMitTay16}.
	
	\begin{lemma}
		Let $p(\cdot)$ satisfy \eqref{eq:p1}, \eqref{eq:p2}. If we have $f \in W^{1,p(\cdot)}(M,\Lambda)$ and $v\in W^{1,p'(\cdot)}(M,\Lambda),$ then the integration-by-parts formula \eqref{eq:by_parts} is valid.
	\end{lemma}
	\begin{proof}
		
		From  \cite{DieHas07, DieHasRou09}, it follows that for $\varphi \in W^{1,p(\cdot)}(\mathbb{R}^n_+)$ there holds
		$$
		\|\varphi(\cdot,0)\|_{L^{p(\cdot)} (\mathbb{R}^{n-1})} \leq C(n,p_{-},p_+,c_{\mathrm{log}}(p)) (\|\varphi\|_{L^{p(\cdot)}(\mathbb{R}^n_+)} + \|\nabla\varphi\|_{L^{p(\cdot)}(\mathbb{R}^n_+)}).
		$$
		Thus, using a partition of unity, one gets the embedding
		$$
		\|f\|_{L^{p(\cdot)}(bM,\Lambda)} \leq C \|f\|_{1,p(\cdot),M}, \quad \|v\|_{L^{p'(\cdot)}(bM,\Lambda)} \leq C \|v\|_{1,p'(\cdot),M}.
		$$
		It remains to use the approximation Lemma~\ref{L:approx1}. \end{proof}

	\begin{lemma}\label{L:ort}
		Let $p(\cdot)$ satisfy \eqref{eq:p1}, \eqref{eq:p2}. Assume that either $f \in W_T^{1,p(\cdot)}(M,\Lambda)$ and $v\in W^{1,p'(\cdot)}(M,\Lambda)$ or $f \in W^{1,p(\cdot)}(M,\Lambda)$ and $v\in W^{1,p'(\cdot)}_N (M,\Lambda)$. Then $(df,\delta v)=0$. The result is also valid if $f \in W_T^{1,1}(M,\Lambda)$ ($W^{1,1}(M,\Lambda)$) and $v\in \mathrm{Lip}(M,\Lambda)$ (resp. $\mathrm{Lip}_N(M,\Lambda)$) or $f\in \mathrm{Lip}_T(M,\Lambda)$ (resp. $\mathrm{Lip} (M,\Lambda)$) and $v\in W^{1,1}(M,\Lambda)$ (resp. $W_N^{1,1}(M,\Lambda)$).
	\end{lemma}
	\begin{proof}
		We consider the case when $tf=0$. By using a partition of unity, one can reduce the proof to the case when the supports of both $f$ and $v$ are in the range of an (admissible) coordinate system. Let in this coordinate system the supports of $f$ and $v$ belong to a bounded domain $D$ for interior patches or to a bounded relatively open subset $D$ of $\{x^n\geq 0\}$ in the case of a boundary coordinate system. We use the Friedrichs mollifier to approximate the metric tensor $g_{ij}$ by smooth functions $g^{(\varepsilon)}_{ij}$ converging to $g_{ij}$ in the uniform metric and in any $W^{1,q}(D)$, $q<\infty$, so that $g^{(\varepsilon)}_{ij}$ is positively defined. In particular, the Christoffel coefficients corresponding to $g^{(\varepsilon)}_{ij}$ are uniformly bounded and converge to the original Christoffel coefficients pointwise and in any $L^q(D)$. Following the proof of Lemma~\ref{L:approx1} we approximate $f$ and $v$ by smooth forms $f^{(\varepsilon)}$ and $v^{(\varepsilon)}$ with compact supports in the same coordinate system so that $tf^{(\varepsilon)}=0$,  $f^{(\varepsilon)}\to f$ in $W^{1,p(\cdot)}(D)$ and $v^{(\varepsilon)}\to v$ in $W^{1,p'(\cdot)}(D)$. 
		The integration-by-parts formula yields 
		$$
		\int\limits_D \langle df^{(\varepsilon)},\delta_\varepsilon v^{(\varepsilon)} \rangle_\varepsilon \sqrt{g_\varepsilon}\, dx=0,
		$$
		where $\langle,\rangle_\varepsilon$ is the scalar product corresponding to the tensor $g^{(\varepsilon)}_{ij}$, $g_\varepsilon$ is the determinant of the matrix $g^{(\varepsilon)}_{ij}$, and $\delta_\varepsilon$ is the codifferential corresponding to the metric $g^{(\varepsilon)}_{ij}$. It remains to pass to the limit, where we use the fact that if $\alpha_j\to \alpha$ in $L^1(D)$ and $\beta_j\to \beta$ a.e. in $D$ and is uniformly bounded then $\alpha_j\beta_j\to \alpha \beta$ in $L^1(D)$.
	\end{proof}   
	
	In particular, Lemma~\ref{L:ort} immediately implies that on $C^{1,1}$ manifold if $\omega, d\omega \in W^{1,1}(M,\Lambda)$ then $d(d\omega)=0$, and if $\omega,\delta\omega \in W^{1,1}(M,\Lambda)$ then $\delta (\delta \omega)=0$.

	\begin{lemma} Let $p(\cdot)$ satisfy \eqref{eq:p1}, \eqref{eq:p2} (but with $p_{-}=1$ allowed).   $\omega\in W^{1,p(\cdot)}(M,\Lambda)$. Then $\omega \in W^{1,p(\cdot)}_T(M,\Lambda)$ ($W^{1,p(\cdot)}_N(M,\Lambda)$) if and only if $(d\omega,\xi)=0$ (resp. $(\delta \omega,\xi)=0$) for all $\xi\in \mathrm{Lip}(M,\Lambda)$.
	\end{lemma}
	\begin{proof}
		We consider the case of the tangential boundary condition. That $\omega \in W_T^{1,p(\cdot)}(M,\Lambda)$ satisfies the required relation  follows from the integration-by-parts formula. On the other hand, if $(d\omega,\xi)=0$ for all $\xi\in \mathrm{Lip}(M,\Lambda)$, from the integration-by-parts formula in any admissible coordinate system we obtain $\omega_I=0$ if $n\notin I$. Thus for a partition of unity $\eta_\alpha$ subordinate to our atlas, in a boundary chart $\eta_\alpha \omega_I$ extended by zero to $x^n<0$ belongs to $W^{1,1}(\mathbb{R}^n).$ The rest follows by the same approximation procedure as in Lemma~\ref{L:approx1} b).
	\end{proof}

	\begin{corollary} Let $M$ be of the class $C^{2,1}$. Then $d\omega \in W^{1,p(\cdot)}_T(M,\Lambda)$ for any  $\omega \in W^{2,p(\cdot)}(M,\Lambda) \cap W^{1,p(\cdot)}_T(M,\Lambda).$  If we have $\omega \in W^{2,p(\cdot)}(M,\Lambda) \cap W^{1,p(\cdot)}_N(M,\Lambda)$ then $\delta\omega \in W^{1,p(\cdot)}_N(M,\Lambda)$.
	\end{corollary}
	\begin{proof}
		
		We prove only the first statement.  Let $\omega$ be a homogeneous form. By using a partition of unity it is sufficient to consider $\omega$ with support in the range of an admissible coordinate system. In this coordinate system we use the Friedrichs mollifier as in Lemma~\ref{L:approx1} b) to approximate the components  $\omega_I$ by smooth functions $\omega_I^{(\varepsilon)}$ so that $\omega_I^{(\varepsilon)}(x',0)=0$ for $n\notin I$. Then $(d\omega)_I^{(\varepsilon)}(x',0)=0$ for $n\notin I$. Passing to the limit, since the trace of $d\omega$ on $\{x^n=0\}$ is well defined we get $(d\omega)_I(x',0)=0$ for $n\notin I$. This means $d\omega \in W^{1,p(\cdot)}_T(M,\Lambda)$. 
	\end{proof}

	\section{Partial Sobolev spaces}\label{sec:partSob}

	We say that $\omega \in L^1_{\mathrm{loc}}(M,\Lambda)$ has (weak) differential $d\omega =f\in L^1_{\mathrm{loc}}(M,\Lambda)$ if for any $\varphi \in \mathrm{Lip}_0(M,\Lambda)$  there holds
	\begin{equation}\label{eq:defd0}
		(\omega, \delta \varphi) = (f,\varphi).
	\end{equation}
	For a $k$-form $\omega$ this is equivalent to 
	\begin{equation}\label{eq:defd1}
		\int\limits_M \omega \wedge d\varphi =(-1)^{k+1} \int\limits_M f\wedge \varphi,
	\end{equation}
	where the $(n-k-1)$-form $\varphi$ has parity opposite to $\omega$.
	
	We say that a form $\omega \in L^1(M,\Lambda)$ with (weak) differential $d\omega =f\in L^1(M,\Lambda)$ has vanishing tangential component at the boundary, $t\omega=0$, if \eqref{eq:defd0} holds for any $\varphi \in \mathrm{Lip}(M,\Lambda)$. Similarly, we say that a form $\omega \in L^1_{\mathrm{loc}}(M,\Lambda)$ has (weak) co-differential $\delta\omega =g\in L^1_{\mathrm{loc}}(M,\Lambda)$ if for any $\varphi \in \mathrm{Lip}_0(M,\Lambda)$ there holds
	\begin{equation}\label{eq:defd2}
		(\omega, d \varphi) = (g,\varphi).
	\end{equation}
	For a $k$-form $\omega$ this is equivalent to 
	\begin{equation}\label{eq:defd3}
		\int\limits_M \omega \wedge \delta\varphi =(-1)^{k} \int\limits_M g\wedge \varphi,
	\end{equation}
	where the $(n-k+1)$-form $\varphi$ has parity opposite to $\omega$.
	
	We say that $\omega \in L^1(M,\Lambda)$ with (weak) codifferential $\delta\omega =g\in L^1(M,\Lambda)$ has vanishing normal component at the boundary, $n\omega=0$, if \eqref{eq:defd2} holds for any $\varphi \in \mathrm{Lip}(M,\Lambda)$.

	From the integration-by-parts formula \eqref{eq:by_parts} it follows that any $\omega \in W^{1,1}(M,\Lambda)$ has a weak differential and codifferential and they coincide with ``strong'' differential and codifferential defined by the standard expression in coordinates \eqref{eq:d}, \eqref{eq:delta}.
	
	We denote
	\begin{align*}
		W^{d,p(\cdot)}(M,\Lambda) &= \{\omega \in L^{p(\cdot)}(M,\Lambda)\ :\ d\omega \in L^{p(\cdot)}(M,\Lambda) \}, \\
		W^{d,p(\cdot)}_T(M,\Lambda) &= \{\omega \in L^{p(\cdot)}(M,\Lambda)\ :\ d\omega \in L^{p(\cdot)}(M,\Lambda)\ \text{and}\ t\omega=0\},\\
		W^{\delta,p(\cdot)}(M,\Lambda) &= \{\omega \in L^{p(\cdot)}(M,\Lambda)\ :\ \delta\omega \in L^{p(\cdot)}(M,\Lambda) \}, \\
		W^{\delta,p(\cdot)}_N(M,\Lambda) &= \{\omega \in L^{p(\cdot)}(M,\Lambda)\ :\ \delta\omega \in L^{p(\cdot)}(M,\Lambda)\ \text{and}\ n\omega=0\}.
	\end{align*}
	On the space $W^{d,p(\cdot)}(M,\Lambda)$ and its subspace $W^{d,p(\cdot)}_T(M,\Lambda)$, the norm is defined as $\|\omega\|_{L^{p(\cdot)}(M,\Lambda)} + \|d\omega\|_{L^{p(\cdot)}(M,\Lambda)}$, and on $W^{\delta,p(\cdot)}(M,\Lambda)$ and its subspace $W^{\delta,p(\cdot)}_N(M,\Lambda)$ the norm is defined as $\|\omega\|_{L^{p(\cdot)}(M,\Lambda)} + \|\delta\omega\|_{L^{p(\cdot)}(M,\Lambda)}$.
	
	Clearly, the expression \eqref{eq:forms} has a (co-)differential if and only if each of its components $f^k_e$, $f^k_o$ has a (co-)differential, and then the corresponding homogeneous components are
	\begin{align*}
		(df)^0_{*}=0,\quad  (df)^k_* = df^{k-1}_*, \quad *\in\{e,o\}, \quad k=1,\ldots, n,\\
		(\delta f)^k_* = \delta f^{k+1}_*, \quad k=0,\ldots,n-1, (\delta f)^n_*=0, \quad *\in\{e,o\}.  
	\end{align*}


	\section{Approximation results for partial Sobolev spaces}\label{ssec:approx2}
	
	Here we continue the line of Section~\ref{ssec:appox1}.
	
	The following statements reflect (locally) the two fact: i) the operator $d$ has constant coefficients in any coordinate system and ii) for a form $\omega$ defined in a domain $D \subset \mathbb{R}^n$ satisfying $t\omega =0$ on $\partial D$ its extension by zero across the boundary of the domain $D$ has differential $\mathbbm{1}_D(x) d\omega$.
	
	\begin{lemma}\label{L:approx2} Let $p(\cdot)$ satisfy \eqref{eq:p1}, \eqref{eq:p2} (with $p_{-}=1$ allowed).\\
		(i) Let $\omega\in W^{d,p(\cdot)}(M,\Lambda)$. Then there exists a sequence $\omega_\varepsilon \in C^{s-1,1}(M,\Lambda)$ such that $\omega_\varepsilon \to \omega$ in $W^{d,p(\cdot)}(M,\Lambda)$. If $\omega \in W^{d,p(\cdot)}_T(M,\Lambda)$ then one can also claim $\omega_\varepsilon\in C_0^{s-1,1}(M,\Lambda)$.\\
		(ii) Let $\omega \in W^{\delta,p(\cdot)}(M,\Lambda)$. Then there exists a sequence $\omega_\varepsilon \in C^{s-1,1}(M,\Lambda)$ such that $\omega_\varepsilon \to \omega$ in $W^{\delta,p(\cdot)}(M,\Lambda)$. If $\omega\in W^{\delta,p(\cdot)}_N(M,\Lambda)$ then one can also claim $\omega_\varepsilon\in C^{s-1,1}_0(M,\Lambda)$.
		
		In the case $p(\cdot)\equiv\infty$ the norms $|\omega_\varepsilon|$, $|d\omega_\varepsilon|$ (for i)) and $|\delta \omega_\varepsilon|$ (for ii)) are uniformly bounded, and the corresponding forms $\omega_\varepsilon$, $d\omega_\varepsilon$ (for i)) or $\omega_\varepsilon$, $\delta\omega_\varepsilon$ (for ii))  converge a.e. in $M$.
	\end{lemma}
	\begin{proof}
		Again we prove this for a homogeneous form $\omega$. We prove the first statement, the second is by duality. Let $\{\eta_j\}$ be a partition of unity subordinate to the coordinate atlas $\{U_j\}$. For $\omega^{(j)}=\omega \eta_j$ we have $\omega^{(j)} \in W^{d,p(\cdot)}(M,\Lambda)$ and
		$$
		d\omega^{(j)} = d\eta_j \wedge \omega + \eta_j d\omega.
		$$
		In interior patches the approximation argument is simpler and we focus on boundary patches. We choose a mollifier $\gamma_\varepsilon$ so that its support is in $\widetilde B^{-}_\varepsilon$. For a form $\eta$ with support in one coordinate patch we extend it in local coordinates by zero outside its support and for $x\in \mathbb{R}^n_+$ we set 
		$$
		\eta_\varepsilon (x)= \int\limits_{\mathbb{R}^{n}_+} \eta(y) \gamma_\varepsilon(x-y)\, dy,
		$$
		where the integration actually goes over $y\in x + \widetilde B^+_{\varepsilon}$. The integral is understood componentwise (that is, each form component $\eta_I$ is mollified separately). By Lemma~\ref{L:Zhikov0},  $\eta_\varepsilon\to \eta$ in $L^{p(\cdot)}(\mathbb{R}^n_+)$. Let us check that if $\eta$ has an integrable (weak) differential then
		\begin{equation}\label{eq:ver0}
			d \eta_{\varepsilon} = (d\eta)_\varepsilon.
		\end{equation}

		For a form with integrable coefficients denote
		\begin{equation}\label{eq:[eps]}
			\phi_{[\varepsilon]}(y):= \int\limits_{\mathbb{R}^n_+} \gamma_\varepsilon(x-y) \phi(x)\, dx,
		\end{equation}
		where again the integral of a form is understood componentwise, i.e. each component is mollified separately.
		If the form $\phi$ is Lipschitz then $\phi_{[\varepsilon]}(y)$ and $(d\phi)_{[\varepsilon]}(y)$ vanish if $y_n<\varepsilon/2$, since $\gamma_\varepsilon(x-y)$ vanishes if $y^n<x^n+\varepsilon/2$. Let $\phi$ be a Lipschitz form with $t\phi=0$. With a slight abuse of notation for $y$ in the support of $\eta$ we calculate
		\begin{equation}\label{eq:dphieps}
			\begin{aligned}
				d \phi_{[\varepsilon]}(y) =\int\limits_{\mathbb{R}^n_+} d_y \gamma_{\varepsilon}(x-y) \wedge \phi(x)\, dx 
				&= - \int\limits_{\mathbb{R}^n_+} d_x \gamma_{\varepsilon}(x-y) \wedge \phi(x)\, dx \\
				&=\int \limits_{\mathbb{R}^n_+} \gamma_\varepsilon(x-y) (d\phi)(x)\, dx=(d\phi)_{[\varepsilon]}(y).
			\end{aligned}
		\end{equation}
		Then
		\begin{align*}
			\int\limits_{\mathbb{R}^n_+}&   \eta_\varepsilon(x) \wedge d\phi(x)\, dx \\&= \int\limits_{\mathbb{R}^n_+} \int\limits_{\mathbb{R}^n_+} \eta(y) \gamma_\varepsilon(x-y) \wedge d\phi(x) \, dy \, dx\\
			&= \int \limits_{\mathbb{R}^n_+} \eta(y)\wedge \int\limits_{\mathbb{R}^n_{+}}  \gamma_\varepsilon(x-y) (d\phi(x))\, dx \, dy:=  \int \limits_{\mathbb{R}^n_+} \eta(y)\wedge (d\phi)_{[\varepsilon]}(y)\, dy\\
			& = (-1)^{k+1}\int\limits_{\mathbb{R}^n_+} d \eta (y) \wedge \phi_{[\varepsilon]}(y)\, dy =(-1)^{k+1} \int \limits_{\mathbb{R}^n_+} (d\eta)_\varepsilon(x)\wedge \phi(x)\, dx. 
		\end{align*}
		This proves \eqref{eq:ver0}. Thus 
		\begin{equation}\label{eq:sumj}
			d \sum_{j} \omega^{(j)}_{\varepsilon} = \sum_j\bigl( d\eta_j \wedge \omega + \sum_j \eta_j d\omega\bigr)_\varepsilon \
			\to \sum_j \bigl(d\eta_j \wedge \omega + \sum_j \eta_j d\omega\bigr)  = d\omega.
		\end{equation}
		in $L^{p(\cdot)}(M,\Lambda)$ as $\varepsilon \to 0$. 
		
		Now let $\omega \in W^{d,p(\cdot)}_T(M,\Lambda)$. From the definition we can infer that $\omega^{(j)} \in W^{d,p(\cdot)}_T(M,\Lambda)$. In this case in local coordinates we extend $\omega^{(j)}$ by zero to $x_n<0$, and take mollifier $\gamma_\varepsilon$ with support in $\widetilde B_\varepsilon^+$. Then $(\omega^{(j)})_\varepsilon (x)=0$ if $x^n<\varepsilon/2$. The relation $(d\phi)_{[\varepsilon]}(y) = d \phi_{[\varepsilon]}(y)$ a.e. in $y^n>0$ is straightforward for any Lipschitz form $\phi$, and for $\eta=\omega^{(j)}$ in the chain of calculations between \eqref{eq:dphieps} and \eqref{eq:sumj} the ``integration-by-parts'' (passage between lines 3 and 4) follows from the definition of $W^{d,p(\cdot)}_T(M,\Lambda)$. All the previous calculations are repeated, which gives \eqref{eq:ver0}, and the rest follows. 
	\end{proof}
	
	\begin{lemma}\label{L:approx3}
		(i) Suppose  $\phi \in \mathrm{Lip}_T(M,\Lambda)$. Then there exists a sequence $\phi_\varepsilon \in C^{s-1,1}_0(M,\Lambda)$ such that $\phi_\varepsilon \to \phi$ in $C(M,\Lambda)$, $\sup_{\varepsilon >0}\max_{x\in M} |d\phi_\varepsilon|(x) \leq C$, and $d\phi_\varepsilon \to d \phi$ a.e. and in any $L^q(M,\Lambda)$, $1\leq q< \infty$.\\ (ii) Let $\phi \in \mathrm{Lip}_N(M,\Lambda)$. Then there exists a sequence  $\phi_\varepsilon \in C^{s-1,1}_0(M,\Lambda)$ such that $\phi_\varepsilon \to \phi$ in $C(M,\Lambda)$, $\sup_{\varepsilon >0}\max_{x\in M} |\delta\phi_\varepsilon|(x) \leq C$, and $\delta\phi_\varepsilon \to \delta \phi$ a.e. and in any $L^q(M,\Lambda)$, $1\leq q< \infty$.
	\end{lemma}
	
	\begin{proof}
		We prove  the first statement, the second is by the Hodge duality. As above, we assume that $\phi$ is a homogeneous form. Take a partition of unity $\eta_j$ subordinate to the coordinate cover. For $\phi_j$ we define $(\phi_j)_{[\varepsilon]}$ as in \eqref{eq:[eps]} with $\gamma_\varepsilon$ supported in $\widetilde B^{-}_{\varepsilon}$. Then $(\phi_j)_{[\varepsilon]}$ vanish near $bM$, by the calculation in \eqref{eq:dphieps} we have $d (\phi_j)_{[\varepsilon]} = (d\phi_j)_{[\varepsilon]}$, $(\phi_{j})_{[\varepsilon]} \to \phi_j$ in $C(M,\Lambda)$ and $(d\phi_{j})_{[\varepsilon]} \to d\phi_j$ in any $L^q(M, \Lambda)$, $1\leq q<\infty$, thus  
		$$
		\phi_\varepsilon = \sum_j (\phi_j)_{[\varepsilon]}
		$$
		is the required approximation (see \eqref{eq:sumj} for the convergence of $d\phi_\varepsilon$). 
	\end{proof}
	
	From approximation results of Lemmas~\ref{L:approx2},~\ref{L:approx3} we immediately get the following 

	\begin{corollary}\label{C:approx2}Let $p(\cdot)$ satisfy \eqref{eq:p1}, \eqref{eq:p2}.
		\begin{itemize}
			\item[(i)]  If $\omega \in W^{d,p(\cdot)}(M,\Lambda)$, then \eqref{eq:defd0} holds for all $\varphi \in W_{N}^{\delta,p'(\cdot)}(M,\Lambda)$. 
			\item[(ii)]  If $\omega \in W^{d,p(\cdot)}_T(M,\Lambda)$, then \eqref{eq:defd0} holds for all $\varphi \in W^{\delta,p'(\cdot)}(M,\Lambda)$.
			\item[(iii)] If $\omega \in W^{\delta ,p(\cdot)}(M,\Lambda)$, then \eqref{eq:defd2} holds for all $\varphi \in W_{T}^{d,p'(\cdot)}(M,\Lambda)$.
			\item[(iv)] If $\omega \in W^{\delta ,p(\cdot)}_N(M,\Lambda)$, then \eqref{eq:defd2} holds for all $\varphi \in W^{d,p'(\cdot)}(M,\Lambda)$.
		\end{itemize}
		This statement is also valid if $p(\cdot)\equiv 1$, $p'(\cdot)\equiv \infty$, or $p(\cdot)\equiv \infty$, $p'(\cdot)\equiv 1$.
	\end{corollary}

	Using the approximation Lemma~\ref{L:approx2} along with the definition of the weak (co)differential we infer that $$d^2=0 \quad \text{ on } W^{d,1} (M,\Lambda) \qquad \text{ and } \qquad \delta^2=0 \quad \text{ on }W^{\delta,1}(M,\Lambda).$$ For $\omega \in W^{d,1}(M,\Lambda)$ and any $\xi \in \mathrm{Lip}_0(M,\Lambda)$ we have
	$$
	(d^2 \omega, \xi) = (d\omega, \delta\xi) = \lim_{\varepsilon\to 0} (d\omega_\varepsilon,\delta \xi)=0.
	$$
	Here $\omega_\varepsilon \in C^{s-1,1}(M,\Lambda)$, $\omega_\varepsilon\to \omega$ in $W^{d,1}(M,\Lambda)$, the first identity is the definition of the weak differential of $d\omega$ and the third one is by Lemma~\ref{L:ort}.

	\section{Harmonic fields}
	
	Let $M$ be of the class $C^{s+2,1}$, $s\in \{0\}\cup \mathbb{N}$, and let $n=\mathrm{dim}\, M$.
	
	We introduce the spaces of even and odd harmonic fields (i.e. $W^{1,2}(M,\Lambda)$ forms with zero differential and codifferential) of degree $r$ with vanishing tangential part on the boundary $\mathcal{H}_T(M,\Lambda^r_*)$, $*\in\{e,o\}$, the spaces of even and odd harmonic forms  of degree r with vanishing normal part on the boundary $\mathcal{H}_N(M,\Lambda^r_*)$, $*\in\{e,o\}$. Then let 
	\begin{align*}
		\mathcal{H}_T(M,\Lambda_*) &= \bigoplus_{r=0}^{\mathrm{dim}\, M} \mathcal{H}_T(M,\Lambda^r_*), \quad *\in\{e,o\},\\
		\mathcal{H}_N(M,\Lambda_*) &= \bigoplus_{r=0}^{\mathrm{dim}\, M} \mathcal{H}_N(M,\Lambda^r_*), \quad *\in\{e,o\},\\
		\mathcal{H}_T(M) &= \mathcal{H}_T(M,\Lambda_e) \bigoplus \mathcal{H}_T(M,\Lambda_o),\\
		\mathcal{H}_N(M) &= \mathcal{H}_N(M,\Lambda_e) \bigoplus \mathcal{H}_N(M,\Lambda_o).
	\end{align*}

	From the classical results of elliptic theory \cite{Morrey1966}, it follows that the spaces $\mathcal{H}_T(M)$ and $\mathcal{H}_N(M)$ are finite dimensional and
	\begin{gather*}
		\mathcal{H}_T(M,\Lambda), \mathcal{H}_N(M,\Lambda) \subset W^{s+2,q}(M,\Lambda) \quad \text{for any}\quad  1\leq q<\infty,\\
		\mathcal{H}_T(M,\Lambda), \mathcal{H}_N(M,\Lambda) \subset C^{s+1,\alpha}(M,\Lambda)  \quad \text{for any}\quad \alpha \in (0,1).
	\end{gather*}
	
	The Hodge star operator sets an $L^2(M,\Lambda)$-isometry between $\mathcal{H}_N(M,\Lambda^r_e)$ and $\mathcal{H}_T(M,\Lambda^{n-r}_o)$, and between $\mathcal{H}_N(M,\Lambda^r_o)$ and $\mathcal{H}_T(M,\Lambda^{n-r}_e)$.
	
	Let $h^r_{*,N,j}$, $j=1,\ldots, B_{*,r}$, be an orthonormal basis in the sense of $L^2(M,\Lambda^r_*)$, $* \in \{e,o\}$, of $\mathcal{H}_N(M,\Lambda^r_*)$. On $L^1(M,\Lambda)$ define the projector 
	\begin{equation}\label{eq:projN}
		\mathcal{P}_N \omega = \sum_{r=0}^{\mathrm{dim}\, M} \sum_{j=1}^{B_{e,r}} (\omega,h^r_{e,N,j})  h^r_{e,N,j} + \sum_{r=0}^{\mathrm{dim}\, M} \sum_{j=1}^{B_{o,r}} (\omega,h^r_{o,N,j})  h^r_{o,N,j}.
	\end{equation}
	Then $h^r_{o,T,j}=  * h^{n-r}_{e,N,j}$, $j=1,\ldots, B_{e,n-r}$, is the orthonormal basis in the sense of $L^2(M,\Lambda^r_*)$ in $\mathcal{H}_T(M,\Lambda^{r}_o)$, and $h^r_{e,T,j}=* h^{n-r}_{o,N,j}$, $j=1,\ldots, B_{o,n-r}$, is the orthonormal basis in $\mathcal{H}_T(M,\Lambda^{r}_e)$. On $L^1(M,\Lambda)$ define the projector 
	\begin{equation}\label{eq:projT}
		\mathcal{P}_T \omega = \sum_{r=0}^{\mathrm{dim}\, M} \sum_{j=1}^{B_{o,n-r}}  (\omega,h^r_{e,T,j}) h^r_{e,T,j} + \sum_{r=0}^{\mathrm{dim}\, M} \sum_{j=1}^{B_{e,n-r}}  (\omega,h^r_{o,T,j}) h^r_{o,T,j}.
	\end{equation}

	The projectors  $\mathcal{P}_T$ and $\mathcal{P}_N$ are characterized by 
	\begin{gather*}
		(\omega - \mathcal{P}_T \omega, h_T) =0 \quad \text{for all} \quad h_T \in \mathcal{H}_T(M),\\
		(\omega - \mathcal{P}_N \omega, h_N) =0 \quad \text{for all} \quad h_N \in \mathcal{H}_N(M)
	\end{gather*}
	In $L^2(M,\Lambda)$ the projectors $\mathcal{P}_T$ and $\mathcal{P}_N$ are orthogonal projectors on $\mathcal{H}_T(M)$ and $\mathcal{H}_N(M)$ respectively.
	
	On an orientable manifold all forms can be considered as even or as odd, so 
	\begin{gather*}
		\mathrm{dim}\,\mathcal{H}_T(M,\Lambda_e^r) = \mathrm{dim}\,\mathcal{H}_T(N,\Lambda_o^r), \quad \mathrm{dim}\,\mathcal{H}_N(M,\Lambda_e^r)=\mathrm{dim}\,\mathcal{H}_N(M,\Lambda_o^r),\\
		\mathrm{dim}\,\mathcal{H}_T(M,\Lambda_e^r) = B_{n-r}, \quad \mathrm{dim}\,\mathcal{H}_N(M,\Lambda_e^r) = B_{r},
	\end{gather*}
	(see \cite{DufSpe56}) where $B_r$ is the $r$-th Betti number of $M$ (rank of the $r$-th homology group of $M$), the number $B_{n-r}$ is equal to the $r$-th relative (mod $bM$) Betti number, the number $B_0$ represents the number of connected components of $M$, $B_n=0$, and if $M$ is contractible then $B_r=0$ for $r=1,\ldots,n-1$. 

In particular, let $M$ be a bounded domain in $\mathbb{R}^3$. Then the dimension of the spaces of harmonic vector fields (i.e. vector fields $\vec{u}$ with $\mathrm{div} \, \vec{u}=0$, $\mathrm{rot}\, \vec{u}=0$ ) with vanishing tangential (normal) part on $\partial M$ is $B_2$ (the number of holes in $M$) (resp. $B_1$ --- the number of independent ``circuits'' in $M$). 

Let $\mathcal{H}(M)$ denote the set of all $h\in L^1_{\mathrm{loc}}(M)$ with $dh =0$ and $\delta h=0$.

\chapter{Potential Theory for Variable Exponent Spaces}\label{sec:auxiliary}

We shall use a number of results on classical integral operators in variable exponent spaces, see \cite{DieHHR11}. 
\section{Maximal function and Riesz potential estimates}
Let
$$
I_\alpha f(x) = \int_{\mathbb{R}^{n}} \frac{f(y)}{|x-y|^{n-\alpha}}\, dy \qquad \text{ and } \qquad Mf\left(x\right) = \sup\limits_{\substack{x \in B\\ B\subset \mathbb{R}^{n} \text{ is a ball}}}\fint_{B} \left\lvert f \left(y\right)\right\rvert\, dy.
$$
For an exponent $p\left(\cdot\right),$ we write $p \in \mathcal{P}^{log}(\Omega)$, respectively $ p \in \mathcal{P}^{log}(\mathbb{R}^{n}),$ if $p$ satisfies \eqref{eq:p1} and \eqref{eq:p2} with $M$ replaced by $\Omega$, respectively $\mathbb{R}^{n}$, with $\mathrm{dist}(x,y)$ the standard Euclidian distance. Also, $c_{\mathrm{log}}(p)$ denotes the best constant for which \eqref{eq:p2} holds.  
\begin{lemma}\label{L:Di1} ({\it Corollary 4.3.11.}, \cite{DieHHR11}) Let $p \in \mathcal{P}^{\mathrm{log}}(\Omega)$ with $p^{-}>1$. Then there exists $K>0$ only depending on $c_{\mathrm{log}}(p)$ and the dimension $n$ such that 
	$$
	\|Mf\|_{L^{p(\cdot)}(\Omega)} \leq K (p^-)' \|f\|_{L^{p(\cdot)}(\Omega)}
	$$
	for all $f\in L^{p(\cdot)}(\Omega)$.
\end{lemma}
Let $B_R$ be a ball with radius $R$ and $B_R(x)$ be the ball of radius $R$ centered at $x\in \mathbb{R}^n$. We shall use the following standard estimate. 

\begin{lemma}\label{L:Di2} ({\it Lemma 6.1.4.}, \cite{DieHHR11}) For $R,\alpha>0$ there holds 
	$$
	\int\limits_{B_R(x)} \frac{|f(y)|}{|x-y|^{n-\alpha}}\, dy \leq C(\alpha) R^\alpha Mf(x).
	$$
\end{lemma}

In \cite{DieHHR11} this result is stated for $0<\alpha<n$, and for $\alpha\geq n$ this estimate is trivial. 

For $n=2$ we shall also use the following estimate for the logarithmic potential which is proved exactly as the previous result:
\begin{equation}\label{eq:log0}
	\int\limits_{B_R(x)} |f(y)|\log \frac{R}{|x-y|}\, dy \leq C(n) R^{n} M f(x).
\end{equation}

\begin{corollary}\label{C:Di}
	For a function with support in $B_R$, $R\leq 1$, and $\alpha>0$ there holds 
	$$
	\|I_\alpha f\|_{L^{p(\cdot)}(B_R)} \leq C(n,p_{-},p_{+},c_{\mathrm{log}}(p),\alpha) R^{\alpha} \|f\|_{L^{p(\cdot)}(B_R)}.
	$$
\end{corollary}
\begin{proof}
	Immediately follows from Lemmas~\ref{L:Di1}, \ref{L:Di2}.
\end{proof}

For the logarithmic potential 
$$
I_{\mathrm{log},R} f (x)=\int\limits_{B_R} f(y) \log \frac{|x-y|}{R}\, dy
$$
an analogue of Corollary~\ref{C:Di} is given by the estimate
\begin{equation}\label{eq:log1}
	\|I_{\mathrm{log},R} f \|_{p(\cdot), B_R} \leq C(n) R^n \|f\|_{p(\cdot),B_R},
\end{equation}
which follows from \eqref{eq:log0} and the maximal function estimate.

\begin{lemma}\label{L:Di3}({\it  Theorem 6.1.9 (page 203)}, \cite{DieHHR11}) Let $p \in \mathcal{P}^{\mathrm{log}}(\mathbb{R}^n)$, $0<\alpha <n$, $1<p_{-} \leq p_{+} <n/\alpha$, and $1/p^\sharp = 1/p - \alpha/n$. Then  
	$$
	\|I_\alpha f\|_{p^\sharp(\cdot)} \leq C(n,p_{-},p_{+},c_{\mathrm{log}}(p),\alpha) \|f\|_{p(\cdot)}.
	$$
\end{lemma}

\begin{lemma}\label{L:DD} Let $p \in \mathcal{P}^{\mathrm{log}}(\mathbb{R}^n)$. Denote $\varkappa = n/(n-s)$, $s\in [1,n-1]$. Let $\operatorname{supp}\, f\subset \overline{B}_R$, $R\leq 1$. Then 
	$$
	\|I_s f\|_{\varkappa p(\cdot), B_R} \leq C(n,p_{-},p_{+},c_{\mathrm{log}} (p)) \|f\|_{p(\cdot)} R^{s/p_{-}'}.
	$$
\end{lemma}
\begin{proof} Let $q$ be such that $\varkappa p = nq/(n-sq)$, or
	$$
	\frac{1}{p} - \frac{s}{np} = \frac{1}{q} - \frac{s}{n}.
	$$ 
	This gives 
	$$
	\frac{1}{q} = \frac{s}{n} + \frac{1}{p} - \frac{s}{np}, \quad q = \frac{np}{n+sp-s},
	$$
	where $1<q<\min(p,n/s)$ and 
	\begin{align*}
		\frac{n}{s}-q = \frac{n}{s}\cdot \frac{n-s}{n+sp-s} &\geq  \frac{1}{n p_{+}},\\ 
		q-1 = \frac{(p-1)(n-s)}{n+s(p-1)} &\geq  \frac{(p_{-}-1)}{np_{-}}.
	\end{align*}
	Then by Lemma~\ref{L:Di3} and the H\"older inequality,
	\begin{align*}
		\|I_s f\|_{\varkappa p(\cdot),B_R} \leq C \|f\|_{q(\cdot), B_R} &\leq 2C \|f\|_{p(\cdot),B_R} \|1\|_{np'(\cdot)/s,B_R} \\&= 2C \|f\|_{p(\cdot),B_R} \|1\|_{p'(\cdot),B_R}^{s/n}. 
	\end{align*}
	
	It remains to use the standard estimate $\|1\|_{z(\cdot),B_R} \lesssim R^{n/\bar{z}}$ for $z\in \mathcal{P}^{\mathrm{log}}(\mathbb{R}^n)$ for the last factor, where $\bar z$ is the average of $z$ over $B_R$.
\end{proof}

\section{Interpolation in variable exponent spaces}
In this Section $B_R$ is ball of radius $R$, if it centered at $x^n=0$ we denote $B_R^{+} = B_R \cap \{x^n>0\}$.

\begin{corollary}\label{C:Di00}
	Let $f\in W^{1,p(\cdot)}(B_R)$ satisfy $f=0$ on $\partial B_R$. Then 
	\begin{align*}
		\|f\|_{p(\cdot),B_R} &\leq C(n,p_{-},p_{+},c_{\mathrm{log}}(p)) R \|\nabla f\|_{p(\cdot),B_R},\\
		\|f\|_{\varkappa p(\cdot), B_R} &\leq C(n,p_{-},p_{+},c_{\mathrm{log}}(p)) R^{1/(np_{-}')} \|\nabla f\|_{p(\cdot),B_R}.
	\end{align*}
	The same estimate is valid in if $B_R$ is replaced by $B_{R}^+$ and $\partial B_R$ is replaced by $\partial B_R \cap \{x_n\geq 0\}$.
\end{corollary}
\begin{proof}
	We prove this for $B_R^+$. Consider the even extension of $f$ to $\{x_n<0\}$ across $\{x_n=0\}$, and then extend $f$ by zero outside $B_R$. We denote this extension again by $f$. Then 
	$$
	|f(x)| \leq C(n) I_1 [|\nabla f|](x).
	$$
	It remains to use Corollary~\ref{C:Di} and Lemma~\ref{L:DD}.
\end{proof}

\begin{corollary}\label{C:Di01}
	Let $f\in W^{k,p(\cdot)}(B_R)$ vanish on $\partial B_R$ together with all its partial derivatives up to the order $k-1$. Then 
	$$
	\sum_{l=0}^{k-1} R^{k-l} \|\nabla^l f\|_{p(\cdot),B_R} \leq C(n,p_{-},p_{+},c_{\mathrm{log}}(p),k) \|\nabla^k f\|_{p(\cdot), B_R}.
	$$
	The same estimate is valid if $B_R$ is replaced by $B_R^+$ and $\partial B_R$ is replaced by $\partial B_R \cap \{x_n>0\}$.
\end{corollary}

\begin{proof}
	Follows from applying Corollary~\ref{C:Di00} to $\nabla^l f$ for $l=k-1$, \ldots, $0$.
\end{proof}

\begin{corollary}\label{C:DiInter}
	For $f\in W^{k,p(\cdot)}(B_R)$, $k\in \mathbb{N}$ there holds 
	$$
	\sum_{l=0}^{k-1}  R^{l+1-k}\|\nabla^l f\|_{p(\cdot),B_R} \leq CR \|\nabla^k f\|_{p(\cdot),B_R} + CR^{-n-1} \|f\|_{1,B_R}
	$$
	with a constant $C=C(n,p_{-},p_{+},c_{\mathrm{log}}(p),k)$.
\end{corollary}

\begin{proof}
	The first statement immediately follows from Lemmas~\ref{L:Di1}, \ref{L:Di2}. 
	We use the Sobolev representation (see \cite{Maz11} {Section 1.1.10}): there exists a polynomial $Q(x)$ such that 
	$$
	f(x)-Q(x)= \sum_{\alpha=k}\int\limits_{B_R} \frac{\varphi_\alpha(x,y) \nabla^{\alpha} f(y)}{|x-y|^{n-k}},
	$$
	where for $l=0,\ldots k-1$ and $x,y\in B_R$ there holds
	$$
	R^l|\nabla^l Q(x)| \leq C R^{-n}\int\limits_{B_R} |f(y)|\, dy,\qquad
	R^l|\nabla_x^l \varphi_\alpha(x,y)|\leq C.
	$$
	Thus 
	\begin{align*}
		\sum_{l=0}^{k-1} R^l \|\nabla^l Q\|_{p(\cdot),B_R} &\leq C R^{-n}\int\limits_{{B}_R} |f(y)|\, dy,\\
		\sum_{l=0}^{k-1} R^l \|\nabla^l (f-Q)\|_{p(\cdot),B_R} &\leq C\sum_{l=0}^{k-1} \sum_{\alpha+\beta=l} R^{-\alpha} \|I_{k-\beta}[\nabla^k f ]\|_{p(\cdot),B_R)}\\
		&\leq C R^k \|\nabla^k f\|_{p(\cdot),B_R},
	\end{align*}
	whence the second claim.
\end{proof}
\begin{lemma}\label{C:Di1}
	Let $k\in \mathbb{N}$. On a  Riemannian $C^{k,1}$ manifold for any $\varepsilon>0$ there exists a constant $C(\varepsilon)$ such that any $\omega\in W^{k,p(\cdot)}(M,\Lambda)$ satisfies
	$$
	\|\omega\|_{k-1,p(\cdot),M} \leq \varepsilon \|\omega\|_{k-1,p(\cdot),M} + C(\varepsilon) \|\omega\|_{1,M}.
	$$
\end{lemma}

\begin{proof}
	To prove the interpolation estimates on manifolds it is sufficient to consider one coordinate chart. Covering the support of $\omega\eta_\alpha$, where $\eta_\alpha$ is an element of the partition of unity, in this chart by balls with sufficiently small radius $C\varepsilon$ and finite overlap, using in each of these cubes the interpolation estimate, and summing over these cubes, we get the required estimate. 
\end{proof}

\section{Calderon-Zygmund operators}

The following theory of Calderon-Zygmund operators in variable exponent spaces was developed by  L. Diening and  M. Růžička, see [\cite{DieHHR11}, Section 6.3].

\noindent{\bf Definition.} A kernel $k(\cdot,\cdot)$ is called a standard kernel if it satisfies the following estimates for all distinct $x,y$ and for all $z$ with $|x-z|< |x-y|/2$:
\begin{align*}
	|k(x,y)| &\leq c|x-y|^{-n},\\
	|k(x,y)-k(z,y)| &\leq c |x-z|^\delta\cdot |x-y|^{-n-\delta},\\
	|k(y,x)- k(y,z)| &\leq c |x-z|^\delta\cdot |x-y|^{-n-\delta},
\end{align*}
with some $\delta>0$. For  $\varepsilon>0$ define the integral operator $T_\varepsilon$ on $f\in C_0^\infty(\mathbb{R}^n)$ by
$$
T_\varepsilon f(x) = \int\limits_{|x-y|>\varepsilon} k(x,y) f(y)\, dy.
$$

\begin{lemma}\label{L:Di5}({\it Corollary 6.3.13. \cite{DieHHR11}}) Let $k$ be a standard kernel on $\mathbb{R}^n \times \mathbb{R}^n$. Assume that $N(x, z):= k(x, x-z)$ is homogeneous of degree $-n$ in $z$ and that for every $x$ the integral of $N(x,z)$ over the sphere $|z| = 1$ equals zero.  Let $p\in \mathcal{P}^{log}(\mathbb{R}^n)$ with $1<p_{-}\leq p_{+}<\infty$. Then the operators $T_\varepsilon$ are uniformly bounded on $L^{p(\cdot)}(\mathbb{R}^n)$ with respect to $\varepsilon > 0$. Moreover, $Tf(x) := \lim\limits_{\varepsilon \to 0} T_\varepsilon f(x)$ exists almost everywhere and $T_\varepsilon f \to Tf$ in $L^{p(\cdot)}(\mathbb{R}^n)$. In particular, $T$ is bounded on $L^{p(\cdot)}(\mathbb{R}^n)$.
\end{lemma}

\section{Potentials in the full space}\label{sec:space}
For $n\geq 3$ denote
$$
K_0(z) =  \frac{|z|^{2-n}}{(2-n)\omega_n}, 
$$
where $\omega_n$ is the $(n-1)$-volume of the unit sphere $|x|=1$ in $\mathbb{R}^n$. For $n=2$ we set 
$$
K_0(z) = \frac{1}{2\pi} \log |x|.
$$
Further we shall mainly deal with the case $n\geq 3$ and write the corresponding representations. So by default $n\geq 3$. All the results are also valid for the two-dimensional case, we mark the cases when the corresponding 2D result deviates from the result for $n\geq 3$.

Let $P[f]$ be the Newtonian potential of $f$,
$$
P[f](x) = \int\limits_{\mathbb{R}^n} K_0(x-y) f(y)\, dy = - (n-2)^{-1} \omega_n^{-1} \int\limits_{\mathbb{R}^n} |x-y|^{2-n} f(y)\, dy.
$$
for $n\geq 3$ and the logarithmic potential
$$
P[f](x) = \frac{1}{2\pi} \int\limits_{\mathbb{R}^2} f(y)\log |x-y|\, dy \quad \text{for}\quad n=2.
$$

Let $Q_\alpha[f]$ be the potential solving $\triangle Q_\alpha[f] = - f_{,\alpha}$,
$$
Q_\alpha[f] = -\int\limits_{\mathbb{R}^n} K_{0,\alpha}(x-y) f (y)\, dy = \omega_n^{-1} \int\limits_{\mathbb{R}^n} |x-y|^{-n} (y_\alpha-x_\alpha) f(y)\, dy.
$$
Let $Q[e]$, $e=e^\alpha$, $\alpha=1,\ldots,n$, be the potential solving $\triangle u = - e^\alpha_{,\alpha}$, that is 
\begin{align*}
	Q[e](x) =Q_\alpha [e^\alpha](x)&= -\int\limits_{\mathbb{R}^n} K_{0,\alpha}(x-y) e^\alpha (y)\, dy \\
	&= \omega_n^{-1} \int\limits_{\mathbb{R}^n} |x-y|^{-n} (y_\alpha-x_\alpha) e^\alpha(y)\, dy.
\end{align*}

Let $s\in \mathbb{N}$.  Equip $W^{s,p(\cdot)}(B_{2R})$ with the norm 
$$
\|u\|^*_{s,p(\cdot),B_{2R}} = \sum_{k=0}^s R^{k-s}\|\nabla^k u\|^*_{p(\cdot),B_{2R}},
$$
The following two lemmas summarize the basic properties of the volume potentials in $\mathbb{R}^n$ in variable exponent spaces. Let $\varkappa = n/(n-1)$. For $n=2$ we define 
	\begin{equation}\label{eq:c1}
		c(f; B_{2R}) = \frac{\log (2R)}{2\pi} \int\limits_{B_{2R}} f\, dx.
	\end{equation}

\begin{lemma}\label{L:pa0}
	For $f\in L^p(B_{2R})$ with support in $B_{2R}$ there holds
	\begin{align*}
		\| Q_\alpha[f]\|^*_{1,p(\cdot), B_{2R}} &\leq C(n, p_{-},p_{+}, c_{\mathrm{log} }(p)) \|f\|_{p(\cdot), B_{2R}},\\
		\|P[f]\|^*_{1,p(\cdot), B_{2R}} &\leq C(n, p_{-},p_{+}, c_{\mathrm{log} }(p)) R \|f\|_{p(\cdot), B_{2R}},\\
		\|P[f]\|^*_{1,\varkappa p(\cdot), B_{2R}} &\leq C(n, p_{-},p_{+}, c_{\mathrm{log} }(p))\|f\|_{p(\cdot), B_{2R}},\\
		\|P[f]\|^*_{2,p(\cdot),B_{2R}} &\leq C (n,p_{-},p_{+}, c_{\mathrm{log} }(p)) \|f\|_{p(\cdot), B_{2R}}.
	\end{align*}
	For $n=2$ the estimates for $P[f]$ are valid with $P[f]$ replaced by $P[f] - c(f;B_{2R})$. 
\end{lemma}
\begin{proof}
	There holds
	\begin{gather*}
		|P[f]| \leq C(n) I_2 [f],\qquad |\nabla P[f]|, |Q_\alpha[f]|\leq C(n) I_1[f],\\
		|\nabla Q_\alpha[f]|,\ |\nabla^2 P[f]| \leq C(n)(|f| + |\mathrm{CZ}[f]|),
	\end{gather*} 
	where $\mathrm{CZ}[g]$ is a Calderon-Zygmund type singular integral operator. It remains to use the Calderon-Zygmund theory in variable exponent spaces from Lemma~\ref{L:Di5}, together with Corollary~\ref{C:Di} and Lemma~\ref{L:DD}. For the third estimate we also use the obvious relation $P[f] \leq C(n)R I_1[f]$ to get the estimate of $P[f]$ in $L^{\varkappa p(\cdot)}(B_{2R})$. 
	
	For $n=2$ we have
		$$
		P[f] - c(f;B_{2R})= \frac{1}{2\pi} \int\limits_{B_{2R}} f(y) \log \frac{|x-y|}{2R}\, dy.
		$$
		It remains to use the estimate \eqref{eq:log1} to evaluate the norm of $P[f]$ in $L^{p(\cdot)}(B_{2R})$, with the rest of the estimates being the same as for $n\geq 3$.
\end{proof}

\begin{remark}
	It is easy to see that for the additional term arising in the planar case we have 
		$$
		|c(f;B_{2R})| \leq C(p_{-},p_{+},c_{\mathrm{log}}(p)) R^2 \|f\|_{p(\cdot),B_{2R}}  |\log (2R)|.
		$$
	
\end{remark}

\begin{corollary}\label{C:pa0}
	Let $f \in W^{s,p(\cdot)}(B_{2R})$, $s\in \{0\}\cup\mathbb{N}$, with support in $B_{2R}$. Then 
	$$
	\|P[f]\|^*_{s+2,p(\cdot),B_{2R}} \leq C \|f\|^*_{s,p(\cdot), B_{2R}}.
	$$
	Let $f\in W^{s,p(\cdot)}(B_{2R})$, $s\in \{0\}\cup\mathbb{N}$, with support in $B_{2R}$. Then 
	$$
	\|Q_{\alpha}[f]\|^*_{s+1,p(\cdot),B_{2R}}  \leq C \|f\|_{s,p(\cdot), B_{2R}}^*.
	$$
	For $n=2$ the estimate for $P[f]$ is valid with $P[f]$ replaced by $P[f] - c(f;B_{2R})$.
\end{corollary}
\begin{proof}
	Follows from the standard relations
	$$
	(P[f])_{,\alpha} = P[f_{,\alpha}], \quad (Q_\alpha[f])_{,\beta} = Q_\alpha[f_{,\beta}],
	$$
	and estimates of Lemma~\ref{L:pa0}.
\end{proof}

For $f\in L^p(B_{2R})$ with support in $B_{2R}$, the potential $v=P[f] \in W^{2,p}_{\mathrm{loc}}(\mathbb{R}^n)$ satisfies
$$
\int\limits_{\mathbb{R}^n} (\nabla v \cdot \nabla \zeta + f\zeta)\, dx =0
$$
for all $\zeta \in C_0^\infty(\mathbb{R}^n)$, and the potential $w = Q_\alpha[f]\in W^{1,p}_{\mathrm{loc}}(\mathbb{R}^n)$ satisfies 
$$
\int\limits_{\mathbb{R}^n} (\nabla v \cdot \nabla \zeta + f\zeta_{,\alpha})\, dx =0
$$
for all $\zeta \in C_0^\infty(\mathbb{R}^n)$. The potentials $P[f]$ and $Q_\alpha[f]$ decay at infinity with the exception of the case $n=2$ where $P[f]$ decays at infinity under the additional condition of zero mean of $f$ over $B_{2R}$.

\section{Half-space potentials}\label{sec:halfspace}

Let $\rho$ be a positive number. In all the subsequent paper but one place (namely, in the proof of Corollary~\ref{C:DiInter}) we use the value $\rho =1$. Denote 
$$
\triangle_\rho = \sum_{j=1}^{n-1} \frac{\partial^2}{\partial (x^j)^2}+\rho^2 \frac{\partial^2}{\partial (x^n)^2},\quad
K_0(x;\rho) = \rho^{-1} K_0(x',\rho^{-1}x^n).
$$
There holds $\triangle_\rho K_0(x;\rho) = \delta(x)$. Clearly, $K_0(x;1)=K_0(x)$. In this Section $C$ is a constant which depends on $n$, $p_{-}$, $p_{+}$, $c_{\mathrm{log}}(p)$, and on $\rho$. 

Let $B_{2R}^+ =B_{2R}\cap \{x_n>0\}$, $\overline{B}_{2R}^+=B_{2R} \cap \{x_n \geq 0\}$.

We can assume that the the variable exponent $p(\cdot)$ is extended to $\mathbb{R}^n$ by first taking even reflection across $\{x_n=0\}$ if we work in boundary patches, and then using the standard McShane-Whitney extension procedure from the ball $B_{2R}$, which keeps the same bounds and log-H\"older modulus of continuity: $$p(x)=\max(\sup_{y\in B_{2R}}( p(y) - \omega(|x-y|)),p_{-}),$$ where $\omega(t) = L(\log (e +t^{-1}))^{-1}$, is the log-H\"older modulus of continuity of $p$ on $B_{2R}$. In particular, for boundary patches this extension is even with respect to $x^n$.

Let 
\begin{itemize}
	\item $y^*$ denote the reflection of $y$ across $y_n=0$, $(y',y^n)^* = (y',-y^n)$;
	
	\item $\mathcal{E}^D$ denote the odd extension of a function across $y_n=0$ from the upper half-space:
	$$
	\mathcal{E}^D[f](y',y_n) = -f(y',-y_n), \quad y_n<0;
	$$
	
	\item  $\mathcal{E}^N$ denote the even extension of a function across $y_n=0$ from the upper half-space, 
	$$
	\mathcal{E}^N[f](y',y_n) = f(y',-y_n), \quad y_n<0;
	$$
	
	\item $ g_D(x,y):= K_0(x-y;\rho) - K_0(x-y^*;\rho)$ ;
	
	\item $g_N(x,y):= K_0(x-y;\rho) + K_0(x-y^*;\rho)$ ;
	
\end{itemize}

For a scalar function $f$ defined on $\mathbb{R}^n_+$ we set
\begin{align*}
	P^D[f] = P[\mathcal{E}^D g] &= \int\limits_{\mathbb{R}^n_+} g_D(x,y) f(y)\, dy,\\
	P^N[f] = P[\mathcal{E}^N g] &= \int\limits_{\mathbb{R}^n_+} g_N(x,y) f(y)\, dy,\\
	Q^D_{\alpha}[f] = Q_{\alpha} [\mathcal{E}^D f] &= \int \limits_{\mathbb{R}^n_+} \frac{\partial g_D(x,y)}{\partial y^\alpha} f(y)\, dy, \quad \alpha<n,\\
	Q^D_{\alpha}[f] = Q_{\alpha} [\mathcal{E}^N f] &= \int \limits_{\mathbb{R}^n_+} \frac{\partial g_D(x,y)}{\partial y^\alpha} f(y)\, dy, \quad \alpha =n,\\
	Q^N_{\alpha}[f] = Q_{\alpha} [\mathcal{E}^N f] &= \int \limits_{\mathbb{R}^n_+} \frac{\partial g_N(x,y)}{\partial y^\alpha} f(y)\, dy, \quad \alpha<n,\\
	Q^D_{\alpha}[f] = Q_{\alpha} [\mathcal{E}^D f] &= \int \limits_{\mathbb{R}^n_+} \frac{\partial g_N(x,y)}{\partial y^\alpha} f(y)\, dy, \quad \alpha =n.
\end{align*}
For $e= \{e^\alpha\}$ we set 
$$
Q^D[e] = Q^D_\alpha[e^\alpha], \quad Q^N[e] = Q^N_\alpha[e^\alpha].
$$

Note that

\begin{itemize}
	
	\item the potential $P^D[f]$ solves the Dirichlet problem $\triangle_\rho u = f$ in $\mathbb{R}^n_{+}$, $u=0$ on $x_n=0$;
	
	\item the potential $P^N[f]$ solves the Neumann problem $\triangle_\rho u  = f$, $\partial u/\partial x_n =0$ on $x_n =0$;
	
	\item the potential $Q^D[e]$ solves the Dirichlet problem $\triangle_\rho u = - (e^{\alpha})_{,\alpha}$ in $\mathbb{R}^n_{+}$, $u=0$ on $x_n=0$,

	\item the potential $Q^N[e]$ solves the Neumann problem $\triangle_\rho u = - (e^{\alpha})_{,\alpha}$ in $\mathbb{R}^n_{+}$, $\rho^2\partial u/\partial x^n + e^{n}=0$ on $x^n=0$.
	
\end{itemize}

Denote $\nabla_\rho v = (\partial v/ \partial x^1,\ldots, \partial v/ \partial x^{n-1}, \rho\, \partial v/ \partial x^n)$. 
For $f\in L^{p(\cdot)}(B_{2R}^+)$ with support in $\overline{B}_{2R}^+$ the potential $v=P^D[f]$ belongs to $W^{2,p(\cdot)}_{\mathrm{loc}}(\mathbb{R}^n_+)$, has zero trace on $\{x^n=0\}$ and satisfies
\begin{equation}\label{eq:P}
	\int\limits_{\mathbb{R}^n_{+}} (\nabla_\rho v \cdot \nabla_\rho \zeta +f\zeta )\, dx=0, 
\end{equation}
for all $\zeta \in C_0^\infty(\mathbb{R}^n_+)$, while $v=P^N[f]$ belongs to $W^{2,p(\cdot)}_{\mathrm{loc}}(\mathbb{R}^n_+)$ and satisfies \eqref{eq:P} for all $\zeta \in C_0^\infty(\overline{\mathbb{R}}^n_+)$. The potential $w = Q_\alpha^D[f]$ belongs to $W^{1,p(\cdot)}_{\mathrm{loc}}(\mathbb{R}^n_+)$, has zero trace on $\{x^n=0\}$ and satisfies 
\begin{equation}\label{eq:Q}
	\int\limits_{\mathbb{R}^n_{+}} (\nabla_\rho w \cdot \nabla_\rho \zeta +\zeta f_{,\alpha} )\, dx=0
\end{equation}
for all $\zeta \in C_0^\infty(\mathbb{R}^n_+)$, while $w=Q^N_\alpha[f]$  belongs to $W^{1,p(\cdot)}_{\mathrm{loc}}(\mathbb{R}^n_+)$ and satisfies \eqref{eq:Q} for all $\zeta \in C_0^\infty(\overline{\mathbb{R}}^n_+)$. All these potentials decay at infinity with the exception of $P^N[f]$ in the dimension $n=2$, which requires additionally that the mean of $f$ over $B_{2R}^+$ is zero.

In the following statements $P^*$ will be either $P^D$ or $P^N$ and $Q^*$ will be either $Q^D$ or $Q^N$. In $W^{s,p(\cdot)}(B_{2R}^+)$ we shall use the norm
$$
\|u\|^*_{s,p(\cdot),B_{2R}^+} = \sum_{k=0}^s R^{k-s}\|\nabla^k u\|^*_{p(\cdot),B_{2R}^+}.
$$
For $n=2$ we denote 
	\begin{equation}\label{eq:c2}
		c(f;B_{2R}^+) = c(\mathcal{E}^N f; B_{2R})= \frac{\log(2R)}{2\pi} \int\limits_{B_{2R}^+} f\, dx.
	\end{equation}

\begin{lemma}\label{L:pa1}
	For $f\in L^{p(\cdot)}(B_{2R}^+)$ with support in $\overline{B}_{2R}^+$ there holds
	\begin{align*}
		\| Q^*_\alpha[f]\|^*_{1,p(\cdot), B_{2R}^+} &\leq C \|f\|_{p(\cdot), B_{2R}^+},\\
		\|P^*[f]\|^*_{1,p(\cdot), B_{2R}^+} &\leq C R \|f\|_{p(\cdot), B_{2R}^+},\\
		\|P^*[f]\|^*_{1,\varkappa p(\cdot), B_{2R}^+} &\leq C\|f\|_{p(\cdot), B_{2R}^+}.
	\end{align*}
	The potentials $P^D[f]$ and $Q^D_\alpha[f]$ have zero trace on $x^n=0$. For $n=2$ the estimates for $P^N[f]$ are valid with $P^N[f]$ replaced by $P^N[f] - c(f;B_{2R}^+)$, while the rest of the estimates is unchanged. 
\end{lemma}
\begin{proof}
	These estimates follow from Lemma~\ref{L:pa0} and the definition of the half-space potentials $P^*$ and $Q^*$ via even or odd extension. Extending $P^D[f]$ and $Q^D_\alpha[f]$ by zero to $x^n<0$ we get functions from $W^{1,1}(B_{4R})$ by the corresponding constant exponent result. But this and the log-H\"older property of the exponent $p(\cdot)$ imply the zero trace property.
\end{proof}

\begin{lemma}\label{L:DR040}
	Let $f \in L^{p(\cdot)}(B_{2R}^+)$ with support in $\overline{B}_{2R}^+$. Then 
	$$
	\|P^*[f]\|^*_{2,p(\cdot),B_{2R}^+} \leq C \|f\|_{p(\cdot), B_{2R}^+}.
	$$
	Let $f\in W^{1,p(\cdot)}(B_{2R}^+)$ with support in $\overline{B}_{2R}^+$. Then 
	$$
	\|Q^*_\alpha[f]\|^*_{2,p(\cdot),B_{2R}^+}  \leq C \|f\|_{1,p(\cdot), B_{2R}^+}^*.
	$$
	For $n=2$ the estimate for $P^N[f]$ is valid with $P^N[f]$ replaced by $P^N[f] - c(f;B_{2R}^+)$, while the rest of the estimates is unchanged.
\end{lemma}
\begin{proof}
	The first estimate follows from Calderon-Zygmund estimates in variable exponent spaces and Corollary~\ref{C:Di}. 
	
	We prove the second estimate for $Q^N$, the case of $Q^D$ is similar. Consider the integral 
	\begin{equation}
		h_\alpha(x)=\int\limits_{\mathbb{R}^n_{+}} \frac{\partial g_N(x,y)}{\partial y^\alpha} f (y)\, dy\label{eq:QQ0}
	\end{equation}
	where $f\in W^{1,p(\cdot)}(G_{2R})$ and $f$ is supported in $G_{2R}$.
	
	For $\alpha = 1,\ldots,n-1$ we can integrate by parts in tangential direction to get
	$$
	h_\alpha(x)=\int\limits_{\mathbb{R}^n_{+}} g_N(x,y) f_{,\alpha}\, dy= \int\limits_{\mathbb{R}^n} K_0(x-y;\rho) \mathcal{E}^N f_{,\alpha}(y)\, dy.
	$$
	From Calderon-Zygmund estimates in variable exponent spaces it follows that 
	$$
	\|h_\alpha\|^*_{2,p(\cdot),B_{2R}^+} \leq C\|\nabla f\|_{p(\cdot),B_{2R}^+}.
	$$
	
	Now consider the integral \eqref{eq:QQ0} for $\alpha = n$, which we can write as 
	$$
	h_n(x) = -\int \limits_{\mathbb{R}^n} K_{0,n}(z;\rho) \mathcal{E}^D f(x-z)\, dz.
	$$
	For $\alpha = 1,\ldots, n-1$ (the tangential derivatives) we have
	$$
	h_{n,\alpha}(x) = -\int \limits_{\mathbb{R}^n} K_{0,n}(x-y;\rho) \mathcal{E}^D f_{,\alpha}(y)\, dy,
	$$
	thus by the Calderon-Zygmund theorem in variable exponent spaces 
	$$
	\|\nabla h_{n,\alpha}\|_{p(\cdot),B_{2R}^+} \leq C \|f_{,\alpha}\|_{p(\cdot), B_{2R}^+}.
	$$
	This guarantees that  
	$$
	\|h_{n,\alpha\beta}\|_{p(\cdot), B_{2R}^+} \leq C\|\nabla f\|_{p(\cdot), B_{2R}^+} \quad \text{for}\quad \alpha+\beta<2n.
	$$ 
	Now, the potential $h_n$ satisfies the integral identity (see, for instance \cite{Morrey1966})
	$$
	\int\limits_{B_{2R}^+} \bigl[(\rho^2 h_{n,n}+f)\zeta_{,n} +  \sum_{\alpha=1}^n h_{n,\alpha}\zeta_{,\alpha}\bigr]\, dx=0
	$$
	for all $\zeta \in C_0^\infty(\overline{B}_{2R}^+)$. Integrating by parts, we get 
	$$
	\int\limits_{B_{2R}^+} (\rho^2 h_{n,n}+f)\zeta_{,n}\, dx = \int\limits_{B_{2R}^+}  \sum_{\alpha=1}^n h_{n,\alpha\alpha}\zeta\, dx=0.
	$$
	This means that the weak derivative $h_{n,nn}$ exists and satisfies the relation
	$$
	\rho^2 h_{n,nn} = -f- \sum_{\alpha=1}^{n-1} h_{n,\alpha \alpha}.
	$$
	Thus, 
	$$
	\|h_{n,nn}\|_{p(\cdot), B_{2R}^+} \leq C\|\nabla f\|_{p(\cdot), B_{2R}^+}. 
	$$
	We have proved that 
	$$
	\|\nabla^2 h_\alpha\|_{p(\cdot), B_{2R}^+} \leq C \|\nabla f\|_{p(\cdot), B_{2R}^+}, \quad \alpha=1,\ldots,n.
	$$
	Combining with the previous estimate 
	$$
	\|Q^*_\alpha[f]\|^*_{1,p(\cdot), B_{2R}^+} \leq C\|f\|_{p(\cdot),B_{2R}^+}, 
	$$ 
	we get the result.
\end{proof}

\begin{lemma}\label{L:DR04c0}
	Let $f \in W^{s,p(\cdot)}(B_{2R}^+)$, $s\in \{0\}\cup\mathbb{N}$, with support in $\overline{B}_{2R}^+$. Then 
	$$
	\|P^*[f]\|^*_{s+2,p(\cdot),B_{2R}^+} \leq C \|f\|^*_{s,p(\cdot), B_{2R}^+}.
	$$
	Let $f\in W^{s+1,p(\cdot)}(B_{2R}^+)$ $s\in \{0\}\cup\mathbb{N}$, with support in $\overline{B}_{2R}^+$. Then 
	$$
	\|Q_\alpha^*[f]\|^*_{s+2,p(\cdot),B_{2R}^+}  \leq C \|f\|_{s+1,p(\cdot), B_{2R}^+}^*.
	$$
	For $n=2$ the estimate for $P^N[f]$ is valid with $P^N[f]$ replaced by $P^N[f] - c(f;B_{2R}^+)$, while the rest of the estimates is unchanged. 
\end{lemma}
\begin{proof}
	We show the estimate only for the potential $Q^*$, the estimate for $P^*$ is similar. We argue by induction. The base of induction for $s=0$ is already known. Let the estimate be known for all values $s=1,\ldots, s_0$ and we show it for $s=s_0+1$. Since we already know that
	$$
	\|Q_\alpha^*[f]\|^*_{s_0+1,p(\cdot), B_{2R}^+} \leq C \|f\|_{s_0,p(\cdot), B_{2R}^+}^*,
	$$
	we only have to establish the estimate
	$$
	\|\nabla^{s_0+2} Q^*[f]\|_{p(\cdot), B_{2R}^+} \leq C \|f\|_{s_0+1,p(\cdot), B_{2R}^+}^*.
	$$
	For a multiindex $\beta = (\beta_1,\ldots,\beta_{n-1},\beta_n)$ let $D^\beta u = D^{\beta_1}_{x_1} \ldots D^{\beta_n}_{x_n}$. Note that for all multiindices $\beta$ with $|\beta|\leq s_0$ and $\beta_n=0$ there holds
	\begin{align*}
		D^\beta Q_\alpha^*[f] &= Q^*_\alpha [D^\beta f] \in W^{2,p(\cdot)}(B_{2R}^+),\\
		\|D^\beta Q_\alpha^*[f]\|^*_{2,p(\cdot),B_{2R}^+} &\leq C \|D^\beta f\|^*_{1,p(\cdot), B_{2R}^+}.
	\end{align*}
	Therefore, for $|\beta|=s_0+1$ with $\beta_n\leq 1$ we get
	\begin{equation}\label{eq:ind00}
		\|D^\beta Q_\alpha^*[f]\|^*_{1,p(\cdot),B_{2R}^+} \leq C \|f\|_{s_0+1,p(\cdot), B_{2R}^+}^*.
	\end{equation}
	Let us show, again by induction in $\beta_n$, that \eqref{eq:ind00} holds for all $\beta$ with $|\beta|=s_0+1$. Assume that \eqref{eq:ind00} holds for all $\beta_n \leq j_0$, $j_0\leq s_0$. We show that it holds for $\beta_n \leq j_0+1$. Applying $D^{\beta}$ with $|\beta|=s_0-1$ and $\beta_n = j_0-1$ to
	$$
	\rho^2(Q_\alpha^*[f])_{,nn} = -\sum_{l=1}^{n-1} (Q_\alpha^*[f])_{,ll} - f_{,\alpha},
	$$
	we get
	$$
	\rho^2 D^\beta (Q_\alpha[f])_{,nn} = -\sum_{l=1}^{n-1} D^\beta(Q[f])_{,ll} - D^\beta f_{,\alpha}.
	$$
	The right-hand side contains the $(s_0+1)$-th derivatives of $Q_\alpha[f]$ with $\beta_n=j_0-1$, thus by the assumption of induction the left-hand side belongs to $W^{1,p(\cdot)}(B_{2R}^+)$, and \eqref{eq:ind00} is satisfied for  $|\beta|=s_0+1$ with $\beta_n = j_0+1$.
\end{proof}


Let 
\begin{align*}
	U_\rho[f](x) &=  \int\limits_{\mathbb{R}^{n-1}} g_N(x,y) f(y')\, dy' \\&=\frac{2}{(2-n)\omega_n\rho}\int\limits_{\mathbb{R}^{n-1}} (|x'-y'|^{2}+(x^n/\rho)^2)^\frac{2-n}{2} f(y')\, dy', 
\end{align*}
where the last expression is given for $n\geq 3$. This potential represents a solution of the Neumann problem $\triangle U[f]=0$ in $\mathbb{R}^n_+$, $\rho^2\partial U[f]/\partial x^n = f$ on $x_n=0$. Clearly, $U_\rho[f](x) = \rho^{-1}U_1[f](x', x_n/\rho)$, where $U_1$ is the standard single-layer potential corresponding to the Laplace operator.
\begin{lemma}\label{L:layer}
	Let $f\in W^{s+1,p(\cdot)}(B_{2R}^+)$, $s=0,1,2,\ldots$ have support in $\overline{B}_{2R}^+$. Then 
	$$
	\|U_\rho[f(\cdot,0)]\|^*_{s+2,p(\cdot),B_{2R}^+} \leq C(n,p_{-},p_{+},c_{\mathrm{log}}(p),s) \|f\|^*_{s+1,p(\cdot),B_{2R}^+}.
	$$
	If $n=2$ then the right-hand side is multiplied by the factor $(1+ \log R^{-1})$.
\end{lemma}
\begin{proof}
	There holds
	\begin{align*}
		U_\rho[f(\cdot,0)] &= -\int\limits_{\mathbb{R}^{n-1}} \left(\int\limits_0^\infty \frac{\partial(g_N(x,y) f(y))}{\partial y^n} dy^n \right)\, dy' \\
		&=-Q_n^N[f] - P^N\left[\frac{\partial f}{\partial x^n}\right].  
	\end{align*}
	The required properties of $U_\rho[f]$ follow then from Lemma~\ref{L:DR04c0}. For $n=2$ we also use the estimate 
		\begin{gather*}
			\|c(f_{,n}; B_{2R}^+)\|_{s+2,p(\cdot),B_{2R}^+} = R^{-s-2}\|c(f_{,n}; B_{2R}^+)\|_{p(\cdot),B_{2R}^+} \\
			\leq C|\log(2R)| R^s \|f_{,n}\|_{p(\cdot),B_{2R}^+} \leq C |\log(2R)| \|f\|^*_{s+1,p(\cdot),B_{2R}^+}.
		\end{gather*}
		which follows by the definition of corresponding norms and the log-H\"older property of the exponent $p(\cdot)$.
\end{proof}

Using this simple statement we prove the following
\begin{corollary}\label{C:ext}
	Let $s\in \{0\}\cup \mathbb{N}$, $M$ be of the class $C^{s+2,1}$, and the forms $f,v\in W^{s+1,p(\cdot)}(M, \Lambda)$. There exists $\gamma\in W^{s+2,p(\cdot)}(M, \Lambda)$ such that $\gamma=0$ on $bM$, $nd\gamma = nf$, $t\delta \gamma =tv$, and 
	$$
	\|\gamma\|_{W^{s+2,p(\cdot)}(M, \Lambda)} \leq C(\mathrm{data}) (\|f\|_{W^{s+1,p(\cdot)}(M, \Lambda)} + \|v\|_{W^{s+1,p(\cdot)}(M, \Lambda)}).
	$$
	If $M$ is only of the class $C^{s+1,1}$, then $\gamma, d\gamma \in W^{s+1,p(\cdot)}(M, \Lambda)$ and
	\begin{gather*}
		\|\gamma\|_{W^{s+1,p(\cdot)}(M, \Lambda)}+ \|d\gamma\|_{W^{s+1,p(\cdot)}(M, \Lambda)} \\
		\leq C(\mathrm{data}) (\|f\|_{W^{s+1,p(\cdot)}(M, \Lambda)} + \|v\|_{W^{s+1,p(\cdot)}(M, \Lambda)}).
	\end{gather*}
	For $C^{1,1}$ manifold the condition $t\delta \gamma=v$ becomes void. 
	Moreover, the map from $f$, $v$ to $\gamma$ is linear. 
\end{corollary}
\begin{proof}
	By using a partition of unity, we may assume that the support of $f$ and $v$ is in one local admissible coordinate chart. Then the boundary condition reads as 
	$$
	\gamma(x',0) =0, \quad \frac{\partial \gamma_I}{\partial x^n}(x',0) = f_{nI}(x',0), \quad n\notin I, \quad \frac{\partial \gamma_{nI}}{\partial x^n}(x',0) = -v_{I}(x',0).
	$$
	Here the supports of $f$ and $v$ belong to $\overline{B}_{R}^+$, while $\overline{B}_{2R}^+$ is still in coordinate map. Let $\eta\in C_0^\infty(B_{2R})$, $0\leq \eta \leq 1$, and $\eta=1$ on $B_R$. Set 
	\begin{align*}
		\gamma_I(x) &=\frac{\eta(x)}{2}( U_{1/2}[f_{nI}(\cdot,0)](x) - 2U_1[f_{nI}(\cdot,0)](x)), \quad n\notin I,\\
		\gamma_{nI}(x) &=\frac{\eta(x)}{2}( U_{1/2}[-v_{I}(\cdot,0)](x) - 2U_1[-v_{I}(\cdot,0)](x)).
	\end{align*}
	By Lemma~\ref{L:layer} this provides the required extension.
\end{proof}

\begin{remark}
	By Hodge duality, on $C^{s+1,1}$ manifold in the situation of Corollary~\ref{C:ext} we can also get $\gamma \in C^{s+1,1}(M,\Lambda)$ with $\delta \gamma \in C^{s+1,1}(M,\Lambda)$ and the estimate 
	\begin{gather*}
		\|\gamma\|_{W^{s+1,p(\cdot)}(M, \Lambda)}+ \|\delta\gamma\|_{W^{s+1,p(\cdot)}(M, \Lambda)} \\
		\leq C(\mathrm{data}) (\|f\|_{W^{s+1,p(\cdot)}(M, \Lambda)} + \|v\|_{W^{s+1,p(\cdot)}(M, \Lambda)}).
	\end{gather*}
	In this case for $C^{1,1}$ manifold the condition $n d\gamma =n f$ becomes void.
\end{remark}

See ~\cite{MazShaBook} for the classical Gagliardo type extension operator and its properties.

	\chapter[Boundary Value Problems for the Hodge Laplacian and Div-Curl Systems]{Boundary Value Problems for the Hodge Laplacian and Div-Curl Systems --- Classical Results}\label{sec:boundary}
	
	In the Chapter we describe classical boundary value problems for the Hodge Laplacian and div-curl systems and some well-known result on their solvability in constant exponent Lebesgue and Sobolev spaces. However, we emphasize that all these results are valid only when $M$ is at least $C^{2,1}.$ So we \emph{can not use} these results to treat the case of $C^{1,1}$ manifolds. 
	
	\section{Classical BVPs for the Hodge Laplacian on \texorpdfstring{$C^{2,1}$}{C21} manifolds} Recall that $\triangle = d\delta + \delta d$ and denote 
	\begin{align}\label{eq:Ddef}
		\mathcal{D}(f,v) = (df,dv) + (\delta f, \delta v).
	\end{align}
	From the integration-by-parts formula it follows that 
	\begin{equation}\label{eq:Green1}
	(\triangle f, v) = \mathcal{D}(f,v) + [\delta f, v] - [v, df]
	\end{equation}
	and 
	\begin{align*}
		(\triangle f, v) - (f, \triangle v) &= [\delta f, v] - [v, df] - [\delta v,f] + [f,dv]\\
		&\begin{multlined}
			= \int\limits_{bM} \langle \nu \wedge \delta f, v \rangle\, d\sigma -  \int\limits_{bM} \langle \nu \wedge v, df \rangle\, d\sigma  - \int\limits_{bM} \langle \nu \wedge \delta v, f \rangle\, d\sigma \\+ \int\limits_{bM} \langle \nu \wedge f, dv\rangle\, d\sigma.  
		\end{multlined}
	\end{align*}
	From this relation and the definition of $[\cdot,\cdot]$ it follows that there are four sets of boundary conditions which make the operator $\triangle$ symmetric in $W^{2,2}(M,\Lambda)$:
	\begin{align*}
		(I)&\quad t\omega =0,\quad  \quad n\omega =0,\\
		(II)&\quad t\omega =0,\quad  \quad t\delta \omega =0,\\
		(III)&\quad n\omega =0,\quad \quad nd\omega =0,\\
		(IV)&\quad  t\delta \omega =0,\quad nd\omega =0.
	\end{align*}
	These leads to four standard boundary value problems for the Hodge Laplacian:
	\begin{align*}
		(I)& \quad\triangle \omega = \eta, \quad t\omega =t\varphi,\quad  n\omega =n\varphi,\\
		(II)& \quad\triangle \omega = \eta, \quad t\omega =t\varphi,\quad  t\delta \omega =t\psi,\\
		(III)& \quad\triangle \omega = \eta, \quad n\omega =n\varphi,\quad nd\omega =n\psi,\\
		(IV)& \quad\triangle \omega = \eta, \quad t\delta \omega =t\varphi,\quad nd\omega =n\psi.
	\end{align*}
	The first three of these problems are elliptic, moreover the first one has zero kernel/co-kernel, while the fourth one is not elliptic and has infinite dimensional kernel/co-kernel which consists of harmonic fields.

 In terms of the classical vector calculus the Green's relation \eqref{eq:Green1} takes form
\begin{equation}\label{eq:Green3D}
\int\limits_D\vec{v}\cdot \triangle \vec{u}\, dV = \int\limits_{\partial D} [(\vec{v}\cdot \vec{n}) \mathrm{div}\, \vec{u} - \vec{v}\cdot(\vec{n}\times\mathrm{rot}\, \vec{u})]\, d\sigma - \int\limits_D [(\mathrm{div}\, \vec{u}) (\mathrm{div}\, \vec{v}) +  \mathrm{rot}\, \vec{u} \cdot \mathrm{rot}\, \vec{v} ]\, dV
\end{equation}
(here $\triangle$ is the standard Laplacian in 3D) and the four BVPs above read as follows: in a domain $D\subset \mathbb{R}^3$ find a vector field $\vec{u} = (u_1,u_2,u_3)$ satisfying
	\begin{align*}
		(I)& \quad\triangle \vec{u} = \vec{\eta}\quad \text{in}\quad D, \quad \vec{u} = \vec{\varphi} \quad \text{on}\quad \partial D,\\
		(II)& \quad\triangle \vec{u} = \vec{\eta} \quad \text{in}\quad D, \quad \vec{n}\times (\vec{u}-\vec{\varphi})=0,\quad  \mathrm{div}\, \vec{u} = \psi\quad \text{on}\quad \partial D,\\
		(III)& \quad\triangle \vec{u} = \vec{\eta}\quad \text{in}\quad D, \quad \vec{n}\cdot (\vec{u}-\vec{\varphi})=0,\quad \vec{n}\times (\mathrm{rot}\, \vec{u} - \vec{\psi})=0\quad \text{on}\quad \partial D,\\
		(IV)& \quad\triangle \vec{u} = \vec{\eta}\quad \text{in}\quad D, \quad \mathrm{div}\, \vec{u} = \varphi,\quad \vec{n}\times (\mathrm{rot}\, \vec{u} - \vec{\psi})=0\quad \text{on}\quad \partial D.
	\end{align*}
The boundary conditions in (II) and (III) are sometimes called ``electric'' and ``magnetic'' boundary conditions, respectively. From \eqref{eq:Green3D} it is easy to get the following representation formula: 
\begin{equation}\label{eq:Int3D}
\begin{gathered}
\vec{u}(x) = \int\limits_D g(x,y) \triangle \vec{u}(y)\, dy + \int\limits_{\partial D} \big( \phi_{(n)}(y) \nabla_y g(x,y) + g(x,y) \vec{\psi}_{(n)}(y)\big)\, d\sigma(y) \\
+ \int\limits_{\partial D} \big(\vec{\psi}_{(t)}(y)\times \nabla_y g(x,y)  - \vec{n} g(x,y) \phi_{(t)}(y) \big)\, d\sigma(y),\\
g(x,y) = -\frac{1}{4\pi |x-y|}, \quad \phi_{(n)} = \vec{u}\cdot\vec{n}, \quad \vec{\psi}_{(n)} = [\vec{n}\times \mathrm{rot}\,\vec{u}],\\ \quad \phi_{(t)}=\mathrm{div}\, \vec{u}, \quad \vec{\psi}_{(t)}= [\vec{n}\times \vec{u}].
\end{gathered}
\end{equation}
A solution to the problem (II) or (III) is then sought as the sum of the volume integral and the boundary integral with $\phi_{(t)}$, $\vec{\psi}_{(t)}$ or $\phi_{(n)}$, $\vec{\psi}_{(n)}$, respectively, with unknown potentials. Using jump relations for potentials one reduces the problem to boundary integral equations. The three-dimensional boundary value problems (II) and (III), including exterior and interior cases, are analyzed in \cite{Kre70} by reducing the problem to the boundary integral equations for single- and double-layer potentials. For the tensorial case this analysis was later written in detail by J. Bolik in \cite{bolikphd} and subsequent works (though of course in essence this approach goes back to \cite{DufSpe51}, \cite{DufSpe56}). Here in the Euclidian setting to get an analogue if \eqref{eq:Int3D} one uses the Green's second identity for the Hodge Laplace for the form $u$ and the double form $g(x,y) = \Gamma(x-y)\sum_{I\in \mathcal{I}(r)} dx^I dy^I$, with $\Gamma$ being the fundamental solution of the standard Laplace operator. In the case of the Riemannian manifold a great role is played by the explicit global construction of the parametrix for the Hodge-Laplacian, see \cite[Chapter V, \S 28]{deRham}. In variable exponent setting, though, the method of boundary integral equations still is not fully developed since it requires the description of the trace spaces and in the case of rough surfaces estimates of singular integral operators. To carry over the results of \cite{MitMitMitTay16} for rough domains to the variable exponent setting presents a challenging problem.

	\section{Variational formulations}\label{ssec:varf} On $C^{1,1}$ manifolds one can not understand the boundary values problems for the Hodge Laplacian in the standard sense. However, these problems are well defined in the variational sense. The four sets of boundary conditions described above correspond to natural variational problems. Consider the equation
	\begin{equation}\label{eq1}
		(d\omega - \varphi,d\zeta) + (\delta \omega-\psi,\delta \zeta) - (\eta,\zeta) =0
	\end{equation}
	\begin{itemize}

		\item[(I)] for $\omega\in W_0^{1,2}(M,\Lambda)$ and all $\zeta \in W_0^{1,2}(M,\Lambda)$ in the case of boundary conditions $t\omega =0$, $n\omega =0$.

		\item[(II)] for $\omega \in W^{1,2}_T(M,\Lambda)$ and  all $\zeta \in W_T^{1,2}(M,\Lambda)$ for the case of Dirichlet boundary condition;

		\item[(III)] for $\omega \in W^{1,2}_N(M,\Lambda)$ and all $\zeta \in W_N^{1,2}(M,\Lambda)$ for the case of Neumann boundary condition;

		\item[(IV)] for $\omega \in W^{1,2}(M,\Lambda)$ and all $\zeta \in W^{1,2}(M,\Lambda)$ for a manifold without boundary  or for the case of ``natural'' boundary conditions $t\delta\omega =0$, $n d\omega =0$.
		
	\end{itemize}


	If $\eta, \varphi, \psi \in L^2(M,\Lambda)$, equation \eqref{eq1} is the Euler-Lagrange equation for the variational problem
	$$
	\mathcal{F}[\omega]=\frac{1}{2}(d\omega,d \omega) + \frac{1}{2}(\delta \omega, \delta \omega) - (\varphi,d\omega) - (\psi, \delta \omega) - (\eta,\omega) \rightarrow \mathrm{min}
	$$
	over 
	\begin{itemize}
		
		\item[(I)]  $\omega \in W_0^{1,2}(M,\Lambda)$ , here in \eqref{eq1} one allows $\zeta \in W_0^{1,2}(M,\Lambda)$.
		
		\item[(II)] $\omega \in W^{1,2}_T(M,\Lambda)$ for the Dirichlet problem, under assumption $\eta \in \mathcal{H}_T^\perp$, here in \eqref{eq1} one allows $\zeta \in W^{1,2}_T(M,\Lambda)$;
		
		\item[(III)] or $\omega\in W^{1,2}_N(M,\Lambda)$ for the Neumann problem, under assumption $\eta \in \mathcal{H}_N^\perp$, here in \eqref{eq1} one allows $\zeta \in W^{1,2}_N(M,\Lambda)$.

		\item[(IV)] $\omega \in W^{1,2}(M,\Lambda)$ for a closed manifold or for ``natural'' boundary conditions, under assumption $\eta \in (\mathcal{H}(M)\cap L^2(M,\Lambda))^\perp$, here in \eqref{eq1} one allows $\zeta\in W^{1,2}(M,\Lambda)$.
		
	\end{itemize}  

    Since $\mathcal{F}[\omega] = \mathcal{F}[\omega + h]$ for  (II) $h\in \mathcal{H}_T (M)$, (III) $h \in \mathcal{H}_N(M)$, (IV) $h\in \mathcal{H}(M) \cap L^2(M,\Lambda)$, the set of admissible forms can be reduced to (II) $\omega \in W^{1,2}_T(M,\Lambda) \cap (\mathcal{H}_T(M))^\perp$, (III) $\omega \in W^{1,2}_N(M,\Lambda) \cap (\mathcal{H}_N(M))^\perp$, (IV) $\omega \in W^{1,2}(M,\Lambda) \cap (\mathcal{H}(M)\cap L^2(M,\Lambda))^\perp$. The existence and uniqueness of the minimizer then follows by the direct method in the calculus of variations. In this the Gaffney inequality (cf. Theorem 7.5.1 in \cite{Morrey1966} for $C^{1,1}$ manifolds) is instrumental:
    $$
    \|\omega\|^2_{W^{1,2}(M,\Lambda)} \leq C_1 (\omega,\omega) + C_2 \mathcal{D}(\omega,\omega), \quad \omega \in W^{1,2}_T(M,\Lambda) \cap W^{1,2}_N(M,\Lambda),
    $$
    which implies by the standard compactness argument   
    \begin{gather*}
    (\omega,\omega) \leq C \mathcal{D}(\omega,\omega), \quad \text{and so}\quad
    \|\omega\|^2_{W^{1,2}(M,\Lambda)} \leq C \mathcal{D}(\omega,\omega), \quad \text{for} \\
    \omega \in W_0^{1,2}(M,\Lambda),\quad \omega \in W^{1,2}_T(M,\Lambda) \cap (\mathcal{H}_T(M))^\perp,\quad  \omega \in W^{1,2}_N(M,\Lambda) \cap (\mathcal{H}_N(M))^\perp.
    \end{gather*}    
    Here one also uses the well-known result on the triviality of the set of harmonic forms in $W_0^{1,2}(M,\Lambda)$. Using the Hodge decomposition one can also obtain similar estimates for $\omega \in W^{1,2}(M,\Lambda) \cap (\mathcal{H}(M)\cap L^2(M,\Lambda))^\perp$ (for instance, Lemma~\ref{L:GaffneyNatur} below).

	We shall be mainly interested in the Dirichlet and Neumann problems. From the classical results of~\cite{Morrey1966} it follows that in these cases solutions to \eqref{eq1} belong to $W^{2,2}(M,\Lambda)$ as soon as $\eta \in L^2(M,\Lambda)$, $\varphi, \psi \in W^{1,2}(M,\Lambda)$ and $M$ is $C^{2,1}$. The integration by parts in \eqref{eq1} yields then the equation and boundary conditions.  
	
	On $C^{2,1}$ manifolds the variational theory of standard boundary value problems for the Hodge Laplacian can be described by the following statements.
	\begin{itemize}
		
		\item The minimizer of $\mathcal{F[\omega]}$ over the space of $W^{1,2}_0(M,\Lambda)$ is in  $W^{2,2}(M,\Lambda)$ and satisfies 
		$$
		\triangle \omega =  \eta + \delta \varphi + d \psi \quad \text{a.e. in $M$}, \quad t\omega =0,\quad n\omega = 0 .
		$$
		
		\item The minimizer of $\mathcal{F[\omega]}$ over the space of $(\mathcal{H}_T(M))^\perp \cap W^{1,2}_T(M,\Lambda)$ belongs to $W^{2,2}(M,\Lambda)$ and satisfies 
		$$
		\triangle \omega =  \eta + \delta \varphi + d \psi \quad \text{a.e. in $M$}, \quad t\omega =0,\quad t\delta \omega = t\psi.
		$$
		
		\item The minimizer of $\mathcal{F[\omega]}$ over the space of $(\mathcal{H}_N(M))^\perp \cap W^{1,2}_N(M,\Lambda)$ belongs to $W^{2,2}(M,\Lambda)$ and satisfies
		\begin{equation*}
			\triangle \omega =  \eta + \delta \varphi + d \psi \quad \text{a.e. in $M$}, \quad n\omega =0,\quad nd \omega = n\varphi.
		\end{equation*}

		\item The minimizer of $\mathcal{F[\omega]}$ over the space of $(\mathcal{H}(M)\cap L^2(M,\Lambda))^\perp \cap W^{1,2}(M,\Lambda)$ belongs to $W^{2,2}(M,\Lambda)$ and  satisfies 
		\begin{equation*}
			\triangle \omega =  \eta + \delta \varphi + d \psi \quad \text{a.e. in $M$}, \quad t\delta\omega =t\psi,\quad nd \omega = n\varphi.
		\end{equation*}
	\end{itemize}
	
	Another classical approach based on the jump relations and boundary integral equations was developed by Duff, Spencer and later extended by R.~Kress, J.~Bolik and some other authors. 
    



\section{Dirichlet and Neumann potentials} If $\eta \in L^2(M,\Lambda) \cap (\mathcal{H}_T (M)^\perp $ then the Dirichlet potential $G_D[\eta]$ is a unique minimizer of the variational problem
	\begin{align}\label{eq:minimizationD}
		\inf \left\lbrace 	\frac{1}{2}(d\omega,d \omega) + \frac{1}{2}(\delta \omega, \delta \omega)  - (\eta,\omega) : \omega \in W^{1,2}_T(M,\Lambda)\cap (\mathcal{H}_T (M))^\perp \right\rbrace.  
	\end{align}
By the Gaffney inequality (cf. Theorem 7.5.1 in \cite{Morrey1966}), which holds for $C^{1,1}$ manifolds, it is easy to prove by standard variational methods that the minimizer exists and is unique. We define this unique minimizer as the Dirichlet potential of $\eta$ and denote this by $G_D[\eta]$. The Dirichlet potential $\omega = G_D[\eta]$ belongs to $W^{1,2}_T(M,\Lambda) \cap(\mathcal{H}_T (M))^\perp$ and satisfies the variational relation 
\begin{equation}\label{eq:Dirichlet}
\mathcal{D}(\omega,\xi) = (\eta,\xi) \quad \text{for all} \quad \xi \in \mathrm{Lip}_T(M,\Lambda).
\end{equation}
On the other hand, for $\eta \in L^2(M,\Lambda)$ any solution $\omega \in W^{1,2}_T(M,\Lambda) \cap (\mathcal{H}_T (M)^\perp$ of this relation minimizes the variational problem \eqref{eq:minimizationD}. 

The variational relation \eqref{eq:Dirichlet} is well defined for $\eta \in L^{p_{-}}(M,\Lambda)$ and $\omega \in W^{1,p_{-}}(M,\Lambda)$. By regularity results of \cite{Morrey1966} any $\omega \in W_T^{1,p_{-}}(M,\Lambda)$ satisfying \eqref{eq:Dirichlet} with $\eta =0$ belongs to $W^{1,q}_T(M,\Lambda)$ for any $q<\infty$. Using approximation it is easy to show that $\mathcal{D}(\omega,\omega)=0$, thus $\omega\in \mathcal{H}_T(M)$, which implies $\omega=0$.

\begin{definition}[Dirichlet potential]
For $\eta \in L^{p_{-}}(M,\Lambda)$ satisfying $(\eta,h_T) =0$ for all $h_T \in \mathcal{H}_T(M)$ by $G_D[\eta] \in W^{1,p_{-}}_T(M,\Lambda)$ we denote a (unique) solution of \eqref{eq:Dirichlet} satisfying $(\omega,h_T)=0$ for all $h_T \in \mathcal{H}_T(M)$.
\end{definition}

An existence result for the Dirichlet potentials will be provided in Section~\ref{sec:Dirichlet}.

Similarly, for $\eta\in L^2(M,\Lambda)\cap (\mathcal{H}_N (M))^\perp $ the Neumann potential $G_N[\eta]$ is a unique minimizer of the variational problem 
\begin{align}\label{eq:minimizationN}
		\inf \left\lbrace 	\frac{1}{2}(d\omega,d \omega) + \frac{1}{2}(\delta \omega, \delta \omega)  - (\eta,\omega) : \omega \in W^{1,2}_N(M,\Lambda)\cap (\mathcal{H}_N (M))^\perp \right\rbrace. 
	\end{align}
Its existence is proved in the same way as for the Dirichlet potential. By construction, it belongs to
$W^{1,2}_N(M,\Lambda) \cap (\mathcal{H}_N (M))^\perp$ and satisfies the variational relation 
\begin{equation}\label{eq:Neumann}
\mathcal{D}(\omega,\xi) = (\eta,\xi) \quad \text{for all} \quad \xi \in \mathrm{Lip}_N(M,\Lambda).
\end{equation}

\begin{definition}[Neumann potential]
For $\eta \in L^{p_{-}}(M,\Lambda)$ satisfying $(\eta,h_N) =0$ for all $h_N \in \mathcal{H}_N(M)$ by $G_N[\eta] \in W^{1,p_{-}}_N(M,\Lambda)$ we denote a (unique) solution of \eqref{eq:Neumann} satisfying $(\omega,h_N)=0$ for all $h_N \in \mathcal{H}_N(M)$.
\end{definition}

An existence result for the Neumann potentials will be provided in Section~\ref{sec:Dirichlet}.

\begin{remark}
	Note that as soon as the manifold is at least $C^{2,1}$, the Dirichlet potential $G_D[\eta]$ can be equivalently defined as the $W^{2,p_{-}}(M,\Lambda)$ solution of 
    $$
    \triangle \omega =\eta, \quad t\omega =0. \quad t\delta \omega =0,
    $$
    satisfying $\mathcal{P}_T \omega =0$. Similarly, assuming the manifold to be at least $C^{2,1}$, the Neumann potential $G_N[\eta]$ can be equivalently defined as the $W^{2,p_{-}}(M,\Lambda)$ solution of the problem 
    $$
    \triangle \omega = \eta,\quad n\omega = 0, \quad nd\omega =0,
    $$ 
    satisfying $\mathcal{P}_N \omega =0$.  
	\end{remark}

\section{Potentials for full Dirichlet data}

We can also consider the following problem:
\begin{align}\label{eq:minimizationF}
		\inf \left\lbrace 	\frac{1}{2}(d\omega,d \omega) + \frac{1}{2}(\delta \omega, \delta \omega)  - (\eta,\omega) : \omega \in W^{1,2}_0(M,\Lambda) \right\rbrace. 
	\end{align}
Due to Gaffney's inequality and triviality of the set of harmonic fields vanishing on $bM$ (this again only requires that the manifold $M$ be $C^{1,1}$, see \cite{AroKrzSza62}) this problem has a unique minimizer in $W^{1,2}_0(M,\Lambda)$, which we denote by $G_0[\eta]$. This minimizer satisfies
\begin{equation}\label{eq:FD0}
\mathcal{D}(\omega,\xi) = (\eta,\xi) \quad \text{for all} \quad \xi \in \mathrm{Lip}_0(M,\Lambda).
\end{equation}    
\begin{definition}[Potential for full Dirichlet data]
For $\eta \in L^{p_{-}}(M,\Lambda)$ by $G_0[\eta] \in W^{1,p_{-}}_N(M,\Lambda)$ we denote a (unique) solution of \eqref{eq:FD0}.
\end{definition}
As soon as $M$ is $C^{2,1}$, the potential $G_0[\eta]$ can be equivalently defined as a unique $W^{2,p_{-}}(M,\Lambda)$ solution of the problem
$$
\triangle \omega = \eta , \quad t\omega =0, \quad n\omega =0.
$$
    
	\section{Results from general elliptic theory}
	Let $M$ be a compact Riemanninan manifold with boundary of the class $C^{s+2,1}$. From the general elliptic theory it follows that (see for instance \cite{Morrey1966, Schwarz}, also \cite{Sil17} for a direct proof which does not rely on results about general elliptic problems) that both Dirichlet and Neumann boundary value problems
	\begin{align}
		\triangle \omega =  \eta, \quad t\omega =t\varphi,\quad t\delta \omega = t\psi\label{D},\\
		\triangle \omega =  \eta, \quad n\omega =n\varphi,\quad nd \omega = n\psi \label{N},
	\end{align}
	are Fredholm solvable, possess the finite-dimensional kernels consisting of Dirichlet fields $\mathcal{H}_T(M)$ and Neumann fields $\mathcal{H}_N(M)$ respectively, a solution satisfies the general elliptic estimate
	\begin{align}\label{ell1}
		\nonumber\|\omega\|_{W^{s+2,p}(M,\Lambda)} 
		\lesssim \|\eta\|_{W^{s,p}(M,\Lambda)} + \|\varphi\|_{W^{s+2-1/p,p}(M,\Lambda)}+\|\psi\|_{W^{s+1-1/p,p}(M,\Lambda)} \\
		+ \|\omega\|_{L^1(M,\Lambda)},
	\end{align}
	where $s\geq 0$ and $p>1$. The problem \eqref{D} is solvable if and only if $(\eta,h_T) = [\psi,h_T]$ for all $h_T \in \mathcal{H}_T(M)$. The problem \eqref{N} is solvable if and only if $(\eta,h_T) = -[h_N,\psi]$ for all $h_N \in \mathcal{H}_N(M)$.

	The last term on the RHS of \eqref{ell1} can be removed if one considers $\omega$ satisfying $(\omega,\mathcal{H}_*(M))=0$, or more generally consider $\omega- \mathcal{P}_*\omega$ where $\mathcal{P}_*$ is a projector on $\mathcal{H}_*(M)$ (further $\mathcal{H}_*(M) = \mathcal{H}_T(M)$ for the Dirichlet problem and $\mathcal{H}_*(M) = \mathcal{H}_T(M)$ for the Neumann problem):
	\begin{equation}\label{ell2}
		\|\omega-\mathcal{P}_*\omega\|_{W^{s+2,p}(M,\Lambda)} \leq C (\|\eta\|_{W^{s,p}(M,\Lambda)} + \|\varphi\|_{W^{s+2-1/p,p}(M,\Lambda)}+\|\psi\|_{W^{s+1-1/p,p}(M,\Lambda)}),
	\end{equation}
	but the constant in \eqref{ell2} comes from a functional analytic argument, and so is not  quantitative. We shall use only the standard projectors $\mathcal{P}_T$, $\mathcal{P}_N$ defined by \eqref{eq:projT}, \eqref{eq:projN}. 

	\section{First order systems}
	
	
	\textbf{Dirichlet condition.} Consider the following problem: find $\omega \in W^{1,p}(M,\Lambda)$, $p>1$, satisfying
	\begin{equation}\label{I}
		d\omega = f \in L^p(M,\Lambda), \quad \delta \omega = v\in L^p(M,\Lambda), \quad t\omega = t\varphi, \quad \varphi \in W^{1,p}(M,\Lambda).
	\end{equation}
	The problem \eqref{I} is solvable if and only if the following conditions hold: 
	\begin{gather*}
		df=0, \quad t(f-d\varphi)=0,\quad (f,h_T) =[\varphi,h_T]\quad \forall h_T\in \mathcal{H}_T(M),\\ \delta v=0,\quad  (v,h_T)=0 \quad \forall h_T\in \mathcal{H}_T(M).
	\end{gather*}

	\textbf{Neumann condition.} Consider the following problem: find $\omega \in W^{1,p}(M)$ satisfying
	\begin{equation}\label{II}
		d\omega = f \in L^p(M,\Lambda), \quad \delta \omega = v\in L^p(M,\Lambda), \quad n\omega = n\varphi, \quad \varphi \in W^{1,p}(M,\Lambda).
	\end{equation}
	The problem \eqref{II} is solvable if and only if the following conditions hold: 
	\begin{gather*}
		\delta v=0, \quad n (v-\delta \varphi)=0, \quad (v,h_N) = -[h_N,\varphi] \quad \forall h_N\in \mathcal{H}_N (M),\\
		df =0, \quad (f,h_N)=0 \quad \forall h_N\in \mathcal{H}_N(M).
	\end{gather*}
	
	In both cases \eqref{I}, \eqref{II}, a solution can be chosen so that \cite{Schwarz}, \cite{Sil17} for $s\geq 0$ and $p>1$ there holds
	$$
	\|\omega\|_{W^{s+1,p}(M, \Lambda)}\leq C(\|f\|_{W^{s,p}(M, \Lambda)}+ \|v\|_{W^{s,p}(M, \Lambda)} + \|\varphi\|_{W^{s+1-1/p,p}(bM, \Lambda)}),
	$$
	but the constant is generally (without presence of some norm of $\omega$ on the right) not quantitative. Any solution satisfies the general elliptic a priori estimate
	\begin{multline*}
		\|\omega\|_{W^{s+1,p}(M, \Lambda)} \\ \leq C(\|f\|_{W^{s,p}(M, \Lambda)}+ \|v\|_{W^{s,p}(M, \Lambda)} + \|\varphi\|_{W^{s+1-1/p,p}(bM, \Lambda)}+\|\omega\|_{L^1(M, \Lambda)}),
	\end{multline*}
	with a \textit{quantitative} constant.
	
 In terms of the classical vector calculus the problems \eqref{I}, \eqref{II} read as follows: in a domain $D \subset \mathbb{R}^3$ find a vector field $\vec{u}$ with given rotor and divergence
$$
\mathrm{rot}\, \vec{u} = \vec{f}, \quad \mathrm{div}\, \vec{u} = \vec{v} \quad \text{in}\quad D
$$
and given tangential component $\vec{n}\times \vec{u} = \vec{n}\times \vec{\varphi}$  or normal component $\vec{n} \cdot \vec{u} = \vec{n} \cdot \vec{\varphi}$, respectively, on $\partial D$.  For $3D$ systems from \eqref{eq:Green3D} one easily gets the following representation formula
\begin{gather*}
\vec{u}(x) = -\int\limits_D ( (\mathrm{div}\, \vec{u})(y)\nabla_y g(x,y) + [\mathrm{rot}\, \vec{u}(y)\times \nabla_y g(x,y)] )\, dy\\
+\int\limits_{\partial D} (\phi(y) \nabla_y g(x,y) + [\vec{\psi}(y)\times \nabla_y g(x,y)])\, d\sigma(y),
\end{gather*}
where $\phi(y) =\vec{u}\cdot \vec{n}$, $\vec{\psi} = \vec{n}\times \vec{u}$, $g(x,y) = -(4\pi |x-y|)^{-1}$. The solution is then sought as the sum of the volume integrals in this representation and the third/fourth integral with an unknown density. To pass to the boundary integral equation it remains to recall the jump relations for derivatives of the single layer potential and for the double layer potential. First-order div-curl type systems for alternating tensor fields were treated in \cite{Kre70} in the Euclidian case using the method of boundary integral equations (in H\"older spaces). Later this approach was expanded by J. Bolik to incorporate theory in Sobolev spaces (with constant exponent). In three-dimensional variable exponent setting certain estimates for the div-curl system were obtained in \cite{SinRi2020} by the boundary integral equations method. However in the main estimate of this work the constant is non-qualified since it comes from the functional-analytic argument (essentially the Fredholm alternative and the fact that the homogeneous boundary integral equation has only trivial solution due to the assumptions on the domain).

	\chapter{The Dirichlet and Neumann Problem in Variable Exponent Spaces}\label{sec:Dirichlet}

	In this Chapter we consider the Dirichlet and Neumann problems for the Hodge Laplacian in variable exponent spaces. Let $s\in \mathbb{N}\cup \{0\}$ and let the manifold $M$ be at least of the class $C^{1,1}$. We shall also consider two other classical variants of boundary conditions for the Hodge Laplacian: namely with $t\omega =0$ and $n\omega =0$ and with $t\delta \omega =0$, $nd\omega =0$, with the latter one postponed to Section~\ref{sec:theory}.
	
	Recall the notation $(\cdot,\cdot)$, $[\cdot,\cdot]$ and $\mathcal{D}(\cdot,\cdot)$ introduced in \eqref{eq:scp}, \eqref{eq:bound}, and \eqref{eq:Ddef}, respectively, and the notation $\|\cdot\|_{p(\cdot),M}$, $\|\cdot\|_{k,p(\cdot),M}$ for the norms in $L^{p(\cdot)}(M,\Lambda)$ and in $W^{k,p(\cdot)}(M,\Lambda)$, respectively.  We also recall the main variational relation \eqref{eq1}:
    $$
    \mathcal{D}(\omega,\zeta) = (\eta,\zeta) + (\varphi, d\zeta) + (\psi, \delta \zeta)
    $$
    for $\omega \in W^{1,p_{-}}_*(M,\Lambda)$, $\zeta \in \mathrm{Lip}_*(M,\Lambda)$, where $* \in \{0,T,N\}$. We present now the results of Section~\ref{Sec:parametrix} which provide the basis for subsequent investigations.

\begin{theorem*}[Theorems~\ref{T:p2},~\ref{T:TrueGaffney}] 
Let $M$ be of the class $C^{1,1}$. Assume $\omega \in W^{1,p_{-}}_{*}(M,\Lambda)$, $* \in \{0,T,N\}$ satisfy \eqref{eq1}  where $\varphi,\psi,\eta \in L^{p(\cdot)}(M,\Lambda)$ for any $\zeta\in \mathrm{Lip}_*(M,\Lambda)$. Then $\omega \in W^{1,p(\cdot)}(M,\Lambda)$ and
		\begin{equation*}
			\|\omega\|_{1,p(\cdot),M} \leq C(\mathrm{data}) (\|\eta\|_{p(\cdot),M} +\|\varphi\|_{p(\cdot),M} + \|\psi\|_{p(\cdot),M} + \|\omega\|_{1,M}).
		\end{equation*}
		If $M$ is of the class $C^{s+2,1}$ for $s\in \mathbb{N}\cup\{0\}$, $\varphi,\psi\in W^{s+1,p(\cdot)}(M,\Lambda)$ and $\eta \in W^{s.p(\cdot)}(M,\Lambda)$, then $\omega \in W^{s+2,p(\cdot)}(M,\Lambda)$ and
		\begin{equation*}
			\|\omega\|_{s+2,p(\cdot),M} \leq C(\mathrm{data},s) (\|\eta\|_{s,p(\cdot),M} +\|\varphi\|_{s+1,p(\cdot),M}+\|\psi\|_{s+1,p(\cdot),M} + \|\omega\|_{1,M}).
		\end{equation*}
\end{theorem*}

For $C^{1,1}$ manifolds we shall also use the following results:
\begin{theorem*}[Theorem~\ref{T:addregD}]
Let $\omega_0, d\omega_0 \in W^{1,p(\cdot)}(M,\Lambda).$ Suppose $\omega\in \omega_0+W_T^{1,p(\cdot)}(M,\Lambda)$, $\eta \in L^{p(\cdot)}(M,\Lambda)$ and $\varphi,\psi \in W^{1,p(\cdot)}(M,\Lambda)$ satisfy \eqref{eq1} for all $\zeta \in \mathrm{Lip}_T(M,\Lambda)$. Then $\alpha =\delta \omega$ and $\beta = d\omega$ belong to $W^{1,p(\cdot)}(M,\Lambda)$ and satisfy $t\alpha =t\psi$, $t\beta =td\omega_0$,
\begin{equation}\label{eq:rel10}
d\alpha + \delta \beta = \eta + d\psi + \delta \varphi,
\end{equation}
and
\begin{equation}\label{eq:rel20}
\mathcal{D}(\alpha,\zeta) = (\eta+d\psi, d\zeta), \quad \mathcal{D}(\beta,\zeta) = (\eta + \delta\varphi,\delta \zeta)
\end{equation}
for all $\zeta \in \mathrm{Lip}_T(M,\Lambda)$. 
\end{theorem*}

\begin{theorem*}[Theorem~\ref{T:addregN}]
Let $\omega_0,\delta \omega_0 \in W^{1,p(\cdot)}(M,\Lambda).$ Suppose $\omega\in \omega_0+ W_N^{1,p(\cdot)}(M,\Lambda)$, $\eta \in L^{p(\cdot)}(M,\Lambda)$ and $\varphi,\psi \in W^{1,p(\cdot)}(M,\Lambda)$ satisfy \eqref{eq1} for all $\zeta \in \mathrm{Lip}_N(M,\Lambda)$. Then $\alpha =\delta \omega$ and $\beta = d\omega$ belong to $W^{1,p(\cdot)}(M,\Lambda)$ and satisfy $n\alpha =n \delta \omega_0$, $n\beta = n\varphi$, \eqref{eq:rel10}, and \eqref{eq:rel20} for all $\zeta \in \mathrm{Lip}_N(M,\Lambda)$. 
\end{theorem*}

\begin{theorem*}[Theorem~\ref{T:addregFD}]
Let $\omega_0, d\omega_0,\delta \omega_0 \in W^{1,p(\cdot)}(M,\Lambda).$ Suppose $\omega\in \omega_0+ W_0^{1,p(\cdot)}(M,\Lambda)$, $\eta \in L^{p(\cdot)}(M,\Lambda)$ and $\varphi,\psi \in W^{1,p(\cdot)}(M,\Lambda)$ satisfy \eqref{eq1} for all $\zeta \in \mathrm{Lip}_0(M,\Lambda)$. Then $\alpha =\delta \omega$ and $\beta = d\omega$ belong to $W^{1,p(\cdot)}(M,\Lambda)$ and satisfy $n\alpha =n \delta \omega_0$, $n\beta = nd\omega_0$, \eqref{eq:rel10}, and \eqref{eq:rel20} for all $\zeta \in \mathrm{Lip}_0(M,\Lambda)$. 
\end{theorem*}

	\section{The Dirichlet problem for the Hodge Laplacian}    
	Our main result in this Section --- Theorem~\ref{T:D3} --- concerns the solvability  of the boundary value problem 
	\begin{equation}\label{eq:DD}
		\triangle \omega = \eta \quad \text{in}\quad M, \quad t\omega =t\varphi, \quad t\delta \omega = t\psi\quad \text{on}\quad bM 
	\end{equation}
 in variable exponent Sobolev spaces. When $M$ is at least $C^{s+2,1}$, $s\in \{0\}\cup \mathbb{N}$, we solve this problem in $W^{s+2,p(\cdot)}(M,\Lambda)$, with the equation understood in the a.e. sense and the boundary conditions in the sense of trace. When $M$ is only $C^{1,1}$, a weak solution $\omega$ to \eqref{eq:DD} is understood as a form $\omega \in \varphi + W_T^{1,p_{-}}(M,\Lambda)$ satisfying the following integral identity: 
	\begin{equation}\label{eq:DD1}
		\mathcal{D}(\omega,\zeta) - (\eta - d\psi,\zeta) - (\psi,\delta \zeta)=0
	\end{equation}
	for all $\zeta \in \mathrm{Lip}_T(M,\Lambda)$.  

    The central result of this Section is Theorem~\ref{T:D3}. To prove this theorem we require several auxiliary results. First, we state a regularity result of Lemma~\ref{T:D1}, which is a partial case of Theorems~\ref{T:p2}, \ref{T:TrueGaffney}. Next, we state and prove Lemma~\ref{L:quant} which allows to get rid of the $\|\omega\|_{1,M}$ term on the right-hand side of the a priori estimate of Lemma~\ref{T:D1}. Combining these results we finally prove Theorem~\ref{T:D3}. A partial case of Theorem~\ref{T:D3} --- an existence and regularity result for the Dirichlet potential --- is stated in Theorem~\ref{T:D2}.

	\begin{lemma}\label{T:D1}
		Let $\omega\in \varphi + W^{1,p_{-}}_T(M,\Lambda)$ be a solution of \eqref{eq:DD1}.

	\begin{enumerate}
		\item Let $M$ be of the class $C^{1,1}$, $\eta \in L^{p(\cdot)}(M,\Lambda)$, $\varphi,\psi \in W^{1,p(\cdot)}(M,\Lambda)$, then $\omega \in W^{1,p(\cdot)}(M,\Lambda)$ and there exists a constant $C = C(\mathrm{data})>0$ such that  
		$$
		\|\omega\|_{1,p(\cdot),M} \leq C(\mathrm{data}) (\|\eta\|_{p(\cdot),M} + \|\psi\|_{1,p(\cdot),M} + \|\varphi\|_{1,p(\cdot),M}  +\|\omega\|_{1,M}). 
		$$
		\item Let $M$ be of the class $C^{s+2,1}$. Let $\eta \in W^{s,p(\cdot)}(M,\Lambda)$, $\psi \in W^{s+1,p(\cdot)}(M,\Lambda)$, and $\varphi \in W^{s+2,p(\cdot)}(M,\Lambda)$ be given.  Then $\omega \in  W^{s+2,p(\cdot)}(M,\Lambda)$ satisfies \eqref{eq:DD} and we have the estimate 
		$$
		\|\omega\|_{s+2,p(\cdot),M} \leq C(\mathrm{data},s) (\|\eta\|_{s,p(\cdot),M} + \|\psi\|_{s+1,p(\cdot),M} + \|\varphi\|_{s+2,p(\cdot),M}  +\|\omega\|_{1,M}). 
		$$
	\end{enumerate}
	\end{lemma}
	\begin{proof}
		The form $\omega'=\omega - \varphi\in W^{1,p_{-}}_T(M,\Lambda)$ satisfies 
		$$
		\mathcal{D}(\omega',\zeta) = (\eta-d\psi,\zeta)+(\psi - d\varphi,d\zeta) +(-\delta \varphi,\delta\zeta)
		$$
		for all $\zeta \in \mathrm{Lip}_T(M,\Lambda)$. It remains to use Theorem~\ref{T:TrueGaffney} for (a) and Theorem~\ref{T:p2} for (b). If $\omega\in W^{2,p_{-}}(M,\Lambda)$ then by the integration-by-parts formula we get  the first and third relations in \eqref{eq:DD}.
	\end{proof}
	
	The constant in Lemma~\ref{T:D1} is quantitative. The constant which comes from the general elliptic results for estimates without the norm on the RHS (see \eqref{ell2}) is generally non-quantitative. But it can be estimated via the constant in the following standard inequality for $\omega \in W^{1,2}_T(M,\Lambda)$ orthogonal to $\mathcal{H}_T(M)$:
	\begin{equation}\label{eq:lambda}
		\|\omega\|_{2,M}^2\leq C_T \mathcal{D}(\omega,\omega) = C_T (\|d\omega\|_{2,M}^2+\|\delta \omega\|_{2,M}^2).
	\end{equation}
	The best such constant $C_T$ is nothing else as the reciprocal of the first positive eigenvalue of the Hodge Laplacian on $M$ with the Dirichlet boundary conditions $t\omega=0$, $t\delta \omega=0$. For the constant $C_T$ for a ball $B_R$ in $\mathbb{R}^n$, an easy scaling argument shows $C_T(B_R) = C(n) R^2$. For more on the constant, see ~\cite{CsaGyuDar12}, ~\cite{Csato13}, ~\cite{CsaDac18}, also the works in spectral geometry ~\cite{GueSav03, RauSav11}.

	\begin{lemma}\label{L:quant}
		Let $\eta,\varphi,\psi \in L^{p_{-}}(M,\Lambda)$, $\mathcal{P}_T\eta =0$. Let $\omega \in W_T^{1,p_{-}}(M,\Lambda)$ satisfy \eqref{eq1} for all $\zeta \in \mathrm{Lip}_T(M,\Lambda)$. Then there exists a constant $C = C(p_{-}, M)$ such that
		\begin{equation}\label{eq:quant}
			\|\omega - \mathcal{P}_T \omega\|_{p_{-},M} \leq C (\|\eta\|_{p_{-},M} + \|\varphi\|_{p_{-},M}  + \|\psi\|_{p_{-},M}).
		\end{equation}
	\end{lemma}
	
	\begin{proof}
		Assume without loss that $p_{-}\leq 2$. Let $\mu\in L^{p_{-}'}(M,\Lambda)$ and $\xi = G_D[\mu-\mathcal{P}_T \mu] \in W^{1,p_{-}'}_T(M,\Lambda)$. There holds
		\begin{equation}\label{omegaxi}
			\begin{gathered}
				(\omega-\mathcal{P}_T\omega,\mu) = (\mu -\mathcal{P}_T \mu,\omega-\mathcal{P}_T\omega )= \mathcal{D}(\xi,\omega - \mathcal{P}_T\omega)\\
				= \mathcal{D}(\omega - \mathcal{P}_T\omega,\xi)= (\eta,\xi) + (\varphi,d\xi) + (\psi,\delta \xi).
			\end{gathered}
		\end{equation}
		From \eqref{eq:lambda} it follows that  
		\begin{align*}
			\|\xi\|_{2,M}^2\leq C_T\mathcal{D} (\xi,\xi) = C_T(\mu-\mathcal{P}_T\mu, \xi) \leq C_T\|\mu\|_{2,M} \|\xi\|_{2,M}.
		\end{align*}
		Then 
		$$
		\|\xi\|_{2,M} \leq C_T \|\mu\|_{2,M}.
		$$
		The \textit{quantitative} elliptic estimate of Lemma~\ref{T:D1} for $\xi$ and the H\"older inequality now give
		$$
		\|\xi \|_{1,p_{-}',M} \leq C \|\mu\|_{p_{-}',M}
		$$
		with the constant $C$ which depends on the data and the constant $C_T$ in \eqref{eq:lambda}. Then from \eqref{omegaxi} we get \eqref{eq:quant}.
	\end{proof}
	
Combining the results of Lemmas~\ref{T:D1}, \ref{L:quant} we get the following

\begin{corollary}\label{C:D1}
Let $\omega\in \varphi + W^{1,p_{-}}_T(M,\Lambda)$ be a solution of \eqref{eq:DD1} and $\mathcal{P}_T(\omega-\varphi)=0$.
	\begin{enumerate}
		\item Let $M$ be of the class $C^{1,1}$, $\eta \in L^{p(\cdot)}(M,\Lambda)$, $\varphi,\psi \in W^{1,p(\cdot)}(M,\Lambda)$, then $\omega \in W^{1,p(\cdot)}(M,\Lambda)$ and there exists a constant $C = C(\mathrm{data})>0$ such that  
		$$
		\|\omega\|_{1,p(\cdot),M} \leq C(\mathrm{data}) (\|\eta\|_{p(\cdot),M} + \|\psi\|_{1,p(\cdot),M} + \|\varphi\|_{1,p(\cdot),M}). 
		$$
		\item Let $M$ be of the class $C^{s+2,1}$. Let $\eta \in W^{s,p(\cdot)}(M,\Lambda)$, $\psi \in W^{s+1,p(\cdot)}(M,\Lambda)$, and $\varphi \in W^{s+2,p(\cdot)}(M,\Lambda)$ be given.  Then $\omega \in  W^{s+2,p(\cdot)}(M,\Lambda)$ satisfies \eqref{eq:DD} and we have the estimate 
		$$
		\|\omega\|_{s+2,p(\cdot),M} \leq C(\mathrm{data},s) (\|\eta\|_{s,p(\cdot),M} + \|\psi\|_{s+1,p(\cdot),M} + \|\varphi\|_{s+2,p(\cdot),M} ). 
		$$
	\end{enumerate}
\end{corollary}
\begin{proof}
We repeat the proof of Lemma~\ref{T:D1}. Now for $\omega'= \omega -\varphi$ by assumption we have $\mathcal{P}_T \omega'=0$. Thus by Lemma~\ref{L:quant} we can estimate 
\begin{gather*}
\|\omega'\|_{p_{-},M} \leq C(\mathrm{data}) (\|\eta-d\psi\|_{p_{-},M} + \|\psi - d\varphi\|_{p_{-},M} + \|\delta \varphi\|_{p_{-},M})\\
\leq C(\mathrm{data}) (\|\eta\|_{p(\cdot),M} + \|\psi \|_{1,p(\cdot),M} + \|\varphi\|_{1,p(\cdot),M}).
\end{gather*}
It remains to use again Theorem~\ref{T:p2} for (b) and Theorem~\ref{T:TrueGaffney} for (a).
\end{proof}
    
    The central result of this Section is
	\begin{theorem}\label{T:D3}
		Consider the boundary value problem \eqref{eq:DD}. 
		\begin{enumerate}
			\item Let $M$ be $C^{1,1}$ and let  $\eta \in L^{p(\cdot)}(M,\Lambda)$, and $\varphi,\psi \in W^{1,p(\cdot)}(M,\Lambda)$ satisfy $(\eta,h_T) = [\psi,h_T]$ for all $h_T \in \mathcal{H}_T(M).$ Then there exists a solution $\omega \in \varphi+W^{1,p(\cdot)}_T(M,\Lambda)$ of \eqref{eq:DD} understood in the sense of  \eqref{eq:DD1} such that for some constant $C = C(\mathrm{data})>0,$ we have 
			$$
			\|\omega\|_{1,p(\cdot),M} \leq C ( \|\eta-d\psi\|_{p(\cdot),M} + \|\varphi\|_{1,p(\cdot),M} + \|\psi\|_{p(\cdot),M}). 
			$$
             Moreover, if $d\varphi \in W^{1,p(\cdot)}(M,\Lambda)$ then $\alpha = \delta \omega$ and $\beta = d \omega$ belong to $W^{1,p(\cdot)}(M,\Lambda)$, satisfy 
            \begin{equation}\label{eq:relAB}
            \begin{gathered}
            d\alpha + \delta \beta = \eta, \quad t\alpha= t\psi, \quad t\beta = td\varphi,\\
            \mathcal{D}(\alpha,\zeta) = (\eta,d\zeta), \quad \mathcal{D}(\beta,\zeta) = (\eta-d\psi,\delta \zeta)
            \end{gathered}
            \end{equation}
            for all $\zeta \in \mathrm{Lip}_T(M,\Lambda)$ and there holds
            \begin{equation}\label{eq:estAB}
            \begin{gathered}
            \|\alpha\|_{1,p(\cdot),M} \leq  C(\mathrm{data}) (\|\eta-d\psi\|_{p(\cdot),M}+ \|\psi\|_{1,p(\cdot),M)}),\\  
            \|\beta\|_{1,p(\cdot),M} \leq  C(\mathrm{data})  (\|\eta - d \psi\|_{p(\cdot),M} + \|d\varphi\|_{1,p(\cdot),M}) .
            \end{gathered}
            \end{equation}
           
			\item Let $M$ be of the class $C^{s+2,1}$. Let $\eta\in W^{s,p(\cdot)}(M,\Lambda)$, $\varphi\in W^{s+2,p(\cdot)}(M,\Lambda)$, and $\psi \in W^{s+1,p(\cdot)}(M,\Lambda)$ be such that $(\eta,h_T) = [\psi,h_T]$ for all $h_T \in \mathcal{H}_T(M)$. Then there exists a solution $\omega\in W^{s+2,p(\cdot)}(M,\Lambda)$ of the boundary value problem \eqref{eq:DD}
			such that for some constant $C = C(\mathrm{data},s)>0,$ we have 
			$$
			\|\omega\|_{s+2,p(\cdot),M} \leq C ( \|\eta\|_{s,p(\cdot),M} + \|\varphi\|_{s+2,p(\cdot),M} + \|\psi\|_{s+1,p(\cdot),M}).
			$$
		  Moreover, the potentials $\alpha = \delta\omega$, $\beta = d \omega$ satisfy \eqref{eq:relAB} and there holds 
                      \begin{equation}\label{eq:estAB1}
            \begin{gathered}
            \|\alpha\|_{s+1,p(\cdot),M} \leq  C(\mathrm{data},s) (\|\eta-d\psi\|_{s,p(\cdot),M} +\|\psi\|_{s+1,p(\cdot),M}),\\
            \|\beta\|_{s+1,p(\cdot),M} \leq C(\mathrm{data},s) ( \|\eta - d \psi\|_{s,p(\cdot),M}+\|d\varphi\|_{s+1,p(\cdot),M}).
            \end{gathered}
            \end{equation}
				\end{enumerate}
	\end{theorem}
	\begin{proof}
		(a) We approximate $\eta$ and $\varphi$, $\psi$  by Lipschitz forms $\eta_\varepsilon$, $\varphi_\varepsilon$, $\psi_\varepsilon$ in $L^{p(\cdot)}(M,\Lambda)$ and $W^{1,p(\cdot)}(M,\Lambda)$, correspondingly. Let $\omega_\varepsilon' \in W_T^{1,2}(M,\Lambda)\cap (\mathcal{H}_T(M))^\perp$ be the (unique) minimizer of the functional
		$$
		\frac{1}{2}\mathcal{D}(\omega',\omega') - ((\eta_\varepsilon-d\psi_\varepsilon) - \mathcal{P}_T (\eta_\varepsilon-d\psi_\varepsilon),\omega') - (\psi_\varepsilon - d\varphi_\varepsilon,d\omega') - (-\delta \varphi_\varepsilon,\delta \omega')
		$$
		over $W_T^{1,2}(M,\Lambda)\cap (\mathcal{H}_T(M))^\perp$. The existence is guaranteed by the Gaffney inequality and the direct method of the calculus of variations (see Section~\ref{ssec:varf} and \cite{Morrey1966}). Then $\omega'_\varepsilon$  and $\omega_\varepsilon = \omega'_\varepsilon + \varphi_\varepsilon$ satisfy
		\begin{gather*}
		\mathcal{D}(\omega'_\varepsilon,\zeta) =  ((\eta_\varepsilon-d\psi_\varepsilon) - \mathcal{P}_T (\eta_\varepsilon-d\psi_\varepsilon),\zeta) + (\psi_\varepsilon - d\varphi_\varepsilon,d\zeta) + (-\delta \varphi_\varepsilon,\delta \zeta),\\
        \mathcal{D}(\omega_\varepsilon,\zeta) = ((\eta_\varepsilon-d\psi_\varepsilon) - \mathcal{P}_T (\eta_\varepsilon-d\psi_\varepsilon),\zeta) + (\psi_\varepsilon,d\zeta)
		\end{gather*}
		for all $\zeta \in \mathrm{Lip}_T(M,\Lambda)$,  respectively. By Corollary~\ref{C:D1}, 
		\begin{gather*}
		\|\omega_\varepsilon\|_{1,p(\cdot),M} \leq C(\mathrm{data}) (\|\eta_\varepsilon\|_{p(\cdot),M} + \|\varphi_\varepsilon\|_{1,p(\cdot),M} + \|\psi_\varepsilon\|_{1,p(\cdot),M})\\
        \leq C(\mathrm{data}) (\|\eta\|_{p(\cdot),M} + \|\varphi\|_{1,p(\cdot),M} + \|\psi\|_{1,p(\cdot),M})
		\end{gather*}
		and the sequence $\omega_\varepsilon$ is uniformly bounded in $W^{1,p(\cdot)}(M,\Lambda)$.  By the same estimates applied to the difference $\omega_{\varepsilon_1}-\omega_{\varepsilon_2} \in (\varphi_{\varepsilon_1} - \varphi_{\varepsilon_2}) + W^{1,p_{-}}_T(M,\Lambda)$, the sequence $\omega_\varepsilon$ converges in $W^{1,p(\cdot)}(M,\Lambda)$. Passing to the limit in $\varepsilon \to 0$ and using that $\mathcal{P}_T(\eta_\varepsilon-d\psi_\varepsilon)$ converges to $\mathcal{P}_T (\eta - d\psi)=0$ in $L^{p(\cdot)}(M,\Lambda)$ we get the existence of a limit $\omega \in W^{1,p(\cdot)}(M,\Lambda)$ which satisfies the integral identity \eqref{eq:DD1} and thus is the required solution of the boundary value problem \eqref{eq:DD}. 
        
         Then by Theorem~\ref{T:addregD} for $\alpha = \delta \omega$, $\beta=d\omega$ we have $\alpha,\beta \in W^{1,p(\cdot)}(M,\Lambda)$ and \eqref{eq:relAB}. Clearly,
        \begin{equation}\label{eq:relAB0}
        \mathcal{D}(\alpha-\psi,\zeta) = (\eta - d\psi,d\zeta) + (-\delta \psi,\delta \zeta), \quad \mathcal{D}(\beta - d\varphi) = (\eta - d\psi - \delta(d\varphi), \delta \zeta)
        \end{equation}
        for all $\zeta \in \mathrm{Lip}_T(M,\Lambda)$. The estimate \eqref{eq:estAB} follow then by Theorem~\ref{T:TrueGaffney} and Lemma~\ref{L:quant} using that $\mathcal{P}_T (\beta-d\varphi) =0$ and  $\mathcal{P}_T(\alpha- \psi) = - \mathcal{P}_T \psi$.         
        
        (b) We argue as in (a). In this case by Corollary~\ref{C:D1} we get 
        \begin{gather*}
        		\|\omega_\varepsilon\|_{s+2,p(\cdot),M} \\
                \leq C(\mathrm{data},s) (\|\eta_\varepsilon\|_{s,p(\cdot),M} + \|\psi_\varepsilon\|_{s+1,p(\cdot),M} + \|\varphi_\varepsilon\|_{s+2,p(\cdot),M})\\
                \leq C(\mathrm{data},s) (\|\eta\|_{s,p(\cdot),M} + \|\psi\|_{s+1,p(\cdot),M} + \|\varphi\|_{s+2,p(\cdot),M})
        \end{gather*}
        Thus $\omega_\varepsilon$ is a bounded sequence in $W^{s+2,p(\cdot)}(M,\Lambda)$, and by a  similar estimate applied to $\omega_{\varepsilon_1} - \omega_{\varepsilon_2}$ we get the convergence of $\omega_\varepsilon$  in $W^{s+2,p(\cdot)}(M,\Lambda)$. In this case, since $\omega$ is at least $W^{2,p_{-}}(M,\Lambda)$, the direct integration-by-parts in \eqref{eq:DD1} yields \eqref{eq:DD}, and we do not need a recourse to Theorem~\ref{T:addregD} to establish the required relations \eqref{eq:relAB} for $\alpha=\delta \omega$ and $\beta = d\omega$. It remains to use Theorem~\ref{T:p2} and Lemma~\ref{L:quant} for relations \eqref{eq:relAB0} to obtain estimates \eqref{eq:estAB1}

        \end{proof}

 We state a partial case of the previous theorem as a separate result for Dirichlet potentials. Recall that on $C^{2,1}$ manifold for $\eta \in L^{p{-}}(M,\Lambda)$, $\mathcal{P}_T \eta=0$, by $G_D[\eta]$ (the Dirichlet potential of $\eta$) we denote a unique solution $\omega \in W^{2,p_{-}}(M,\Lambda)$ of 
	$$
	\triangle \omega = \eta,\quad t\omega =0,\quad t\delta \omega =0, \quad \mathcal{P}_T \omega=0.
	$$
	On $C^{1,1}$ manifold we understand $G_D[\eta]$ as the unique $W^{1,p_{-}}_T(M,\Lambda)$ solution of 
	$$
	\mathcal{D}(\omega,\zeta) = (\eta,\zeta)\quad \text{for all} \quad \zeta \in \mathrm{Lip}_T(M,\Lambda)
	$$
	satisfying $\mathcal{P}_T \omega =0$. On $C^{2,1}$ manifold these two notions are equivalent.

\begin{theorem}\label{T:D2}
	Let $M$ be $C^{1,1}$ and let $\eta\in L^{p(\cdot)}(M,\Lambda)$ and $\mathcal{P}_T \eta =0$.
	\begin{enumerate}
		\item There exist a unique Dirichlet potential $\omega:= G_D[\eta]\in W^{1,p(\cdot)}(M,\Lambda)$ and  a constant $C= C(\mathrm{data})>0$ such that 
		$$
		\|\omega\|_{1,p(\cdot),M} \leq C(\mathrm{data})\|\eta\|_{p(\cdot),M}.
		$$
		Moreover,  $\alpha = \delta G_D[\eta]$ and $\beta = d G_D[\eta]$ belong to $W_T^{1,p(\cdot)}(M,\Lambda)$, satisfy \begin{equation}\label{eq:HD}
			t\alpha=0, \quad t\beta =0, \quad\eta = d\alpha + \delta \beta,
		\end{equation}
		along with the estimate
		\begin{equation}\label{eq:estAB_GD} 
			\|\alpha\|_{1,p(\cdot),M}, \|\beta\|_{1,p(\cdot),M} \leq C(\mathrm{data}) \|\eta\|_{p(\cdot),M}.
		\end{equation}
		and the relations
		\begin{equation}\label{eq:AB}
			\mathcal{D}(\alpha,\zeta) = (\eta,d \zeta), \quad \mathcal{D}(\beta, \zeta) = (\eta,\delta \zeta)
		\end{equation}
		for all $\zeta \in \mathrm{Lip}_T (M,\Lambda)$.	
	
		\item Furthermore, if $M$ is $C^{s+2,1}$ and $\eta\in W^{s,p(\cdot)}(M,\Lambda)$ with $s\geq 0$ integer,  then $\omega: = G_D[\eta] \in W^{s+2,p(\cdot)}(M,\Lambda)\cap W_T^{1,p(\cdot)}(M,\Lambda)$ and there exists a constant $C= C(\mathrm{data},s)>0$ such that 
		$$
		\|\omega\|_{s+2,p(\cdot),M} \leq C\|\eta\|_{s,p(\cdot),M}. 
		$$ Moreover, $\alpha = \delta G_D[\eta]$ and $\beta = d G_D[\eta]$ belong to $W^{s+1,p(\cdot)}(M,\Lambda)$, 	satisfy \eqref{eq:HD}, \eqref{eq:AB},  and the estimate
		$$
		\|\alpha\|_{s+1,p(\cdot),M}, \|\beta\|_{s+1,p(\cdot),M} \leq C(\mathrm{data},s) \|\eta\|_{s,p(\cdot),M}.
		$$
		\end{enumerate}
\end{theorem}
	\begin{corollary}\label{C:commT}
		If $\eta \in W^{\delta,p_{-}}(M,\Lambda)$ then $\delta G_D[\eta-\mathcal{P}_T\eta] = G_D[\delta \eta]$ and if $\eta \in W_T^{d,p_{-}}(M,\Lambda)$ then $d G_D[\eta-\mathcal{P}_T \eta] = G_D[d\eta]$.
	\end{corollary}
	\begin{proof}
		If $\delta \eta$ is defined and at least in $L^{p_{-}}(M,\Lambda)$ (clearly this condition is nontrivial only if $s=0$) then again using the integration-by-parts formula we see that the relation in \eqref{eq:AB} for $\alpha = \delta G_D[\eta-\mathcal{P}_T\eta]$ coincides with the relation defining $G_D[\delta \eta]$. Since $t\alpha =0$ and $(\alpha,h_T) =0$ for all $h_T\in \mathcal{H}_T(M)$, we immediately arrive at $\alpha =G_D[\delta \eta]$. The same reasoning for $\beta = d G_D[\eta-\mathcal{P}_T\eta]$ yields $\beta = G_D[d\eta]$ under the additional assumption $t\eta =0$.
	\end{proof}

	\section{The Neumann problem for the Hodge Laplacian}

    In this Section we consider the boundary value problem 
    \begin{equation}\label{eq:NN}
    \triangle \omega = \eta \quad \text{in}\quad M, \quad n\omega = n\varphi, \quad nd\omega = n \psi\quad \text{on}\quad bM.
    \end{equation}
    If $M$ is at least $C^{s+2,1}$, $s\in \{0\}\cup \mathbb{N}$ then we solve this problem in $W^{s+2,p(\cdot)}(M,\Lambda)$. If $M$ is only $C^{1,1}$ then a weak solution to \eqref{eq:NN} is understood as $\omega \in \varphi + W^{1,p(\cdot)}_N(M,\Lambda)$ satisfying
    \begin{equation}\label{eq:NN1}
    \mathcal{D}(\omega,\zeta) - (\eta - \delta \psi,\zeta)  -(\psi, d\zeta)=0
    \end{equation}
    for all $\zeta \in \mathrm{Lip}_N(M,\Lambda)$.

	\begin{lemma}\label{T:N1}Let $\omega\in \varphi + W^{1,p_{-}}_N(M,\Lambda)$ be a solution of \eqref{eq:NN1}.
	\begin{enumerate}
		\item Let $M$ be of the class $C^{1,1}$, $\eta \in L^{p(\cdot)}(M,\Lambda)$, $\varphi,\psi \in W^{1,p(\cdot)}(M,\Lambda)$, then $\omega \in W^{1,p(\cdot)}(M,\Lambda)$ and there exists a constant $C = C(\mathrm{data})>0$ such that  
		$$
		\|\omega\|_{1,p(\cdot),M} \leq C(\mathrm{data}) (\|\eta\|_{p(\cdot),M} + \|\psi\|_{1,p(\cdot),M} + \|\varphi\|_{1,p(\cdot),M}  +\|\omega\|_{1,M}). 
		$$
		\item Let $M$ be of the class $C^{s+2,1}$. Let $\eta \in W^{s,p(\cdot)}(M,\Lambda)$, $\psi \in W^{s+1,p(\cdot)}(M,\Lambda)$, and $\varphi \in W^{s+2,p(\cdot)}(M,\Lambda)$ be given.  Then $\omega \in  W^{s+2,p(\cdot)}(M,\Lambda)$ satisfies \eqref{eq:NN} and we have the estimate 
		$$
		\|\omega\|_{s+2,p(\cdot),M} \leq C(\mathrm{data},s) (\|\eta\|_{s,p(\cdot),M} + \|\psi\|_{s+1,p(\cdot),M} + \|\varphi\|_{s+2,p(\cdot),M}  +\|\omega\|_{1,M}). 
		$$
	\end{enumerate}
	\end{lemma}
	
	For $\omega \in W_N^{1,2}(M,\Lambda)$ such that $\mathcal{P}_N \omega =0$ there exists a constant $C_N=C_N(M)$ such that
	$$
	\|\omega\|_{2,M}^2\leq C_N \mathcal{D}(\omega,\omega).
	$$
	Recall that if $M$ is at least $C^{2,1}$ for $\eta\in L^{p_{-}}(M,\Lambda)$, $\mathcal{P}_N\eta=0$, by $G_N[\eta]$ (the Neumann potential of $\eta$) we denote a unique solution $\omega \in W^{2,p_{-}}(M,\Lambda)$ of
	$$
	\triangle \omega = \eta,\quad n\omega =0,\quad n d\omega =0, \quad \mathcal{P}_N \omega=0.
	$$
	On $C^{1,1}$ manifold we understand $G_N[\eta]$ as the unique $W^{1,p_{-}}_N(M,\Lambda)$ solution of 
	$$
	\mathcal{D}(\omega,\zeta) = (\eta,\zeta)\quad \text{for all} \quad \zeta \in \mathrm{Lip}_N(M,\Lambda),
	$$
	satisfying $\mathcal{P}_N \omega =0$. On $C^{2,1}$ manifold these two notions coincide.
	
	The following statements are proved in the same way as for the Dirichlet boundary conditions or can be obtained by duality.
	\begin{lemma}\label{L:quant1}
		Let $\eta,\varphi,\psi \in L^{p_{-}}(M,\Lambda)$, $\mathcal{P}_N\eta =0$. Let $\omega \in W_N^{1,p_{-}}(M,\Lambda)$ satisfy \eqref{eq1} for all $\zeta \in \mathrm{Lip}_N(M,\Lambda)$. Then there exists a constant $C = C(p_{-}, M)$ such that
		\begin{equation}\label{eq:quant1}
			\|\omega - \mathcal{P}_N \omega\|_{p_{-},M} \leq C (\|\eta\|_{p_{-},M} + \|\varphi\|_{p_{-},M}  + \|\psi\|_{p_{-},M}).
		\end{equation}
	\end{lemma}

	\begin{corollary}\label{C:N1} Let $\omega\in \varphi + W^{1,p_{-}}_N(M,\Lambda)$ be a solution of \eqref{eq:NN1} and $\mathcal{P}_T (\omega-\varphi)=0$.
	\begin{enumerate}
		\item Let $M$ be of the class $C^{1,1}$, $\eta \in L^{p(\cdot)}(M,\Lambda)$, $\varphi,\psi \in W^{1,p(\cdot)}(M,\Lambda)$, then $\omega \in W^{1,p(\cdot)}(M,\Lambda)$ and there exists a constant $C = C(\mathrm{data})>0$ such that  
		$$
		\|\omega\|_{1,p(\cdot),M} \leq C(\mathrm{data}) (\|\eta\|_{p(\cdot),M} + \|\psi\|_{1,p(\cdot),M} + \|\varphi\|_{1,p(\cdot),M}). 
		$$
		\item Let $M$ be of the class $C^{s+2,1}$. Let $\eta \in W^{s,p(\cdot)}(M,\Lambda)$, $\psi \in W^{s+1,p(\cdot)}(M,\Lambda)$, and $\varphi \in W^{s+2,p(\cdot)}(M,\Lambda)$ be given.  Then $\omega \in  W^{s+2,p(\cdot)}(M,\Lambda)$ satisfies \eqref{eq:NN} and we have the estimate 
		$$
		\|\omega\|_{s+2,p(\cdot),M} \leq C(\mathrm{data},s) (\|\eta\|_{s,p(\cdot),M} + \|\psi\|_{s+1,p(\cdot),M} + \|\varphi\|_{s+2,p(\cdot),M}). 
		$$
	\end{enumerate}
	\end{corollary}

	\begin{theorem}\label{T:N3}
Consider the boundary value problem \eqref{eq:NN}. 
		\begin{enumerate}
			\item Let $M$ be $C^{1,1}$ and let  $\eta \in L^{p(\cdot)}(M,\Lambda)$, and $\varphi,\psi \in W^{1,p(\cdot)}(M,\Lambda)$ satisfy $(\eta,h_N) = -[\psi,h_N]$ for all $h_N \in \mathcal{H}_N(M).$ Then there exists a solution $\omega \in \varphi+W^{1,p(\cdot)}_N(M,\Lambda)$ of \eqref{eq:NN} understood in the sense of  \eqref{eq:NN1} such that for some constant $C = C(\mathrm{data})>0$ we have 
			$$
			\|\omega\|_{1,p(\cdot),M} \leq C ( \|\eta-d\psi\|_{p(\cdot),M} + \|\varphi\|_{1,p(\cdot),M} + \|\psi\|_{p(\cdot),M}). 
			$$
             Moreover, if $\delta \varphi \in W^{1,p(\cdot)}(M,\Lambda)$ then $\alpha = \delta \omega$ and $\beta = d \omega$ belong to $W^{1,p(\cdot)}(M,\Lambda)$, satisfy 
            \begin{equation}\label{eq:relAB_N}
            \begin{gathered}
            d\alpha + \delta \beta = \eta, \quad n\alpha= n\delta\varphi, \quad n\beta = n\psi,\\
            \mathcal{D}(\alpha,\zeta) = (\eta-\delta \psi,d\zeta), \quad \mathcal{D}(\beta,\zeta) = (\eta,\delta \zeta)
            \end{gathered}
            \end{equation}
            for all $\zeta \in \mathrm{Lip}_T(M,\Lambda)$ and there holds
            \begin{equation}\label{eq:estAB_N}
            \begin{gathered}
            \|\alpha\|_{1,p(\cdot),M} \leq  C(\mathrm{data})( \|\eta-\delta\psi\|_{p(\cdot),M}+\|\delta\varphi\|_{1,p(\cdot),M}),\\
            \|\beta\|_{1,p(\cdot),M} \leq C(\mathrm{data})  (\|\eta - \delta \psi\|_{p(\cdot),M}+\|\psi\|_{1,p(\cdot),M} ).
            \end{gathered}
            \end{equation}
               
			\item Let $M$ be of the class $C^{s+2,1}$. Let $\eta\in W^{s,p(\cdot)}(M,\Lambda)$, $\varphi\in W^{s+2,p(\cdot)}(M,\Lambda)$, and $\psi \in W^{s+1,p(\cdot)}(M,\Lambda)$ be such that $(\eta,h_N) = -[\psi,h_N]$ for all $h_N \in \mathcal{H}_N(M)$. Then there exists a solution $\omega\in W^{s+2,p(\cdot)}(M,\Lambda)$ of the boundary value problem \eqref{eq:NN} 
			such that for some constant $C = C(\mathrm{data},s)>0,$ we have 
			$$
			\|\omega\|_{s+2,p(\cdot),M} \leq C ( \|\eta\|_{s,p(\cdot),M} + \|\varphi\|_{s+2,p(\cdot),M} + \|\psi\|_{s+1,p(\cdot),M}).
			$$
		  Moreover, the potentials $\alpha = \delta\omega$, $\beta = d \omega$ satisfy \eqref{eq:relAB_N} and there holds 
                      \begin{equation}\label{eq:estAB1_N}
            \begin{gathered}
            \|\alpha\|_{s+1,p(\cdot),M} \leq C(\mathrm{data}) (\|\eta-\delta\psi\|_{s,p(\cdot),M}+\|\delta\varphi\|_{s+1,p(\cdot),M}),\\
            \|\beta\|_{s+1,p(\cdot),M} \leq  C(\mathrm{data}) (\|\eta - \delta \psi\|_{s,p(\cdot),M}+\|\psi\|_{s+1,p(\cdot),M} ).
            \end{gathered}
            \end{equation}
				\end{enumerate}
	\end{theorem}

   We state a partial case of this theorem as a separate result for Neumann potentials.
    
	\begin{theorem}\label{T:N2}
		(i) Let $M$ be a $C^{s+2,1}$ manifold. Let $\eta\in W^{s,p(\cdot)}(M,\Lambda)$ and $\mathcal{P}_T \eta =0$. Then there exist a unique Neumann potential $\omega: = G_N[\eta] \in W^{s+2,p(\cdot)}(M,\Lambda)\cap W_T^{1,p(\cdot)}(M,\Lambda)$ of $\eta$ satisfying 
		$$
		\|\omega\|_{s+2,p(\cdot),M} \leq C\|\eta\|_{s,p(\cdot),M}
		$$ 
		where $C= C(\mathrm{data},s)$. Moreover, we also have that $\alpha = \delta G_N[\eta]$ and $\beta = d G_N[\eta]$ belong to $W^{s+1,p(\cdot)}(M,\Lambda)$,  satisfy 
		\begin{equation}\label{eq:HN}
			n\alpha=0, \quad n\beta =0, \quad\eta = d\alpha + \delta \beta,
		\end{equation}
		the estimate
		$$
		\|\alpha\|_{s+1,p(\cdot),M}, \|\beta\|_{s+1,p(\cdot),M} \leq C(\mathrm{data},s) \|\eta\|_{s,p(\cdot),M}.
		$$
		and the relations
		\begin{equation}\label{eq:ABN}
			\mathcal{D}(\alpha,\zeta) = (\eta,d \zeta), \quad \mathcal{D}(\beta, \zeta) = (\eta,\delta \zeta)
		\end{equation}
		for all $\zeta \in \mathrm{Lip}_N (M,\Lambda)$. 
		
		(ii) If $s=0$ and $M$ is only $C^{1,1}$ then $G_N[\eta]\in W^{1,p(\cdot)}(M,\Lambda)$ and satisfies
		$$
		\|\omega\|_{1,p(\cdot),M} \leq C(\mathrm{data})\|\eta\|_{p(\cdot),M}.
		$$
		Moreover,  $\alpha = \delta G_N[\eta]$ and $\beta = d G_N[\eta]$ belong to $W_N^{1,p(\cdot)}(M,\Lambda)$, satisfy \eqref{eq:ABN}, \eqref{eq:HN} and
		$$
		\|\alpha\|_{1,p(\cdot),M}, \|\beta\|_{1,p(\cdot),M} \leq C(\mathrm{data}) \|\eta\|_{p(\cdot),M}.
		$$
	\end{theorem}
	
	\begin{corollary}\label{C:commN}
		If $\eta \in W^{\delta,p_{-}}_N(M,\Lambda)$ then $\delta G_N[\eta-\mathcal{P}_N\eta] = G_N[\delta \eta]$ and if $\eta \in W^{d,p_{-}}(M,\Lambda)$ then $d G_N[\eta-\mathcal{P}_N\eta] = G_N[d\eta]$.
	\end{corollary}

	\section{Full Dirichlet data}

In this Section we consider the boundary value problem
	\begin{equation}\label{eq:ZZ}
		\triangle \omega = \eta \quad \text{in}\quad M, \quad \omega =\varphi\quad \text{on}\quad bM 
	\end{equation}
 in variable exponent Sobolev spaces. When $M$ is at least $C^{s+2,1}$, $s\in \{0\}\cup \mathbb{N}$, we solve this problem in $W^{s+2,p(\cdot)}(M,\Lambda)$, with the equation understood in the a.e. sense and the boundary conditions in the sense of trace. When $M$ is only $C^{1,1}$, a weak solution $\omega$ to \eqref{eq:ZZ} is understood as a form $\omega \in \varphi + W_0^{1,p_{-}}(M,\Lambda)$ satisfying the following integral identity: 
	\begin{equation}\label{eq:ZZ1}
		\mathcal{D}(\omega,\zeta) - (\eta,\zeta) =0
	\end{equation}
	for all $\zeta \in \mathrm{Lip}_0(M,\Lambda)$.

	\begin{lemma}\label{T:Z1} (i) Let $M$ be $C^{s+2,1}$. Assume $\eta \in W^{s,p(\cdot)}(M,\Lambda)$ and $\varphi \in W^{s+2,p(\cdot)}(M,\Lambda)$ be given. Let $\omega\in \varphi + W^{1,p_{-}}_0(M,\Lambda)$ satisfy \eqref{eq:ZZ1} for all $\zeta \in \mathrm{Lip}_0(M,\Lambda)$. Then $\omega \in  W^{s+2,p(\cdot)}(M,\Lambda)$, there holds 
		$$
		\|\omega\|_{s+2,p(\cdot),M} \leq C(\mathrm{data},s) (\|\eta\|_{s,p(\cdot),M} + \|\varphi\|_{s+2,p(\cdot),M}  +\|\omega\|_{1,M}),
		$$
		and $\omega$ satisfies \eqref{eq:ZZ}.
        
		(ii) If $M$ is of the class $C^{1,1}$, $\eta \in L^{p(\cdot)}(M,\Lambda)$, $\varphi \in W^{1,p(\cdot)}(M,\Lambda)$, then $\omega \in \varphi+ W^{1,p(\cdot)}_0(M,\Lambda)$ and there holds 
		$$
		\|\omega\|_{1,p(\cdot),M} \leq C(\mathrm{data},s) (\|\eta\|_{p(\cdot),M} + \|\varphi\|_{1,p(\cdot),M}  +\|\omega\|_{1,M})
		$$
	\end{lemma}
	
	It is well known (see, for instance \cite{Morrey1966}) that $\mathcal{H}_T(M) \cap \mathcal{H}_N (M) = \{0\}$, so the problem \eqref{eq:ZZ1} has a unique solution in $ W^{1,2}_0(M)$ for all $\eta\in L^2(M)$ and for all $\omega \in W_0^{1,2}(M)$ there holds 
	$$
	(\omega,\omega) \leq C_0 \mathcal{D}(\omega,\omega),
	$$
	where $C_0$ is the reciprocal of the first eigenvalue of the Hodge Laplacian with boundary conditions $t\omega=0$, $n\omega=0$. Further for $\eta \in L^{p_{-}}(M,\Lambda)$ we denote the solution $\omega \in W_0^{1,p_{-}}(M,\Lambda)$ to \eqref{eq:ZZ1} by $G_0[\eta]$.
	
	The same arguments as before give the following results. 
	
	\begin{lemma}\label{L:quant0}
		Let $\eta,\varphi,\psi \in L^{p_{-}}(M,\Lambda)$. Let $\omega \in W_0^{1,p_{-}}(M,\Lambda)$ satisfy \eqref{eq1} for all $\zeta \in \mathrm{Lip}_0(M,\Lambda)$. Then there exists a constant $C = C(p_{-}, M)$ such that
		\begin{equation}\label{eq:quant0}
			\|\omega\|_{p_{-},M} \leq C (\|\eta\|_{p_{-},M} + \|\varphi\|_{p_{-},M}  + \|\psi\|_{p_{-},M}).
		\end{equation}
	\end{lemma}

\begin{corollary}\label{C:Z1}
Let $\omega\in \varphi + W^{1,p_{-}}_0(M,\Lambda)$ be a solution of \eqref{eq:ZZ1}.
	\begin{enumerate}
		\item Let $M$ be of the class $C^{1,1}$, $\eta \in L^{p(\cdot)}(M,\Lambda)$, $\varphi \in W^{1,p(\cdot)}(M,\Lambda)$, then $\omega \in W^{1,p(\cdot)}(M,\Lambda)$ and there exists a constant $C = C(\mathrm{data})>0$ such that  
		$$
		\|\omega\|_{1,p(\cdot),M} \leq C(\mathrm{data}) (\|\eta\|_{p(\cdot),M} + \|\varphi\|_{1,p(\cdot),M}). 
		$$
		\item Let $M$ be of the class $C^{s+2,1}$. Suppose $\eta \in W^{s,p(\cdot)}(M,\Lambda)$ and $\varphi \in W^{s+2,p(\cdot)}(M,\Lambda)$ be given.  Then $\omega \in  W^{s+2,p(\cdot)}(M,\Lambda)$ satisfies \eqref{eq:ZZ} and we have the estimate 
		$$
		\|\omega\|_{s+2,p(\cdot),M} \leq C(\mathrm{data},s) (\|\eta\|_{s,p(\cdot),M} + \|\varphi\|_{s+2,p(\cdot),M}). 
		$$
	\end{enumerate}
\end{corollary}

\begin{theorem}\label{T:Z3}
Consider the boundary value problem \eqref{eq:ZZ}. 
		\begin{enumerate}
			\item Let $M$ be $C^{1,1}$ and let  $\eta \in L^{p(\cdot)}(M,\Lambda)$, and $\varphi \in W^{1,p(\cdot)}(M,\Lambda)$. Then there exists a solution $\omega \in \varphi+W^{1,p(\cdot)}_0(M,\Lambda)$ of \eqref{eq:ZZ} understood in the sense of  \eqref{eq:ZZ1} such that for some constant $C = C(\mathrm{data})>0$ we have 
			$$
			\|\omega\|_{1,p(\cdot),M} \leq C (\|\eta\|_{p(\cdot),M} + \|\varphi\|_{1,p(\cdot),M}). 
			$$
            The potentials $\alpha = \delta \omega$ and $\beta = d \omega$ belong to $W^{1,p(\cdot)}_{\mathrm{loc}}(M,\Lambda)$ and satisfy
            $d\alpha + \delta \beta = \eta$.

			\item Let $M$ be $C^{s+2,1}$. Let $\eta\in W^{s,p(\cdot)}(M,\Lambda)$ and $\varphi\in W^{s+2,p(\cdot)}(M,\Lambda)$. Then there exists a solution $\omega\in W^{s+2,p(\cdot)}(M,\Lambda)$ of the boundary value problem \eqref{eq:ZZ}
			such that for some constant $C = C(\mathrm{data},s)>0,$ we have 
			$$
			\|\omega\|_{s+2,p(\cdot),M} \leq C ( \|\eta\|_{s,p(\cdot),M} + \|\varphi\|_{s+2,p(\cdot),M}).
			$$
				\end{enumerate}
\end{theorem}

	\section{Hodge decomposition in variable exponent spaces}\label{sec:Hodge}
	
	The following theorems are standard Hodge decomposition results stated in variable exponent setting. The proofs are fairly standard and included here only for the reader's convenience. Let $s\in \mathbb{N} \cup\{0\}$, and $M$ be of the class $C^{s+1,1}$. In the following statements $C$ is a quantitative constant which depends only on $p_{-}$,  $p_{+}$,  $c_{\mathrm{log}}(p)$, $M$, and $s$.

	\begin{theorem}\label{T:HodgeD}
		Let $\omega\in W^{s,p(\cdot)}(M,\Lambda)$. There exist $\alpha,\beta \in W^{s+1,p(\cdot)}(M,\Lambda)$ and $h\in \mathcal{H}_T(M)$ such that 
		\begin{gather*}
			\omega = h + d\alpha + \delta \beta,\\
			t\alpha =0, \quad \delta \alpha =0,\quad t\beta =0, \quad d \beta =0,\\
			\|\alpha\|_{s+1,p(\cdot),M}, \|\beta\|_{s+1,p(\cdot),M} \leq C \|\omega\|_{ s,p(\cdot),M}.
		\end{gather*}
		These forms are given by $h = \mathcal{P}_T \omega$, $\alpha = \delta G_D[\omega - h]$, $\beta = d G_D[\omega -h]$.
	\end{theorem}
	\begin{proof}
		First we use Theorem ~\ref{T:D2}. Since we assume that $M$ is only $C^{s+1,1}$ we get only $G_D[\omega-h] \in W^{s+1,p(\cdot)}(M,\Lambda)$ and $\alpha,\beta \in W^{\max(s,1),p(\cdot)}(M.\Lambda)$. Since $\alpha,\beta$ satisfy the relations \eqref{eq:AB} (with $\eta=\omega$), we can use Theorem~\ref{T:p2} (or Theorem~\ref{T:TrueGaffney} for $s=0$) to establish that $\alpha,\beta \in W^{s+1,p(\cdot)}(M,\Lambda)$ and satisfy the required estimate.
	\end{proof}
	
	\begin{theorem}\label{T:HodgeN}
		Let $\omega\in W^{s,p(\cdot)}(M,\Lambda)$, $s\in \mathbb{N} \cup\{0\}$. Then there exist $\alpha,\beta \in W^{s+1,p(\cdot)}(M,\Lambda)$ and $h\in \mathcal{H}_N(M)$ such that 
		\begin{gather*}
			\omega = h + d\alpha + \delta \beta,\\
			n\alpha =0, \quad \delta \alpha =0,\quad n\beta =0, \quad d \beta =0,\\
			\|\alpha\|_{s+1,p(\cdot),M}, \|\beta\|_{s+1,p(\cdot),M} \leq C \|\omega\|_{s,p(\cdot),M}.
		\end{gather*}
		These forms are given by $h = \mathcal{P}_N \omega$, $\alpha = \delta G_N[\omega - h]$, $\beta = d G_N[\omega -h]$.
	\end{theorem}
	\begin{proof}
		We use Theorem~\ref{T:N2}, and the same argument as in the previous theorem to establish the $W^{s+1,p(\cdot)}(M,\Lambda)$ regularity for $\alpha$ and $\beta$.
	\end{proof}

	\begin{theorem}\label{T:Hodge}
		Let $\omega\in W^{s,p(\cdot)}(M,\Lambda)$, $s\in \mathbb{N} \cup\{0\}$. Then there exist $\alpha,\beta \in W^{s+1,p(\cdot)}(M,\Lambda)$ and $h\in W^{s,p(\cdot)}(M,\Lambda) \cap \mathcal{H}(M)$ such that 
		\begin{gather*}
			\omega = h + d\alpha + \delta \beta,\\
			t\alpha =0, \quad \delta \alpha =0,\quad n\beta =0, \quad d \beta =0,\\
			\|h\|_{s,p(\cdot),M},  \|\alpha\|_{s+1,p(\cdot),M}, \|\beta\|_{s+1,p(\cdot),M} \leq C \|\omega\|_{s,p(\cdot),M}.
		\end{gather*}
	\end{theorem}
	
	\begin{proof}
		Take $\alpha = \delta G_D[\omega - \mathcal{P}_T\omega]$, $\beta = d G_N[\omega -\mathcal{P}_N \omega]$ and let
		$$
		h = \omega - d\alpha - \delta \beta.
		$$
		By the same reasoning as in the proof of Theorem~\ref{T:HodgeD} we get $\alpha,\beta \in W^{s+1,p(\cdot)}(M,\Lambda)$ and the required estimate. 
		
		For any $\xi \in \mathrm{Lip}_N(M,\Lambda)$ using the integration-by-parts formula (see Lemma~\ref{L:ort}) and the relation \eqref{eq:ABN} for $\beta$ we find
		\begin{align*}
			(h,\delta \xi) = (\omega,\delta\xi) - (\delta \beta, \delta \xi) = (\omega- \mathcal{P}_N\omega,\delta\xi) - (\delta \beta, \delta \xi) \\
			= (\delta\beta,\delta\xi) - (\delta\beta , \delta \xi)=0.
		\end{align*} 
		Similarly we get $(h,d\xi)=0$ for any $\xi \in \mathrm{Lip}_T(M,\Lambda)$. Thus $dh=0$, $\delta h=0$.
	\end{proof}

With the principal results on the solvability and regularity of solutions of boundary value problems for the Hodge Laplacian established we proceed to the study of first-order systems of div-curl type.

	\chapter{Theory of First-Order Systems in Variable Exponent Spaces.}\label{sec:theory}
	
	In this Chapter we present results on solvability of \eqref{I}, \eqref{II} in variable exponent Sobolev spaces, and related results. Let $M$ be of the class $C^{s+1,1}$, $s\in \{0\}\cup \mathbb{N}$. Recall that by $G_D[\eta]$ and $G_N[\eta]$ we denote the Dirichlet and Neumann potentials of $\eta$, respectively, while $\mathcal{P}_T$ and $\mathcal{P}_N$ denote the standard projectors on the spaces of Dirichlet and Neumann harmonic fields. 

	\section{Cohomology resolution in variable exponent spaces}
	In the lemmas below $C$ is a quantitative constant which depends on $p_{-}$, $p_{+}$, $c_{\mathrm{log}}(p)$, $M$, and $s$. 
	
	\begin{lemma}\label{L:S1}
		Let $f\in W^{s,p(\cdot)}(M,\Lambda)$ satisfy $df=0$, $tf=0$ and  $\mathcal{P}_T f=0$. Let $G_D[f]$ be the Dirichlet potential of $f$ and $\alpha =\delta G_D[f]$. Then 
		
		\begin{itemize}
			\item  $f=d\alpha$, $\delta \alpha =0$, $\mathcal{P}_T \alpha =0$;
			
			\item $\alpha \in W^{s+1,p(\cdot)}(M,\Lambda)\cap W_T^{1,p(\cdot)}(M,\Lambda)$, $\|\alpha\|_{s+1,p(\cdot),M} \leq C \|f\|_{s,p(\cdot),M}$.
		\end{itemize}
	\end{lemma}
	
	\begin{proof}
		By Theorem~\ref{T:HodgeD} we have $f =d \alpha + \delta\beta$   
		where $\alpha = \delta G_D[f]$, $\beta = d G_D[f]$ and
		$$
		\|\alpha\|_{s+1,p(\cdot),M} + \|\beta\|_{s+1,p(\cdot),M} \leq C(\mathrm{data},s) \|f\|_{s,p(\cdot),M}.
		$$
		By Corollary~\ref{C:commT}, $\beta = d G_D[f] = G_D[df]=G_D[0]=0$, thus $f=d\alpha$.
	\end{proof}

	\begin{lemma}\label{L:S2}
		Let $f\in W^{s,p(\cdot)}(M,\Lambda)$ satisfy $df=0$ and $\mathcal{P}_N f=0$. Let $G_N[f]$ be the Neumann potential of $f$ and $\alpha =\delta G_N[f]$. Then
		\begin{itemize} 
			\item  $f=d\alpha$, $\delta \alpha =0$, $\mathcal{P}_N \alpha =0$;
			
			\item $\alpha \in W^{s+1,p(\cdot)}(M,\Lambda)\cap W_N^{1,p(\cdot)}(M,\Lambda)$, $\|\alpha\|_{s+1,p(\cdot),M} \leq C \|f\|_{s,p(\cdot),M}$.
			
		\end{itemize}
		
	\end{lemma}
	
	\begin{proof}
		By Theorem~\ref{T:HodgeN} we have $f =d \alpha + \delta\beta$   
		where $\alpha = \delta G_D[f]$, $\beta = d G_D[f]$ and
		$$
		\|\alpha\|_{s+1,p(\cdot),M} + \|\beta\|_{s+1,p(\cdot),M} \leq C(\mathrm{data},s) \|f\|_{s,p(\cdot),M}.
		$$
		By Corollary~\ref{C:commN}, $\beta = d G_N[f] = G_N[df]=G_N[0]=0$, thus $f=d\alpha$.
	\end{proof}

	In the same way we obtain the following results.
	
	\begin{lemma}\label{L:S3}
		Let $g\in W^{s,p(\cdot)}(M,\Lambda)$ satisfy $\delta g=0$, $ng=0$ and $\mathcal{P}_N g=0$. Let $G_N[g]$ be the Neumann potential of $g$ and $\beta =d G_N[g]$. Then
		\begin{itemize} 
			\item  $g=\delta\beta$, $d \beta =0$, $\mathcal{P}_N \beta =0$;
			
			\item  $\beta\in W^{s+1,p(\cdot)}(M,\Lambda)\cap W_N^{1,p(\cdot)}(M,\Lambda)$, $\|\beta\|_{s+1,p(\cdot),M} \leq C \|g\|_{s,p(\cdot),M}$.
			
		\end{itemize}
	\end{lemma}
	
	\begin{lemma}\label{L:S4}
		Let $g\in W^{s,p(\cdot)}(M,\Lambda)$ satisfy $\delta g=0$ and $\mathcal{P}_T g=0$. Let $G_D[g]$ be the Dirichlet potential of $g$ and $\beta =d G_D[g]$. Then
		\begin{itemize} 
			\item $g=\delta\beta$, $d \beta =0$, $\mathcal{P}_T \beta=0$;
			
			\item $\beta \in W^{s+1,p(\cdot)}(M,\Lambda) \cap W_T^{1,p(\cdot)}(M,\Lambda)$, $\|\beta\|_{s+1,p(\cdot),M} \leq C \|g\|_{s,p(\cdot),M}$.
			
			
		\end{itemize}
	\end{lemma}
	

	\section{Gauge fixing}
	The following statements are immediate corollaries of Lemmas~\ref{L:S1}, \ref{L:S2}, \ref{L:S3}, \ref{L:S4} respectively.
	
	\begin{corollary}\label{C:I0}
		Let $\eta \in W_T^{d,p_{-}}(M,\Lambda)$ with $d\eta \in W^{s,p(\cdot)}(M,\Lambda)$. Then there exists $\omega \in W^{s+1,p(\cdot)}(M,\Lambda)\cap W_T^{1,p(\cdot)}(M,\Lambda)$ such that $d\omega = d\eta$, $\delta \omega =0$, $\mathcal{P}_T \omega=0$, and 
		\begin{equation}\label{L:I0_est}
			\|\omega\|_{s+1,p(\cdot),M} \leq C \|d\eta\|_{s,p(\cdot),M}
		\end{equation}
		with a constant $C=C(\mathrm{data},s)$. Precisely, the solution is given by $\omega = \delta G_D[d\eta]$.
	\end{corollary}
	
	\begin{corollary}\label{C:I1}
		Let $\eta \in W^{d,p_{-}}(M,\Lambda)$ with $d\eta \in  W^{s,p(\cdot)}(M,\Lambda)$. Then there exists $\omega \in W^{s+1,p(\cdot)}(M,\Lambda)\cap W_N^{1,p(\cdot)}(M,\Lambda)$ such that $d\omega = d\eta$, $\delta \omega =0$, $\mathcal{P}_N \omega=0$, and 
		\begin{equation}\label{L:I1_est}
			\|\omega\|_{s+1,p(\cdot),M} \leq C \|d\eta\|_{s,p(\cdot),M}
		\end{equation}
		with a constant $C=C(\mathrm{data},s)$. Precisely, the solution is given by $\omega = \delta G_N[d\eta]$.
	\end{corollary}

	\begin{corollary}\label{C:I2}
		Let $\eta \in W_N^{\delta ,p_{-}}(M,\Lambda)$ with $\delta\eta \in W^{s,p(\cdot)}(M,\Lambda)$. Then there exists $\omega \in W^{s+1,p(\cdot)}(M,\Lambda)\cap W_N^{1,p(\cdot)}(M,\Lambda)$ such that $\delta\omega = \delta\eta$, $d \omega =0$, $\mathcal{P}_N \omega=0$, and 
		\begin{equation}\label{L:I2_est}
			\|\omega\|_{s+1,p(\cdot),M} \leq C \|\delta\eta\|_{s,p(\cdot),M}
		\end{equation}
		with a constant $C=C(\mathrm{data},s)$. Precisely, the solution is given by $\omega = d G_N[\delta \eta]$.
	\end{corollary}
	
	\begin{corollary}\label{C:I3}
		Let $\eta \in W^{\delta ,p_{-}}(M,\Lambda)$ with $d\eta \in  W^{s,p(\cdot)}(M,\Lambda)$. Then there exists $\omega \in W^{s+1,p(\cdot)}(M,\Lambda)\cap W_T^{1,p(\cdot)}(M,\Lambda)$ such that $\delta\omega = \delta\eta$, $d \omega =0$, $\mathcal{P}_T \omega=0$, and 
		\begin{equation}\label{L:I3_est}
			\|\omega\|_{s+1,p(\cdot),M} \leq C \|\delta\eta\|_{s,p(\cdot),M}
		\end{equation}
		with a constant $C=C(\mathrm{data},s)$. Precisely, the solution is given by $\omega = d G_T[\delta \eta]$.
	\end{corollary}

	Now we return to the Hodge decomposition. Corollaries~\ref{C:I0}--\ref{C:I3} lead to the following statement.
	
	\begin{corollary}
		All of the following subspaces are closed in $W^{s,p(\cdot)}(M,\Lambda):$ 
		\begin{align*}
			d W^{s+1,p(\cdot)}_T(M,\Lambda), \quad  d W^{s+1,p(\cdot)}(M,\Lambda), \quad dW_N^{s+1,p(\cdot)}(M,\Lambda) \\
			\delta W^{s+1,p(\cdot)}_N(M,\Lambda), \quad \delta W^{s+1,p(\cdot)}(M,\Lambda), \quad \delta W^{s+1,p(\cdot)}_T(M,\Lambda). 
		\end{align*}
	\end{corollary}

	\section{Solvability of first-order systems}
	In this Section we extend the classical results on solvability of first-order ``div-curl'' type systems to variable exponent Lebesgue spaces.
	\begin{theorem}\label{L:sysD}
		Let 
		\begin{itemize}
			
			\item $f\in W^{s,p(\cdot)}(M,\Lambda)$, $df=0$; 
			
			\item $v\in W^{s,p(\cdot)}(M,\Lambda)$, $\delta v=0$, $\mathcal{P}_T v =0$;
			
			\item $\varphi \in W^{s+1,p(\cdot)}(M,\Lambda)$ satisfy $(f,h_T) =[\varphi,h_T]$ for all $h_T \in \mathcal{H}_T(M)$;
			
			\item and $t(f-d\varphi)=0$.
			
		\end{itemize}
		
		Then there exists a unique solution $\omega \in W^{s+1,p(\cdot)}(M,\Lambda)$ of the boundary value problem
		$$
		d\omega = f, \quad \delta \omega =v, \quad t\omega = t\varphi
		$$
		such that $\mathcal{P}_T (\omega - \varphi)=0$ and 
		\begin{align*}
			\|\omega\|_{s+1,p(\cdot),M} \leq C ( \|f\|_{s,p(\cdot),M} + \|v\|_{s,p(\cdot),M} + \|\varphi\|_{s+1,p(\cdot),M})
		\end{align*}
		with a constant $C=C(\mathrm{data},s)$.
		This solution is given by the formula
		\begin{equation}\label{eq:sysDsol}
			\omega = \varphi +\delta G_D[f-d\varphi]+dG_D[v-\delta \varphi].
		\end{equation}
	\end{theorem}
	
	\begin{proof}
		To prove the existence, we introduce $\widetilde \omega = \omega - \varphi$, then split it as $\widetilde \omega = \omega_1+\omega_2$ and look for solutions to the following two problems:
		\begin{gather*}
			d \omega_1 = f-d\varphi, \quad \delta \omega_1 =0, \quad t\omega_1 =0,\\
			d\omega_2 =0, \quad \delta \omega_2 = v-\delta \varphi, \quad t\omega_2=0.
		\end{gather*}
		By Lemma~\ref{L:S1}, 
		\begin{gather*}
			f-d\varphi = d\alpha, \quad \alpha = \delta G_D[f-d\varphi]\in W^{s+1,p(\cdot)}(M,\Lambda) \cap W_T^{1,p(\cdot)}(M,\Lambda), \quad \delta \alpha=0,\\
			\|\alpha\|_{s+1,p(\cdot),M} \leq C (\|f-d\varphi\|_{s,p(\cdot),M}).
		\end{gather*}  
		By Lemma~\ref{L:S4}, 
		\begin{gather*}
			v-\delta \varphi = \delta \beta, \quad \beta = dG_D[v-\delta \varphi]\in  W^{s+1,p(\cdot)}(M,\Lambda) \cap  W_T^{1,p(\cdot)}(M,\Lambda), \quad d \beta=0,\\
			\|\beta\|_{s+1,p(\cdot),M} \leq C (\|v-\delta\varphi\|_{s,p(\cdot),M}).
		\end{gather*}
		It remains to put $\omega_1= \alpha$, $\omega_2=\beta$. Thus, the required solution has the form \eqref{eq:sysDsol}. 
		The uniqueness follows from the fact that a solution of the homogeneous problem is itself in $\mathcal{H}_T(M)$. 
	\end{proof}
	
	Dual form of this statement for the Neumann boundary conditions has the following form.
	\begin{theorem}\label{L:sysN}
		Let 
		\begin{itemize}
			
			\item $f\in W^{s,p(\cdot)}(M,\Lambda)$, $df=0$, $\mathcal{P}_N f=0$; 
			
			\item $v\in W^{s,p(\cdot)}(M,\Lambda)$, $\delta v=0$;
			
			\item $\varphi \in W^{s+1,p(\cdot)}(M,\Lambda)$ satisfy $(v,h_N) =- [h_N,\varphi] $ for all $h_N \in \mathcal{H}_N(M)$;
			
			\item and $n(g-\delta\varphi) = 0$.
			
		\end{itemize}
		
		Then there exists a unique solution $\omega \in W^{s+1,p(\cdot)}(M,\Lambda)$ of the boundary value problem
		$$
		d\omega = f, \quad \delta \omega =v, \quad n\omega = n\varphi
		$$
		such that $\mathcal{P}_N (\omega - \varphi)=0$ and 
		\begin{align*}
			\|\omega\|_{s+1,p(\cdot),M} \leq C ( \|f\|_{s,p(\cdot),M} + \|v\|_{s,p(\cdot),M} + \|\varphi\|_{s+1,p(\cdot),M})
		\end{align*}
		with a constant $C=C(\mathrm{data},s)$. This solution is given by the formula
		\begin{equation}\label{eq:sysNsol}
			\omega = \varphi + \delta G_N[f-d\varphi] + d G_N[v-\delta \varphi]. 
		\end{equation}
	\end{theorem}
	\begin{proof}
		In this case, using Lemmas~\ref{L:S2}, \ref{L:S3} we construct the required solution  in the form \eqref{eq:sysNsol}. 
		The uniqueness follows from the fact that a solution of the homogeneous problem is itself in $\mathcal{H}_N(M)$.
	\end{proof}
	
	\section{Gaffney's inequality}\label{ssec:Gaffney}
	In this Section we extend the classical Gaffney inequality to variable exponents Sobolev spaces. 
	\begin{lemma}\label{L:GaffneyD1}
		Let $\omega\in L^{p_{-}}(M,\Lambda)$ have differential $d\omega = f \in W^{s, p(\cdot)}(M,\Lambda)$, satisfy $t\omega =0$, and have   codifferential $\delta \omega =v\in W^{s, p(\cdot)}(M,\Lambda)$. Then 
		\begin{align*}
			\omega \in W^{s+1, p(\cdot)}(M,\Lambda) \cap W^{1,p(\cdot)}_T(M,\Lambda), \quad \omega = \mathcal{P_T} \omega +  \delta G_D [f] + d G_D[v],\\
			\|\omega - \mathcal{P}_T\omega\|_{s+1,p(\cdot),M} \leq C (\|f\|_{s,p(\cdot),M} + \|v\|_{s,p(\cdot),M})
		\end{align*}
		where the constant $C=C(\mathrm{data},s)$.
	\end{lemma}
	
	\begin{proof}
		The Hodge decomposition of Theorem~\ref{T:HodgeD} together with Corollary~\ref{C:commT} give
		$$
		\omega = \mathcal{P}_T \omega + d \delta G_D[\omega] + \delta d G_D[\omega] = \mathcal{P}_T \omega + d G_D[f] + \delta G_D[v].
		$$ 
		The differentiability properties of $\omega$ now follow from Theorem~\ref{T:D2}.
	\end{proof}
	
	The dual form of this statement is
	
	\begin{lemma}\label{L:GaffneyN1}
		Let $\omega\in L^{p_{-}}(M,\Lambda)$ have differential $d\omega =f \in W^{s, p(\cdot)}(M,\Lambda)$, codifferential $\delta \omega =v\in W^{s, p(\cdot)}(M,\Lambda)$, and satisfy $n\omega =0$. Then 
		\begin{align*} 
			\omega \in W^{s+1, p(\cdot)}(M,\Lambda) \cap W^{1,p(\cdot)}_N(M), \quad \omega = \mathcal{P}_N\omega + \delta G_N[f] + d G_D [v],\\
			\|\omega - \mathcal{P}_N\omega\|_{s+1,p(\cdot),M} \leq C (\|f\|_{s,p(\cdot),M} + \|v\|_{s,p(\cdot),M})
		\end{align*}
		where the constant $C=C(\mathrm{data},s)$.
	\end{lemma}
	
	An alternative proof of the Gaffney inequality which does not require solving auxiliary boundary value problems will be presented in Section~\ref{ssec:SimpleGaffney}. 

\section{Cohomological minimizers in variable exponent spaces}

    As a consequence, we can now prove the following result, which is instrumental for dealing with quasilinear problems in variable exponent spaces. With our Gaffney inequality in Section~\ref{ssec:Gaffney} (or Section~\ref{ssec:SimpleGaffney}), this can now be proved analogously to the proof of Theorem 6.4 in \cite{Sil19} or Theorem 5.1 in \cite{BandDacSil}.


 In the following statement we work in a domain $\Omega$ in the Euclidian space, so we do not make distinction between even and odd forms. Let $W^{1,p(\cdot)}_{\delta,T} = \{u\in W^{1,p(\cdot)}_{T}(\Omega,\Lambda)\,:\, \delta u=0\}$ . The condition $\int\limits_{\Omega} u = \int\limits_{\Omega} u_0$ in the statement below relates only to the top degree  components, since the integrals of other components over $\Omega$ vanish by definition. Note that while the integral functional below does not depend on the top degree component of the form $u$, this component will be exactly the top degree component of the form $u_0$.

    \begin{corollary}\label{nonlinear minimization}
       Let $\Omega \subset \mathbb{R}^{n}$ be a $C^{1,1}$ bounded, contractible domain. Let $a:\Omega \rightarrow [\gamma, L]$ be measurable, where $ 0 < \gamma < L < \infty.$ Let $p(\cdot)$ satisfy \eqref{eq:p1}, \eqref{eq:p2};  $F\in L^{p'(\cdot)}\left( \Omega, \Lambda\right)$, $p'(x)=\frac{p(x)}{p(x)-1}$, $u_{0} \in W^{1,p(\cdot)}\left( \Omega, \Lambda\right)$, and 
       $$
       I[u]:= \int\limits_{\Omega} \left[ \frac{a\left( x\right)}{{p\left( x\right)}}\left\lvert du \right\rvert^{p(x)} - \left\langle F, du \right\rangle\right]\, dx \qquad \text{ for } u \in u_{0} + W^{1,p(\cdot)}_{\delta, T}\left( \Omega, \Lambda\right).
       $$
       Then the following minimization problem 
		\begin{align*}
			\inf \left\lbrace I[u]\,:\, u \in u_{0} + W^{1,p(\cdot)}_{\delta, T}\left( \Omega, \Lambda\right),\ \int\limits_{\Omega}u = \int\limits_\Omega u_0\right\rbrace 
		\end{align*}
		admits a unique minimizer. Moreover, the minimizer $\bar{u} \in W^{1,p(\cdot)}\left( \Omega,\Lambda\right)$ is the unique weak solution to the system 
		\begin{align*}
			\left\lbrace \begin{aligned}
				\delta \left( a\left(x\right)\left\lvert d\bar{u}\right\rvert^{p(x)-2}d\bar{u}\right) &= \delta F &&\text{ in } \Omega, \\
				\delta\bar{u} &= \delta u_{0} &&\text{ in } \Omega, \\
				t\bar{u} &=t u_{0} &&\text{ on } \partial\Omega. 
			\end{aligned}\right. 
		\end{align*}  
    \end{corollary}
    \begin{proof}

        For any $\varepsilon \in (0,1)$, using Young's inequality and integrating, we have 
        \begin{align*}
          \left\lvert \int\limits_{\Omega} \left\langle F, du \right\rangle\, dx \right\rvert &\leq \int\limits_{\Omega} \left\lvert  F \right\rvert \left\lvert  du  \right\rvert \, dx   \\
          &\leq \int\limits_{\Omega} \frac{\varepsilon^{p(x)}}{p(x)}\left\lvert du \right\rvert^{p(x)}\, dx + \int\limits_{\Omega} \frac{1}{\varepsilon^{p'(x)}p'(x)}\left\lvert F \right\rvert^{p'(x)}\, dx \\
          &\leq {\varepsilon^{p_{-}}\int\limits_{\Omega} \frac{1}{p(x)}\left\lvert du \right\rvert^{p(x)}\, dx} + C_{p, \varepsilon}\int\limits_{\Omega}\left\lvert F \right\rvert^{p'(x)}\, dx,
        \end{align*}
        where $C_{p, \varepsilon}>0$ depends only on $\varepsilon, p_{-}$ and $p_{+}.$ Thus, by choosing $\varepsilon \in (0,1)$ sufficiently small, we deduce 
        \begin{align*}
            I[u] &\geq \gamma {\int\limits_{\Omega} \frac{1}{p(x)}\left\lvert du \right\rvert^{p(x)}\, dx} - \left\lvert \int\limits_{\Omega} \left\langle F, du \right\rangle\, dx \right\rvert \\
            &\geq  {\left( \gamma  - \varepsilon^{p_{-}}\right)\int\limits_{\Omega} \frac{1}{p(x)}\left\lvert du \right\rvert^{p(x)}\, dx}  - C_{p, \varepsilon}\int\limits_{\Omega}\left\lvert F \right\rvert^{p'(x)} \leq -C \int\limits_{\Omega}\left\lvert F \right\rvert^{p'(x)} . 
        \end{align*}
        This shows $I[u]$ is bounded below. The same estimate shows that we can choose $\varepsilon>0$ small enough to have the bound 
        \begin{align*}
            I[u] &\geq \frac{\gamma}{2}{\int\limits_{\Omega} \frac{1}{p(x)}\left\lvert du \right\rvert^{p(x)}\, dx} \geq {\frac{\gamma}{2p_{-}}\int\limits_{\Omega} \left\lvert du \right\rvert^{p(x)}\, dx}.  
        \end{align*}
        Thus, if $\left\lbrace u_{s}\right\rbrace_{s \in \mathbb{N}} \subset u_{0} + W^{1,p(\cdot)}_{\delta, T}\left( \Omega, \Lambda\right)$ is a minimizing sequence for $I,$ then $\left\lVert du_{s} \right\rVert_{L^{p(\cdot)}}$ is uniformly bounded and $\delta u_{s} = \delta u_{0}$ for every $s\in \mathbb{N}.$ Since $\Omega$ is contractible, 
        $$
        \mathcal{H}_{T}\left( \Omega, \Lambda^r \right) = \left\lbrace 0 \right\rbrace,\quad r<n,\quad \text{and}\quad \mathcal{H}_{T}\left( \Omega, \Lambda^n \right) = \left\lbrace cdx^1\ldots dx^n,\, c\in \mathbb{R} \right\rbrace.
        $$
        So $\mathcal{P}_T (u_s-u_0)=0$ and applying the variable exponent Gaffney inequality of Lemma~\ref{L:GaffneyD1} to $u_{s} -u_{0,}$ we deduce 
        \begin{align*}
            \left\lVert u_{s}\right\rVert_{W^{1,p(\cdot)}\left( \Omega, \Lambda \right)} &\leq   \left\lVert u_{s} -u_{0}\right\rVert_{W^{1,p(\cdot)}\left( \Omega, \Lambda \right)} +  \left\lVert u_{0}\right\rVert_{W^{1,p(\cdot)}\left( \Omega, \Lambda \right)} \\
            &\leq   \left\lVert du_{s} -du_{0}\right\rVert_{L^{p(\cdot)}\left( \Omega, \Lambda \right)} +  \left\lVert u_{0}\right\rVert_{W^{1,p(\cdot)}\left( \Omega, \Lambda \right)} 
            \\
            &\leq   \left\lVert du_{s} \right\rVert_{L^{p(\cdot)}\left( \Omega, \Lambda \right)} +  c\left\lVert u_{0}\right\rVert_{W^{1,p(\cdot)}\left( \Omega, \Lambda \right)}. 
        \end{align*}
        Thus, $\left\lbrace u_{s}\right\rbrace_{s \in \mathbb{N}} $ is uniformly bounded in $W^{1,p(\cdot)}\left( \Omega, \Lambda\right).$ As $W^{1,p(\cdot)}_{\delta, T}\left( \Omega, \Lambda\right)$ is a linear subspace and consequently a convex subset of $W^{1,p(\cdot)}\left( \Omega, \Lambda\right),$ this implies that up to the extraction of a subsequence that we do not relabel, there exists $u \in u_{0} + W^{1,p(\cdot)}_{\delta, T}\left( \Omega, \Lambda\right)$ such that 
        $$ u_{s} \rightharpoonup u \qquad \text{ weakly in } W^{1,p(\cdot)}\left( \Omega, \Lambda\right).$$ This implies  $$ du_{s} \rightharpoonup du \qquad \text{ weakly in } L^{p(\cdot)}\left( \Omega, \Lambda\right).$$ Now set $\phi\left( x, t \right) = a(x)\left\lvert t \right\rvert^{p(x)}.$ Using Theorem 2.2.8 in \cite{DieHHR11}, we have 
        \begin{align*}
            \int\limits_{\Omega}  a\left( x\right)\left\lvert du \right\rvert^{p(x)}\, dx \leq \liminf\limits_{s \rightarrow \infty}  \int\limits_{\Omega}  a\left( x\right)\left\lvert du_{s} \right\rvert^{p(x)}\, dx. 
        \end{align*} Since the weak convergence in $L^{p(\cdot)}$ also implies $$ \int_{\Omega} \left\langle F, du_{s} \right\rangle\, dx \rightarrow  \int_{\Omega} \left\langle F, du \right\rangle\, dx,$$  standard techniques in the direct methods in the calculus of variations implies $u \in u_{0} + W^{1,p(\cdot)}_{\delta, T}\left( \Omega, \Lambda\right)$ is a minimizer, which also satisfies the weak form of the Euler-Lagrange equations 
        \begin{equation}\label{eq:EulerLagrange}
        \begin{gathered}
            \int\limits_{\Omega} \left\langle a(x)\left\lvert du \right\rvert^{p(x)-2} du - F, d\phi \right\rangle  = 0 \qquad \text{ for all } \phi \in W^{1,p(\cdot)}_{\delta, T}\left( \Omega, \Lambda\right). 
        \end{gathered} 
        \end{equation}
        By the gauge fixing result of Corollary~\ref{C:I0}, the Euler-Lagrange equation \eqref{eq:EulerLagrange} holds also for all $\varphi \in W^{1,p(\cdot)}_T(\Omega,\Lambda)$.         Uniqueness follows from the fact that the function $t \mapsto \phi\left( x, t \right)$ is strictly convex for a.e. $x \in \Omega$.
    \end{proof}

	\section{Natural boundary conditions for the Hodge Laplacian}
	
	In this Section we assume that $M$ is of the class $C^{s+1,1}$, $s\in \mathbb{N}\cup\{0\}$. For Theorem~\ref{T:Natur3} on solvability of the boundary value problem with natural boundary conditions (that is, setting the tangential part of the codifferential and the normal part of the differential) we shall assume $M$ to be $C^{s+2,1}$. The case of $C^{1,1}$ manifolds will be considered separately in Theorem~\ref{T:NaturC11}. As is well known, this problem is non-elliptic, with infinite dimensional (co)kernel so its solution is usually constructed in a roundabout way using already constructed solutions to problems with Dirichlet and Neumann boundary conditions.
	
	For $\omega \in W^{s,p(\cdot)}(M,\Lambda)$ recall the decomposition $\omega = h + d\alpha + \delta \beta$ from Theorem~\ref{T:Hodge} and denote 
	$$
	\mathcal{P} \omega = h = \omega - d\alpha - \delta \beta =  \omega - d \delta G_D[\omega-\mathcal{P}_T\omega] - \delta d G_N[\omega-\mathcal{P}_N\omega].
	$$
	By construction, $\mathcal{P}\omega$ is a bounded operator on $W^{s,p(\cdot)}(M,\Lambda)$:
	$$
	\|\mathcal{P} \omega\|_{s,p(\cdot),M} \leq C(\mathrm{data},s) \|\omega\|_{s,p(\cdot),M},
	$$
	and $d \mathcal{P} \omega =0$, $\delta \mathcal{P}\omega =0$. By the Gaffney inequality, clearly $\mathcal{P}\omega \in W^{s+1,q}_{loc}(M,\Lambda)$ for any $q<\infty$. 
    
    If $\omega \in L^{p(\cdot)}(M,\Lambda)$ and $\mathcal{P}\omega =0$ then $\omega = d\alpha +d \beta$ with $\alpha \in W^{1,p(\cdot)}_T(M,\Lambda)$ and $\beta \in W_N^{1,p(\cdot)}(M,\Lambda)$ and so $(\omega,h)=0$ for any $h\in\mathcal{H}(M,\Lambda) \cap  L^q(M,\Lambda)$ with $q\geq p_{-}'=p_{-}/(p_{-}-1)$. In particular, $\mathcal{P}\omega=0$ implies that $\mathcal{P}_T \omega=0$ and $\mathcal{P}_N \omega=0$, so 
    $$
    \omega = d \alpha + \delta \beta, \quad \alpha =\delta G_D[\omega], \quad \beta = d G_N[\omega].
	$$
	
	From Corollaries~\ref{C:commT} and \ref{C:commN} it follows that for $\omega \in \mathcal{H}(M)\cap L^q(M,\Lambda)$, $q>1$, there holds $\mathcal{P}\omega = \omega$. For $\omega \in L^q(M,\Lambda)$, $q\geq 2$, the projector $\mathcal{P}$ coincides with the usual orthogonal projector on the space $\mathcal{H}(M) \cap L^2(M,\Lambda)$. 
	
	\begin{lemma}
		Let $\omega \in \mathcal{H}(M) \cap L^{p(\cdot)}(M,\Lambda)$ and $(\omega,h)=0$ for all $h \in \mathcal{H}(M) \cap (\cap_{q>1} W^{s,q}(M,\Lambda))$. Then $\omega =0$.
	\end{lemma}
	\begin{proof}
		Let $\eta\in W^{s,q}(M,\Lambda)$. Then in the Hodge decomposition of Theorem~\ref{T:Hodge} we get $\eta =h+d\alpha+\delta \beta$, where $h\in W^{s,q}(\Omega)$, $\alpha,\beta \in W^{s+1,q}(\Omega)$ for any $q<\infty$, $t\alpha=0$ and $n\beta=0$. Thus  $(\omega,d\alpha)=0$, $(\omega,\delta \beta)=0$, and so  
		$$
		(\omega,\eta) = (\omega,h) =0. 
		$$
		Since $\eta$ is arbitrary, $\omega =0$.
	\end{proof}
	
	\begin{corollary}\label{C:harm}
		If $(\omega,h)=0$  for all $h \in \mathcal{H}(M) \cap (\cap_{q>1} W^{s,q}(M,\Lambda))$ then $\mathcal{P}\omega =0$.
	\end{corollary}
	\begin{proof}
		Follows from the fact that 
        $$
        (\mathcal{P}\omega,h) = (\omega -d\alpha-\delta\beta,h)=(\omega,h)=0
        $$ 
        for such $h$. By the previous Lemma, $\mathcal{P}\omega=0$.
	\end{proof}
    \begin{corollary}\label{C:harm1}
    If $\omega = dA+\delta B$, $A,B \in W^{s+1,p(\cdot)}(M,\Lambda)$, $tA =0$, $nB=0$, then $\mathcal{P}\omega =0$. 
    \end{corollary}
    \begin{proof}
    By Corollary~\ref{C:harm} it is sufficient to check that $(\omega,h)=0$ for all $h \in \mathcal{H}(M) \cap (\cap_{q>1} W^{s,q}(M,\Lambda))$. But this follows from the assumptions of the Corollary and approximation results: $(dA+\delta B,h)=0$.
    \end{proof}
    
	If one assumes more regularity on $M$, say $M$ is of the class $C^{k+2,1}$  (or $C^{k+3,\alpha}$), then one can test $\omega$ only against harmonic fields from $\cap_{q>1} W^{k,q}(M,\Lambda)$ (resp. $C^{k,\alpha}(M,\Lambda)$).
	
	Also directly from Theorem~\ref{T:Hodge} and  Corollaries~\ref{C:commT}, \ref{C:commN} we obtain a result similar to Lemmas~\ref{L:GaffneyD1}, \ref{L:GaffneyN1}.
	
	\begin{lemma}\label{L:GaffneyNatur}
		Let $\omega \in L^{p_{-}}(M,\Lambda)$ have differential $d\omega =f\in W^{s,p(\cdot)}(M,\Lambda)$ and codifferential $\delta \omega=v \in W^{s,p(\cdot)}(M,\Lambda)$. Then $\omega - \mathcal{P}\omega \in W^{s+1,p(\cdot)}(M,\Lambda)$,
		$$
		\omega - \mathcal{P}\omega = d G_D [v] + \delta G_N[f],
		$$
		and 
		$$
		\|\omega - \mathcal{P}\omega\|_{s+1,p(\cdot),M} \leq C ( \|f\|_{s,p(\cdot),M} + \|v\|_{s,p(\cdot),M} )
		$$
		where $C=C(\mathrm{data},s)$.
	\end{lemma}
    
	\begin{theorem}\label{T:Natur2}

        (i) Let $M\in C^{1,1}$ and and $\eta \in L^{p(\cdot)}(M,\Lambda)$ with $\mathcal{P}\eta=0$. Then there exists $\omega\in W^{1,p(\cdot)}(M,\Lambda)$ such that $\mathcal{P}\omega=0$ and $d\omega,\delta \omega \in W^{1,p(\cdot)}(M,\Lambda)$  which solves the boundary value problem 
		\begin{equation}\label{eq:natural}
			\triangle \omega = \eta, \quad t\delta \omega =0,\quad nd\omega =0,
		\end{equation}
        and satisfies
        $$
        \|\omega\|_{1,p(\cdot),M} + \|d\omega\|_{1,p(\cdot),M}+ \|\delta \omega\|_{1,p(\cdot),M} \leq C(\mathrm{data}) \|\eta\|_{p(\cdot),M}.
        $$
        (ii) Let $M \in C^{s+2,1}$. Let $\eta \in W^{s,p(\cdot)}(M,\Lambda)$ be such that $\mathcal{P}\eta =0$. Then there exists a solution $\omega \in W^{s+2,p(\cdot)}(M,\Lambda)$ of the boundary value problem \eqref{eq:natural} such that $\mathcal{P}\omega=0$ and
		$$
		\|\omega\|_{s+2,p(\cdot),M} \leq C \|\eta\|_{s,p(\cdot),M}
		$$
		where $C=C(\mathrm{data},s)$. 
	\end{theorem}
	\begin{proof}
		By Theorem~\ref{T:Hodge} and the definition of the projector $\mathcal{P}$ we have $\mathcal{P}_T \eta =0$, $\mathcal{P}_N \eta=0$ and the Hodge decomposition 
		$$
		\eta = d\alpha + \delta \beta, \quad \alpha=\delta G_D[\eta] ,\quad \beta =d G_N[\eta].
		$$     
        We look for a solution to the problem 
		$$
		d\omega =\beta= dG_N[\eta], \quad \delta \omega =\alpha= \delta G_D[\eta]
		$$
        in the form 
        $$
        \omega = G_D[\eta] + G_N[\eta] +\mu_1 + \mu_2
        $$
        where 
        \begin{gather*}
        d\mu_1 =-d G_D[\eta], \quad \delta \mu_1 =0, \quad t\mu_1=0,\\
        d\mu_2 =0, \quad \delta \mu_2 = -\delta G_N[\eta], \quad n\mu_2=0.
        \end{gather*}
        By Corollaries~\ref{C:I0},~\ref{C:I2}, we can construct $\mu_1$ and $\mu_2$ in the form
        $$
        \mu_1 = \delta G_D[-d G_D[\eta]], \quad \mu_2 = d G_N[-\delta G_N[\eta]].
        $$
        By Corollaries~\ref{C:commT}, \ref{C:commN},
        $$
        \mu_1 = - G_D [ \delta d G_D [\eta]], \quad \mu_2 = - G_N [d\delta G_N[\eta]].
        $$
        Thus 
		\begin{gather*}
		\omega = G_D[\eta- \delta d G_D[\eta]] +G_N[\eta - d\delta G_N[\eta]]\\
        =G_D[d\delta G_D[\eta]] + G_N[ \delta d G_N[\eta]]\\
        =d G_D[\delta G_D[\eta]] + \delta G_N[d G_N [\eta]].
		\end{gather*}
		It clearly satisfies \eqref{eq:natural}, $\mathcal{P}\omega=0$, and provides a solution with the required estimates.
	\end{proof}

    Now we shall deal with the non-homogeneous``natural'' boundary value problem for the Hodge Laplacian assuming that $M$ is of the class $C^{s+2,1}$, $s\in \{0\}\cup \mathbb{N}$.
	
	\begin{theorem}\label{T:Natur3}
		Let $M\in C^{s+2,1}$, $s\in \{0\}\cup \mathbb{N}$. Let $\eta \in W^{s,p(\cdot)}(M,\Lambda)$ and $\varphi,\psi \in W^{s+1,p(\cdot)}(M,\Lambda)$ be such that $\mathcal{P}(\eta-d\varphi-\delta\psi)=0$. Then there exists a solution $\omega \in W^{s+2,p(\cdot)}(M,\Lambda)$ of the boundary value problem 
		\begin{equation}\label{eq:natural1}
			\triangle \omega = \eta, \quad t\delta \omega =t\varphi,\quad nd\omega =n\psi,
		\end{equation}
		such that 
		$$
		\|\omega\|_{s+2,p(\cdot),M} \leq C (\|\eta\|_{s,p(\cdot),M} + \|\varphi\|_{s+1,p(\cdot),M} + \|\psi\|_{s+1,p(\cdot),M})
		$$
		where $C=C(\mathrm{data},s)$.
	\end{theorem}
	
	\begin{proof}
		First, by Corollary~\ref{C:ext} we construct $\gamma \in W^{s+2,p(\cdot)}(M,\Lambda)$ such that 
		\begin{align*}
			t\gamma=0,\quad n\gamma=0, \quad t\delta \gamma =t\varphi, \quad n d \gamma = n \psi,\\
			\|\gamma\|_{s+2,p(\cdot),M} \leq C (\|\varphi\|_{s+1,p(\cdot),M} + \|\psi\|_{s+1,p(\cdot),M}).
		\end{align*}
		Consider the boundary value problem 
		$$
		\triangle \widetilde \omega = \widetilde \eta= \eta - d \alpha - \delta \beta, \quad t\delta \widetilde \omega =0, \quad nd\widetilde \omega =0,
		$$
		where 
		$$
		\alpha = \delta \gamma, \quad \beta = d \gamma, \quad t\alpha = t\varphi, \quad n \beta = n \psi.
		$$
        By Corollary~\ref{C:harm1} there holds 
        $$
        \mathcal{P}(\widetilde \eta)=\mathcal{P}(\eta - d\alpha - \delta \beta)=\mathcal{P}(\eta -d \varphi - \delta \psi) + \mathcal{P}(d(\varphi-\alpha)+\delta (\psi-\beta))=0.
        $$
	     It remains to use Theorem~\ref{T:Natur2} and then set $\omega = \widetilde \omega + \gamma - \mathcal{P}[\widetilde \omega + \gamma]$.
	\end{proof}
    
On $C^{1,1}$ manifolds for $\eta \in L^{p(\cdot)}(M,\Lambda)$, $\varphi,\psi \in W^{1,p(\cdot)}(M,\Lambda)$ we can understand a solution of the boundary value problem \eqref{eq:natural1} in the weak sense, namely
\begin{equation}\label{eq:natural1weak}
\mathcal{D}(\omega,\zeta) = (\eta,\zeta) + [\zeta,\psi] - [\varphi,\zeta]
\end{equation}
for all $\zeta \in \mathrm{Lip}\,(M,\Lambda)$. On $C^{2,1}$ manifolds, assuming that $\omega \in W^{2,p_{-}}$, this definition yields \eqref{eq:natural1} upon integration-by-parts. 

Assuming that $\eta\in L^2(M,\Lambda)$, $\varphi,\psi \in W^{1,2}(M,\Lambda)$, and $\mathcal{P}(\eta-d\varphi-\delta \psi)=0$ a solution to \eqref{eq:natural1weak} is easily constructed. Consider first the variational problem
$$
\frac{1}{2}\mathcal{D}(\omega,\omega)- (\eta,\omega) - [\omega,\psi] + [\varphi,\omega]\rightarrow \min
$$
over $\omega \in W^{1,2}(M,\Lambda)$. A unique minimizer with $\mathcal{P}\omega=0$ exists by the direct method of the calculus of variations and the Gaffney inequality of Lemma~\ref{L:GaffneyNatur} (see Section~\ref{ssec:varf}). This minimizer satisfies \eqref{eq:natural1weak} as the Euler-Lagrange equation. 

Now let $M$ be still only $C^{1,1}$, $\eta\in L^{p(\cdot)}(M,\Lambda)$, $\varphi,\psi \in W^{1,p(\cdot)}(M,\Lambda)$. Consider the Lipschitz approximations $\eta_\varepsilon \to \eta$ in $L^{p(\cdot)}(M,\Lambda)$, $\varphi_\varepsilon\to \varphi$, $\psi_\varepsilon \to \psi$ in $W^{1,p(\cdot)}(M,\Lambda)$. Let
$$
p_\varepsilon = \mathcal{P}(\eta_\varepsilon - \delta \psi_\varepsilon - d \varphi_\varepsilon).
$$
The sequence $p_{\varepsilon} \to 0$ in $L^{p(\cdot)}(M,\Lambda)$ as $\varepsilon \to 0$. Let $\omega_\varepsilon$ be the unique minimizer of the corresponding variational problem 
$$
\frac{1}{2}\mathcal{D}(\omega,\omega)- (\eta_\varepsilon- p_\varepsilon,\omega) - [\zeta,\psi_\varepsilon] + [\varphi_\varepsilon,\zeta] \to \min
$$
over $\omega \in W^{1,2}(M,\Lambda)$ with $\mathcal{P}\omega =0$. Then  $\omega_\varepsilon$ satisfies the variational relation
$$
\mathcal{D} (\omega_\varepsilon,\zeta) = (\eta_\varepsilon - \delta \psi_\varepsilon- d \varphi_\varepsilon-p_\varepsilon,\zeta) + (d\zeta,\psi_\varepsilon)+ (\delta\zeta, \varphi_\varepsilon)
$$
for all $\zeta \in \mathrm{Lip}\,(M,\Lambda)$. By Theorem~\ref{L:addregNatur} there holds $\alpha_\varepsilon=\delta \omega_\varepsilon$ and $\beta_\varepsilon =d\omega_\varepsilon$ belong to $W^{1,p(\cdot)}(M,\Lambda)$ and satisfy the variational relations 
\begin{gather*}
\mathcal{D}(\alpha_\varepsilon, \zeta) = (\eta_\varepsilon-\delta \psi_\varepsilon - p_\varepsilon, d\zeta), \\
\mathcal{D}(\beta_\varepsilon, \zeta )= (\eta_\varepsilon - d \varphi_\varepsilon - p_\varepsilon, \delta\zeta)
\end{gather*}
for all $\zeta \in \mathrm{Lip}\,(M,\Lambda)$ together with the boundary conditions $t\alpha_\varepsilon = t \varphi_\varepsilon$, $n\beta_\varepsilon = n \psi_\varepsilon$. By Theorem~\ref{T:TrueGaffney} and Lemmas~\ref{L:quant},~\ref{L:quant1}, there holds
\begin{gather*}
\|\alpha_\varepsilon\|_{1,p(\cdot),M} \leq \|\varphi_\varepsilon\|_{1,p(\cdot),M}+ C(\mathrm{data}) (\|\eta_\varepsilon - \delta \psi_\varepsilon - d\varphi_\varepsilon- p_{\varepsilon}\|_{p(\cdot),M} + \|\delta \varphi_{\varepsilon}\|_{p(\cdot),M}),\\
\|\beta_\varepsilon\|_{1,p(\cdot),M} \leq \|\psi_\varepsilon\|_{1,p(\cdot),M}+  C(\mathrm{data}) (\|\eta_\varepsilon - \delta \psi_\varepsilon - d\varphi_\varepsilon- p_{\varepsilon}\|_{p(\cdot),M} + \|d\psi_{\varepsilon}\|_{p(\cdot),M}).
\end{gather*}
By Lemma~\ref{L:GaffneyNatur} we obtain the estimate 
$$
\|\omega_{\varepsilon}\|_{1,p(\cdot),M} \leq C (\mathrm{data}) (\|\eta_\varepsilon - \delta \psi_\varepsilon - d\varphi_\varepsilon- p_{\varepsilon}\|_{p(\cdot),M} + \|\varphi_\varepsilon\|_{1,p(\cdot),M} +\|\psi_\varepsilon\|_{1,p(\cdot),M})
$$
which provides the uniform bound on $\omega_\varepsilon$ in $W^{1,p(\cdot)}(M,\Lambda)$. By the same estimates applied to the difference of $\omega_{\varepsilon_1} - \omega_{\varepsilon_2}$ we conclude that the sequences $\alpha_\varepsilon, \beta_\varepsilon,\omega_\varepsilon$ are fundamental and therefore converge in $W^{1,p(\cdot)}(M,\Lambda)$. The limit form $\omega = \lim_{\varepsilon \to 0} \omega_\varepsilon$ possesses the differential $d\omega = \lim_{\varepsilon \to 0} \beta_\varepsilon$ and codifferential $\delta \omega_\varepsilon = \lim_{\varepsilon \to 0} \alpha_\varepsilon$, which belong to $W^{1,p(\cdot)}(M,\Lambda)$. We formulate this result as
\begin{theorem}\label{T:NaturC11}
Let $M\in C^{1,1}$. Let $\eta \in L^{p(\cdot)}(M,\Lambda)$ and $\varphi,\psi \in W^{1,p(\cdot)}(M,\Lambda)$ be such that $\mathcal{P}(\eta-d\varphi-\delta\psi)=0$. Then there exists a (weak) solution $\omega \in W^{1,p(\cdot)}(M,\Lambda)$ of the boundary value problem \eqref{eq:natural1} such that $d\omega,\delta\omega \in W^{1,p(\cdot)}(M,\Lambda)$, $\mathcal{P}\omega=0$, and
		\begin{gather*}
		\|\omega\|_{1,p(\cdot),M} + \|d\omega\|_{1,p(\cdot),M} + \|\delta \omega\|_{1,p(\cdot),M}\\
        \leq C(\mathrm{data}) (\|\eta\|_{p(\cdot),M} + \|\varphi\|_{1,p(\cdot),M} + \|\psi\|_{1,p(\cdot),M}).
		\end{gather*}
\end{theorem}

We can also construct a solution to the boundary value problem \eqref{eq:natural1} in an explicit form using solutions of auxiliary Dirichlet and Neumann problems, similar to the construction used in the proof of Theorem~\ref{T:Natur2}. Consider the boundary value problems 
\begin{gather*}
\triangle A = \eta - \delta \psi \quad \text{in}\quad M, \quad t A =0\quad \text{and}\quad t\delta A = t\varphi \quad \text{on}\quad bM,\\
\triangle B = \eta - d \varphi \quad \text{in}\quad M, \quad n B =0\quad \text{and}\quad nd B = n\psi \quad \text{on}\quad bM.
\end{gather*}
Let 
$$
\omega = d G_D[\delta A] + \delta G_N[dB].
$$
Then using the commutation relations of Corollaries~\ref{C:commT}, \ref{C:commN} we get 
\begin{gather*}
d\omega = d\delta G_N[dB ] = dB - \delta d G_N[dB] = dB, \\
\delta \omega = \delta d G_D [\delta A] = \delta A - d \delta G_D[\delta A] = \delta A.
\end{gather*}
Thus $\omega$ satisfies the required boundary conditions $t\delta \omega = t\varphi$ and $nd\omega = n\psi$, and $\mathcal{P}\omega =0$. It remains to check that $\triangle \omega = \eta$. The potentials $dA$, $\delta A$, $dB$, $\delta B$ satisfy the variational relations
\begin{gather*}
\mathcal{D}(\delta A,\zeta )= ( \eta - \delta \psi, d \zeta ), \quad \mathcal{D}(dA,\zeta) = (\eta - \delta \psi - d \varphi, \delta \zeta), \quad \zeta \in \mathrm{Lip}_T(M,\Lambda),\\
\mathcal{D}(\delta B,\zeta )= ( \eta - \delta \psi-d\varphi, d \zeta ), \quad \mathcal{D}(d B,\zeta) = (\eta - d \varphi, \delta \zeta), \quad \zeta \in \mathrm{Lip}_N(M,\Lambda).
\end{gather*}
In particular, $dA = d G_D[\eta - d\varphi - \delta \psi]$, $\delta B = \delta G_N [\eta - d\varphi - \delta \psi]$. 

The condition $\mathcal{P}(\eta -d\varphi - \delta \psi)=0$ is equivalent to 
\begin{gather*}
\eta -d\varphi - \delta \psi = du + \delta v, \quad u \in W_T^{1,p(\cdot)}(M,\Lambda), \quad v \in W_N^{1,p(\cdot)}(M,\Lambda),\\
u = \delta G_D[\eta -d\varphi - \delta \psi], \quad v = d G_N[\eta -d\varphi - \delta \psi].
\end{gather*}
Finally we calculate 
$$
\triangle \omega = d\delta A + \delta d B =\eta + (\eta-d\varphi - \delta \psi)- \delta dA - d \delta B = \eta.
$$
This argument provides an alternative proof of Theorems~\ref{T:Natur3}, \ref{T:NaturC11}. The proof of Theorem~\ref{T:NaturC11} via the variational argument is given to stress the fact that the problem, while being non-elliptic, can still be understood in the variational setting.

    
	\section{Hodge-Dirac system}
	
	In a standard way, solvability of the boundary value problems for the Hodge Laplacian gives solvability of boundary value problems for the Hodge-Dirac operator. 

	Let $D = d+\alpha \delta$, $\alpha \in \mathbb{R}\setminus \{0\}$. The important cases are (symmetric) $d+\delta$ and (skew-symmetric) $d-\delta$ (on forms with vanishing tangential/normal part on $bM$). Let $s\in \{0\}\cup \mathbb{N}$.
	
	\begin{theorem}\label{Theorem on Hodge-Dirac tangential}
		Let $f\in W^{s,p(\cdot)}(M,\Lambda)$ be such that $\mathcal{P}_T f =0$. Then there exists $\omega \in  W^{s+1,p(\cdot)}(M,\Lambda)$  satisfying $D \omega = F$ a.e. in $M$ and $t\omega =0$, such that $(\omega, \mathcal{H}_T(M)) =0$ and 
		\begin{equation}\label{eq:DE}
			\|\omega\|_{s+1,p(\cdot),M} \leq C(\mathrm{data},s,\alpha) \|f\|_{s,p(\cdot),M}.
		\end{equation}
	\end{theorem}
	\begin{proof}
		Take $\omega = D G_D[\alpha^{-1}f]$, that is $f$ defined by \eqref{eq:forms} is mapped to $$\omega= \sum_{r=0}^n \omega_e^r +  \sum_{r=0}^n \omega_o^r, $$  where $\omega_*^{(0)} = \delta G_D[f_*^{1}]$, $\omega_*^{j}= \alpha^{-1} d G_D[f_*^{j-1}] + \delta G_D[f_*^{j+1}]$, $j=1,\ldots, n-1$ and $\omega_*^{n}= \alpha^{-1} d G_D[f_*^{n-1}]$, $*\in\{e,o\}$.
	\end{proof}
	\begin{theorem}
		Let $f\in W^{s,p(\cdot)}(M,\Lambda)$ be such that $\mathcal{P}_N f =0$. Then there exists $\omega \in  W^{s+1,p(\cdot)}(M,\Lambda)$  satisfying $D \omega = F$ a.e. in $M$ and $n\omega =0$, such that $(\omega, \mathcal{H}_N(M)) =0$ and \eqref{eq:DE} holds. 
	\end{theorem}
	\begin{proof}
		Take $\omega = D G_N[\alpha^{-1}f]$. 
	\end{proof}
	
The following a priori estimate is valid.
\begin{corollary}
Let $f\in W^{s,p(\cdot)}(M,\Lambda)$ and let $\omega \in W_T^{1,p_{-}}(M,\Lambda)$ satisfy $D\omega =f$, a.e. in $M$. Then $\omega \in W^{s+1,p(\cdot)}(M,\Lambda)$ and 
$$
\|\omega - \mathcal{P}_T \omega\|_{s+1,p(\cdot),M} \leq C(\mathrm{data},s,\alpha) \|f\|_{s,p(\cdot),M}. 
$$
\end{corollary}
\begin{proof}
First note that $D\omega =f$ and $t\omega=0$ imply $(f,h_T)=0$ for all $h_T \in \mathcal{H}_T(M)$. Let $\eta \in W^{s+1,p(\cdot)}(M)$, $t\omega=0$, be a solution to $D \omega = \eta$ guaranteed by Theorem~\ref{Theorem on Hodge-Dirac tangential}, then
$$
\|\eta\|_{s+1,p(\cdot),M}\leq C(\mathrm{data},s,\alpha) \|f\|_{s,p(\cdot),M}
$$
Then $D(\omega-\eta)=0$. Let $\zeta\in \mathrm{Lip}_T(M,\Lambda)$.  Then from the integration-by-parts formula (for $C^{1,1}$ manifolds see Lemma~\ref{L:ort}) it follows that 
$$
0=(D(\omega-\eta), d\zeta + \alpha \delta^{-1}\zeta)= (d(\omega-\zeta) + \alpha \delta(\omega-\zeta),d\zeta + \alpha \delta^{-1}\zeta ) =\mathcal{D}(\omega-\eta,\zeta).
$$
From Theorems~\ref{T:p2}, \ref{T:TrueGaffney} it follows that $\omega-\eta \in W^{s+1,q}(M,\Lambda)$ for any $q< \infty$, and also $\omega-\eta \in W^{s,p(\cdot)}(M,\Lambda)$. By approximation one can take the test form $\zeta = \omega-\eta$, so $\mathcal{D}(\omega-\eta,\omega-\eta)=0$, which implies $\omega-\eta \in \mathcal{H}_T(M)$. This proves the claim.
\end{proof}

\begin{corollary}
Let $f\in W^{s,p(\cdot)}(M,\Lambda)$ and let $\omega \in W_N^{1,p_{-}}(M,\Lambda)$ satisfy $D\omega =f$,  a.e. in $M$. Then $\omega \in W^{s+1,p(\cdot)}(M,\Lambda)$ and 
$$
\|\omega - \mathcal{P}_N \omega\|_{s+1,p(\cdot),M} \leq C(\mathrm{data},s,\alpha) \|f\|_{s,p(\cdot),M}. 
$$
\end{corollary}

\begin{corollary}
Let $f\in W^{s,p(\cdot)}(M,\Lambda)$ and let $\omega \in W_0^{1,p_{-}}(M,\Lambda)$ satisfy $D\omega =f$,  a.e. in $M$. Then $\omega \in W^{s+1,p(\cdot)}(M,\Lambda)$ and 
$$
\|\omega\|_{s+1,p(\cdot),M} \leq C(\mathrm{data},s,\alpha) \|f\|_{s,p(\cdot),M}. 
$$
\end{corollary}

We also state a variant of Gaffney's inequality. 
\begin{theorem}\label{T:HodgeDiracGaffney}
Let $M$ be $C^{1,1}$. Let $\omega\in L^{p_{-}}(M,\Lambda)$ and $f\in L^{p(\cdot)}(M,\Lambda)$ satisfy 
\begin{equation}\label{eq:HDG}
(\omega, \alpha d \zeta + \delta \zeta) = (f,\zeta)
\end{equation}
for all $\zeta \in \mathrm{Lip}_T(M,\Lambda)$. Then $\omega \in W^{1,p_{-}}_T(M,\Lambda)$ and $D\omega = f$. 
\end{theorem}
\begin{proof}
Let $\widetilde \omega \in W_T^{1,p(\cdot)}(M,\Lambda)$ be a solution to $D\widetilde \omega =f$ guaranteed by Theorem~\ref{Theorem on Hodge-Dirac tangential}. Then $\hat \omega = \omega - \widetilde \omega$ satisfies
$$
(\hat \omega, \alpha d \zeta +\delta \zeta )=0 \quad \text{for all} \quad \zeta \in \mathrm{Lip}_T(M,\Lambda).
$$
Let $g \in L^{p_{-}'}(M,\Lambda)$ and $\mathcal{P}_T g=0$. Then by Theorem~\ref{Theorem on Hodge-Dirac tangential} (with $\alpha$ replaced by its reciprocal) there exists $\zeta \in W_T^{1,p_{-}'}(M,\Lambda)$ such that $\alpha d \zeta_g + \delta \zeta_g =g$. By approximation, we have
$$
(\hat \omega, g)= (\hat \omega,\alpha d \zeta_g + \delta \zeta_g) =0.
$$
Then 
$$
(\hat \omega -\mathcal{P}_T\hat\omega, g)=0 \quad \text{for all}\quad g\in L^{p_{-}'}(M,\Lambda).
$$
Thus $\hat \omega -\mathcal{P}_T \hat \omega =0$ and so $\omega = \widetilde \omega + h_T$, $h_T \in \mathcal{H}_T(M)$. Once we know that $\omega \in W^{1,p(\cdot)}(M,\Lambda)$ we can use integration-by-parts in \eqref{eq:HDG} to get
$$
(d\omega + \alpha \delta \omega,\zeta) =(f,\zeta) +[\omega,\zeta]\quad \text{for all} \quad \zeta \in \mathrm{Lip}_T(M,\Lambda).
$$
This immediately implies $D\omega =f$ and $t\omega =0$ and completes the proof of Theorem~\ref{T:HodgeDiracGaffney}.
\end{proof}

The proof of the following theorem is similar or may be obtained by the Hodge duality.

\begin{theorem}\label{T:HodgeDiracGaffneyN}
Let $M$ be $C^{1,1}$. Let $\omega\in L^{p_{-}}(M,\Lambda)$ and $f\in L^{p(\cdot)}(M,\Lambda)$ satisfy \eqref{eq:HDG} for all $\zeta \in  \mathrm{Lip}_N(M,\Lambda)$. Then $\omega \in W^{1,p_{\cdot}}_N(M,\Lambda)$ and $D\omega = f$. 
\end{theorem}

\begin{remark}
It is also possible to prove Theorems~\ref{T:HodgeDiracGaffney}, \ref{T:HodgeDiracGaffneyN} in vein of local a priori estimates similar to Lemma~\ref{existence of weak derivative tangential}. See the end of Section~\ref{ssec:SimpleGaffney}.
\end{remark}

	\section{Non-elliptic first-order systems}
	
	In this Section we extend to the variable exponent setting some well-known results on the boundary value problems for the operators $d$ and $\delta$. 
	
	\begin{theorem}\label{Poincare lemma with dirichlet data tangential}
		Let $f\in W^{s,p(\cdot)}(M,\Lambda)$, $\varphi \in W^{s+1,p(\cdot)}(M,\Lambda)$, be such that
		$$
		df =0, \quad t(f-d\varphi)=0,\quad  (f,h_T) =[\varphi,h_T]
		$$
		for all $h_T \in \mathcal{H}_T(M)$. Then the problem 
		\begin{equation}\label{eq:1probl}
			d\omega = f, \quad \omega=\varphi \quad \text{on}\quad bM.
		\end{equation}
		has a solution in $W^{s+1,p(\cdot)}(M,\Lambda)$ such that 
		\begin{equation}\label{eq:1pest}
			\|\omega\|_{s+1,p(\cdot),M} \leq C(\mathrm{data},s) ( \|f\|_{s,p(\cdot),M} + \|\varphi\|_{s+1,p(\cdot),M} ).
		\end{equation}
	\end{theorem}
	
	\begin{proof}
		By Lemma~\ref{L:S1} we get
		\begin{align*}
			f- d\varphi = d \delta G_D[f-d\varphi],\\ 
			\|\delta G_D[f-d\varphi]\|_{s+1,p(\cdot),M} \leq C ( \|f\|_{s,p(\cdot),M} + \|\varphi\|_{s+1,p(\cdot),M} ).
		\end{align*}
		By Corollary~\ref{C:ext}, there exists $\gamma \in W^{s+1,p(\cdot)}(M,\Lambda) $ such that $d\gamma \in W^{s+1,p(\cdot)}(M,\Lambda)$ and
		$$
		t\gamma=0,\quad n\gamma=0,\quad nd\gamma = n \delta G_D[f-d\varphi],
		$$ 
		and 
		$$
		\|d\gamma\|_{s+1,p(\cdot),M} \leq C \|\delta G_D[f-d\varphi]\|_{s+1,p(\cdot),M}.
		$$
		Then  
		$$
		\omega =  \delta G_D[f-d\varphi] - d\gamma + \varphi  
		$$
		solves \eqref{eq:1probl} and satisfies \eqref{eq:1pest}.
	\end{proof}

	The dual form of this statement is 
	
	\begin{theorem}
		Let $v\in W^{s,p(\cdot)}(M,\Lambda)$, $\psi \in W^{s+1,p(\cdot)}(M,\Lambda)$, be such that
		$$
		\delta v =0, \quad n(g-\delta\psi)=0,\quad  (g,h_N) = -[h_N,\psi]
		$$
		for all $h_N \in \mathcal{H}_N(M)$. Then the problem 
		\begin{equation}\label{eq:2probl}
			\delta\omega = v, \quad \omega=\psi \quad \text{on}\quad bM.
		\end{equation}
		has a solution in $W^{s+1,p(\cdot)}(M,\Lambda)$ such that 
		\begin{equation}\label{eq:2pest}
			\|\omega\|_{s+1,p(\cdot),M} \leq C(\mathrm{data},s) ( \|v\|_{s,p(\cdot),M} + \|\psi\|_{s+1,p(\cdot),M} ).
		\end{equation}
	\end{theorem}

	\chapter{A Priori Estimates in Variable Exponent Spaces}\label{Sec:parametrix}

		This Chapter contains a general regularity result, Lemmas~\ref{T:D1},~\ref{T:N1},~\ref{T:Z1} above are its special cases.  First, in Section~\ref{sec:coord} we write variational relation~\eqref{eq1} in coordinates. In Section~\ref{sec:potentialHodge} we present construction of potentials corresponding to the boundary value problems for the Hodge Laplacian. In Section~\ref{ssec:local} we pass to the localized version of \eqref{eq1} and represent its solution via potentials from Section~\ref{sec:potentialHodge}. In Sections~\ref{ssec:W1p}, \ref{ssec:W2p} this representation is used to derive elliptic estimates in variable exponent spaces. Section~\ref{ssec:SimpleGaffney} we obtain Gaffney's inequality under natural regularity assumptions on the manifold. In Section~\ref{sec:refined} we complete the proof of $W^{1,p(\cdot)}$ estimates on $C^{1,1}$ manifolds. In the last Section~\ref{sec:additional} we prove additional regularity for potentials on $C^{1,1}$ manifolds following the line of Section~\ref{ssec:SimpleGaffney}.
	
	By $B_r$ we denote an open ball of radius $r$ in $\mathbb{R}^n$, $B_r(x)$ is the open ball of radius $r$ centered at $x$, $B_r^+(x) = B_r(x)\cap \{x^n>0\}$, $\overline{B}_r^+(x) = B_r(x) \cap \{x^n\geq 0\}$. Without loss it is sufficient to consider only the case when $\omega$ is a homogeneous form (the only exception is the Hodge-Dirac system in Remark~\ref{remark:HD}).
	
	\section{Variational formulation in coordinates}\label{sec:coord}
	Below in boundary coordinate patches we shall use the admissible coordinate systems \cite{Morrey1966}[Lemma 7.5.1]. Let $(U_\alpha,h_\alpha)$ be a finite coordinate atlas such that if $U_\alpha$ intersects the boundary, the coordinate system $h_\alpha$ is admissible. Let $\varepsilon>0$. We shall use a finite refinement $(V_\beta,\widetilde h_\beta)$ of the coordinate atlas $(U_\alpha,h_\alpha)$ such that
	\begin{itemize}
		
		\item If $V_\beta$ intersects the boundary the coordinate system $\widetilde h_\beta$ is admissible,
		
		\item for all $\beta$ there exists $\alpha$ and an affine transformation $T_{\alpha\beta}$ satisfying $V_\beta \subset U_\alpha$, $T_{\alpha\beta}\widetilde h_{\beta} = h_\alpha$ on $V_\beta$, and the norms of differentials of such transformations and their inverses are uniformly bounded ($\|DT_{\alpha\beta}\|^2 \leq\lambda_{\mathrm{min}}^{-1}$, $\|DT_{\alpha\beta}^{-1}\|^2\leq \lambda_{\mathrm{max}}$, where $\lambda_{\mathrm{max}}$, $\lambda_{\mathrm{min}}^{-1}$ are the maximal and minimal eigenvalue of the matrix $g_{ij}$ in coordinate system $h_{\alpha}$).
		
		
		\item $\widetilde h_\beta(V_\beta) =B_\varepsilon(0)$ for  interior patches or half-ball $B_\varepsilon^+(0)$ for boundary patches and $T_{\alpha\beta} B_{2\varepsilon}(0) \subset h_\alpha(U_\alpha)$ (resp. $T_{\alpha\beta} \overline{B}^{+}_{2\varepsilon}(0) \subset h_\alpha(U_\alpha)$).   
		
		\item and in any coordinate system $\widetilde h_\beta$ we have $g_{ij}(0) = \delta_{ij}$.
		
	\end{itemize} 
	Such atlas exists for all $\varepsilon >0$ and will be further referred to as $\varepsilon$-atlas. In particular, this guarantees that in all the coordinate charts of this atlas the coefficients (and their derivatives) will have the uniform bounds independent on $\varepsilon$.
	
	By the Lebesgue lemma, for a sufficiently small $r_*>0$ each open ball of radius $r_*$ in $M$ belongs to one of $U_\alpha$. Thus for a given $c\geq 1$ there exists $r_0>0$ such that for any $r\in (0,r_0)$ there exists a finite cover of manifold by balls $\mathcal{B}_r(P_j)$, $j=1,\ldots, N$, of radius $r$ such that each concentric ball $\mathcal{B}_{cr}(P_j)$ of radius $cr$ belongs to one of $U_\alpha$ and either has an empty intersection with $bM$ or $P_j \in bM$. It immediately follows that for each $c\geq 1$ there exists $r_1>0$ such that for each $r\in (0,r_1)$ there exists a finite cover of $M$ by the open sets of the form $\varphi_\alpha^{-1} (B_r(P_{j,\alpha}))$ where $B_{cr}(P_{j,\alpha}) \subset \varphi_\alpha(U_\alpha)$ and in boundary patches either $B_{4r}(P_{j,\alpha})\subset \{x^n>0\}$ or $P_{j,\alpha}\in \{x^n=0\}$. Let $\widetilde B_r (P_{j,\alpha})$ be a ball in $\varphi_\alpha(U_\alpha)$ centered at $P_{j,\alpha}$ with respect to the metric defined by the metric tensor in its center. There exists $r_2>0$ such that for each $r\in (0,r_2)$ there exists a finite open cover of $M$ by the sets of the form $\varphi_\alpha^{-1} (\widetilde B_r(P_{j,\alpha}))$ where $\widetilde B_{cr}(P_{j,\alpha}) \subset \varphi_\alpha(U_\alpha)$ and in boundary patches either $\widetilde B_{4r}(P_{j,\alpha})\subset \{x^n>0\}$ or $P_{j,\alpha}\in \{x^n=0\}$.

	We shall work with forms satisfying the variational equation \eqref{eq1}, that is 
	$$
	(d\omega - \varphi, d\zeta)+(\delta \omega - \psi,\delta \zeta) - (\eta,\zeta)=0
	$$
	for a suitable class of test forms $\zeta$.
	
	Let $\omega$ be such that one of the following holds
	\begin{enumerate}
		\item (\textit{problem with full Dirichlet data}) $\omega \in W_0^{1,p_{-}}(M,\Lambda)$ and the variational equation \eqref{eq1} holds for  all $\zeta \in \mathrm{Lip}_0(M,\Lambda)$;
		
		\item (\textit{Dirichlet problem}) $\omega \in W_T^{1,p_{-}}(M,\Lambda)$ and the variational equation \eqref{eq1} holds for  all $\zeta \in \mathrm{Lip}_T(M,\Lambda)$;
		
		\item (\textit{Neumann problem})  $\omega \in W_N^{1,p_{-}}(M,\Lambda)$ and the variational equation \eqref{eq1} holds for all $\zeta \in \mathrm{Lip}_N(M,\Lambda)$.
	\end{enumerate}
	
	Further we can assume that $\omega$ is a homogeneous form of degree $r$.
	
	Let $\mathcal{I}(r)$ be the set of ordered multi-indices $I = 1\leq i_1<i_2\ldots<i_r\leq n$ and let $\mathcal{I}_0$ denote a subset of $\mathcal{I}(r)$ such that in boundary coordinate patches for the Dirichlet problem $\mathcal{I}_0$ consists of $I\in \mathcal{I}(r)$ with $i_r<n$, and for the Neumann problem $\mathcal{I}_0 \in \mathcal{I}(r)$ consists of $I$ with $i_r =n$. The case when $\mathcal{I}_0 = \mathcal{I}(r)$ corresponds to the boundary conditions $t\omega=0$, $n\omega =0$. For interior coordinate patches we set $\mathcal{I}_0 = \emptyset$.

	Let $G_{2R} = \widetilde {G}_{2R}= B_{2R}$ for an interior coordinate patch, or $G_{2R} = B_{2R}^+=B_{2R} \cap\{x_n > 0\}$ and $\widetilde G_{2R} = \overline{B}_{2R}^+$  for a boundary coordinate patch. Recall that by $\overline{B}_{2R}^+$ we denote $B_{2R} \cap \{x^n \geq 0\}$. 
	
	By $C_0^\infty(\mathcal{I}_0, G_{2R})$ we shall denote the set of functions $\{\zeta_I\}$, $I\in \mathcal{I}(r)$, from $C_0^\infty(\overline{B}_{2R}^+)$ such that $\zeta_I =0$ near $B_{2R} \cap \{x_n=0\}$ for all $I \in \mathcal{I}_0$. For a (variable) exponent $q(\cdot)$ by $W^{1,q(\cdot)}(\mathcal{I}_0,G_{2R})$, we denote the set of functions $\{\omega_I\}$, $I\in \mathcal{I}(r)$, such that $\omega_I =0$  on $B_{2R} \cap \{x_n=0\}$ for all $I \in \mathcal{I}_0$  (in the sense of trace). For $s\in \mathbb{N}$ by $W^{s,p(\cdot)}(\mathcal{I}_0, G_{2R})$ we mean $W^{1,p(\cdot)}(\mathcal{I}_0,G_{2R}) \cap W^{s,p(\cdot)}(G_{2R})$.

	Then $\omega \in W^{1,p_{-}}(\mathcal{I}_0,G_{2R})$ and equation \eqref{eq1}, that is 
	$$
	(d\omega - \varphi,d\zeta)+(\delta \omega-\psi,\delta \zeta)-(\eta,\zeta)=0, 
	$$
	takes the form:
	\begin{equation}\label{iid1}
		\begin{aligned}
			\int\limits_{G_{2R}} \zeta_{I,\alpha} &[a^{IJ\alpha \beta} \omega_{J ,\beta}+b^{IJ \alpha} \omega_J + e^{I\alpha}] \, dx\\
			&+ \int\limits_{G_{2R}} \zeta_I [b^{*IJ\beta} \omega_{J,\beta}+ c^{IJ}\omega_J + f^I] \, dx = 0 
		\end{aligned}
	\end{equation}
	for all  $\zeta \in C_0^\infty(\mathcal{I}_0, G_{2R})$.

	The coefficients
	\begin{itemize}
		
		\item  $a^{IJ\alpha \beta} = \widetilde a^{IJ\alpha \beta} \sqrt{g}$, where $\widetilde a^{IJ\alpha \beta}$ is a tensor field on $M$ involving only the metric: for co-vector fields $z$, $\widetilde z$ there holds 
		\begin{gather*}
			a^{IJ\alpha \beta} \omega_J \xi_I z_\alpha \widetilde z_\beta = \langle \widetilde z\wedge \omega, z\wedge \zeta\rangle + \langle\widetilde  z\lrcorner \omega, z\lrcorner \zeta\rangle,\\
			a^{IJ\alpha \beta} = a^{JI\beta\alpha},\quad  a^{IJ\alpha \beta} + a^{IJ\beta \alpha} = 2G^{IJ}g^{\alpha \beta} 
		\end{gather*} 
		
		
		\item  $b^{IJ\alpha}$, $b^{*IJ\alpha}$, $c^{IJ}$ involve the Christoffel symbols and the metric, $b^{*IJ\alpha} = b^{JI\alpha}$ and $c^{IJ} = c^{JI}$;
		
		\item  $e^{I\alpha}$ are linear combinations of the components of $\varphi$ and $\psi$ with coefficients depending only on the metric.
		
		\item  $f^I$ are linear combinations of the components of $\psi$ and $\eta$; 
		
		\item of $\eta$ in $f$ involve only the metric coefficients;
		
		\item  of $\psi$ in $f$ involve also the derivatives of metric via Christoffel symbols.
		
	\end{itemize}
	
	Thus, for a manifold of the class $C^{s+2,1}$, $s\geq -1$, we have
	\begin{itemize}
		\item $a^{IJ\alpha\beta} \in C^{s+1,1}(G_{2R})$;
		
		\item $b^{IJ\alpha}, b^{*IJ\alpha}, c^{IJ} \in C^{s,1}(G_{2R})$ if $s\geq 0$ and $b^{IJ\alpha}, b^{*IJ\alpha}, c^{IJ} \in L^\infty(G_{2R})$ if $s=-1$;
		
		\item $f^I = C_1^{IJ} \eta_J  +  C_2^{IJ} \psi_J: = C_1 \eta + C_2 \psi$, where $C_1 \in C^{s+1,1}(G_{2R})$, $C_2 \in C^{s,1}(G_{2R})$ if $s\geq 0$, $C_2 \in L^\infty(G_{2R})$ for $s=-1$;
		
		\item $e^{I\alpha}  =  C_3^{IJ\alpha}\varphi_J +  C_4^{IJ\alpha} \psi_J:= C_3\varphi + C_4 \psi$, where $C_3, C_4 \in C^{s+1,1}(G_{2R})$;
		
	\end{itemize}

	In particular, for all $s\geq -1$ the coefficients $a^{IJ\alpha\beta}$, the coefficients of $\varphi$, $\psi$ in $e$, the coefficients of $\eta$ in $f$ are Lipschitz,
	\begin{equation}\label{eq:basic}
		\begin{gathered}
			|a^{IJ\alpha \beta}(x) - a^{IJ\alpha \beta}(0)| \leq C R \quad \text{for} \quad x\in G_{2R},\\
			|f| \leq C(|\eta| + |\psi|), \quad |e| \leq C(|\varphi| + |\psi|)
		\end{gathered}
	\end{equation}
	with some universal constants which will be the same for any $\varepsilon$-atlas from our collection. For $s\geq 0$ there holds 
	\begin{equation}\label{eq:fereg}
		\begin{aligned}
			\|f\|_{s,p(\cdot),G_{2R}} &\leq C (\|\eta\|_{s,p(\cdot),M} + \|\psi\|_{s,p(\cdot),M}),\\
			\|e\|_{s+1,p(\cdot),G_{2R}} &\leq C (\|\varphi\|_{s+1,p(\cdot),M} + \|\psi\|_{s+1,p(\cdot),M}).
		\end{aligned}
	\end{equation}

	

	In the Euclidean metric for a function $\zeta$ with compact support in $G_{2R}$ (and vanishing tangential/ normal part on $\{x_n=0\}$ in the boundary coordinate patch the Gaffney formula holds 
	\begin{equation}\label{Gaffney}
		\begin{aligned}
			\int\limits_{G_{2R}} a^{IJ\alpha \beta}(0) \zeta_{I,\alpha} \omega_{J,\beta}\, dx &=  (d\omega,d\zeta)_0 + (\delta_0 \omega,\delta_0\zeta)_0 \\
			&= (\nabla \omega,\nabla \zeta)_{0;G_{2R}} := \int\limits_{G_{2R}} \delta^{IJ} \delta^{\alpha \beta} \omega_{J,\beta} \zeta_{I,\alpha}\, dx.
		\end{aligned}
	\end{equation}
	Here $\delta_0$ and $(\cdot,\cdot)_0$ stand for the codifferential and the scalar product in the standard Euclidean metric, and vanishing of the boundary integral here is equivalent to the relation $a^{IJ\alpha n}(0) \omega_J \zeta_I z_\alpha - a^{IJn\beta}(0) \omega_J \zeta_I z_\beta =0$ on $x^n=0$ if $t\omega=t\zeta=0$ or $n\omega=n\zeta=0$. 
	
	
	Further we shall use the following notation. Let 
	\begin{equation}\label{eq:nota0}
		\begin{gathered}
			(\nabla \zeta)_{I\alpha} = \zeta_{I,\alpha}, \quad (a\nabla \omega)^{I\alpha} = a^{IJ\alpha \beta} \omega_{J,\beta},\quad \nabla \omega : \nabla \zeta = \delta^{IJ}\delta^{\alpha \beta}\omega_{J,\beta}\zeta_{I,\alpha},\\
			(b\omega)^{I\alpha} =  b^{IJ \alpha} \omega_J,\quad (b^* \nabla \omega)^{I} = b^{*IJ\beta}  \omega_{J,\beta}, \quad (c\omega)^I = c^{IJ}\omega_J,
		\end{gathered}
	\end{equation}
	Products of such expressions are understood in the usual way as contractions over repeated lower and upper indices:
	$$
	a\nabla \omega \nabla \zeta = a^{IJ\alpha \beta} \omega_{J,\beta} \zeta_{I,\alpha}, \quad b\omega \nabla \zeta = b^{IJ \alpha} \omega_J \zeta_{I,\alpha}, \quad b^* \nabla \omega = b^{*IJ\beta}  \omega_{J,\beta} \zeta_{I},
	$$
	and so forth. We shall reserve the indices $J$, $\beta$ to $\omega$ and its localized version $\Omega$ below, and $I,\alpha$ to test functions. Rewrite now  \eqref{iid1} as 
	\begin{equation}\label{iid1a}
		\int\limits_{G_{2R}} ((a\nabla \omega+ b\omega + e) \nabla \zeta  +(b^* \nabla\omega + c\omega +f) \zeta)\, dx =0.
	\end{equation}

	Further we assume that $\omega$ is a homogeneous form of degree $r$ solving one of the three BVPs formulated above (full Dirichlet data/ Dirichlet/ Neumann) for \eqref{eq1} and in a local coordinate chart in \eqref{iid1} (equivalently \eqref{iid1a}) the form $\omega\in W^{1,p_{-}}(\mathcal{I}_0, G_{2R})$, and the data $e,f \in L^{p(\cdot)}(G_{2R})$. We assume that $M$ is at least of the class $C^{s+2,1}$, $s\geq -1$.
	

	\section{Potentials for the Hodge Laplacian}\label{sec:potentialHodge}
	
	In the interior patches for $F = F^I$ and $E = E^{I\alpha}$ we introduce the volume potentials $\mathcal{P}$ and $\mathcal{Q}$ as 
	$$
	(\mathcal{P}[F])_I = P[F^I], \quad  (\mathcal{Q}[E])_I = Q [E^I] =Q_\alpha [E^{I\alpha}]
	$$
	where the volume potentials $P$, $Q$ are defined in Section~\ref{sec:space}.

	In the case of boundary patches we use the following extensions of densities  of these potentials, which correspond to solving the Dirichlet/Neumann problem in the half-space $x^n>0$ by odd/even reflection. Using the notation of Section~\ref{sec:halfspace}  where the parameter $\rho =1$, that is the potentials corresponding to the standard Laplace operator are used, we define

	Let 
	\begin{itemize}
		\item $g_I(x,y) = g_D(x,y)$ for $I \in \mathcal{ I}_0$;
		
		\item $g_I(x,y) = g_N(x,y)$ for $I \notin \mathcal{ I}_0$;
		
		\item $g_{I\alpha}(x,y) =\frac{\partial g_D(x,y)}{\partial y^\alpha}$ for $I \in  \mathcal{ I}_0$;
		
		\item $g_{I\alpha}(x,y) =\frac{\partial g_N(x,y)}{\partial y^\alpha}$ for $I \notin  \mathcal{ I}_0$;
		
		\item $ \mathcal{E}^I =  \mathcal{E}^D $ for $I \in \mathcal{ I}_0$,;
		
		\item $ \mathcal{E}^I =  \mathcal{E}^N $ for $I \notin \mathcal{ I}_0$;
		
		\item $\mathcal{E}^{I\alpha} = \mathcal{E}^D$ if $I \in  \mathcal{ I}_0$ and $\alpha<n$ or $I \notin \mathcal{ I}_0$ and $\alpha =n$;
		
		\item $\mathcal{E}^{I\alpha} = \mathcal{E}^D$ if $I \notin  \mathcal{ I}_0$ and $\alpha=n$ or $I \in \mathcal{ I}_0$ and $\alpha =n$.
		
	\end{itemize}
	
	Then, for boundary patches we define (in the following formula there is no summation in $I$)
	\begin{equation}\label{def:PQ}
		\begin{aligned}
			(\mathcal{P}[F])_I &= \int \limits_{\mathbb{R}^n_{+}} \delta_{IJ}  g_I(x,y) F^J(y)\, dy = \int \limits_{\mathbb{R}^n} K_0(x-y) \mathcal{E}^I [F^I](y)\, dy,\\
			(\mathcal{Q}[E])_I &=  \int\limits_{\mathbb{R}^n_{+}} \delta_{IJ} g_{I\alpha}(x,y) E^{J\alpha}(y)\, dy = -\int\limits_{\mathbb{R}^n} K_{0,\alpha}(x-y) \mathcal{E}^{I\alpha}[E^{I\alpha}](y)\, dy.
		\end{aligned}
	\end{equation}
	In the notation of Section~\ref{sec:halfspace} this means
	\begin{align*}
		(\mathcal{P}[F])_I &= P^D[F^I], \quad I \in \mathcal{I}_0, \\ (\mathcal{P}[F])_I &= P^N[F^I], \quad I \notin \mathcal{I}_0,\\
		(\mathcal{Q}[E])_I &= Q^D[E^I] = Q^D_\alpha[E^{I\alpha}], \quad I \in \mathcal{I}_0, \\
		(\mathcal{Q}[E])_I &= Q^N[E^I] = Q^N_\alpha[E^{I\alpha}], \quad I \notin \mathcal{I}_0.
	\end{align*}

	Note that
	
	\begin{enumerate}
		
		\item for $I\in \mathcal{I}_0,$ the potential $(\mathcal{P}[F])_I$ solves the Dirichlet problem $$\triangle u = F^I  \text{ in } \mathbb{R}^n_{+}, \qquad u=0 \text{ on }x_n=0 ;$$
		
		\item for $I \notin \mathcal{I}_0,$ the potential $(\mathcal{P}[F])_I$ solves the Neumann problem $$\triangle u  = F^I \text{ in } \mathbb{R}^n_{+}, \qquad \partial u/\partial x_n =0 \text{ on } x_n =0;$$
		
		\item for $I \in \mathcal{I}_0, $ the potential $(\mathcal{Q}[E])_I$ solves the Dirichlet problem $$\triangle u = - (E^{I\alpha})_{x^\alpha} \text{ in } \mathbb{R}^n_{+}, \qquad u=0 \text{ on } x_n=0;$$

		\item for $I \notin \mathcal{I}_0, $  the potential $(\mathcal{Q}[E])_I$ solves the Neumann problem $$\triangle u = - (E^{I\alpha})_{x^\alpha} \text{ in }\mathbb{R}^n_{+}, \qquad \partial u/\partial x^n + E^{In}=0 \text{ on }x^n=0.$$
		
	\end{enumerate}

	For $\omega=\{\omega_I\}$ we denote (for functions the norm $\|\cdot\|^*_{s,p,G_{2R}}$ was introduced in Sections~\ref{sec:space}, \ref{sec:halfspace})
	$$
	\|\omega \|^*_{s,p, G_{2R}} = \sum_{I} \|\omega_I\|^*_{s,p, G_{2R}}.
	$$
	
	If $n=2$ we also need the following notation. For $F=F^I$ introduce the set of constants $c(F;G_{2R}) = \{c_I(F,G_{2R})\}$ by
		\begin{equation}\label{eq:c3}
			c_I (F;G_{2R}) = \begin{cases} c(F^I;G_{2R}) \quad \text{if}\quad I \notin \mathcal{I}_0,\\
				0, \quad \text{otherwise}
			\end{cases}
		\end{equation}
		Here the constants $c(f;G_{2R})$ are introduced by the expressions in \eqref{eq:c1}, \eqref{eq:c2}.

	From Lemma~\ref{L:pa0} and Lemma~\ref{L:pa1} we get
	\begin{lemma}\label{L:pa}
		For $F = F^I$ and $E= E^{I\alpha}$ with support in $\widetilde G_{2R}$ there holds
		\begin{align*}
			\| \mathcal{Q}[E]\|^*_{1,p(\cdot), G_{2R}} &\leq C(n,p_{-},p_{+},c_{\mathrm{log}} (p)) \|E\|_{p(\cdot), G_{2R}},\\
			\|\mathcal{P}[F]\|^*_{1,p(\cdot), G_{2R}} &\leq C(n,p_{-},p_{+},c_{\mathrm{log}} (p)) R \|F\|_{p(\cdot), G_{2R}},\\
			\|\mathcal{P}[F]\|^*_{1,\varkappa p(\cdot), G_{2R}} &\leq C(n,p_{-},p_{+},c_{\mathrm{log}} (p))\|F\|_{p(\cdot), G_{2R}}.
		\end{align*}
		Moreover, $\mathcal{Q}[E], \mathcal{P}[F] \in W^{1,p(\cdot)}(\mathcal{I}_0, G_{2R})$. If $n=2$ the estimates for $\mathcal{P}[F]$ are valid if $\mathcal{P}[F]$ is replaced by $\mathcal{P}[F] - c(F;G_{2R})$.
	\end{lemma}
	
	From Corollary~\ref{C:pa0} and Lemma~\ref{L:DR04c0} we get

	\begin{lemma}\label{L:DR04c}
		Let $F \in W^{s,p(\cdot)}(G_{2R})$, $s\in \{0\}\cup \mathbb{N}$, with support in $\widetilde G_{2R}$. Then 
		$$
		\|\mathcal{P}[F]\|^*_{s+2,p(\cdot),G_{2R}} \leq C \|F\|_{s,p(\cdot), G_{2R}}.
		$$
		Let $E\in W^{s+1,p(\cdot)}(G_{2R})$,  $s\in \{0\}\cup \mathbb{N}$, with support in $G_{2R}$. Then 
		$$
		\|\mathcal{Q}[E]\|^*_{s+2,p(\cdot),G_{2R}}  \leq C \|E\|_{s+1,p(\cdot), G_{2R}}^*.
		$$
		Here $C=C(n,p_{-},p_{+},c_{\mathrm{log}} (p)).$ If $n=2$ the estimates for $\mathcal{P}[F]$ are valid if $\mathcal{P}[F]$ is replaced by $\mathcal{P}[F] - c(F;G_{2R})$.
	\end{lemma}

	\section{Local a priori estimates}\label{ssec:local}

	Let $\xi \in C_0^\infty(B_{7R/4})$ be such that $\xi=1$ on $B_R$, $0\leq \xi \leq 1$, $|\nabla \xi| \leq C(n) R^{-1}$, and let $\xi_*\in C_0^\infty(B_{2R})$ be such that $0\leq \xi_*\leq 1$, $\xi_*=1$ on $B_{7R/4}$, and $|\nabla \xi_*| \leq C(n) R^{-1}$. 
	
	Using in \eqref{eq1} test form $\xi \zeta$ instead of $\zeta$ one obtains the following localized relation:
	\begin{equation}\label{eq:localized}
		\begin{gathered}
			\mathcal{D}(\omega \xi, \zeta) = (\widetilde \varphi, d\zeta) + (\widetilde \psi, \delta \zeta) + (\widetilde \eta,\zeta),\\
			\widetilde \varphi = \xi \varphi+ d\xi \wedge \omega, \quad \widetilde \psi = \xi \psi - d\xi \lrcorner \omega,\\
			\widetilde \eta = \xi \eta + d\xi \wedge \delta\omega - d\xi \lrcorner d\omega + d\xi \lrcorner \varphi - d\xi \wedge \psi.
		\end{gathered}
	\end{equation}
	Denote
	\begin{equation}\label{eq:nota1}
		\begin{gathered}
			\Omega = \omega \xi,\quad E =\xi e- a\omega\nabla \xi, \quad F = \xi f + e \nabla \xi+ a\nabla \omega \nabla \xi + (b-b^*) \omega\nabla \xi,\\
			(a\omega\nabla \xi)^{I\alpha} = a^{IJ\alpha \beta}\xi_{,\beta} \omega_J, \quad
			(a\nabla\omega \nabla \xi )^I = \xi_{,\alpha} a^{IJ\alpha \beta} \omega_{J,\beta},\\
			((b-b^*) \omega\nabla \xi)^I = (b^{IJ\alpha}- b^{*IJ\alpha}) \xi_{,\alpha} \omega_J, \quad (e\nabla \xi)^I = e^{I\alpha} \xi_{,\alpha}.
		\end{gathered}
	\end{equation}
	Recalling also the notation \eqref{eq:nota0}, for $\Omega \in W^{1,p_{-}}(\mathcal{I}_0, G_{2R})$ from \eqref{iid1} with $\zeta$ replaced by $\xi \zeta$, and using \eqref{Gaffney}, we get the formulation for \eqref{eq:localized} in coordinates: 
	\begin{multline}\label{iid1c}
		\int\limits_{G_{2R}} \left[ \nabla \Omega : \nabla \zeta + (\xi_*[(a(\cdot)-a(0)) \nabla \Omega + b \Omega]+E) \nabla \zeta \right. \\+ \left. (\xi_*[b^* \nabla \Omega + c\Omega]+F) \zeta \right]  \, dx=0,
	\end{multline}
	where $\zeta \in C_0^\infty(X)$, $X=\mathbb{R}^n$ for $G_{2R}=B_{2R}$ and $X=\overline{\mathbb{R}^n_+}$ if $G_{2R}=\overline{B}_{2R}^+$, $\zeta_I=0$ on $\{x^n=0\}$ for $I \in \mathcal{I}_0$. The relation between some of the terms in the RHS of \eqref{eq:localized} and \eqref{eq:nota1}-\eqref{iid1c} is 
	\begin{align*}
		(d\xi \lrcorner d\omega - d\xi \wedge \delta \omega,\zeta) &= \int\limits_{G_{2R}} (a\nabla \omega \nabla \xi + b\omega \nabla \xi)\zeta\, dx  \\
		- (d\xi \wedge \omega,d\zeta)+ (d\xi \lrcorner \omega,\delta \zeta) &= \int\limits_{G_{2R}} [-(a\omega \nabla \xi) \nabla \zeta -(b^*\omega \nabla \xi) \zeta]\, dx,\\
		(-d\xi \lrcorner \varphi + d\xi \wedge \psi,\zeta) &= \int\limits_{G_{2R}} (e\nabla \xi) \zeta\, dx.
	\end{align*} 
	
	Note that there holds 
		\begin{equation}\label{eq:mean0}
			\int\limits_{G_{2R}} [\xi_* (b^* \nabla \Omega + c \Omega ) +F]^K\, dx =\quad \text{for}\quad K \notin \mathcal{I}_0 .
		\end{equation}
		To check this relation it is sufficient to take the test form with coordinates $\zeta_I = \delta_{IK}$ in \eqref{iid1c}.
	
	Recall that in notation of Section~\ref{sec:coord} 
	$$
	f=C_1 \eta + C_2 \psi, \quad e = C_3 \varphi + C_4 \psi,
	$$
	and the regularity of the coefficients is
	\begin{itemize}
		\item $a, C_1, C_3, C_4\in C^{s+1,1}(G_{2R})$,
		
		\item $b,b^*, c, C_2\in C^{s,1}(G_{2R})$ if $s\geq 0$ and $b,b^*,c, C_2 \in L^\infty(G_{2R})$ if $s=-1$.
	\end{itemize}
	The corresponding bounds of the coefficients are uniform with respect to $R$ and the chosen element of the cover.

	
	In $G_{2R}$ there holds
	\begin{equation}\label{eq:modAest}
		\begin{aligned}
			|(a(\cdot)-a(0))\xi_*| &\leq CR,\\
			|\nabla [(a(\cdot)-a(0))\xi_*]|, |b\xi_*|, |b^* \xi_*|, |c\xi_*| &\leq C,
		\end{aligned}
	\end{equation}
	and if $M$ is at least of the class $C^{2,1}$ there holds 
	\begin{equation}\label{eq:modBest}
		|\nabla [b\xi_*]|, |\nabla[b^* \xi_*]|, |\nabla [c\xi_*]| \leq C R^{-1}.
	\end{equation}

	\begin{lemma} There holds 
		\begin{equation}\label{ieq1}
			\Omega =\mathcal{Q}[\xi_*((a(\cdot)-a(0)\nabla \Omega + b\Omega) +E] + \mathcal{P}[\xi_*(b^* \nabla \Omega + c\Omega)  +F] .
		\end{equation}
	\end{lemma}

	\begin{proof} By Lemma~\ref{L:pa} the expression on the right is well-defined and belongs at least to $W^{1,p_{-}}(\mathcal{I}_0,G_{2R})$. Moreover, (for $n=2$ this is due to \eqref{eq:mean0}) the potentials on the right decay at infinity. 
		
		We shall treat the case of boundary patches, the interior case being simpler. Denote 
		$$
		A = \xi_*[a(\cdot)-a(0)\nabla \Omega + b\Omega] +E, \quad B = \xi_*(b^* \nabla \Omega + c\Omega)  +F.
		$$
		There holds
		\begin{align*}
			\int\limits_{\mathbb{R}^n_+} (\nabla \mathcal{Q}[A]: \nabla \zeta+ A\nabla \zeta) \, dx=0,\\
			\int\limits_{\mathbb{R}^n_+} (\nabla \mathcal{P}[B]: \nabla \zeta +B\zeta)\, dx=0,
		\end{align*}
		where the test functions $\zeta_I \in C_0^\infty(\overline{\mathbb{R}^n_{+}})$ satisfy $\zeta_I=0$ on $x_n=0$ for $I\in \mathcal{I}_0$. 
		
		Substracting these relations from \eqref{iid1c} we get 
		$$
		\int\limits_{\mathbb{R}^n_+} \nabla V: \nabla \zeta =0,
		$$
		where 
		$$
		V:=\Omega - Q[A] - P[B]
		$$ 
		vanishes on $x_n=0$ for indices $I\in \mathcal{I}_0$. Therefore, the odd extension of $V_I$, if $I\in \mathcal{I}_0$, and even extension of $V_I$ for $I\notin \mathcal{I}_0$ across $x_n=0$ satisfies 
		$$
		\int\limits_{\mathbb{R}^n} \nabla V : \nabla \zeta\, dx=0
		$$
		for all $\zeta \in C_0^\infty(\mathbb{R}^n)$. By the Weyl lemma all the functions $V_I$ are harmonic, and since $V_I\to 0$ as $x\to \infty$ we get $V_I\equiv 0$. 
	\end{proof}
	
	
	
	
	Define the operator  
	\begin{equation}\label{Tdef}
		T[\Omega]= \mathcal{Q}[\xi_*((a(\cdot)-a(0))\nabla \Omega +b\Omega)] + \mathcal{P}[\xi_*(b^*\nabla \Omega + c \Omega)].
	\end{equation}
	The operator $T$ depends on $\xi_*$. The integral equation \eqref{ieq1} rewrites as 
	\begin{equation}\label{ieq2}
		\Omega - T[\Omega] = \mathcal{Q}[E] + \mathcal{P}[F].
	\end{equation}
	 For $n=2$ this needs the following modification. We set 
		$$
		T'[\Omega] =T[\Omega] - c(\xi_*(b^*\nabla \Omega + c \Omega); G_{2R}), 
		$$
		and instead of \eqref{ieq2} by using \eqref{eq:mean0} we work with the integral equation
		\begin{equation}\label{eq:mod2}
			\Omega - T'[\Omega] = \mathcal{Q}[E] + \mathcal{P}[F] - c(F;G_{2R}).
		\end{equation}
		With this modification all the estimates work in exactly the same way as for $n\geq 3$, so further we use the equation in the form \eqref{ieq2}.

	\section{Estimates in \texorpdfstring{$W^{1,p(\cdot)}(M,\Lambda)$}{W1pxM}}\label{ssec:W1p}
	
	Consider the operator $T$ as a mapping from $W^{1,p(\cdot)}(\mathcal{I}_0, G_{2R})$ to the same space. We show that for sufficiently small $R>0$ the operator $T$ is a contraction in the norm $\|\cdot\|^*_{1,p(\cdot),G_{2R}}$. 
	
	\begin{lemma}\label{L:c0} For sufficiently small $R=R(\mathrm{data})>0$ there holds
		\begin{equation}\label{me}
			\|T[\Omega]\|^*_{1,p(\cdot),G_{2R}} \leq \frac{1}{2} \|\Omega\|^*_{1,p(\cdot),G_{2R}}.
		\end{equation}
	\end{lemma}
	
	\begin{proof} From Lemma~\ref{L:pa}  it follows that
		\begin{align*}
			\| \mathcal{Q}[\xi_*(a(\cdot)-a(0)) \nabla \Omega]\|^*_{1,p(\cdot),G_{2R}} &\leq C(n,p_{-},p_{+},c_{\mathrm{log}} (p)) R \|\nabla \Omega\|_{p(\cdot), G_{2R}},\\
			\|\mathcal{Q}[\xi_*b\Omega]\|^*_{1,p(\cdot), G_{2R}} &\leq C(n,p_{-},p_{+},c_{\mathrm{log}} (p))\|b\|_{\infty} \| \Omega\|_{p(\cdot), G_{2R}},\\
			\|\mathcal{P}[\xi_*b \nabla \Omega]\|^*_{1,p(\cdot),G_{2R}} &\leq C(n,p_{-},p_{+},c_{\mathrm{log}} (p))\|b\|_{\infty} R \|\nabla \Omega\|_{p(\cdot),G_{2R}},\\
			\| \mathcal{P}[\xi_*c\Omega]\|^*_{1,p(\cdot),G_{2R}} &\leq C(n,p_{-},p_{+},c_{\mathrm{log}} (p)) \|c\|_{\infty} R \|\Omega\|_{p(\cdot),G_{2R}}.
		\end{align*}
		Then
		$$
		\|T[\Omega]\|^*_{1,p(\cdot),G_{2R}} \leq C R \|\Omega\|^*_{1,p(\cdot),G_{2R}},
		$$
		where the constant $C$ depends only on $n$, $p_{-}$, $p_{+}$, $c_{\mathrm{log}} (p)$, the Lipschitz constant of $a^{IJ\alpha\beta}$, and on $\|b\|_\infty+\|c\|_\infty$. It remains to choose $R\leq 1/(2C)$. 
	\end{proof}
	
	Let $R$ be sufficiently small such that the estimate \eqref{me} holds for all exponents $q\in \mathcal{P}^{\mathrm{log}}(p_{-}, p_{+},c_{\mathrm{log}}(p))$. 

	\begin{lemma}\label{L:c1} Let $\omega\in W^{1,p_{-}}(\mathcal{I}_0,G_{2R})$ be a solution of \eqref{iid1}. Let $E \in L^{p(\cdot)}(G_{2R})$ and $F \in L^{p_{-}}(G_{2R})$ in \eqref{eq:nota1} be such that $\mathcal{P}[F] \in W^{1,p(\cdot)}(G_{2R})$. Then we have $\omega \in W^{1,p(\cdot)}(\mathcal{I}_0,G_{3R/2})$ and
		$$
		\|\xi\omega\|^*_{1,p(\cdot),G_{2R}} \leq 2 (\|\mathcal{Q}[E]\|^*_{1,p(\cdot),G_{2R}} + \|\mathcal{P}[F]\|^*_{1,p(\cdot),G_{2R}}).
		$$
	\end{lemma}
	\begin{proof} 
		In $W^{1,p_{-}}(G_{2R})$ for $\Omega = \omega \xi$ defined in \eqref{eq:nota1} there holds \eqref{ieq2} and so
		\begin{equation}\label{eq:series}
			\Omega = \sum_{j=0}^\infty T^j (\mathcal{Q}[E] + \mathcal{P} [F]).
		\end{equation}
		But the series in the RHS converge also in $W^{1,p(\cdot)}(G_{2R})$. Thus $\Omega \in W^{1,p(\cdot)}(G_{2R})$ and so $\omega \in W^{1,p(\cdot)}(G_{R})$ with the required estimate.
	\end{proof}
	
	\begin{lemma}\label{L:p1} Assume $\omega \in W^{1,p_{-}}(\mathcal{I}_0,G_{2R})$ be a solution to \eqref{iid1} with $e,f\in L^{p(\cdot)}(G_{2R})$. Then $\omega \in W^{1,p(\cdot)}(\mathcal{I}_0, G_{3R/2})$. Moreover,
		\begin{equation}\label{P1e}
			\|\omega\|^{*}_{1,p(\cdot),G_{3R/2}} \leq C(\mathrm{data})(\|\omega\|^{*}_{1,p_{-},G_{3R/2}}+ \|e\|_{p(\cdot), G_{2R}}+\|f\|_{p(\cdot), G_{2R}}).
		\end{equation}
	\end{lemma}
	\begin{proof} Let $\varkappa = n/(n-1)$ and $l_0$ be the minimum integer such that $\varkappa^l p_{-} \geq p(\cdot)$.  Denote  
		$$
		q_l(x) = \min (\varkappa^l p_{-}, p(x)), \quad x\in G_{2R}.
		$$
		Clearly $l_0 \leq 1+\log_{\varkappa} (p_{+}/p_{-})$. Define
		$$
		R_l = \frac{7R}{4} - \frac{R l}{4l_0}, \quad l=0,\ldots, l_0, \quad \widetilde R_{l}= \frac{R_{l-1}+R_l}{2}
		$$
		Let the functions $\xi_l$, $l=1,\ldots, l_0$ be such that $\xi_l\in C_0^\infty(B_{\widetilde R_{l}})$, $0\leq \xi_l\leq 1$, $\xi_l=1$ on $B_{R_{l}}$, $R|\nabla \xi_l| \leq C(n,p_{-},p_{+})$. Let the functions $\widetilde \xi_l$, $l=1,\ldots, l_0$ be such that $\widetilde\xi_l\in C_0^\infty(B_{R_{l-1}})$, $0\leq \widetilde\xi_l\leq 1$, $\widetilde\xi_l=1$ on $B_{\widetilde R_{l}}$, $R|\nabla \widetilde\xi_l| \leq C(n,p_{-},p_{+})$. Let $E_l$ and $F_l$ be defined according to \eqref{eq:nota1} with $\xi=\xi_l$:
		$$
		E_l =\xi_l e- a\omega\nabla \xi_l, \quad F_l = \xi_l f + e \nabla \xi_l + a\nabla \omega \nabla \xi_l + (b-b^*) \omega\nabla \xi_l.
		$$
		Let $\omega \in W^{1,q_{l-1}(\cdot)}(G_{R_{l-1}})$. By the Sobolev embedding of Corollary~\ref{C:Di00} and the H\"older inequality we get 
		\begin{align*}
			\|\omega \nabla \xi_l\|_{q_l(\cdot), G_{R_{l-1}}} &\leq C(n,p_{-},p_{+}) R^{-1} \|\omega \widetilde \xi_l\|_{q_l(\cdot), G_{R_{l-1}}}\\
			&\leq C(n,p_{-},p_{+},c_{\mathrm{log}}(p)) R^{-1} R^{n/p_{-}'} \|\omega\|^*_{1,q_{l-1}(\cdot),G_{R_{l-1}}}.
		\end{align*}
		Now Lemma~\ref{L:pa} and the H\"older inequality yield 
		\begin{multline*}
			\|\mathcal{Q}[a\omega\nabla \xi_l]\|^*_{1, q_l(\cdot), G_{2R}} + \|\mathcal{P}[a \nabla \omega \nabla\xi_l ]\|^*_{1,q_l(\cdot), G_{2R}}+ \|\mathcal{P}[(b-b^*)\omega\nabla \xi_l]\|^*_{1, q_l(\cdot), G_{2R}}\\
			\leq C(\mathrm{data})R^{-1} \|\omega\|^*_{1,q_{l-1}(\cdot),G_{R_{l}}} 
		\end{multline*}
		and 
		\begin{align*}
			\|\mathcal{Q}[\xi_l e]\|^*_{1, q_l(\cdot), G_{2R}} &\leq C(\mathrm{data})\|e\|_{p(\cdot), G_{2R}},\\
			\|\mathcal{P}[\xi_l f]\|^*_{1, q_l(\cdot), G_{2R}} &\leq C(\mathrm{data}) \|f\|_{p(\cdot), G_{2R}},\\
			\|\mathcal{P}[e\nabla\xi_l]\|^*_{1, q_l(\cdot), G_{2R}} &\leq C(\mathrm{data}) \|e\|_{p(\cdot), G_{2R}}.
		\end{align*}
		Thus from Lemma~\ref{L:c1} we get 
		\begin{align*}
			\|\omega\|^*_{1,q_l(\cdot), G_{R_l}} &\leq 2\|\mathcal{Q}[E_l]\|^*_{1,q_l(\cdot),G_{2R}} + 2 \|\mathcal{P}[F_l]\|^*_{1,q_l(\cdot),G_{2R}} \\
			&\leq C (\mathrm{data}) ( \|e\|_{p(\cdot), G_{2R}} + \|f\|_{p(\cdot), G_{2R}} + R^{-1}\|\omega\|^*_{1,q_{l-1}(\cdot), G_{R_l}}).
		\end{align*}
		Iterating this inequality we obtain $\omega \in W^{1,p(\cdot)}(G_{3R/2})$ and the required estimate.
	\end{proof}
	
	The estimate we obtain in Lemma~\ref{L:p1} is very rough, with ``bad'' constant, but sufficient for our needs.
	\begin{corollary}\label{L:corr1} Let $\omega \in W_*^{1,p_{-}}(M,\Lambda)$, $*\in \{T,N,0\}$ be a solution to \eqref{eq1} (with $\zeta \in \mathrm{Lip}_* (M,\Lambda)$) where $\eta,\varphi,\psi\in L^{p(\cdot)}(M,\Lambda)$. Then $\omega \in W^{1,p(\cdot)}_*(M,\Lambda)$. Moreover,
		\begin{equation}\label{P1e1}
			\begin{aligned}
				\|\omega\|_{W^{1,p(\cdot)}(M,\Lambda)} &\leq C(\mathrm{data})\|\omega\|_{W^{1,p_{-}}(M,\Lambda)} \\
				&\ \ + C(\mathrm{data}) (\|\eta\|_{L^{p(\cdot)}(M,\Lambda)}+\|\varphi\|_{L^{p(\cdot)}(M,\Lambda)} + \|\varphi\|_{L^{p(\cdot)}(M,\Lambda)}).
			\end{aligned}
		\end{equation}
	\end{corollary}
	\begin{proof}
		It is sufficient to prove this for $\omega$ being a homogeneous form of degree $r$. Let $(V_\alpha,\widetilde h_\alpha)$ be an $R$-atlas with sufficiently small $R$ such that in each $G_{2R}=\widetilde h_\alpha (V_\alpha)$ the statement of Lemma~\ref{L:p1} holds. Using estimate of Lemma~\ref{L:p1} and \eqref{eq:basic} we get 
		\begin{equation*}
			\begin{aligned}
				\|\omega\|_{W^{1,p(\cdot)}(G_R)} \leq C&\|\omega\|_{W^{1,p_{-}}(G_{2R})} \\
				&+ C (\|\eta\|_{L^{p(\cdot)}(G_{2R})}+\|\varphi\|_{L^{p(\cdot)}(G_{2R})} + \|\varphi\|_{L^{p(\cdot)}(G_{2R})}).
			\end{aligned}
		\end{equation*}
		Since $\widetilde h_\alpha^{-1}(G_R)$ cover $M$ the claim easily follows. 
	\end{proof}
	
	\begin{remark}
		The dependence on $\nabla \omega$ in $F$ (see \eqref{eq:nota1}) arises only via $d\omega$ and $\delta \omega$, thus 
		\begin{align*}
			&\|\mathcal{P}[a\nabla \omega \nabla \xi + b\omega \nabla \xi]\|^*_{1,p(\cdot),G_{2R}} \\
			&\qquad \leq C(\mathrm{data}) R \|a\nabla \omega \nabla \xi + b\omega \nabla \xi\|_{p(\cdot),G_{2R}}\\
			&\qquad \leq C(\mathrm{data}) (\|d\omega\|^*_{p(\cdot),G_{2R}} +\|\delta \omega\|_{p(\cdot),G_{2R}} + \|\omega\|_{p(\cdot),G_{2R}}).
		\end{align*}
		That is, the term $\|\omega\|_{1,p_{-},M}$ on the RHS of the estimate \eqref{P1e1} can be replaced by $ \|d\omega\|_{p_{-},M}+ \|\delta \omega\|_{p_{-},M} + \|\omega\|_{p_{-},M}$ without direct use of the Gaffney inequality.  
	\end{remark}
	
	Together with the standard Gaffney inequality Corollary~\ref{L:corr1} provides the following a priori estimate.
	\begin{corollary}\label{L:corr1a} Let $p_{-}\geq 2$, $\omega \in W_*^{1,p_{-}}(M,\Lambda)$, $*\in \{T,N,0\}$ be a solution to \eqref{eq1} (with $\zeta \in \mathrm{Lip}_* (M,\Lambda)$) where $\eta,\varphi,\psi\in L^{p(\cdot)}(M,\Lambda)$. Then $\omega \in W^{1,p(\cdot)}_*(M,\Lambda)$. Moreover,
		\begin{equation}\label{P1e1a}
			\begin{aligned}
				\|\omega\|_{W^{1,p(\cdot)}(M,\Lambda)} &\leq C(\mathrm{data})\|\omega\|_{1,M} \\
				&\ \ + C(\mathrm{data}) (\|\eta\|_{p(\cdot),M}+\|\varphi\|_{p(\cdot),M} + \|\varphi\|_{p(\cdot),M}).
			\end{aligned}
		\end{equation}
	\end{corollary}
	\begin{proof}
		In this case we can use $\omega$ itself as the test form in \eqref{eq1}. This gives 
		\begin{gather*}
			\mathcal{D}(\omega,\omega)  = (\eta,\omega) + (\varphi, d\omega) + (\psi, d\omega) \\
			\leq \|\eta\|_{2,M} \|\omega\|_{2,M} + \|\varphi\|_{2,M}\|d\omega\|_{2,M} +\|\psi\|_{2,M}\|\delta\omega\|_{2,M} .
		\end{gather*}
		Thus
		$$
		\mathcal{D}(\omega,\omega) \leq \|\eta\|_{2,M}^2 + \|\omega\|_{2,M}^2 + 2\|\varphi\|_{2,M}^2 + 2\|\psi\|_{2,M}^2.
		$$
		By the classical Gaffney inequality we have 
		$$
		\|\omega\|_{1,2,M}^2 \leq C(\mathcal{D}(\omega,\omega) + \|\omega\|_{2,M}^2) \leq C (\|\eta\|_{2,M}^2 +\|\varphi\|_{2,M}^2 + \|\psi\|_{2,M}^2+ \|\omega\|_{2,M}^2).
		$$
		It remains to use Corollary~\ref{L:corr1} with $p_{-}=2$, the H\"older inequality, and interpolation to replace $\|\omega\|_{2,M}$ by $\|\omega\|_{1,M}$.
	\end{proof}

	\section{Estimates in \texorpdfstring{$W^{s+2,p(\cdot)}(M,\Lambda)$}{Ws+2pxM}}\label{ssec:W2p}
	
	In this section, $M$ will be at least of the class $C^{s+2,1}$, $s\geq 0$.
	
	In \eqref{eq:nota1} choose the cut-off function $\xi \in C_0^\infty(B_{7R/4})$  such that $\xi =1$ on $B_R$, $0\leq \xi \leq 1$, $|\nabla^k \xi| \leq C(n,s) R^{-k}$, $k\leq s+2$ and $\xi_*\in C_0^\infty(B_{2R})$ such that $0\leq \xi_*\leq 1$ and $\xi_* =1 $ on $B_{7R/4}$, and $\sum_{k=0}^{s+2} R^k|\nabla^k \xi_*|\leq C(n,s)$.
	
	Consider the operator $T$ from \eqref{Tdef} as a mapping from $W^{s+2,p(\cdot)}(G_{2R})$ to the same space. Let us show that for sufficiently small $R$ the operator $T$ is a contraction in the norm $\|\cdot\|^*_{s+2,p(\cdot),G_{2R}}$.
	
	Denote 
	$$
	\|\gamma\|_{s,\infty,G_{2R}} =  \sum_{k=0}^s R^k \|\nabla^k \gamma\|_{\infty,G_{2R}}.
	$$
	By definition, $\|\xi_*\|_{s,\infty,G_{2R}} \leq C_*(n)$
	
	Using the Leibnitz formula one easily sees that
	
	\begin{lemma}\label{L:simple}
		For $\gamma_1,\gamma_2 \in C^{s-1,1}(G_{2R})$ there holds
		$$
		\|\gamma_1 \gamma_2\|_{s,\infty,G_{2R}} \leq C(n,s) \|\gamma_1\|_{s,\infty,G_{2R}} \|\gamma_1\|_{s,\infty,G_{2R}}.
		$$
		For $\gamma\in C^{s-1,1}(G_{2R})$ there holds
		$$
		\| \xi_*\gamma \Omega\|^*_{s,p(\cdot),G_{2R}}\leq \|\gamma\|_{s,\infty,G_{2R}} C(n,s) \|\Omega\|^*_{s,p(\cdot),G_{2R}}.
		$$
	\end{lemma}

	Let $s\in \{0\}\cup \mathbb{N}$.
	
	\begin{lemma}\label{L:c0a} For sufficiently small $R = R(\mathrm{data},s)$ there holds
		\begin{equation}\label{me1}
			\|T[\Omega]\|^*_{s+2,p(\cdot),G_{2R}} \leq \frac{1}{2} \|\Omega\|^*_{s+2,p(\cdot),G_{2R}}.
		\end{equation}
	\end{lemma}
	\begin{proof}
		Recall that in the notation of \eqref{Tdef},
		$$
		T[\Omega] = \mathcal{Q}[\xi_* a((\cdot)-a(0))\nabla \Omega + \xi_*b\Omega] + \mathcal{P}[\xi_* b^*\nabla \Omega + \xi_* c\Omega]. 
		$$
		Since $a\in C^{s+1,1}$, $b,c\in C^{s,1}$, $s\geq 0$, we have
		\begin{align*}
			\|\xi_*a((\cdot)-a(0))\nabla \Omega \|^*_{s+1,p(\cdot),G_{2R}} &\leq  C(n,s)\|\nabla a\|_{s,\infty,G_{2R}} R \|\nabla \Omega\|^*_{s+1,p(\cdot),G_{2R}}\\
			&\leq   C(n,s) \|\nabla a\|_{s,\infty,G_{2R}}R \|\Omega\|^*_{s+2,p(\cdot,G_{2R})},\\
			\|\xi_*b\Omega \|^*_{s+1,p(\cdot),G_{2R}} 
			&\leq C(n,s) \|b\|_{s+1,\infty,G_{2R}} \|\Omega\|^*_{s+1,p(\cdot),G_{2R}} \\
			&\leq C(n,s) \|b\|_{s+1,\infty,G_{2R}} R \|\Omega\|^*_{s+2,p(\cdot),G_{2R}},\\
			\|\xi_*b^*\nabla \Omega\|^*_{s,p(\cdot),G_{2R}} 
			&\leq C(n,s) \|b^*\|_{s,\infty,G_{2R}} \|\nabla \Omega\|^*_{s,p(\cdot),G_{2R}} \\
			&\leq C(n,s) \|b^*\|_{s,\infty,G_{2R}} R \|\nabla \Omega\|^*_{s+2,p(\cdot),G_{2R}},\\
			\|\xi_*c\Omega\|_{s,p(\cdot),G_{2R}} 
			&\leq  C(n,s) \|c\|_{s,\infty,G_{2R}} \|\Omega\|^*_{s,p(\cdot),G_{2R}} \\
			&\leq  C(n,s) \|c\|_{s,\infty,G_{2R}} R^2 \|\Omega\|^*_{s+2,p(\cdot),G_{2R}}.
		\end{align*} 
		Lemma~\ref{L:DR04c} then gives 
		\begin{align*}
			&\|T[\Omega]\|^*_{s+2,p(\cdot),G_{2R}} \\
			&\qquad\leq \|\xi_*a((\cdot)-a(0))\nabla \Omega + \xi_*b\Omega \|^*_{s+1,p(\cdot),G_{2R}} + \|\xi_*b^*\nabla \Omega + \xi_*c\Omega\|_{s,p(\cdot),G_{2R}}\\
			&\qquad \leq  CR \|\Omega\|^*_{s+2,p(\cdot),G_{2R}}.
		\end{align*}
		The constant $C$ here depends only on $n$, $p(\cdot)$ (via $p_{-}$, $p_{+}$, $c_{\mathrm{log}} p$), the corresponding bounds for the coefficients. It remains to choose $R \leq 1/(2C)$. 
	\end{proof}
	
	By Lemma~\ref{L:c0a}, we can choose $R>0$ sufficiently small such that the mapping $T$ is a contraction, both in  $W^{s+1,p(\cdot)}(G_{2R})$ and  $W^{s+2,p(\cdot)}(G_{2R}),$ with the norm not exceeding $1/2$:
	\begin{equation}\label{eq:contr}
		\|T\Omega\|_{k,p(\cdot),G_{2R}} \leq \frac{1}{2} \|\Omega\|_{k,p(\cdot),G_{2R}}, \quad k=s+1,s+2.
	\end{equation}

	\begin{corollary}\label{C:2}
		Let $\omega\in W^{s+1,p(\cdot)}(\mathcal{I}_0,G_{2R})$ be a solution of \eqref{iid1} where $f\in W^{s,p(\cdot)}(G_{2R})$ and $e \in W^{s+1,p(\cdot)}(G_{2R})$. Then $\omega \in W^{s+2,p(\cdot)}(G_R)$ and 
		$$
		\|\omega\|^*_{s+2,p(\cdot),G_{R}} \leq C(\mathrm{data},s) (\|f\|^*_{s,p(\cdot),G_{2R}}+ \|e\|^*_{s+1,p(\cdot), G_{2R}} + R^{-1}\|\omega\|^*_{s+1,p(\cdot), G_{2R}}).
		$$
	\end{corollary}
	
	\begin{proof}
		By Lemmas~\ref{L:DR04c},~\ref{L:simple} we get 
		\begin{align*}
			\|\mathcal{Q}[E]\|^*_{s+2,p(\cdot), G_{2R}} &\leq C \| \xi e - a\omega\nabla \xi\|^*_{s+1,p(\cdot),G_{2R}} \\
			&\leq CR^{-1} \|\omega\|^*_{s+1,p(\cdot), G_{2R}} + C \|e\|^*_{s+1,p(\cdot), G_{2R}}  ,\\
			\|\mathcal{P}[F]\|^*_{s+2,p(\cdot),G_{2R}} &\leq C\| \xi f+e\nabla \xi+a  \nabla \omega \nabla \xi+ (b-b^*) \omega\nabla \xi\|^*_{s,p(\cdot),G_{2R}}\\
			&\leq  C\|f\|^*_{s,p(\cdot),G_{2R}} +C\|e\|^*_{s+1,p(\cdot)} + CR^{-1} \|\omega\|^*_{s+1,p(\cdot), G_{2R}}.
		\end{align*}
		
		By \eqref{eq:contr} equation \eqref{ieq2} is uniquely solvable in $W^{s+1,p(\cdot)}(G_{2R})$, and the solution is given by \eqref{eq:series}. The series converge also in $W^{s+2,p(\cdot)}(G_{2R})$ and
		\begin{align*}
			\|\Omega\|^*_{s+2,p(\cdot),G_{2R}} &\leq 2\|\mathcal{Q}[E]\|^*_{s+2,p(\cdot),G_{2R}} + 2 \|\mathcal{P}[F]\|^*_{s+2,p(\cdot),G_{2R}}\\
			&\leq  C (\|f\|_{s,p(\cdot),G_{2R}}+ \|e\|^*_{s+1,p(\cdot), G_{2R}}+ R^{-1}\|\omega\|^*_{s+1,p(\cdot), G_{2R}}).
		\end{align*}
		It remains to use that $\omega=\Omega$ in $G_R$. 
	\end{proof}

	\begin{theorem}\label{T:p2} Let $\omega$ be a $W_*^{1,p_{-}}(M,\Lambda)$, $*\in \{T,N,0\}$, solution of \eqref{eq1}, where $\zeta\in \mathrm{Lip}_*(M,\Lambda)$, with $\eta \in W^{s,p(\cdot)}(M,\Lambda)$, $\varphi, \psi\in W^{s+1,p(\cdot)}(M,\Lambda)$. Then $\omega \in W^{s+2,p(\cdot)}(M,\Lambda)$ and 
		\begin{multline*}
			\|\omega\|_{W^{s+2,p(\cdot)}(M,\Lambda)} \\
			\leq C(\|\eta\|_{W^{s,p(\cdot)}(M,\Lambda)} + \|\varphi\|_{W^{s+1,p(\cdot)}(M,\Lambda)}+\|\psi\|_{W^{s+1,p(\cdot)}(M,\Lambda)}+ \|\omega\|_{L^{1}(M,\Lambda)})
		\end{multline*}
		where $C = C (\mathrm{data},s)$.
	\end{theorem}
	
	\begin{proof} By Corollary~\ref{L:corr1} we have $\omega \in W^{1,p(\cdot)}(M,\Lambda)$. Assume without loss that $\omega$ is a homogeneous even form of degree $r$.
		
		We argue by induction in $s$. Let $1\leq k\leq s+1$ and $\omega \in W^{k,p(\cdot)}(M,\Lambda)$. Then in a sufficiently fine $\varepsilon$-atlas  by Corollary~\ref{C:2} we have  
		$$
		\|\omega\|_{k+1,p(\cdot),G_R} \leq C (\|\eta\|_{k-1,p(\cdot),M} + \|\varphi\|_{k,p(\cdot),M}+ \|\psi\|_{k,p(\cdot),M} + \|\omega\|_{k,p(\cdot),G_{2R}}).
		$$
		Since the neighborhoods $\widetilde h_\alpha^{-1} (G_R)$ cover $M$ we get
		\begin{equation}\label{eq:inter}
			\|\omega\|_{k+1,p(\cdot),M} \leq C (\|\eta\|_{k-1,p(\cdot),M} + \|\varphi\|_{k,p(\cdot),M}+ \|\psi\|_{k,p(\cdot),M} + \|\omega\|_{k,p(\cdot),M}).
		\end{equation}
		The interpolation inequality of Lemma~\ref{C:Di1} then gives
		$$
		\|\omega\|_{k+1,p(\cdot),M} \leq C (\|\eta\|_{k-1,p(\cdot),M} + \|\varphi\|_{k,p(\cdot),M}+ \|\psi\|_{k,p(\cdot),M} + \|\omega\|_{1,M}).
		$$
		When $k$ reaches $s+1$ we get the required estimate.
	\end{proof}
	
	\begin{remark}
		The dependence on $\nabla \omega$ in $F$ (see \eqref{eq:nota1}) arises only via $d\omega$ and $\delta \omega$, thus for $k=0,\ldots,s+1$ there holds
		\begin{align*}
			&\|\mathcal{P}[a\nabla \omega \nabla \xi + b\omega \nabla \xi]\|^*_{k+1,p(\cdot),G_{2R}} \\&\qquad\leq C(\mathrm{data},k) R \|a\nabla \omega \nabla \xi + b\omega \nabla \xi\|^*_{k,p(\cdot),G_{2R}}\\
			&\qquad\leq C(\mathrm{data},s) (\|d\omega\|^*_{k, p(\cdot),G_{2R}} +\|\delta \omega\|^*_{k, p(\cdot),G_{2R}} + \|\omega\|^*_{k, p(\cdot),G_{2R}}).
		\end{align*}
		That is, the term $\|\omega\|_{k,p(\cdot),M}$ on the RHS in the intermediate estimate \eqref{eq:inter} can be replaced by $ \|d\omega\|_{k-1,p(\cdot,M)}+ \|\delta \omega\|_{k-1,p(\cdot),M} + \|\omega\|_{k-1,p(\cdot),M}$. 
	\end{remark}
	
	\section{Gaffney inequality with natural regularity assumptions}\label{ssec:SimpleGaffney}
	In this Section we obtain Gaffney's inequality in variable exponent spaces by a direct argument without using the Hodge decomposition of Section~\ref{ssec:Gaffney}.

	\begin{theorem}\label{T:NNaturalGaffney}
		Let $M$ be of the class $C^{s+1,1}$, $s\in \{0\}\cup \mathbb{N}$, $\omega \in L^{p_{-}}(M,\Lambda)$,  $d\omega, \delta \omega \in W^{s,p(\cdot)}(M,\Lambda)$ and $t\omega=0$ or $n\omega=0$. Then $\omega \in W^{s+1,p(\cdot)}(M,\Lambda)$ and 
		\begin{equation}\label{eq:GS0}
			\|\omega - \mathcal{P}_* \omega\|_{s+1,p(\cdot),M} \leq C(\mathrm{data},s) (\|d\omega\|_{s,p(\cdot),M}+\|\delta\omega\|_{s,p(\cdot),M})
		\end{equation} 
		where $*$ is $T$ in the case of $t\omega=0$ and $*=N$ for $n\omega=0$.
	\end{theorem}
	
	The proof of this theorem will comprise of several lemmas.

	\begin{lemma}\label{T:SimpleGaffney}
		Let $M$ be of the class $C^{s+1,1}$, $s\in \{0\}\cup \mathbb{N}$, $\omega \in W^{1,p_{-}}_*(M,\Lambda)$, $*\in \{0,T,N\}$, and $d\omega, \delta \omega \in W^{s,p(\cdot)}(M,\Lambda)$. Then $\omega \in W^{s+1,p(\cdot)}(M,\Lambda)$ and 
		\begin{equation}\label{eq:GS}
			\|\omega\|_{s+1,p(\cdot),M} \leq C(\mathrm{data},s) (\|d\omega\|_{s,p(\cdot),M}+\|\delta\omega\|_{s,p(\cdot),M} + \|\omega\|_{1,M}).
		\end{equation} 
	\end{lemma}
	\begin{proof}
		Clearly, $\omega$ satisfies \eqref{eq1} with $\varphi = d\omega$, $\psi = \delta \omega$, $\eta=0$, for all $\zeta \in \mathrm{Lip}_*(M,\Lambda)$. Using Lemma~\ref{L:corr1} we get $\omega \in W_*^{1,p(\cdot)}(M,\Lambda)$.
		
		Let $\xi$ be a $C^{s+1,1}$ function supported in one of the coordinate neighborhoods. Then $\Omega \in   W^{1,p(\cdot)}_*(M,\Lambda)$ satisfies
		\begin{gather*}
			\mathcal{D}(\Omega,\zeta)=(\varphi,d\zeta)+(\psi,\delta \zeta)\quad \text{for all} \quad \zeta \in \mathrm{Lip}_*(M,\Lambda)\\
			\text{where}\quad \varphi = \xi d\omega +d\xi\wedge \omega, \quad \psi = \xi \delta \omega - d\xi \lrcorner \omega.
		\end{gather*}
		As above, we assume without loss that $\omega$ is a homogeneous form of degree $r$. In coordinates following the notation of Section~\ref{ssec:local} we get a relation similar to \eqref{iid1c}:
		\begin{multline}\label{iid1e}
			\int\limits_{G_{2R}} \left[ \nabla \Omega : \nabla \zeta + (\xi_*[(a(\cdot)-a(0)) \nabla \Omega + b \Omega]+ \hat{e}) \nabla \zeta \right. \\+ \left. (\xi_*[b^* \nabla \Omega + c\Omega]+ \hat{f}) \zeta \right]  \, dx=0,
		\end{multline}
		where $\zeta \in C_0^\infty(\mathcal{I}_0, G_{2R})$, and the dependence of $e$, $f$ on $\varphi$, $\psi$ is described in Section~\ref{sec:coord}. As in Section~\ref{ssec:local}, we get (with the the same modification for $n=2$, see \eqref{eq:mod2})
		$$
		\Omega - T[\Omega] = \mathcal{P}[\hat f] + \mathcal{Q}[\hat e].
		$$
		From Lemmas~\ref{L:c0}, \ref{L:c0a} for sufficiently small $R>0$ the operator $T$ acts as a contraction on $W^{k+1,p(\cdot)}(\mathcal{I}_0,G_{2R})$, $k=0,\ldots,s$. Therefore  
		\begin{align*}
			\|\Omega\|_{k+1,p(\cdot),G_{2R}} &\leq 2(\|\mathcal{P}[\hat f]\|_{k+1,p(\cdot),G_{2R}} + \|\mathcal{Q}[\hat e]\|_{k+1,p(\cdot),G_{2R}})\\
			&\leq C (\|d\omega\|_{k,p(\cdot),G_{2R}} + \|\delta \omega\|_{k,p(\cdot),G_{2R}}+\|\omega\|_{k,p(\cdot),G_{2R}}).
		\end{align*}
		Since we can cover our manifold by an open cover where each element is the image of $G_R$, we get
		$$
		\|\omega\|_{k+1,p(\cdot),M} \leq C (\|d\omega\|_{k,p(\cdot),M}+ \|\delta \omega\|_{k,p(\cdot),M}+\|\omega\|_{k,p(\cdot),M}).
		$$
		By iterating this inequality and using the interpolation Lemma~\ref{C:Di1} we finally arrive at the Gaffney type estimate \eqref{eq:GS}.
	\end{proof}
	
	\begin{lemma}\label{T:SimpleGaffney1}
		Let $M$ be of the class $C^{s+1,1}$, $s\in \{0\}\cup \mathbb{N}$, $\omega \in W^{1,p_{-}}_*(M,\Lambda)$, $*\in \{0,T,N\}$, and $d\omega, \delta \omega \in W^{s,p(\cdot)}(M,\Lambda)$. Then $\omega \in W^{s+1,p(\cdot)}(M,\Lambda)$ and 
		\begin{equation}\label{eq:GS1}
			\|\omega-\mathcal{P}_*\omega\|_{s+1,p(\cdot),M} \leq C(\mathrm{data},s) (\|d\omega\|_{s,p(\cdot),M}+\|\delta\omega\|_{s,p(\cdot),M}).
		\end{equation} 
	\end{lemma}
	\begin{proof}
		We prove this for the tangential boundary condition. Assume without loss that $p_{-}\leq 2$ and let $\eta\in L^{p_{-}'}(M,\Lambda)$, $\mathcal{P}_T\eta=0$. Since for $\Phi =G_D[\eta]$ we have
		$$
		\|\Phi\|_{2,M}\leq C(M) (\|d\Phi\|_{2,M}+ \|\delta\Phi\|_{2,M}) = C(M) \sqrt{\mathcal{D}(\Phi,\Phi)},\quad \mathcal{D}(\Phi,\Phi) = (\eta,\Phi),
		$$
		we get $\mathcal{D}(\Phi,\Phi)\leq C(M) (\eta,\eta)$ and
		$$
		\|\Phi\|_{2,M} \leq C (M) \|\eta\|_{2,M}.
		$$
		By the classical Gaffney inequality,
		$$
		\|\Phi\|_{1,2,M} \leq C(M) (\|d\Phi\|_{2,M}+ \|\delta\Phi\|_{2,M}) \leq C(M) \|\eta\|_{2,M}.
		$$
		By Lemma~\ref{L:corr1} and Lemma~\ref{T:SimpleGaffney},
		$$
		\|\Phi\|_{1,p_{-}',M}\leq C(M,p_{-}) \|\eta\|_{p_{-}',M}.
		$$
		Then
		$$
		(\omega-\mathcal{P}_T \omega, \eta)=\mathcal{D}(\Phi, \omega) \leq C(M,p_{-})(\|d\omega\|_{p_{-},M} +\|\delta\omega\|_{p_{-},M})\|\eta\|_{p_{-}',M}.
		$$
		It follows that for any $\zeta \in L^{p_{-}'}(M,\Lambda)$ there holds
		$$
		(\omega-\mathcal{P}_T \omega, \zeta) \leq C(M,p_{-})(\|d\omega\|_{p_{-},M} +\|\delta\omega\|_{p_{-},M})\|\zeta\|_{p_{-}',M}.
		$$
		Thus
		$$
		\|\omega - \mathcal{P}_T \omega\|_{p_{-},M} \leq C(M,p_{-}) (\|d\omega\|_{p_{-},M} +\|\delta\omega\|_{p_{-},M}).
		$$
		Combined with Lemma~\ref{T:SimpleGaffney} this yields \eqref{eq:GS1}.
	\end{proof}
	
	To prove Theorem~\ref{T:NNaturalGaffney} it remains to show that the existence of generalized differential and codifferential in $L^{p_{-}}(M,\Lambda)$ implies that $\omega \in W^{1,p_{-}}(M,\Lambda)$.  This is proved in Lemma \ref{existence of weak derivative tangential} and Lemma \ref{existence of weak derivative normal} for tangential and normal boundary conditions, respectively. To prove them, we shall use the following technical lemma.

	\begin{lemma} \label{L:VT}
		(i) Let $a^{\alpha\beta},f^\alpha,w\in L^{p_-}(\mathbb{R}^n)$ and the function $u\in L^1(\mathbb{R}^n)$ with compact support satisfy 
		\begin{equation}\label{eq:rel00}
			\int\limits_{\mathbb{R}^n} u \triangle \varphi \, dx = \int\limits_{\mathbb{R}^n} [a^{\alpha\beta}\varphi_{,\alpha\beta}+f^{\alpha}\varphi_{,\alpha} + w\varphi]\, dx
		\end{equation}
		for all $\varphi \in C_0^\infty(\mathbb{R}^n)$. Then 
		\begin{equation}\label{eq:sol0}
			u=(P[a^{\alpha\beta}])_{,\alpha\beta}- (P[f^{\alpha}])_{,\alpha}+P[w].
		\end{equation}
		If $n=2$ one additionally requires that the integral of $w$ over $\mathbb{R}^2$ is zero.

		(ii) Let $a_{\alpha\beta},f^\alpha,w\in L^{p_-}(\mathbb{R}^n_+)$ and the function $u\in L^1(\mathbb{R}^n_+)$ with compact support in $\overline{\mathbb{R}^n_+}$ satisfy 
		\begin{equation}\label{eq:rel01}
			\int\limits_{\mathbb{R}^n_+} u \triangle \varphi \, dx = \int\limits_{\mathbb{R}^n_+} [a^{\alpha\beta}\varphi_{,\alpha\beta}+f^{\alpha}\varphi_{,\alpha}+w\varphi]\, dx
		\end{equation}
		for all $\varphi \in C_0^\infty(\overline{\mathbb{R}^n_+})$ satisfying $\varphi =0$ on $\{x^n=0\}$. Then
		\begin{align}\label{eq:sol01}
		 u =	\begin{aligned}[t]
				P^D[w]&-\sum_{\alpha=1}^{n-1} (P^D[f^{\alpha}])_{,\alpha} - (P^N[f^n])_{,n}\\
				&+ \sum_{\alpha,\beta<n}(P^D[a^{\alpha\beta}])_{,\alpha\beta} + (P^D[a^{nn}])_{,nn} \\&\qquad + \sum_{\alpha=1}^{n-1} (P^N[a^{\alpha n}])_{,\alpha n} + \sum_{\beta=1}^{n-1} (P^N[a^{n\beta }])_{,n\beta}.
			\end{aligned}
		\end{align}
		(iii) If the integral identity \eqref{eq:rel01} holds for all $\varphi \in C_0^\infty(\overline{\mathbb{R}^n_+})$ satisfying $\varphi_{,n} =0$ on $\{x^n=0\}$ then 
		\begin{align}\label{eq:sol02}
			u = \begin{aligned}[t]
				 P^N[w]&-\sum_{\alpha=1}^{n-1} (P^N[f^{\alpha}])_{,\alpha} - (P^D[f^n])_{,n} \\
				&+\sum_{\alpha,\beta<n}(P^N[a^{\alpha\beta}])_{,\alpha\beta} + (P^N[a^{nn}])_{,nn} \\&\qquad + \sum_{\alpha=1}^{n-1} (P^D[a^{\alpha n}])_{,\alpha n} + \sum_{\beta=1}^{n-1} (P^D[a^{n\beta }])_{,n\beta}.
			\end{aligned}
		\end{align}
		 If $n=2$ this holds if the integral of $w$ over $\mathbb{R}^n_+$ is zero.
		\end{lemma}
	
	\begin{proof} To prove the first statement consider the function $v$ defined by the right-hand side of \eqref{eq:sol0}. This function satisfies relation \eqref{eq:rel00}, thus by the Weyl lemma the difference $u-v$ is harmonic. Since it vanishes at infinity, $u=v$. 
		
		The second statement follows by considering the odd extension of $u$ across $x^n=0$, which satisfies \eqref{eq:rel00} for all $\varphi \in C_0^\infty(\mathbb{R}^n)$ with $a^{\alpha\beta}$ obtained by odd extension if $\alpha,\beta<n$ or $\alpha = \beta=n$ and even extension otherwise, $f^\alpha$ obtained by odd extension if $\alpha<n$ and by even extension if $\alpha=n$, and $w$ obtained by the odd extension.
		
		The third statement follows by considering the even extension of $u$ across $x^n=0$, which satisfies \eqref{eq:rel00} for all $\varphi \in C_0^\infty(\mathbb{R}^n)$ with $a^{\alpha\beta}$ obtained by even extension if $\alpha,\beta<n$ or $\alpha = \beta=n$ and odd extension otherwise, $f^\alpha$ obtained by even extension if $\alpha<n$ and by odd extension if $\alpha=n$, and $w$ obtained by the even extension.  
	\end{proof}
	In \eqref{eq:sol0}, \eqref{eq:sol01}, and \eqref{eq:sol02} the terms on the RHS corresponding to $f$ can be written as $Q[f]$, $Q^D[f]$, and $Q^N[f]$, respectively.
	
	\begin{lemma}\label{existence of weak derivative tangential}  Let $M \in C^{1,1}$ and $\omega,d\omega,\delta\omega \in L^{p_-}(M,\Lambda)$, $t\omega =0$. Then $\omega \in W^{1,p_{-}}(M,\Lambda)$.
	\end{lemma}
	
	\begin{proof} By definition, for any $\varphi \in \mathrm{Lip}(M,\Lambda)$ and $\psi \in \mathrm{Lip}_T(M,\Lambda)$ there holds
		$$
		(\omega,\delta \varphi) = (f,\varphi), \quad (\omega,d\psi)=(v,\psi),
		$$
		where $f,v \in L^p(M,\Lambda)$.
		
		Therefore, 
		\begin{equation}\label{eq:Gaffney_test}
			(\omega,\delta \varphi) + (\omega,d\psi) = (f,\varphi)+ (v,\psi),
		\end{equation}
		for any $\varphi \in \mathrm{Lip}(M,\Lambda)$, $\psi \in \mathrm{Lip}_T(M,\Lambda)$.
		
		Now, replacing the test forms $\varphi$ and $\psi$ by $\xi \varphi, \xi\psi$, where $\xi$ is an arbitrary Lipschitz function with support in one of coordinate charts, we get 
		$$
		(\xi \omega,\delta \varphi + d\psi) = (\xi f +d\xi \wedge \omega,\varphi) + (\xi v -d\xi \lrcorner \omega,\psi),
		$$
		or 
		\begin{equation}\label{eq:start_rel}
			\begin{gathered}
				(\Omega, \delta \varphi + d\psi) = (\widetilde f, \varphi) + (\widetilde v, \psi),\\
				\Omega = \xi \omega,\quad \widetilde f = \xi f +d\xi \wedge \omega, \quad \widetilde v = \xi v -d\xi \lrcorner \omega.
			\end{gathered}
		\end{equation}
		We shall prove that $\Omega \in W^{1,p_{-}}(M,\Lambda)$.
		
		We can assume that $\omega$ is a homogeneous form of degree $r$. Then in coordinates the relation \eqref{eq:start_rel} takes the form 
		\begin{align}\label{eq:rel001}
		\int\limits_{G_{2R}} G^{IJ}_r \Omega_I &[ (\delta \varphi)_J + (d\psi)_J] \sqrt{g}\, dx \\
				&=\int\limits_{G_{2R}} (G^{IJ}_{r+1} \widetilde f_I \varphi_J + G^{IJ}_{r-1} \widetilde v_I \psi_J) \sqrt{g}\, dx. \notag
				\end{align}
		Here $\varphi$, $\psi$ are Lipshitz on $\mathbb{R}^n$ ($\overline{\mathbb{R}^n_+}$ in boundary patches) with $\psi_I =0$ on $\{x_n=0\}$ if $I \in \mathcal{I}_0$ (that is if $I$ does not contain $n$). Further we work in a boundary coordinate system, the interior case being simpler.
		
		Since $\Omega=0$ outside $G_{2R}$,  relation \eqref{eq:rel001} can be extended to the whole $\mathbb{R}^n_+$ :
		\begin{align*}
			\int\limits_{\mathbb{R}^n_+} G^{IJ}_r \Omega_I &[ (\delta \varphi)_J + (d\psi)_J] \sqrt{g}\, dx \\
			&=\int\limits_{\mathbb{R}^n_+} (G^{IJ}_{r+1} \widetilde f_I \varphi_J + G^{IJ}_{r-1} \widetilde v_I \psi_J) \sqrt{g}\, dx.
		\end{align*}
		
		Let $W=W_I dx^I$ be a smooth form in with compact support in $\overline{\mathbb{R}^n_+}$ such that (for boundary coordinate charts)  
		\begin{equation}\label{eq:Wcond}
			W_I =0, \quad I \in \mathcal{I}_0, \quad \frac{\partial W_{I}}{\partial x^n}=0, \quad I \notin \mathcal{I}_0.
		\end{equation}
		Recall that in this case $I\in\mathcal{I}_0$ if it does not contain $n$.
		
		Let 
		\begin{equation}\label{eq:test_pp}
			\varphi = dW, \quad \psi = \delta_0 W,
		\end{equation}
		where $\delta_0$ is the codifferential operator in standard Euclidian metric:
		$$
		\psi_I= -\frac{\partial W_{jI}}{\partial x^j}.
		$$
		Then $t\psi=0$, and we have 
		\begin{align*}
			\int\limits_{\mathbb{R}^n_+} G^{IJ}_r \Omega_I &[ (\delta dW)_J + (d\delta_0 W)_J] \sqrt{g}\, dx \\
			&=\int\limits_{\mathbb{R}^n_+} (G^{IJ}_{r+1} \widetilde f_I (dW)_J + G^{IJ}_{r-1} \widetilde v_I (\delta_0 W)_J) \sqrt{g}\, dx.
		\end{align*}
		By our assumption on the admissible coordinate system, $G^{IJ}(0) = \delta^{IJ}$, $\sqrt{g}(0)=1$. Using that $\delta_0 d +d\delta_0 = -\triangle$ rewrite the relation above as 
		\begin{align*}
			-\int\limits_{\mathbb{R}^n_+} \Omega_I \triangle W_I\, dx 
			=\begin{aligned}[t]
				\int\limits_{\mathbb{R}^n_+} &(G^{IJ}_{r+1} \widetilde f_I (dW)_J + G^{IJ}_{r-1} \widetilde v_I (\delta_0 W)_J) \sqrt{g}\, dx
				\\ &+\int\limits_{\mathbb{R}^n_+}\xi_* G^{IJ}_r \Omega_I [ (\delta_0-\delta)d W]_J\, \sqrt{g}\, dx\\
				&\qquad \qquad- \int\limits_{\mathbb{R}^n_+} \xi_*[G^{IJ}_r(0)-G_r^{IJ}\sqrt{g}] \Omega_I \triangle W_J\, dx.
			\end{aligned}
		\end{align*}
		Rewrite this as 
		\begin{align*}
			\int\limits_{\mathbb{R}^n_+} \Omega_I \triangle W_I\, dx
			=\int\limits_{\mathbb{R}^n_+}  &F^{I\beta} W_{I,\beta}  dx\\
			&+\int\limits_{\mathbb{R}^n_+}\xi_* \Omega_J(A^{IJ\alpha\beta}  W_{I,\alpha\beta} +B^{IJ\alpha} W_{I,\alpha} )\, dx,
		\end{align*}
		where by the Lipschitz continuity of the metric on $G_{2R}$ we have 
		\begin{equation}\label{eq:Aest}
			|A^{IJ\alpha\beta}|\leq C R,\quad |B^{IJ\alpha}|, |(A^{IJ\alpha\beta})_{,\gamma}|\leq C
		\end{equation}
		with $C$ independent of $R$. Here $F^{I\beta}$ are linear combinations of $\widetilde f_I$, $\widetilde v_I$ with Lipchitz continuous coefficients which depend only on the metric:
		$$
		F^{I\beta}W_{I,\beta} = - G^{IJ}_{r+1}\widetilde f_I (dW)_J - G^{IJ}_{r-1} \widetilde v_I (\delta_0 W)_J.
		$$
		
		Denote the terms on the RHS of \eqref{eq:sol01} and \eqref{eq:sol02} corresponding to $a^{\alpha\beta}$ by $H_{\alpha\beta}[a]$, $H_{\alpha\beta}^D[a]$, and $H_{\alpha\beta}^N[a]$, respectively. Then by Lemma~\ref{L:VT} we have
		\begin{align*}
			\Omega_I &= Q_\alpha^D[F^{I\alpha}+B^{IJ\alpha}\xi_*\Omega_J] + H_{\alpha\beta}^D[\xi_* A^{IJ\alpha\beta} \Omega_J], \qquad I \in \mathcal{I}_0,\\
			\Omega_I &= Q_\alpha^N[F^{I\alpha}+B^{IJ\alpha}\xi_*\Omega_J] + H_{\alpha\beta}^N[\xi_* A^{IJ\alpha\beta} \Omega_J], \qquad I \in \mathcal{I}_0.
		\end{align*}
		We abbreviate this to 
		$$
		\Omega = \mathcal{Q}[F] + \mathcal{Q}[\xi_* B\Omega] + \mathcal{H}[\xi_* A \Omega].
		$$
		Consider the operator 
		$$
		\mathcal{T}_0 [\Omega] = \mathcal{Q}[\xi_* B\Omega] + \mathcal{H}[\xi_* A \Omega]
		$$
		on $L^p(G_{2R})$. By Lemmas~\ref{L:pa1}, \ref{L:DR040} and using \eqref{eq:Aest} we get
		\begin{align*}
			\|\mathcal{Q}[\xi_* B\Omega]\|_{p_-,G_{2R}} &\leq C R \|\Omega\|_{p_-,G_{2R}},\\
			\|\mathcal{H}[\xi_* A \Omega]\|_{p_-,G_{2R}} &\leq CR \|\Omega\|_{p_-,G_{2R}}.
		\end{align*}
		Thus
		$$ 	\|\mathcal{T}_0[\Omega]\|_{p_-,G_{2R}} \leq C(M,p_-) R \|\Omega\|_{p_-,G_{2R}}.		$$
		For sufficiently small $R$ the operator $\mathcal{T}_0$ is a contraction on $L^{p_-}(G_{2R})$ with the norm less than $1/2$, and there holds
		\begin{equation}\label{eq:series0}
			\Omega = \sum_{j=0}^\infty (\mathcal{T}_0)^j \mathcal{Q}[F].
		\end{equation}
		Now, using the same Lemmas~\ref{L:pa1}, \ref{L:DR040} and estimate \eqref{eq:Aest} we have
		\begin{align}\label{eq:T0contr}
			\begin{aligned}
				\|\mathcal{Q}[\xi_* B\Omega]\|^*_{1,p_-,G_{2R}} &\leq C  \|\Omega\|_{p_-,G_{2R}} \leq C R \|\Omega\|^*_{1,p_-,G_{2R}},\\
				\|\mathcal{H}[\xi_* A \Omega]\|_{1,p_-,G_{2R}} &\leq CR \|\xi_* A\Omega\|^*_{1,p_-,G_{2R}} \leq C R \|\Omega\|^*_{1,p_-,G_{2R}},
			\end{aligned}
		\end{align}
		which implies 
		$$ 	\|\mathcal{T}_0[u]\|^*_{1,p_-,G_{2R}} \leq C R \|u\|^*_{1,p_-,G_{2R}}. $$
		For sufficiently small $R>0$ the operator $\mathcal{T}_0$ is also a contraction on $W^{1,p_-}(G_{2R})$ with the norm less than $1/2$. Thus, the series in \eqref{eq:series0} converge also in $W^{1,p_-}(G_{2R})$ and we get $\Omega \in W^{1,p_-}(G_{2R})$.
	\end{proof}
	
	In the same way one obtains the following
	
	\begin{lemma}\label{existence of weak derivative normal}  Let $M \in C^{1,1}$ and $\omega,d\omega,\delta\omega \in L^{p_-}(M,\Lambda)$, $n\omega =0$. Then $\omega \in W^{1,p_-}(M,\Lambda)$.
	\end{lemma}
	
	\begin{proof} The proof repeats the previous case. We describe the differences. Now the identity \eqref{eq:Gaffney_test} holds for all $\varphi \in \mathrm{Lip}_N(M,\Lambda)$ and $\psi \in \mathrm{Lip}(M,\Lambda)$. The condition \eqref{eq:Wcond} is replaced by 
		$$
		W_{I}=0, \quad I \in \mathcal{I}_0, \quad \frac{\partial W_I}{\partial x^n}=0 \quad \quad n\notin \mathcal{I}_0.
		$$
		Recall that in this case $I \in \mathcal{I}_0$ if $I$ contains $n$. Then the test forms $\varphi = dW$, $\psi = \delta_0 W$, defined in  \eqref{eq:test_pp} satisfy $n\varphi= ndW=0$, and the rest goes through as above.
	\end{proof}

    Combining the Lemmas above, the proof of Theorem \ref{T:NNaturalGaffney} is complete. 

\begin{remark} \label{remark:HD}   
Similar to Lemmas~\ref{existence of weak derivative tangential}, \ref{existence of weak derivative normal}, one can also obtain corresponding statements for the Hodge-Dirac equation, thus providing an alternative ``localized'' proof of Theorems~\ref{T:HodgeDiracGaffney},~\ref{T:HodgeDiracGaffneyN}. To this end, we first localize \eqref{eq:HD} which yields the relation
$$
(\Omega, \alpha d\zeta + \delta \zeta)= (\xi f+d\xi \wedge \omega - \alpha d\xi \lrcorner \omega,\zeta)
$$
for all $\zeta \in \mathrm{Lip}_T(M,\Lambda)$, compare \eqref{eq:start_rel} (we work here with the case of tangential boundary values). Now in local admissible coordinates, set $\zeta =dW + \alpha^{-1}\delta_0 W$ where $W\in C_0^\infty(\overline{\mathbb{R}^n_+})$ and $tW =0$ on $x^n=0$. This gives 
$$
(\Omega, \delta d W + d \delta_0 W + \alpha^{-1}\delta \delta_0 W) = (\xi f+d\xi \wedge \omega - \alpha d\xi \lrcorner \omega,dW + \alpha^{-1} \delta W).
$$
This integral identity is of the same form as the one used in the proof of Lemma~\ref{existence of weak derivative tangential}. The difference is that generally we can not split it into independent equations corresponding to homogeneous forms. Nevertheless, in local coordinates we have
$$
\delta d W + d \delta_0 W + \alpha^{-1}\delta \delta_0 W+ \triangle_{\mathrm{eucl}}W = (\delta-\delta_0)dW + \alpha^{-1}(\delta-\delta_0) \delta_0 W 
$$ 
with the coefficients by the second derivatives in the RHS being small, so the rest of the proof of Lemma~\ref{existence of weak derivative tangential} goes modulo some heavier notation --- we also have to sum integral identities over all $2(n+1)$ homogeneous components.
\end{remark}

	\section{Refined local a priori estimates}\label{sec:refined}

	In this Section we obtain refined local a priori estimates which may be of use of certain applications to low-regularity cases. We slightly change the solution representation obtained in Section~\ref{ssec:local}, keeping the same notation as above. This allows to obtain a refinement of the estimate obtained in Corollary~\ref{L:corr1}. First we show how to obtain the required result by using the estimate of Corollary~\ref{L:corr1} and duality. However, this estimate is not a local a priori estimate since it requires solving a dual problem.

	\begin{lemma}\label{L:TrueGaffney}
		Let $M$ be of the class $C^{1,1}$. Let $\omega \in W^{1,p_{-}}_{*}(M,\Lambda)$, $* \in \{0,T,N\}$ satisfy \eqref{eq1}  where $\varphi,\psi,\eta \in L^{p(\cdot)}(M,\Lambda)$ for any $\zeta\in \mathrm{Lip}_*(M,\Lambda)$. Then $\omega \in W^{1,p(\cdot)}(M,\Lambda)$ and
		\begin{equation}\label{eq:critest}
			\|\omega\|_{1,p(\cdot),M} \leq C(\mathrm{data}) (\|\eta\|_{p(\cdot),M} +\|\varphi\|_{p(\cdot),M} + \|\psi\|_{p(\cdot),M} + \|\omega\|_{1,M}).
		\end{equation}
	\end{lemma}
	\begin{proof}
		For $p_{-}\geq 2$ this statement is already contained in Corollary~\ref{L:corr1a}.  Assume without loss that $p_{-}\leq 2$. We treat the Dirichlet problem, the other two cases being similar. Let $\xi,\mu \in L^{p_{-}'}(M,\Lambda)$. By minimizing the functionals 
		$$
		\mathcal{F}_1[\Phi]=\frac{1}{2}\mathcal{D}(\Phi,\Phi) - (\xi,d\Phi), \quad \mathcal{F}_2[\psi]=\frac{1}{2}\mathcal{D}(\Psi,\Psi) - (\mu,d\Psi)
		$$
		over $W^{1,2}_T(M,\Lambda) \cap (\mathcal{H}_T(M))^\perp ,$ we find the minimizers satisfying 
		\begin{equation}\label{eq:aux}
			\mathcal{D}(\Phi,\zeta) = (\xi,d\Phi), \quad \mathcal{D}(\Psi,\zeta) = (\mu,d\Psi)
		\end{equation}
		for all $\zeta \in W_{T}^{1,2}(M,\Lambda)$. In particular,  we have
		$$
		\mathcal{D}(\Phi,\Phi) \leq (\xi,\xi), \quad \mathcal{D}(\Psi,\Psi) \leq (\mu,\mu).
		$$
		By \eqref{eq:lambda}, 
		$$
		(\Phi,\Phi) \leq C_T(\xi,\xi), \quad (\Psi,\Psi) \leq C_T(\mu,\mu).
		$$
		By Corollary~\ref{L:corr1a} we obtain $\Phi,\Psi \in W^{1,p_{-}'}(M,\Lambda)$ and
		\begin{equation}\label{eq:aux1}
			\|\Phi\|_{1,p_{-}',M} \leq C(\mathrm{data}) \|\xi\|_{p_{-}',M}, \quad \|\Psi\|_{1,p_{-}',M} \leq C(\mathrm{data})\|\mu\|_{p_{-}',M}).
		\end{equation}
		Using approximation,  the space of test forms $\zeta$ in \eqref{eq:aux}  can be extended to $W^{1,p_{-}}_T(M,\Lambda)$. Thus
		\begin{gather*}
			(\xi,d\omega) = \mathcal{D}(\Phi,\omega)  = (\eta,\Phi) + (\varphi,d\Phi ) + (\psi, \delta \Phi),\\
			(\mu,\delta\omega) =  \mathcal{D}(\Psi,\omega) = (\eta,\Psi) + (\varphi,d\Psi) + (\psi,\delta \Psi).
		\end{gather*}
		Using \eqref{eq:aux1} we have
		\begin{gather*}
			|(\xi,d\omega)|\leq C(\mathrm{data}) \|\xi\|_{p_{-}',M}(\|\eta\|_{p_{-},M} +\|\varphi\|_{p_{-},M}  +\|\psi\|_{p_{-},M}),\\
			|(\mu,\delta\omega)| \leq C(\mathrm{data}) \|\mu\|_{p_{-}',M}(\|\eta\|_{p_{-},M} +\|\varphi\|_{p_{-},M}  +\|\psi\|_{p_{-},M}).
		\end{gather*}
		By the duality characterization of the norm this yields
		\begin{gather*}
			\|d\omega\|_{p_{-},M} \leq C(\mathrm{data}) (\|\eta\|_{p_{-},M} +\|\varphi\|_{p_{-},M}  +\|\psi\|_{p_{-},M}),\\
			\|\delta\omega\|_{p_{-},M} \leq C(\mathrm{data})(\|\eta\|_{p_{-},M} +\|\varphi\|_{p_{-},M}  +\|\psi\|_{p_{-},M}).
		\end{gather*}
		Now we can use the Gaffney inequality of Lemma~\ref{T:SimpleGaffney} and  Corollary~\ref{L:corr1} to derive the inequality \eqref{eq:critest}. 
	\end{proof}
	
	In the following theorem we obtain the same result without using duality argument. This requires a different representation of terms arising at the localization stage in Section~\ref{ssec:local}. Namely, by using integraion-by-parts and special structure of the equation one can get rid of the derivatives of $\omega$ on the right-hand side of the resulting equation.
	
	\begin{theorem}\label{T:TrueGaffney}
		Let $M$ be of the class $C^{1,1}$. Let $\omega \in W^{1,p_{-}}_{*}(M,\Lambda)$, $* \in \{0,T,N\}$ satisfy \eqref{eq1}  where $\varphi,\psi,\eta \in L^{p(\cdot)}(M,\Lambda)$ for any $\zeta\in \mathrm{Lip}_*(M,\Lambda)$. Then $\omega \in W^{1,p(\cdot)}(M,\Lambda)$ and
		\begin{equation}\label{eq:TG1}
			\|\omega\|_{1,p(\cdot),M} \leq C(\mathrm{data}) (\|\eta\|_{p(\cdot),M} +\|\varphi\|_{p(\cdot),M} + \|\psi\|_{p(\cdot),M} + \|\omega\|_{1,M}).
		\end{equation}
		If $M$ is of the class $C^{s+1,1}$, $s\in \mathbb{N}$, $\varphi,\psi\in W^{s,p(\cdot)}(M,\Lambda)$ and $\eta \in W^{s-1.p(\cdot)}(M,\Lambda)$, then $\omega \in W^{s+1,p(\cdot)}(M,\Lambda)$ and
		\begin{equation}\label{eq:TG2}
			\|\omega\|_{s+1,p(\cdot),M} \leq C(\mathrm{data},s) (\|\eta\|_{s-1,p(\cdot),M} +\|\varphi\|_{s,p(\cdot),M}+\|\psi\|_{s,p(\cdot),M} + \|\omega\|_{1,M}).
		\end{equation}
	\end{theorem}
	
	\begin{proof}
		We again assume that $\omega$ is a homogeneous form.
		
		Assuming the function $\xi$ to be radial we get that $d\xi$ is tangential at $x^n=0$ and thus using the integration-by-parts formula one can rewrite 
		\begin{align*}
			(-d\xi \lrcorner d\omega + d\xi \wedge \delta \omega,\zeta) &= -(d\omega,d\xi \wedge \zeta) + (\delta \omega,d\xi \lrcorner \zeta) \\
			&=(\omega,  -\delta (d\xi \wedge \zeta) + d(d\xi \lrcorner\zeta)).
		\end{align*}
		Note that the last expression does not contain derivatives of $\omega$, and with respect to $\zeta$ it contains only the components of $\zeta$ and their first derivatives. We can further transform this (though this is not essential) by writing
		\begin{align*}
			-\delta (d\xi \wedge \zeta) + d(d\xi \lrcorner\zeta) &= -\delta ( d(\xi \zeta) - \xi d\zeta) - d (\delta (\xi \zeta)-\xi \delta \zeta )\\
			&= -\triangle (\xi \zeta)+ \xi \triangle \zeta - d \xi \lrcorner d\zeta +d\xi \wedge \delta \zeta  
		\end{align*}
		and using that the Bochner-Weizenb\"ock formula gives (see for instance \cite[Chapter V, \S 26]{deRham}) 
		$$
		-\triangle (\xi \zeta)+ \xi \triangle \zeta = -\zeta\triangle \xi +2 \nabla \xi \cdot \nabla \zeta, \quad (\nabla \xi \cdot \nabla \zeta)_I= (\nabla^j \xi)\nabla_j \zeta_I,
		$$
		the last expression containing only partial derivatives of $\zeta$ of orders $0$ and $1$, and the coefficients are expressed in terms of the metric tensor and Christoffel symbols only. Formally this calculation requires that the metric tensor be twice differentiable, but since its second order derivatives are not present in the final expression the relation readily extends to $C^{1,1}$ Riemannian manifolds. Thus, instead of \eqref{eq:localized} after cancellations we get 
		\begin{equation}\label{eq:localized1}
			\begin{aligned}
				\mathcal{D}(\omega \xi, \zeta) &= (\xi \varphi, d\zeta) + (\xi \psi, \delta \zeta) + (\hat \eta,\zeta) -(\omega,\zeta)\triangle \xi + (\omega, 2 \nabla \xi \cdot \nabla \zeta),\\
				\hat \eta &= \xi \eta + d\xi \lrcorner \varphi - d\xi \wedge \psi.
			\end{aligned}
		\end{equation}
		This means that instead of \eqref{iid1c} we get a similar relation 
		\begin{multline}\label{iid1d}
			\int\limits_{G_{2R}} \left[ \nabla \Omega : \nabla \zeta + (\xi_*[(a(\cdot)-a(0)) \nabla \Omega + b \Omega]+\widetilde E) \nabla \zeta \right. \\+ \left. (\xi_*[b^* \nabla \Omega + c\Omega]+\widetilde F) \zeta \right]  \, dx=0,
		\end{multline}
		for all $\zeta \in C_0^\infty(\mathcal{I}_0, G_{2R})$, with $F$, $E$ replaced by 
		$$
		\widetilde F = \xi f + e\nabla \xi+ b_1\omega \nabla^2\xi +b_2 \omega \nabla \xi, \quad
		\widetilde E = \xi e + a_1\omega \nabla \xi,
		$$
		where 
		\begin{gather*}
			(a_1 \omega \nabla \xi ) \nabla\zeta = a_1^{IJ\alpha \beta}\omega_J \xi_{,\beta} \zeta_{I,\alpha}, \quad a_1^{IJ\alpha\beta} = -2 G^{IJ}g^{\alpha \beta}\sqrt{g},\\
			(b_1 \omega \nabla^2\xi)\zeta = b_1^{IJ\alpha \beta}\omega_J \xi_{,\alpha\beta} \zeta_I, \quad  b_1^{IJ\alpha\beta} = - g^{\alpha\beta} G^{IJ}\sqrt{g}=\frac{1}{2}a_1^{IJ\alpha\beta},\\
			(b_2 \omega \nabla \xi)\zeta = b_2^{IJ\alpha}\omega_J \xi_{,\alpha}\zeta_I.
		\end{gather*}
		Here $a_1$ and $b_1$ depend only on the coefficients of the metric, and thus on $C^{s+1,1}$ manifold belong to the class $C^{s,1}(\widetilde G_{2R})$, while $b_2$ also depends on the derivatives of the metric tensor, and thus belong to $C^{s-1,1}(G_{2R})$ if $s>0$ and $L^\infty(G_{2R})$ if $s=0$. The corresponding bounds are uniform with respect to $R$ and to the element of the corresponding cover ($R$-atlas).
		
		Arguing as above, we obtain the integral equation  (for $n=2$ this requires the same modification as in \eqref{eq:mod2})
		$$
		\Omega - T[\Omega] = \mathcal{P}[\widetilde F] + \mathcal{Q}[\widetilde E]
		$$
		and instead of \eqref{eq:series} we get 
		$$
		\Omega = \sum_{j=0}^\infty T^j ( \mathcal{P}[\widetilde F] + \mathcal{Q}[\widetilde E]).
		$$
		Therefore, choosing sufficiently small $R=R(\mathrm{data},s)$ as above (see Lemmas~\ref{L:c0}, \ref{L:c0a})  so that $T$ is a contraction in the norm $\|\cdot\|^*_{k,p(\cdot),G_{2R}}$ with the norm $\leq 1/2$ for all $k=1,\ldots,s+1$, we get 
		$$
		\|\omega\|^*_{k,p(\cdot),G_{R}} \leq 2 \|\mathcal{P}[\widetilde F]\|^*_{k,p(\cdot),G_{2R}} + 2 \|\mathcal{Q}[\widetilde E]\|^*_{k,p(\cdot),G_{2R}}. 
		$$
		We start from $s=0$. In this case all the coefficients of $a_1$, $b_1$, $b_2$ are uniformly bounded and using Lemma~\ref{L:pa} we evaluate 
		\begin{align*}
			\|\mathcal{P}[\widetilde F]\|^*_{1,p(\cdot),G_{2R}} &\leq C R \|\widetilde F\| _{p(\cdot), G_{2R}} \\ &\leq C R \|f\|_{p(\cdot),G_{2R}} +C\|e\|_{p(\cdot),G_{2R}}+ C R^{-1}\|\omega\|_{p(\cdot), G_{2R}},\\
			\|\mathcal{Q}[\widetilde E]\|^*_{1,p(\cdot),G_{2R}} &\leq C \|\widetilde E\|_{p(\cdot),G_{2R}} \leq C \|e\|_{p(\cdot),G_{2R}} + C R^{-1} \|\omega\|_{p(\cdot), G_{2R}}.
		\end{align*}
		Thus 
		$$
		\|\omega\|^*_{1,p(\cdot),G_{R}} \leq C (R\|f\|_{p(\cdot),G_{2R}} +\|e\|_{p(\cdot),G_{2R}} + R^{-1}\|\omega\|_{p(\cdot), G_{2R}}).  
		$$
		Gluing these estimates together and using interpolation we obtain \eqref{eq:TG1}.
		
		Now let $s>0$. If $k> 1$, using again Lemmas~\ref{L:DR04c},~\ref{L:simple} we obtain 
		\begin{align*}
			\|\mathcal{P}[\widetilde F]\|^*_{k,p(\cdot),G_{2R}} 
			&\leq C \|f\|^*_{k-2,p(\cdot),G_{2R}} + C \|e\|^*_{k-1,p(\cdot),G_{2R}} + C R^{-1} \|\omega\|^*_{k-1,p(\cdot),G_{2R}}, \\
			\|\mathcal{Q}[\widetilde E]\|^*_{k,p(\cdot),G_{2R}} &\leq C \|\xi e  + a_1\omega \nabla \xi\|^*_{k-1,p(\cdot),G_{2R}}\\
			&\leq C \|e\|^*_{k-1,p(\cdot),G_{2R}} + C R^{-1} \|\omega\|^*_{k-1,p(\cdot),G_{2R}} 
		\end{align*}
		Thus we have 
		$$
		\|\omega\|^*_{k,p(\cdot),G_{R}} \leq C (\|f\|^*_{k-2,p(\cdot),G_{2R}} +\|e\|^*_{k-1,p(\cdot),G_{2R}} + R^{-1}\|\omega\|^*_{k-1,p(\cdot), G_{2R}})
		$$
		which implies 
		$$
		\|\omega\|_{k,p(\cdot),M} \leq C (\|\eta\|_{k-2,p(\cdot),M} +\|\varphi\|_{k-1,p(\cdot),M}+\|\psi\|_{k-1,p(\cdot),M} + \|\omega\|_{k-1,p(\cdot),M}).
		$$
		Iterating this estimate in $k=1,\ldots,s+1$ and using interpolation we get \eqref{eq:TG2}.
	\end{proof}

	\section{Additional regularity of the potential on \texorpdfstring{$C^{1,1}$}{C11} manifolds}\label{sec:additional}
	
	In this Section we prove that on $C^{1,1}$ manifold for $\eta \in L^{p(\cdot)}(M,\Lambda)$ the Dirichlet/Neumann potentials $G_D[\eta]$, $G_N[\eta]$, which belong to $W^{1,p(\cdot)}(M,\Lambda)$, have differential and codifferential also from $W^{1,p(\cdot)}(M,\Lambda)$.

\begin{theorem}\label{T:addregD}
Let $\omega_0, d\omega_0 \in W^{1,p(\cdot)}(M,\Lambda)$, $\omega\in \omega_0+W_T^{1,p(\cdot)}(M,\Lambda)$, $\eta \in L^{p(\cdot)}(M,\Lambda)$ and $\varphi,\psi \in W^{1,p(\cdot)}(M,\Lambda)$ satisfy \eqref{eq1}, that is
$$
\mathcal{D}(\omega,\zeta) = (\eta,\zeta) + (\varphi,d\zeta) + (\psi,\delta \zeta),
$$
for all $\zeta \in \mathrm{Lip}_T(M,\Lambda)$. Then $\alpha =\delta \omega$ and $\beta = d\omega$ belong to $W^{1,p(\cdot)}(M,\Lambda)$ and satisfy $t\alpha =t\psi$, $t\beta =td\omega_0$,
\begin{equation}\label{eq:rel0Da}
d\alpha + \delta \beta = \eta + d\psi + \delta \varphi
\end{equation}
and
\begin{equation}\label{eq:rel0D}
\mathcal{D}(\alpha,\zeta) = (\eta+d\psi, d\zeta), \quad \mathcal{D}(\beta,\zeta) = (\eta + \delta\varphi,\delta \zeta)
\end{equation}
for all $\zeta \in \mathrm{Lip}_T(M,\Lambda)$. 
\end{theorem}

	\begin{proof}
		Let $\omega$ be a homogeneous form of degree $r$ and $\Omega =\omega \xi $. Let the cut-off functions $\xi$, $\xi_*$ be defined as in Sections~\ref{ssec:local}--\ref{ssec:W2p} There holds (see\eqref{eq:localized})
        
        $$
        \mathcal{D}(\Omega,\zeta) = (\widetilde \eta, \zeta) + (\widetilde \varphi,d\zeta)+(\widetilde \psi,\delta \zeta)
        $$
        where 
        \begin{gather*}
         \widetilde \varphi = \xi \varphi + d\xi \wedge \omega,\quad \widetilde \psi = \xi \psi - d\xi \lrcorner \omega,\\
         \widetilde \eta = \xi \eta +d\xi \wedge (\delta \omega-\psi) -d\xi \lrcorner (d\omega-\varphi)
        \end{gather*}

		Here $\widetilde \varphi, \widetilde \psi\in W^{1,p(\cdot)}(M,\Lambda)$ and $\widetilde \eta \in L^{p(\cdot)}(M,\Lambda)$.
        This could be further transformed to a somewhat simpler expression (see \eqref{eq:localized1}) but we do not need this form here.
		
		Further we shall work with boundary coordinate charts, the interior case being much simpler. Recall that for tangential boundary values $I \in \mathcal{I}_0$ iff $n\notin I$. Take $\zeta =d W$ where
        \begin{equation}\label{eq:W1}
	      W \in C_0^\infty(\overline{\mathbb{R}^n_+}),\quad W_I(x',0)=0 \quad \text{if}\quad n\notin I, \quad \frac{\partial W_I(x',0)}{\partial x^n}=0 \quad \text{if}\quad n\in I.
		\end{equation}
		We obtain (note that here we work in one particular coordinate chart)
		\begin{equation}\label{eq:pp1}
		(\delta \Omega, \delta d W)= (\widetilde \psi, \delta dW)+(\widetilde \eta,dW)
		\end{equation}
		Now, the form which is defined as $\delta_0 W$ in our local coordinate system, and denoted as such:
		$$
		(\delta_0 W)_I = - W_{jI,j},
		$$
		satisfies  $t\delta_0 W=0$ and thus the integration-by-parts formula yields
		\begin{equation}\label{eq:pp2}
		(\delta\Omega, d \delta_0 W)=0.
		\end{equation}
        Summing \eqref{eq:pp1} and \eqref{eq:pp2} we arive at 
        \begin{equation}\label{eq:pp3}
        (\delta \Omega, \delta d W + d \delta_0 W)=(\widetilde \psi, \delta dW)+(\widetilde \eta,dW).
        \end{equation}
		Denote $\delta \Omega = \Psi$. In coordinates \eqref{eq:pp3} becomes
		\begin{align*}
			\int\limits_{\mathbb{R}^n_+} G^{IJ}_r &\Psi_I [(\delta d + d\delta_0)W]_J \sqrt{g}\, dx \\
			&=\int\limits_{\mathbb{R}^n_+} G^{IJ}_{r-1} \widetilde \psi_I (\delta d W)_J\sqrt{g}\, dx 
			+\int\limits_{\mathbb{R}^n_+} G^{IJ}_{r} \widetilde \eta_I (d W)_J\sqrt{g}\, dx.
		\end{align*} 
		Rewrite this as 
		\begin{multline*}
			-\int\limits_{\mathbb{R}^n_+} \Psi_I \triangle W_I\, dx = 
			\int\limits_{\mathbb{R}^n_+} \xi_*G^{IJ}_r \Psi_I [(\delta_0-\delta)d W]_J \sqrt{g}\, dx \\
			+\int\limits_{\mathbb{R}^n_+} \xi_*(G^{IJ}_r\sqrt{g} - G^{IJ}(0))\Psi_I \triangle W_J\, dx\\
			+\int\limits_{\mathbb{R}^n_+} G^{IJ}_{r-1} \widetilde \psi_I (\delta d W)_J\sqrt{g}\, dx 
			+\int\limits_{\mathbb{R}^n_+} G^{IJ}_{r} \widetilde \eta_I (d W)_J\sqrt{g}\, dx.
		\end{multline*}
		
		We further rewrite this as
		\begin{equation}\label{eq:form_Psi}
			\begin{aligned}
				\int\limits_{\mathbb{R}^n_+} \Psi_I \triangle W_I\, dx
				=\int\limits_{\mathbb{R}^n_+}  F^{I\beta} &W_{I,\beta}  dx + \int\limits_{\mathbb{R}^n_+}  C^{I\alpha\beta} W_{I,\alpha\beta} dx \\
				&+\int\limits_{\mathbb{R}^n_+}\xi_* \Psi_J(A^{IJ\alpha\beta}  W_{I,\alpha\beta} +B^{IJ\alpha} W_{I,\alpha} )\, dx,
			\end{aligned}
		\end{equation}
		where $A^{IJ\alpha \beta}$, $B^{IJ\alpha}$ satisfy \eqref{eq:Aest}, the coefficients $C^{I\alpha\beta}$, $F^{I\beta}$ vanish outside $G_{2R}$, and $C^{I\alpha\beta} \in W^{1,p(\cdot)}(G_{2R})$, $F^{I\beta} \in L^{p(\cdot)}(G_{2R})$.
		
		Then by Lemma~\ref{L:VT} and our choice \eqref{eq:W1} of $W_I$ we get
		\begin{align*}
			\Psi_I &= Q_\alpha^D[F^{I\alpha}+B^{IJ\alpha}\xi_*\Psi_J] + H_{\alpha\beta}^D[\xi_* A^{IJ\alpha\beta} \Omega_J + C^{I\alpha \beta}], \quad I \in \mathcal{I}_0,\\
			\Psi_I &= Q_\alpha^N[F^{I\alpha}+B^{IJ\alpha}\xi_*\Omega_J] + H_{\alpha\beta}^N[\xi_* A^{IJ\alpha\beta} \Psi_J + C^{I\alpha\beta}], \quad I \notin \mathcal{I}_0.
		\end{align*}
		We abbreviate this to 
		$$
		\Psi = \mathcal{Q}[F] +\mathcal{H}[C] + \mathcal{Q}[\xi_* B\Psi] + \mathcal{H}[\xi_* A \Psi].
		$$
		Note that by Lemmas~\ref{L:pa1}, \ref{L:DR040} there holds $\mathcal{Q}[F], \mathcal{H}(C) \in W^{1,p(\cdot)}(G_{2R})$. Using the contraction property of the operator $\mathcal{T}_0$ (see \eqref{eq:T0contr} with $p_{-}$ replaced by $p(\cdot)$) we get
		$$
		\Psi = \sum_{j=0}^\infty (\mathcal{T}_0)^j (Q[F] + \mathcal{H}(C))
		$$
		first in $L^{p(\cdot)}(G_{2R})$ and then in $W^{1,p(\cdot)}(G_{2R})$ (with the norm $\|\cdot\|^*_{1,p(\cdot),G_{2R}}$). Thus we get $\delta \Omega = \Psi \in W^{1,p(\cdot)}(M,\Lambda)$, which immediately implies $\delta G_D[\eta]=\delta \omega \in W^{1,p(\cdot)}(M,\Lambda)$.
		
		To prove that $d \omega \in W^{1,p(\cdot)}(M,\Lambda)$ we take first $\zeta = \delta_0 W$, which gives
		\begin{equation}\label{eq:p4}
		(d \Omega, d\delta_0 W) = - (\delta \Omega, \delta \delta_0 W)+ (\widetilde \varphi, d \delta_0 W)+ (\widetilde \psi, \delta \delta_0 W) + (\widetilde \eta, \delta_0 W).
		\end{equation}
		On the other hand, by the integration-by-parts formula,
		\begin{equation}\label{eq:p5}
		(d\Omega, \delta_0 d W)= (d\Omega, (\delta_0-\delta) dW) + (d \Omega_0,\delta d W), \quad \Omega_0 = \xi \omega_0 .
		\end{equation}
		Finally, combining \eqref{eq:p4}, \eqref{eq:p5}, and denoting $d\Omega = \Phi$, we get 
		\begin{equation}\label{eq:pp6}
        \begin{aligned}
			(\Phi, (d\delta_0 W + \delta_0 d W))&=  (\Psi, (\delta_0-\delta) \delta_0 W)+ (\Phi, (\delta_0-\delta) dW) + (\widetilde \varphi, d \delta_0 W) \\
			&\qquad+ (\widetilde \psi, (\delta-\delta_0) \delta_0 W)  + (\widetilde \eta, \delta_0 W) + (d\Omega_0, \delta d W).
		\end{aligned}
        \end{equation}
		In coordinates \eqref{eq:pp6} takes the form similar to \eqref{eq:form_Psi}:
		\begin{equation}\label{eq:form_Phi}
			\begin{aligned}
				\int\limits_{\mathbb{R}^n_+} \Phi_I \triangle W_I\, dx
				=\int\limits_{\mathbb{R}^n_+}  \widetilde F^{I\beta} &W_{I,\beta}  dx + \int\limits_{\mathbb{R}^n_+}  \widetilde C^{I\alpha\beta} W_{I,\alpha\beta} dx \\
				&+\int\limits_{\mathbb{R}^n_+}\xi_* \Phi_J(A^{IJ\alpha\beta}  W_{I,\alpha\beta} +B^{IJ\alpha} W_{I,\alpha} )\, dx,
			\end{aligned}
		\end{equation}
		where the coefficients $\widetilde F^{I\beta}$, $\widetilde C^{I\alpha \beta}$ have the same properties as $F^{I\beta}$, $C^{I\alpha \beta}$. The rest follows as above, and we get $d \Omega = \Phi \in W^{1,p(\cdot)}(M,\Lambda)$. 
		
		Therefore, $d\omega, \delta \omega \in W^{1,p(\cdot)}(M,\Lambda)$. Using the integration-by-parts formula
		we get 
        $$
        (\delta (d\omega) + d (\delta \omega) - \eta - d\psi - \delta \varphi,\zeta) =[\delta \omega-\psi,\zeta] 
        $$
        for all $\zeta \in \mathrm{Lip}_T(M,\Lambda)$ and thus
		\begin{equation}\label{eq:Laplace1}
			\delta (d\omega) + d (\delta \omega) = \eta+d\psi + \delta \varphi, \quad t(\delta \omega-\psi)=0.
		\end{equation}
		
		Now, by Lemma~\ref{L:ort} and \eqref{eq:Laplace1} we have
		$$
		(d(\delta \omega), d\zeta) = (\eta+d\psi + \delta \varphi - \delta(d\omega), d\zeta) = (\eta+d\psi,d\zeta)
		$$
		and
		$$
		(\delta(d\omega), \delta\zeta) = (\eta+d\psi + \delta \varphi - d (\delta \omega), \delta\eta) = (\eta+d\psi + \delta \varphi-d\psi, \delta\zeta) = (\eta +\delta \varphi, \delta \zeta).
		$$
		Since $d(d\omega)=0$, $\delta(\delta \omega)=0$, this proves \eqref{eq:rel0D}.
	\end{proof}

We state the following corollary for Dirichlet potentials. 
    	\begin{lemma}\label{L:addregD}
		If $\eta\in L^{p(\cdot)}(M,\Lambda)$ with $\mathcal{P}_T[\eta]=0$, then for $\alpha =\delta G_D[\eta]$ and $\beta = d G_D[\eta]$ we have $\alpha,\beta\in W_T^{1,p(\cdot)}(M,\Lambda)$ and they satisfy
		\begin{equation}\label{eq:rel1}
		d\alpha + \delta \beta= \eta
		\end{equation}
		and
		\begin{equation}\label{eq:rel2}
			\mathcal{D}(\alpha, \zeta) = (\eta, d\zeta), \quad \mathcal{D}(\beta, \zeta) = (\eta, \delta\zeta)
		\end{equation}
		for all $\zeta \in \mathrm{Lip}_T(M,\Lambda)$. As a corollary, if additionally $\eta \in W^{\delta,p_{-}} (M,\Lambda)$ then $\delta G_D[\eta] = G_D[\delta \eta]$ and if additionally $\eta \in W^{d,p_{-}}_T(M,\Lambda)$ then $d G_D[\eta] = G_D[d\eta]$.
	\end{lemma}
\begin{proof} 
From Theorem~\ref{T:addregD} it follows that $\alpha,\beta \in W_T^{1,p(\cdot)}(M,\Lambda)$ and satisfy \eqref{eq:rel1},\eqref{eq:rel2}. It remains to see that $\mathcal{P}_T\alpha, \mathcal{P}_T\beta=0$, the first relation in \eqref{eq:rel2} coincides with the relation defining the Dirichlet potential of $\delta \eta$ and the second relation in \eqref{eq:rel2} coincides with the relation defining the Dirichlet potential of $d\eta$ if $t\eta=0$. 
\end{proof}

The following result and its corollary for Neumann potentials is obtained in exactly the same manner or by the Hodge duality 

\begin{theorem}\label{T:addregN}
Let $\omega_0,\delta \omega_0 \in W^{1,p(\cdot)}(M,\Lambda)$, $\omega\in \omega_0+ W_N^{1,p(\cdot)}(M,\Lambda)$, $\eta \in L^{p(\cdot)}(M,\Lambda)$ and $\varphi,\psi \in W^{1,p(\cdot)}(M,\Lambda)$ satisfy \eqref{eq1}, that is
$$
\mathcal{D}(\omega,\zeta) = (\eta,\zeta) + (\varphi,d\zeta) + (\psi,\delta \zeta),
$$
for all $\zeta \in \mathrm{Lip}_N(M,\Lambda)$. Then $\alpha =\delta \omega$ and $\beta = d\omega$ belong to $W^{1,p(\cdot)}(M,\Lambda)$ and satisfy $n\alpha =n \delta \omega_0$, $n\beta = n\varphi$, \eqref{eq:rel0Da}, and \eqref{eq:rel0D} for all $\zeta \in \mathrm{Lip}_N(M,\Lambda)$. 
\end{theorem}

In this case $I \in\mathcal{I_0}$ iff $n\in I$ and we take $W\in C_0^\infty(\overline{\mathbb{R}^n_{+}})$ such that
\begin{equation}\label{eq:W2}
W_I(x',0) =0\quad \text{if}\quad n\in I, \quad \frac{\partial W_I}{\partial x^n}(x',0)\quad \text{if}\quad n\notin I.
\end{equation}
Then $n d W =0$, $n \delta_0 W=0$. The relation \eqref{eq:pp1} takes exactly the same form, the relation \eqref{eq:pp2} becomes
$$
(\delta \Omega, d\delta_0 W) = (\delta \Omega_0, d\delta_0 W), \quad \Omega_0 = \xi \omega_0,
$$
and thus instead of \eqref{eq:pp3} we arrive at (denoting $\Psi = \delta \Omega$)
\begin{equation}\label{eq:pp7}
(\Psi, \delta d W + d \delta_0 W)=(\widetilde \psi, \delta dW)+(\widetilde \eta,dW) + (\delta \omega_0, d\delta_0 W). 
\end{equation}
This yields $\Psi \in W^{1,p(\cdot)}(\Omega,\Lambda)$. Now, for $\Phi = d\Omega$ instead of \eqref{eq:p5} we have 
$$
(d\Omega, \delta_0 d W) = (d\Omega, (\delta_0-\delta) d W)
$$ 
and \eqref{eq:pp6} is replaced by 
\begin{multline*}
(\Phi, (d\delta_0 + \delta_0 d)W) 
= - (\Psi, (\delta-\delta_0) \delta_0 W)+ (\widetilde \varphi, d \delta_0 W) \\ + (\widetilde \psi, (\delta-\delta_0) \delta_0 W) + (\widetilde \eta, \delta_0 W).
\end{multline*}
This relation yields $\Phi \in W^{1,p(\cdot)}(\Omega,\Lambda)$.

	\begin{lemma}\label{L:addregN}
		If $\eta\in L^{p(\cdot)}(M,\Lambda)$ with $\mathcal{P}_N[\eta]=0$, then for $\alpha = \delta G_N[\eta]$ and $\beta =  d G_N[\eta]$ we have $\alpha,\beta\in W_N^{1,p(\cdot)}(M,\Lambda)$ and they satisfy \eqref{eq:rel1} and \eqref{eq:rel2} for all $\zeta \in \mathrm{Lip}_N(M,\Lambda)$. As a corollary, if additionally $\eta \in W^{d,p_{-}}(M,\Lambda)$ then $d G_N[\eta] = G_N[d \eta]$ and if additionally $\eta \in W^{\delta,p_{-}}_N(M,\Lambda)$ then $\delta G_N[\eta] = G_N[\delta\eta]$.
	\end{lemma}

For the problem with natural boundary conditions we have the following result:
\begin{theorem}\label{L:addregNatur}
Let $\omega\in  W^{1,p(\cdot)}(M,\Lambda)$, $\eta \in L^{p(\cdot)}(M,\Lambda)$ , $\varphi,\psi \in W^{1,p(\cdot)}(M,\Lambda)$ satisfy \eqref{eq1} for all $\zeta \in \mathrm{Lip}(M,\Lambda)$. Then $\alpha =\delta \omega$ and $\beta = d\omega$ both belong to $W^{1,p(\cdot)}(M,\Lambda)$ and satisfy $t\alpha =t \psi$, $n\beta = n\varphi$, \eqref{eq:rel0Da}, and \eqref{eq:rel0D} for all $\zeta \in \mathrm{Lip}(M,\Lambda)$. 
\end{theorem}
In this case to establish the regularity of $\delta \omega$ we repeat the first part of the proof of Theorem~\ref{T:addregD} and use the set of test forms described in \eqref{eq:W1}, which gives \eqref{eq:pp3}, and finally leads to $\delta \omega \in W^{1,p(\cdot)}(M,\Lambda)$. To establish the regularity if $d\omega$ we use the set of test forms \eqref{eq:W2} corresponding to the proof of Theorem~\ref{T:addregN} which leads to \eqref{eq:pp6} and by the same arguments this yields $d\omega \in W^{1,p(\cdot)}(M,\Lambda)$. The rest follows from integration-by-parts.

While formally the set of test forms in relations \eqref{eq:rel0D} for $\alpha$ and $\beta$ in Theorem~\ref{L:addregNatur} is free of any boundary values, in fact the relation for $\alpha$ essentially requires only test forms with vanishing tangential part and in the relation for $\beta$ one can use only test forms with vanishing normal part. Indeed, if the first of the relations in \eqref{eq:rel0D} holds for test forms with vanishing tangential component then replacing $\zeta$ by $\zeta-d\gamma$ with $nd\gamma = n\zeta$,  $t\gamma=0$ and $n\gamma=0$ (see Corollary~\ref{C:ext}) we immediately see that it holds for all sufficiently regular test forms without any assumption on the boundary values. The same for the second relation in \eqref{eq:rel0D}, here one can take $\zeta - \delta \gamma$ instead of $\zeta$, with $t\delta \gamma = t\zeta$, $t\gamma=0$, $n\gamma=0$.

In case of the full Dirichlet data the potentials $\alpha, \beta \in W^{1,p(\cdot)}_{\mathrm{loc}}(M,\Lambda)$ and satisfy \eqref{eq:rel0Da}. But for the boundary regularity we do not have results similar to the above. To see what the complication here is, assume that $\omega_0, d\omega_0, \delta \omega_0\in W^{1,p(\cdot)}(M,\Lambda)$, $\omega\in \omega_0+ W_0^{1,p(\cdot)}(M,\Lambda)$,  $\eta \in L^{p(\cdot)}(M,\Lambda)$ and $\varphi,\psi \in W^{1,p(\cdot)}(M,\Lambda)$ satisfy \eqref{eq1} for all $\zeta \in \mathrm{Lip}_0(M,\Lambda)$. If these potentials had the required regularity $\alpha, \beta \in W^{1,p(\cdot)}(M,\Lambda)$ (this is definitely so if $M$ is $C^{2,1}$), we would clearly have $n\alpha =n \delta \omega_0$, $t\beta = td\omega_0$, and 
\begin{gather*}
\mathcal{D}(\alpha,\zeta) = (\eta + d\psi+\delta \varphi-\delta \beta,d\zeta)\quad \text{for all} \quad \zeta \in \mathrm{Lip}(M,\Lambda),\\
\mathcal{D}(\beta,\zeta) = (\eta + d\psi+\delta \varphi-d\alpha,\delta\zeta)\quad \text{for all} \quad \zeta \in \mathrm{Lip}(M,\Lambda).
\end{gather*}
In particular, we would have
\begin{gather*}
\alpha = \delta \omega_0 + \delta G_N[\eta - \mathcal{P}_N\eta+ d (\psi - \delta \omega_0) + \delta (\varphi-\beta)],\\
\beta = d\omega_0 + d G_D[\eta -\mathcal{P}_T \eta + d(\psi-\alpha)+ \delta (\varphi-d\omega_0)].
\end{gather*}
Unlike the relations for the Dirichlet and Neumann potentials, here the right-hand side contains $d\alpha$ and $\delta \beta$.  Thus in this case we state only an interior regularity result.
\begin{theorem}\label{T:addregFD}
Let $\omega\in  W^{1,p(\cdot)}(M,\Lambda)$, $\eta \in L^{p(\cdot)}(M,\Lambda)$ and $\varphi,\psi \in W^{1,p(\cdot)}(M,\Lambda)$ satisfy \eqref{eq1} for all $\zeta \in \mathrm{Lip}_0(M,\Lambda)$. Then $\alpha =\delta \omega$ and $\beta = d\omega$ belong to $W^{1,p(\cdot)}_{\mathrm{loc}}(M,\Lambda)$ and satisfy $d\alpha + \delta \beta = \beta + d\psi + \delta \varphi$ a.e. in $M$.
\end{theorem}

\backmatter

\bibliographystyle{amsalpha}
\bibliography{LG_diff}
    

\printindex
\end{document}